\documentclass[11pt]{amsart}
\usepackage[utf8]{inputenc}
\usepackage[T1]{fontenc}
\usepackage{amsfonts}
\usepackage{amsmath,amscd}
\usepackage{amsthm}
\usepackage{latexsym}
\usepackage{amssymb}
\usepackage{enumerate}
\usepackage{url}
\markleft{\hfill G. Raposo\hfill }
\usepackage{hyperref}

\newtheorem{theorem}{Theorem}[section]
\newtheorem*{theorem*}{Theorem}
\newtheorem{proposition}[theorem]{Proposition}
\newtheorem{corollary}[theorem]{Corollary}
\newtheorem{definition}[theorem]{Definition}
\newtheorem{example}[theorem]{Example}
\newtheorem{lemma}[theorem]{Lemma}
\newtheorem{remark}[theorem]{Remark}
\newtheorem*{claim*}{Claim}
\usepackage{tikz-cd}

\newtheorem*{stheorem}{Sample Theorem}

\usepackage{mathdots}

\usepackage{geometry}
\usepackage{tikz}
\usetikzlibrary{arrows,matrix,positioning}

\usepackage{caption}
\usepackage{subcaption}
\usepackage{graphicx}

\numberwithin{equation}{section}

\newcommand{\R}{\mathbb{R}} 
\newcommand{\C}{\mathbb{C}} 
\newcommand{\N}{\mathbb{N}} 
\newcommand{\Prob}{\mathbb{P}} 
\newcommand{\E}{\mathbb{E}} 

\newcommand{\Per}{\textup{Per}} 

\newcommand{\ff}{\downharpoonright} 
\newcommand{\rf}{\upharpoonright} 

\title{Global fluctuations for standard Young tableaux}
\author{Gabriel Raposo}
\address{Department of Statistics, University of California-Berkeley, \newline 367 Evans Hall, Berkeley, CA 94720, USA}
\email{raposo@berkeley.edu}

\begin{document}

\maketitle

\begin{abstract}
We introduce the notion of a Young generating function for a probability measure on integer partitions. We use this object to characterize probability distributions over integer partitions satisfying a law of large numbers and those that satisfy a central limit theorem. We further establish a multilevel central limit theorem, which enables the study of random standard Young tableaux. As applications of these results, we describe the fluctuations of height functions associated with (i) the Plancherel growth process, (ii) random standard Young tableaux of fixed shape, and (iii) probability distributions induced by extreme characters of the infinite symmetric group $S_\infty$. In all cases, we identify the limiting fluctuations as a conditioned Gaussian Free Field.  
\end{abstract}

\tableofcontents

\section{Introduction}

\subsection{Overview}

The study of random integer partitions is an extensive topic deeply intertwined with probability theory, representation theory, algebraic combinatorics, and statistical mechanics. Starting in 1977, Logan and Shepp \cite{LS}, and Vershik and Kerov \cite{VK}, obtained limit shape results for partitions under the Plancherel distribution, resolving Ulam’s problem on the asymptotic behavior of the expected length of the longest increasing subsequence in uniformly random permutations. Subsequent advances, notably the work of Baik, Deift, and Johansson \cite{BDJ}, who proved that the fluctuations of the longest row in Plancherel distributed partitions converge to the Tracy--Widom distribution, and later breakthroughs by Johansson, Borodin--Okounkov--Olshanski, and Okounkov \cite{Jo,BOkOl,O} who established the convergence of the fluctuations for the first rows to the Airy ensemble, revealed profound connections to random matrix theory. These developments have driven a rich interplay between combinatorial probability and integrable systems. While a full survey of the literature lies beyond our present scope, we refer the reader to Romik’s comprehensive book \cite{Rom} for an in-depth treatment of the subject. 

In recent years, this field has expanded significantly. Besides the Plancherel measure, the limiting behavior of random partitions is being studied for diverse families of distributions, including Schur--Weyl measures \cite{Sni, BO07, Mel11}, Gelfand measure \cite{Mel11}, $q$-Plancherel measures \cite{St08, FM12}, Jack--Plancherel measures \cite{Ma08, DF, DS19, GH}, Jack measures \cite{Mo23}, Plancherel--Hurwitz measures \cite{CLW24}, and Jack--Thoma measures \cite{CDM}. Concurrently, substantial progress has been made in analyzing random standard Young tableaux, which can be understood as random sequences of integer partitions, with a major focus on characterizing their limiting surfaces  \cite{Su,Gor,KP22,Pr,BBFM}. Describing the fluctuations of these models is a central challenge, which remains unaddressed in the case of random standard Young tableaux. 

The aim of this paper is to fill this gap by providing robust methods to study global fluctuations of random partitions and random standard Young tableaux for a broad variety of distributions. To formalize this analysis, we represent partitions graphically via Young diagrams in Russian notation. As such we embed the Young diagram from the lower right quadrant of $\R^2$ with coordinates $r$, $s$, denoting row and column of the partition, respectively, into the upper half plane with coordinates $x = s - r$ and $ y = r + s$. The boundary of the shape can be viewed as the graph of a continuous piecewise linear function (see Figure \ref{RussianNotation}). This rotated coordinate system both simplifies the analysis and aligns with the convention in previous literature.  

\begin{figure}[h]
    \centering
    \includegraphics[scale=0.35]{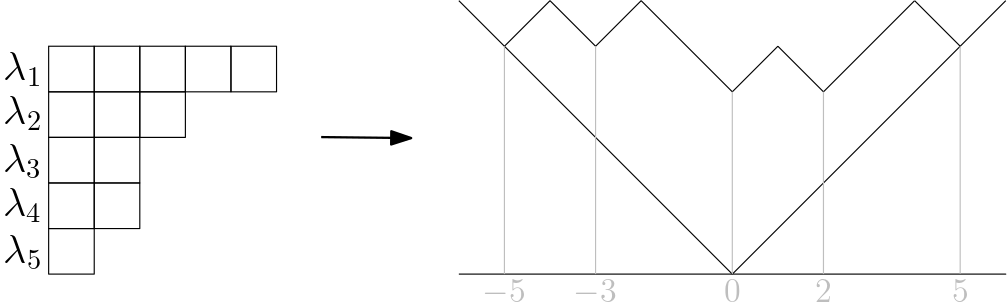}
    \caption{Young Diagram of shape $\lambda=(5,3,2,2,1)$ in Russian notation with its minima (-5,-3,0,2,5).}
    \label{RussianNotation}
\end{figure}

When a sequence of random partitions converges to a deterministic limit shape, we observe a Law of Large Numbers (LLN) phenomenon. Likewise, rescaling the fluctuations around this limit yields a Gaussian process, corresponding to a Central Limit Theorem (CLT) phenomenon. See \cite{Ve94} for a review of the LLN phenomenon under the Plancherel distribution and uniform distribution among others, and Kerov's CLT \cite{IO} for the CLT phenomenon in the Plancherel case. Building on these notions, we introduce multilevel analogs. Given two partitions $\mu$ and $\lambda$, we write $\mu \nearrow \lambda$ if $\lambda$ is obtained from $\mu$ by adding a single square box at a corner of $\mu$. For a fixed partition $\lambda$, consider a uniformly random sequence of increasing partitions $\emptyset \nearrow\lambda^1\nearrow\cdots\nearrow\lambda$, equivalently a uniformly random standard Young tableau of shape $\lambda$. In the limit, this process converges to a deterministic surface, a Multilevel Law of Large Numbers, while its two-dimensional fluctuations correspond to a Multilevel Central Limit Theorem.

In this article, we introduce the Young generating function, an object which plays the role that the characteristic function has in classical probability. While its rigorous definition is deferred to the next section, we informally outline its role: it enables a unified framework to characterize sequences of probability measures $\rho_n$ in $\mathbb{Y}_n$, the set of integer partitions of $n$, that exhibit LLN or CLT behavior. Specifically, we identify necessary and sufficient conditions for $\rho_n$ to satisfy a LLN as $n\to \infty$, and establish parallel criteria for $\rho_n$ to satisfy a CLT as $n\to \infty$. We extend these characterizations to multilevel ensembles, which allow the study of random increasing sequences of partitions and in particular of random standard Young tableaux, as such we obtain a Multilevel LLN and a Multilevel CLT. See section \ref{SectionMainResults} for formal statements of these results. 

Furthermore, for the distributions satisfying a CLT, we obtain explicit covariance formulas, which allow us to characterize the two-dimensional Gaussian fields associated with a variety of models. Three new applications are of particular interest, $(1)$ The Plancherel growth process, $(2)$ Random standard Young tableaux of fixed shape and $(3)$ Distributions induced by restrictions to $S_n$, the symmetric group over $n$ elements, of extreme characters of the infinite symmetric group $S_\infty$. In striking contrast to prior literature, we establish that fluctuations in these models are not governed by a Gaussian Free Field (GFF), but by a conditioned variant of the GFF. This distinction reveals intrinsic limitations in the analogies between random matrix theory and random partitions--analogies rigorously developed by Biane, Kerov, and Okounkov \cite{Bi,Ke,O}, among others. See section \ref{SectionApplications} for formal statements of these applications. 

To illustrate what our applications look like we present an informal version of our result concerning the Plancherel growth process. Denote by $\dim(\lambda)$ the number of standard Young tableaux of shape $\lambda$ and by $|\lambda|$ the size of the partition $\lambda$. The \textbf{Plancherel growth process} is a Markov process on the set of integer partitions $\mathbb{Y}$ with initial state being the empty partition $\emptyset$ at time $0$ and with transition probabilities given by
\begin{equation*}
    p(\mu,\lambda):=\begin{cases}
        \frac{\dim(\mu)}{|\mu|\dim(\lambda)} & \textup{ if } \mu \nearrow \lambda,\\\
        0 &\, \textup{otherwise.}
    \end{cases}
\end{equation*}

Notably, the distribution at time $n$ is precisely the Plancherel measure $\Prob(\lambda)=\frac{\dim(\lambda)^2}{n!}$.

\begin{stheorem}[Theorem \ref{ThmCGFFforPlancherel}]
Denote $\textup{H}(x,t)$ the length of the $x$th diagonal at time $t$ of the Plancherel growth process, then
    \begin{equation*}
    \sqrt{\pi}\big[\textup{H}(\sqrt{n}x,nt)-\E\textup{H}(\sqrt{n}x,nt)\big]\to \mathfrak{C}(x,t), \hspace{1mm} \text{ as }  \hspace{1mm} n\to \infty.
    \end{equation*}
Here $\mathfrak{C}$ denotes the conditioned Gaussian Free Field, which is defined in more detail in section \ref{SubsectionGaussianFields}. 
\end{stheorem}
This result substantially generalizes Kerov's CLT \cite{IO}: whereas Kerov’s CLT describes fluctuations of the Plancherel measure at a fixed time $t$, our theorem establishes the joint asymptotic behavior of the entire process $\textup{H}(\sqrt{n}x, nt)$ as $n\to \infty$. The resulting limiting fluctuations are described by the Gaussian field $\mathfrak{C}$, which encodes both spatial and temporal correlations.

Our proofs are based on the moment method, where we use the representation theory of the symmetric group $S_n$ to extract combinatorial data. More specifically, since the irreducible representations of $S_n$ are indexed by integer partitions $\mathbb{Y}_n$, we can construct families of central elements of $\C[S_n]$ whose action on irreducible characters allows us to recover the moments of certain probability distributions associated with integer partitions. This idea was first developed by Biane \cite{Bi} to study the LLN phenomenon and later expanded by Śniady to study the CLT phenomenon \cite{Sni3, Sni}. The idea of using operators acting on representation-theoretic objects to extract combinatorial data associated with a random process has been very fruitful, leading to new and interesting phenomena beyond the reach of other techniques. A prominent example is the work of Borodin and Corwin, who introduced Macdonald processes and studied them by using differential operators acting on Macdonald symmetric functions \cite{BC}. More recently, Bufetov and Gorin, in a series of papers \cite{BuG1,BuG2,BuG3}, employed differential operators acting on Schur symmetric functions to analyze discrete particle configurations. 

Our work builds on these advances and contributes several novel ideas. The primary technical challenges arise in the multilevel setting, where developing tractable expansion formulas for the operators demands both a level of generality and precision beyond what is available in the literature. Unlike prior approaches restricted to the center of the symmetric group algebra, we develop new techniques within the Gelfand--Tsetlin algebra of the symmetric group. Working within this richer algebraic structure introduces unexpected combinatorial obstacles that require delicate arguments---including the introduction of generalized falling factorials, and new summation formulas over set partition lattices---to rigorously describe the operator expansion formulas.

In light of these difficulties, we reconstruct the theory from scratch. Consequently, instead of relying on previous intricate arguments based on free probability, the semigroup of partial permutations, or genus expansions for characters, we propose a new perspective purely based on elementary combinatorial arguments. As a benefit, this point of view enables particularly transparent explanations that yield explicit formulas, which, in turn, offers new combinatorial proofs of the previously known operator expansions. These ideas are developed in section \ref{SectionTechnicalLemmasProofs}.

In addition to the main results and applications discussed above, other applications of interest contained in this text are the following: (1) Other representation-theoretic induced distributions on integer partitions that fall within our framework are the Gelfand distribution (see Example \ref{ExampleGelfandDist} and Example \ref{ExampleCGFFGelfand}) and the Schur--Weyl distributions (see Example \ref{ExampleSchurWeylDist}). For both models we obtain conditioned GFF fluctuations. (2) We describe the limiting shape of sublinear random standard Young tableaux, showing that after rescaling it converge to the celebrated Vershik--Kerov--Logan--Shepp curve (see Example \ref{TheoremSublinearn}). Furthermore, we also describe the fluctuations of the limiting shape in this regime. (3) In 2004, Pittel and Romik formulated a conjecture regarding the fluctuations of the probability distribution for the location of the box containing $n$ in a random standard Young tableau (see \cite[section 6]{PiR}). By examining different descriptions of Young diagrams, we prove central limit theorems for various statistics associated with integer partitions. As a consequence, we resolve this conjecture (see section \ref{SubsectionCoordinateSystems}).

Another combinatorial result of independent interest is a new elementary proof of a family of polynomial identities, originally due to Rosas \cite{Ros}, that characterize the expansion of products of basis elements in the polynomial ring with respect to both falling and rising factorial bases (see section \ref{CombRosas}). We further establish multivariate generalizations of these identities and use them to derive new formulas that extend the Möbius inversion formula on the set partition lattice.

\subsection{Previous work on the subject} 
The generality of our Young generating function approach offers a distinct advantage over prior techniques: it systematically unifies diverse models while enabling access to their two-dimensional fluctuations, a regime previously inaccessible. We contextualize this advance through comparison with prior literature, noting that this overview is by no means exhaustive.
\begin{itemize}
    \item \textbf{Free probability and representation theory of $S_n$.} Following the limit shape result for the Plancherel distribution, the first result regarding global fluctuations is the celebrated Kerov's CLT for the Plancherel measure, proven by Ivanov and Olshanski \cite{IO}, based on previous work by Kerov. In a more general setting, Biane studied probability measures satisfying a LLN in \cite{Bi} and in \cite{Bi2}, our characterization of the LLN (see Theorem \ref{TheoremLLN}) can be understood as a restatement of these results. Regarding global fluctuations, Śniady introduced a criteria, the asymptotic factorization property, which allows to establish Gaussian fluctuation \cite{Sni} for some models of interest, including the Plancherel and Schur--Weyl measures. A limitation of this criteria is that it isn't exhaustive as it doesn't include some distributions such as Gelfand measures as shown by Méliot \cite{Mel11}. Notably, the two-dimensional fluctuation regime\,---\,the primary innovation of our work\,---\,lies beyond the scope of these previous results.
    \item \textbf{Differential operators for global asymptotics.} In \cite{BuG1,BuG2,BuG3}, the authors introduced Schur generating functions and use it to characterize the CLT for discrete particle configurations. Similar results were later obtained by Huang who introduced Jack generating functions \cite{Hu21}. The Young generating function can be understood as the integer partition version of these notions, while our main results are analogues of the ones obtained in those papers. In representation-theoretic terms, our work establishes symmetric group analogs of unitary group results developed in the cited literature. See also \cite{CDM} where the single-level LLN and CLT have been studied for random partitions under Jack--Thoma measures.
    
    \item \textbf{Determinantal Point Process.} Determinantal point processes have played a central role in the study of random partitions, starting with the seminal work of Borodin, Okounkov and Olshanski \cite{BOkOl} who showed that Poissonized random Plancherel partitions converges to a determinantal point process, this set of tools has expanded an applied in different settings. Recently, Gorin and Rahman \cite{GR} introduced a determinantal point process for a Poissonized version of random Young tableaux, this was used in \cite{BBFM} to find limiting surfaces for random standard Young tableaux of fixed shape. Notably, Borodin--Ferrari \cite{BF} and Petrov \cite{Pe} used determinantal point processes to obtain Gaussian Free field fluctuations on different settings. These techniques, however, appear inapplicable to our framework, the absence of Poissonization in our setup prevents the existence of an underlying determinantal point process. 
    
    \item \textbf{Variational methods.} After the earlier results of Biane \cite{Bi} in 1998 who established the existence of limiting surfaces for random standard Young tableaux, the problem of describing these limiting surfaces has gained a lot of attention. The use of variational and entropy optimization methods has been used since the work of Pittel and Romik \cite{PiR} who found explicit formulas for limiting surface induced from uniform random Young tableaux of rectangular shapes. Recently, considerable effort has been devoted to describe these limiting surfaces for more general shapes, see for example \cite{Su,Gor,KP22,Pr}. While these methods can be used to predict fluctuations, see for example \cite[Lecture 12]{Gorin}, so far no progress has been made to make that approach rigorous. 

    \item \textbf{Discrete loop equations.} This adapted version of the Dyson--Schwinger equations, already used in random matrix theory, was first used in \cite{BGG} to obtain CLT statements for discrete log-gases. These ideas were later used in \cite{GH} to characterize the edge universality of discrete $\beta$-ensembles, in particular they characterize the fluctuations of a Jack deformation of the Plancherel measure to be a Tracy--Widom $\beta$ distribution. Recently, in \cite{GorinHuang}, the dynamical loop equation is introduced, using this Gorin and Huang obtain Gaussian field type fluctuations for multiple models. The applicability of this method to the setting of the present paper is unclear.
\end{itemize}

\subsection{Organization of the article}
The rest of the text is organized as follows. We introduce the Young generating function and state our main results concerning the law of large numbers, central limit theorem and multilevel central limit theorem in section \ref{SectionMainResults}. Applications of these results for the Plancherel growth process, random standard Young tableaux of fixed shape and extreme characters of the infinite symmetric group $S_\infty$ are stated in section \ref{SectionApplications}. We introduce the operators and state the formulas concerning the expansion of these operators in the Gelfand--Tsetlin algebra in section \ref{SectionTechnicalLemmasProofs}. Proof of the main results concerning the LLN behavior is contained in sections \ref{SectionProofsLLN}, while proofs concerning the CLT behavior are contained in section \ref{SectionProofsCLT}. Finally, proofs of the applications are contained in section \ref{SectionProofsOfApplications}.  

\subsection{Acknowledgments}
I am grateful to my advisor Vadim Gorin for suggesting me this problem and the very valuable discussions and corrections throughout all this work. The project was partially supported by NSF grant DMS -- 2246449.

\section{Main results}\label{SectionMainResults}

\subsection{The Young generating function}

We are interested in studying random distributions $\rho$ on $\mathbb{Y}_n$, to this end we will introduce the Young generating function for the distribution $\rho$ on integer partitions. Given a tower of finite symmetric groups $(S_n)_{n \in \N}$ with their natural embeddings $S_n \hookrightarrow S_{n+1}$, the infinite symmetric group is defined to be $S_{\infty}:=\varinjlim S_n$. We start by showing that each probability measure $\rho$ on $\mathbb{Y}_n$ arises a central function on the infinite symmetric group. Let $\chi$ be a character of the symmetric group $S_n$, we can extend $\chi$ to be defined on $S_\infty$ as follows, if $\sigma \in S_\infty$ is conjugate to some $\tau \in S_n$, then $\chi(\sigma):=\chi(\tau)$, otherwise $\chi(\sigma):=0$. Since characters are invariant under conjugation, this extension is well defined.

\begin{definition}
Let $\mathbb{Y}_n$ the set of integer partitions of size $n$. For a probability measure $\rho$ on $\mathbb{Y}_n$ define its associated character to be $M_{\rho}:S_{\infty}\to \R$ as follows,
\begin{equation*}
M_{\rho}(\cdot):=\sum_{\lambda \in \mathbb{Y}_n} \rho(\lambda) \frac{\chi_\lambda(\cdot)}{\dim(\lambda)}
\end{equation*}
here we denote by $\chi_\lambda$ the irreducible character corresponding to the partition $\lambda$ and we denote by $\dim(\lambda)$ the dimension of the corresponding irreducible representation indexed by $\lambda$. 
\end{definition}

Let $R$ be a ring, denote by $R_n=R[[x_1,\dots,x_n]]$ the ring of formal power series over the variables $x_1,\dots,x_n$ with coefficients on $R$. For each $i\leq j$ denote by $\pi_i$ the projection $\pi_i:R_j\to R_i$ defined as the identity over $R_i \subseteq R_j$ and linearly extended to $R_j$ such that $\pi_i(x_s)=0$ if $s>i$. We are interested in the inverse limit
\begin{equation*}
\varprojlim R_n :=\{\vec{r}\in \prod_{n\in\N} R_n: \pi_i(r_j)=r_i \textup{ for all } i\leq j\}.
\end{equation*}

We will focus on two cases, the case in which $R=\R$ is just the field of real numbers and the case in which $R=\R[S_\infty]$ is the group ring of $S_\infty$, that is, the ring of formal finite linear combinations of elements of $S_\infty$ over the real numbers with product induced by the group product of $S_\infty$. In this situation we will denote
\begin{equation*}
\R[S_\infty][\vec{x}]=\varprojlim_{n \in \N} \R[S_\infty][[x_1,\dots,x_n]]\hspace{2mm} \text{ and }\hspace{2mm}\R[\vec{x}]=\varprojlim_{n \in \N} \R[[x_1,\dots,x_n]].
\end{equation*}

Elements of $\R[\vec{x}]$ are (infinite) linear combinations of finite degree monomials in $x_i$ with no convergence requirement. Elements of $\R[S_\infty][\vec{x}]$ are (infinite) linear combinations of formal products of elements $g\in S_\infty$ and finite degree monomials in $x_i$, with the additional restriction that for each finite degree monomial, only finitely many distinct $g\in S_\infty$ appear with non-zero coefficient. 

Notice that for each $n\in \N$, the character function $M_{\rho}$ can be extended as a linear transformation $M_\rho: \R[S_\infty][[x_1,\dots,x_n]]\to\R[[x_1,\dots,x_n]]$. Moreover, for each $i\leq j$, the following diagram is commutative
\begin{center}
\begin{tikzcd}
\R[S_\infty][[x_1,\dots,x_j]] \arrow{r}{M_\rho} \arrow[swap]{d}{\pi_i} & \R[[x_1,\dots,x_j]] \arrow{d}{\pi_i} \\
\R[S_\infty][[x_1,\dots,x_i]] \arrow{r}{M_{\rho}} & \R[[x_1,\dots,x_i]]
\end{tikzcd}
\end{center}
This guarantees that $M_\rho$ can be extended to a linear transformation $M_{\rho}:\R[S_\infty][\vec{x}]\to \R[\vec{x}]$. We can now define the Young generating function to be the image of a carefully constructed element of $\R[S_\infty][\vec{x}]$ under $M_\rho$. We use the following notation, for each partition $\lambda\in \mathbb{Y}$, denote by $\sigma[\lambda]$ an element in $S_\infty$ with cycles of length indexed by the rows of $\lambda$. 

\begin{definition}
Let $\rho$ be a probability distribution over $\mathbb{Y}_n$ and let $M_\rho: \R[S_\infty][\vec{x}] \to \R[\vec{x}]$ be its associated character. Let $\{\sigma[(k)^i]\}_{i\geq 1,k \geq 1}$ be an infinite family of permutations in $S_\infty$ with disjoint support such that $\sigma[(k)^i]$ is a product of $i$ cycles of length $k$ and
\begin{equation*}
U_\infty=\prod_{k=1}^\infty \Big(1+\sum_{i=1}^\infty n^{\frac{i(k-1)}{2}} \sigma[(k)^i] \frac{x_k^{i}}{i!}\Big) \in \R[S_\infty][\vec{x}].
\end{equation*}
Define the \textbf{Young generating function} of $\rho$ to be
\begin{equation*}
\textup{A}_{\rho}(x_1,x_2,\dots) := M_{\rho}(U_\infty) \in \R[\vec{x}].
\end{equation*}
\end{definition}

While an infinite product is, in principle, not well defined on $\R[S_\infty][\vec{x}]$ it can be formally defined to be the inverse limit of the products truncated up to $m$, that is, for $U_m = \prod_{k=1}^m \big(1+\sum_{i=1}^\infty n^{\frac{i(k-1)}{2}} \sigma[(k)^i] \frac{x_k^{i}}{i!}\big) \in \R[S_\infty][x_1,\dots,x_k]$, since for each $i\leq j$ we have $\pi_i(U_j)=U_i$, hence the element $U_\infty = (U_m)_{m\in \N} \in \R[S_\infty][\vec{x}]$ is well defined.

\begin{remark}
Note that the element $U_\infty$ is not unique and depends on the choice of the infinite family of permutation $\{\sigma[(k)^i]\}_{i\geq 1,k \geq 1}$, however since $M_\rho$ is invariant under conjugation, the image of $U_\infty$ under $M_\rho$ is uniquely determined. This guarantees that the Young generating function $\textup{A}_{\rho}$ is uniquely determined by $\rho$. 
\end{remark}

\begin{remark}
One may like to push further the similitude between the elements inside the product in $U_\infty$ and exponential functions. We can define a monoid of partitions where the product is given by the disjoint union and identity element is the empty partition, equivalently we can construct a monoid over the conjugacy classes of elements of $S_\infty$ with product being the disjoint product. Then, instead of working with the group ring of $S_\infty$ we could consider the monoid ring of partitions. In this setting, for $\exp(x)=\sum_{i=0} \frac{x^i}{i!}$ we can define $U_m=\prod_{k=1}^m\exp\big( (k) \cdot x_k\big)$ where $(k)$ denotes the partition with a single row of length $k$. This observation further establishes the similitude with the characteristic function of the classical probability setting and justify the fact that these objects only allow us to study partitions rather than the infinite symmetric group $S_\infty$. 
\end{remark}

At this point the reader may be worried that the calculation of $\textup{A}_{\rho}$ might be difficult. We will provide some examples to show that $\textup{A}_{\rho}$ is well behaved.

\begin{example}\label{ExampleArhoPlancherel}
We compute the Young generating function of the Plancherel distribution $\rho$ over $\mathbb{Y}_n$. That is for each $\lambda \in \mathbb{Y}_n$, $\rho(\lambda)=\frac{\dim(\lambda)^2}{n!}$. This distribution is induced by the regular representation of $S_n$, hence $M_{\rho}$ corresponds to the trivial character, that is $M_{\rho}(e)=1$ and $M_{\rho}(\sigma)=0$ for each $\sigma \neq e$ on $S_\infty$. It is straightforward that
\begin{equation*}
    \textup{A}_{\rho_n}(x_1,x_2,\dots)= \sum_{i=0}^\infty \frac{x_1^i}{i!}\in \R[\vec{x}].
\end{equation*}
\end{example}

\begin{example}
We can also consider the distribution $\rho_n$ associated with the natural permutation representation of the symmetric group, where $S_n$ acts on a $n$-dimensional space by permuting coordinates. In this setting we have that $\rho_n\big((1)^n\big)=\tfrac{1}{n}$ while $\rho_n\big((n-1,1)\big)=\tfrac{n-1}{n}$. Additionally, $M_{\rho_n}(\sigma)=\frac{\#\{\textup{Fixed points of } \sigma\}}{n}$ for $\sigma \in S_n$. A direct computation gives
\begin{equation*}
    \textup{A}_{\rho_n}(x_1,x_2,\dots)= \Big(\sum_{i=0}^\infty \frac{x_1^i}{i!}\Big)\cdot\Big(\sum_{\substack{i_2,\dots,i_n\geq 0,\\2i_2+\dots+ni_n\leq n}} \frac{n-(2i_2+\dots+ni_n)}{n} \prod_{k=2}^n n^{\frac{i_k(k-1)}{2}}\frac{x_k^{i_k}}{i_k!}\Big)\in \R[\vec{x}].
\end{equation*}
\end{example}

See also \ref{ExampleSchurWeylDist} for a computation of $\textup{A}_\rho$ for the Schur--Weyl distribution and section \ref{SectionProofThmCGFFforSinfty} for the general computation for probability measures induced from extreme characters of $S_\infty$.

\begin{remark}
Note that the rescaling by $\sqrt{n}$ is included in the definition of the Young generating function $\textup{A}_\rho$, while the choice of this rescaling allow us to simplify the statements of our main results it shouldn't be understood as a restriction. We believe that different rescaling in the Young generating function will allow to study other regimes of convergence besides the ones considered in this paper. 
\end{remark}

We end this section by noting that natural operations on rings of power series can be lifted into the inverse limits. For instance denote $\partial_i$ the \textit{partial derivative operation}, $\partial_i:\R[\vec{x}]\to \R[\vec{x}]$ for $i \geq 1$, defined such that $\partial_i x_i^k = k x_i^{k-1}$ for each $k\geq 0$. We can also denote the linear operator of evaluation at $0$ by $f(x_1,x_2,\dots) \bigr|_{\vec{x}=0}=f(0,0,\dots) \in \R$. In other words the evaluation at $0$ returns the constant term of $f$.

Similarly, we would like to define the exponential and logarithm operation such that $\exp(x)=\sum_{i=0}^\infty \frac{x^i}{i!}$ and $-\ln(1-x)=\sum_{i=1}^\infty \frac{x^i}{i}$. However these operators are only well defined for a subset of $\R[\vec{x}]$. For instance, for $f\in  \R[\vec{x}]$, $\exp(f)$ is well defined if $f\bigr|_{\vec{x}=0} = 0$ while $\ln(f)$ is well defined if $f\bigr|_{\vec{x}=0} =1$. Notice that for any probability measure $\rho$ on $\mathbb{Y}_n$, we have that $\textup{A}_{\rho}(0,0,\dots)=1$, hence $\ln\big(\textup{A}_\rho(x_1,x_2,\dots)\big)$ is a well defined element of $\R[\vec{x}]$. 

\subsection{Preliminaries}

We start by introducing a space of functions where both integer partitions and their limit shapes exists. This can be achieved by taking the completion of the space Young diagrams in Russian notation under the uniform norm.

\begin{definition}
    We say that a continuous real valued function $f(x)$ is a continuous \textit{diagram} if 
    \begin{enumerate}
        \item $|f(x)-f(y)|\leq |x-y|$ for all $x,y\in \R$.
        \item  $f(x)=|x-z|$ for some $z\in \R$ and all $|x|$ large enough. 
    \end{enumerate}
We call $z$ the center of the diagram and say that a diagram is centered if $z=0$.
\end{definition}

We can naturally embed Young diagrams in the space of centered continuous diagrams by drawing them as Young diagrams in Russian notation, see figure \ref{RussianNotation}. Note that in this coordinate system we ensure that both the minima and the maxima of the Young diagram are integers. Additionally, the area of a box in the Young diagram in Russian notation is equal to  $2$ rather than $1$. Although this coordinate system is preferred to state our applications, all main theorems will be stated in terms of Kerov's transition measure.

\begin{theorem}[Markov--Krein correspondence \cite{Ke}] 
    Let $\mathcal{D}[a,b]$ the set of diagrams on the interval $[a,b]$ with the uniform topology and $\mathcal{M}[a,b]$ the set of probability measures on the interval $[a,b]$ with the weak topology. There is a homeomorphism of spaces $\mathcal{D}[a,b] \to \mathcal{M}[a,b]$. Given a continuous diagram $\omega$ we call its corresponding measure $m_K[\omega]$ it's \textit{transition measure}. This correspondence is given by the relation
    \begin{equation*}
    \int_a^b \frac{1}{z-t}dm_K[\omega](t)=\frac{1}{z}\exp \int_a^b \frac{1}{t-z}d\Big(\frac{\omega(t)-|t|}{2}\Big),
    \end{equation*}
    for all $z\in \C\backslash[a,b]$.
    Furthermore, given a partitions $\lambda$ we relax the notation to $m_K[\lambda]$.
\end{theorem}

While the previous formula is given in terms of Stieltjes transforms, the calculation of the transition measure can be done explicitly in most cases of interest.

\begin{example}\label{ExamplesOfMarkovKrein}
If $\lambda$ is a partition, then 
\begin{equation*}
m_K[\lambda]=\sum_i \mu_i \delta_{x_i}
    \text{ where } \displaystyle \mu_i=\frac{\prod_j (x_i-y_j)}{\prod_{i\neq j} (x_i-x_j)},
\end{equation*}
and $\{x_i\}_i$, $\{y_i\}_i$ denotes the collection of minima and, respectively, maxima of the Young diagram $\lambda$. Note that $\mu_i=\frac{\dim(\Lambda)}{(n+1)\dim(\lambda)}$ where $\Lambda$ is the partition obtained from $\lambda$ when adding a square at the location of the minima $x_i$.  See figure \ref{FigureMKtransform} for an explicit example. 

\begin{figure}[h]
    \centering
    \includegraphics[scale=0.3]{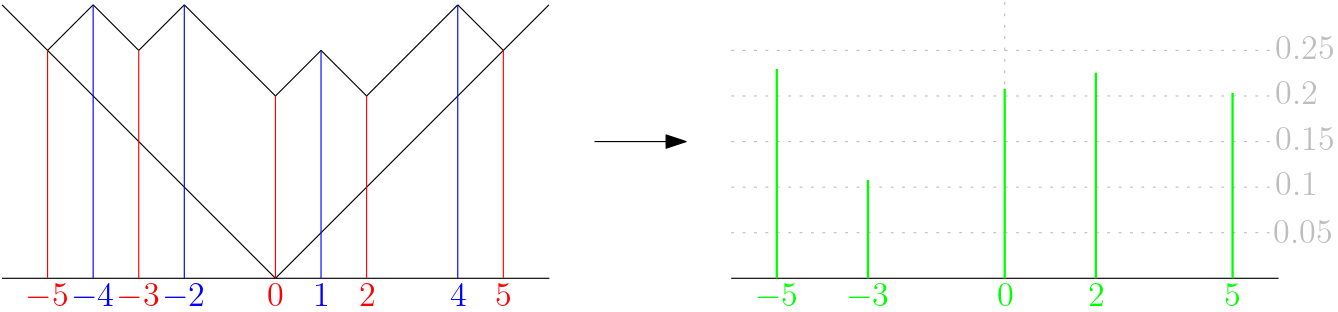}
    \caption{Markov--Krein transform of $\lambda=(5,3,2,2,1)$.}
    \label{FigureMKtransform}
\end{figure}

If $\omega$ is the Vershik--Kerov--Logan--Shepp curve, that is
\begin{equation*}
\omega(t) = \begin{cases}
(2/\pi)(t\arcsin(t/2)+\sqrt{4-t^2}) & \text{ if } |t|\leq 2,\\
|t| & \text{ if } |t|\geq 2,\\
\end{cases}
\end{equation*}
then $dm_K[\omega](t)=(2\pi)^{-1}\sqrt{4-t^2}\,dt$ is the semicircle distribution.
\end{example}

While the Markov--Krein transform may seem complicated to grasp due to it's non-linearity, we will show how to overcome this difficulty in section \ref{SubsectionCoordinateSystems} and section \ref{PreliminariesSection7}. In fact, this transformation has been largely studied from different perspectives in the last few decades. For instance, explicit formulas for Kerov's transition measure are given in \cite[Theorem 1]{Rom2} and more recently, Śniadi found the modulus of continuity for the Markov--Krein transform  \cite{Sni2}.

We will state our main theorems in terms of cumulants, which we briefly review here. 

\begin{definition}\label{DefintionclassicalCumulant}
Given $r \in \N$, a collection of random variables $X_1,\dots,X_r$. The $r$th order cumulant of $X_1,\dots,X_r$ is defined through
\begin{equation*}
\kappa(X_1,\dots,X_r) := \sum_{\pi \in \Theta_r} (-1)^{|\pi|-1} (|\pi|-1)! \prod_{B \in \pi} \E\Big[ \prod_{j\in B} X_{j}\Big].
\end{equation*}
Where $\Theta_r$ is the set of all set partitions of $\{1,2,\dots,r\}$ and $|\pi|$ denotes the number of elements in the set partition $\pi$. 
\end{definition}

Cumulants can be understood as an alternative to the joint moments, in fact, there is an explicit bijection between the set of joint moments and the set of cumulants of the random variables $X_1,\dots,X_r$. The first cumulant and second cumulant coincide with the mean and covariance, respectively, that is 
\begin{equation*}
    \kappa(X_1)=\E[X_1] \hspace{2mm} \text{ and } \hspace{2mm} \kappa(X_1,X_2)=\textup{Cov}(X_1,X_2). 
\end{equation*}
Additionally, Definition \ref{DefintionclassicalCumulant} can be inverted as
\begin{equation*}
\E[X_1X_2\cdots X_r]= \sum_{\pi \in \Theta_r} \prod_{\{i_1,\dots,i_s\}\in \pi} \kappa(X_{i_1},X_{i_2},\dots,X_{i_s}). 
\end{equation*}

These explicit formulas between joint cumulants and joint moments are obtained by Möbius inversion on the partition lattice, see  \cite[Page 154]{Ai}. We will further deepen into this connection in the next few sections. The main interest into working with cumulants rather than joint moments is the following characterization of a Gaussian process in terms of cumulants, see \cite[Example 3.2.3 (ii)]{PT}. 

\begin{lemma}
A collection of random variables $(X_i)_{i \in I}$ defines a Gaussian process if and only if for each $r\geq 3$ and $i_1,\dots,i_r\in I$, the $r$th joint cumulant vanishes. 
\begin{equation*}
\kappa(X_{i_1},X_{i_2},\dots,X_{i_r})=0.
\end{equation*}
\end{lemma}

\subsection{Law of large numbers}
For $n=1,2,\dots$ let $\rho_n$ be a probability measure on $\mathbb{Y}_n$. 

\begin{definition}\label{defLLNappropriate}
    We say that $\rho_n$ is ${\bf{LLN-appropriate}}$ if there exists a countable collection of numbers $(c_i)_{i\in \N}$ such that
\begin{enumerate}
    \item for each index $i \in \N$, we have
    $$ \lim_{n \to \infty } \partial_i \ln\big(\textup{A}_{\rho_n}(\vec{x})\big)\bigr|_{\vec{x}=0} = c_i.$$
    \item for each $r\geq 2$ and $i_1,\dots,i_r \in \N$, we have
        $$ \lim_{n \to \infty } \partial_{i_1} \cdots  \partial_{i_r}\ln\big(\textup{A}_{\rho_n}(\vec{x})\big)\bigr|_{\vec{x}=0} = 0.$$
\end{enumerate}
Denote $F_\rho$ the corresponding power series $F_\rho(z):=\sum_{i=1}^\infty c_i z^{i-1}$.
\end{definition}

\begin{definition}\label{defSatiesfiesLLN}
    Let $X_k=\frac{1}{\sqrt{n^k}}\int_\R x^k \,m_K[\lambda](dx)$.
    We say that $\rho_n$ {\bf{satisfies a LLN}} as $n\to \infty$ if there exists a countable collection of numbers $(a_k)_{k\in\N}$ such that 
\begin{enumerate}
    \item for each $k \in \N$, we have
    $$\lim_{n \to \infty} \E_{\rho_n}[X_k] = a_k.$$
    \item For each $r\geq 2$ and $k_1,k_2,\dots, k_r \in \N$, we have
    $$\lim_{n \to \infty} \kappa(X_{k_1},X_{k_2},\dots,X_{k_r}) = 0.$$
\end{enumerate}
\end{definition}

By keeping track of the coefficients of the monomials $x_i$ we note that $c_1=1$. This is a consequence of the fact that $M_\rho(e)=1$. 

\begin{theorem}\label{TheoremLLN}
A sequence of measures $\rho_{n}$ is ${\bf{LLN-appropriate}}$ if and only if it {\bf{satisfies a LLN}} as $n\to \infty$. The numbers $a_k$ are polynomials in $c_k$ that can be computed through
\begin{equation*}
a_k=[z^{-1}] \frac{1}{k+1} \big(z^{-1}+zF_\rho(z)\big)^{k+1}.    
\end{equation*}
Where $[z^{-1}]$ denotes the coefficient of $z^{-1}$ in the power series. 
\end{theorem}

\begin{remark}
Following the notations of Example \ref{ExamplesOfMarkovKrein}, the identity between limiting moments $a_k$ of the transition measure and the function $F_\rho(z)$ stated in Theorem \ref{TheoremLLN} can be interpreted \cite[Corollary 0]{Speicher} as a relationship between the Stieltjes  transform of the limiting rescaled measure $\hat{m}_K[\lambda]=\sum_i \mu_i \delta_{x_i/\sqrt{n}}$ and its Voiculescu's R-transform. That is, by denoting $C_\rho(z)=\lim_{n \to \infty} \int_a^b \frac{1}{z-t}d\hat{m}_K[\lambda](t)$ the Stieltjes transform of the transition measure, we have that $zF_\rho(z)=C^{-1}_\rho(z)-1/z$. This means that the coefficients $c_i$ are the free cumulants of the limiting probability distribution. This was originally observed by Biane \cite{Bi}. For a further in-depth treatment of free probability and it's combinatorial identities see \cite{NS}.
\end{remark}

\subsection{Central limit theorem} As before, $\rho_n$ denotes a probability measure on $\mathbb{Y}_n$.

\begin{definition} \label{defCLTappropriate}
    We say that $\rho_n$ is ${\bf{CLT-appropriate}}$ if there exists two countable collection of numbers $(c_i)_{i\in \N}$ and $(d_{i,j})_{i,\,j\in \N}$ such that 
\begin{enumerate}
    \item for each index $i \in \N$, we have  $$ \lim_{n \to \infty } \partial_i \ln\big(\textup{A}_{\rho_n}(\vec{x})\big)\bigr|_{\vec{x}=0} = c_i.$$
    \item for each indices $i, j \in \N$, we have
    $$ \lim_{n \to \infty } n \partial_i  \partial_j\ln\big(\textup{A}_{\rho_n}(\vec{x})\big)\bigr|_{\vec{x}=0} = d_{i,j}.$$
    \item for each $r\geq 3$ and $i_1,\dots,i_r \in \N$, we have
        $$ \lim_{n \to \infty } n^{r/2} \partial_{i_1} \dots \partial_{i_r} \ln\big(\textup{A}_{\rho_n}(\vec{x})\big)\bigr|_{\vec{x}=0} = 0.$$
\end{enumerate}
Denote $F_\rho$ and $Q_\rho$ the corresponding power series 
\begin{equation*}
    F_\rho(z):=\sum_{i=1}^\infty c_i z^{i-1} \hspace{2mm} \text{ and } \hspace{2mm} Q_\rho(z,w):=\sum_{i,j=1}^\infty d_{i,j} z^{i}w^j.
\end{equation*}
\end{definition}

By keeping track of the coefficients of the monomials $x_ix_j$ we note that $d_{i,j}=0$ if $i=1$ or $j=1$. This is a consequence of the fact that $M_\rho(e\cdot \sigma)-M_\rho(\sigma)M_\rho(e)=0$ for any permutation $\sigma$.

\begin{definition}\label{defSatiesfiesCLT}
    Let $X_k=\frac{1}{\sqrt{n^k}}\int_\R x^k \,m_K[\lambda](dx)$.
    We say that $\rho_n$ {\bf{satisfies a CLT}} as $n\to \infty$ if there exists two countable collection of numbers $(a_k)_{k\in \N}$ and $(b_{k,k'})_{k,\,k'\in \N}$ such that 
\begin{enumerate}
    \item for each $k \in \N$, we have
    $$\lim_{n \to \infty} \E_{\rho_n}[X_k] = a_k.$$
    \item for each $k,k' \in \N$, we have
    $$\lim_{n \to \infty} n \kappa(X_k,X_{k'}) = b_{k,k'}.$$
    \item For each $r\geq 3$ and $k_1,k_2,\dots, k_r \in \N$, we have
    $$\lim_{n \to \infty} n^{r/2} \kappa(X_{k_1},X_{k_2},\dots,X_{k_r}) = 0.$$
\end{enumerate}
\end{definition}

\begin{theorem}\label{TheoremCLT}
A sequence of measures $\rho_{n}$ is ${\bf{CLT-appropriate}}$ if and only if it {\bf{satisfies a CLT}} as $n\to \infty$. The numbers $a_k, b_{k,k'}$ are polynomials in $c_k,d_{k,k'}$ that can be computed through
\begin{equation*}
a_k=[z^{-1}] \frac{1}{k+1} \big(z^{-1}+zF_\rho(z)\big)^{k+1}
\end{equation*}
and
\begin{align*}
b_{k,k'} &= [z^{-1}w^{-1}] \bigg[ \Big(z^{-1}+zF_\rho(z)\Big)^{k}\Big(w^{-1}+wF_\rho(w)\Big)^{k'}
\Big(Q_\rho(z,w)\\
& \hspace{3.5cm}- zw \partial_z \partial_w\big(zwF_\rho(z)F_\rho(w)+\ln(1 -zw\frac{zF_\rho(z)-wF_\rho(w)}{z-w})\big)\Big)\bigg]. 
\end{align*}
Where $[z^{-1}]$ and $[z^{-1}w^{-1}]$ denotes respectively the coefficient of $z^{-1}$ and $z^{-1}w^{-1}$ in the power series. 
\end{theorem}

We provide a short example on how to apply this theorem. 

\begin{example}\label{ExampleLLNandCLTPlancherel}
Following Example \ref{ExampleArhoPlancherel}, denote $\rho_n$ the Plancherel distribution on $\mathbb{Y}_n$. We showed earlier that $A_{\rho_n}=\sum_{i=0}^\infty \frac{x_1^i}{i!}$, which implies $\ln(A_{\rho_n})=x_1$. This gives that $c_1=1$ while $c_i=0$ for $i\geq 2$ and $d_{i,j}=0$ for any $i,j\geq 1$. Hence $F_{\rho}(z)=1$ and $Q(z,w)=0$. It follows from Theorem \ref{TheoremCLT} that
\begin{align*}
    a_k=[z^{-1}] \frac{1}{k+1} \big(z^{-1}+z\big)^{k+1}=\begin{cases}
    \frac{\binom{k}{k/2}}{k/2+1} &\text{ for } k \text{ even,}\\
    0 & \text{ for } k \text{ odd.}
    \end{cases}
\end{align*}
Which are precisely the moments of the semi-circle distribution. Similarly,
\begin{align*}
    b_{k,k'}&=[z^{-1}w^{-1}] \bigg[ \Big(z^{-1}+z\Big)^{k}\Big(w^{-1}+w\Big)^{k'}
\Big(-zw\partial_z \partial_w\ln(1 -zw\big)- zw\Big)\bigg]\\
&=[z^{-1}w^{-1}] \bigg[ \Big(\sum_{j=0}^k \binom{k}{j} z^{k-2j}\Big)\Big(\sum_{j=0}^{k'} \binom{k'}{j} z^{k'-2j}\Big)
\Big(\sum_{i=2}^{\infty} i(zw)^{i}\Big)\bigg]\\
&=\sum_{i=2}^{\infty} i\binom{k}{\frac{k-i-1}{2}}\binom{k'}{\frac{k'-i-1}{2}}.
\end{align*}
Where the binomials vanish when the indices $\frac{k-i-1}{2}$ or $\frac{k'-i-1}{2}$ are distinct from $0,\dots,k$ or $0,\dots,k'$ respectively. By using the inversion formulas for the Chebyshev polynomials of the first kind, it is a short computation to verify that this coincides with Kerov's central limit theorem for the transition measure. See \cite[Theorem 8.8]{IO}.
\end{example}

\subsection{Multilevel theorems} \label{SectionMultilevelTheorems}

We are interested in  sampling an increasing sequence of partitions. Given two partitions $\lambda$ and $\mu$, denote $\lambda \subseteq \mu$ when the Young diagram induced from $\lambda$ is contained in the Young diagram induced from $\mu$. Given a sequence of integers $n_1<n_2<\dots<n_s$ we would like to sample $ \vec{\lambda}=(\lambda^1 \subsetneq \lambda^2 \subsetneq \cdots \subsetneq \lambda^s) \in \mathbb{Y}_{n_1}\times \mathbb{Y}_{n_2}\times \dots \times \mathbb{Y}_{n_s}$. One option to formalize this is to proceed as follows, denoting by $\lambda \vdash n$ a partition of size $n$, and $n' <n$, sample $\mu \vdash n'$ with probability $p(\lambda \to \mu )=\frac{\dim(\lambda\backslash \mu)\dim (\mu)}{\dim(\lambda)}$, here $\dim(\lambda\backslash \mu)$ denotes the number of standard Young tableaux with skew shape $\lambda\backslash \mu$. Note that if $\mu=\emptyset$, the empty partition, then $\dim(\lambda\backslash \mu)=\dim(\lambda)$. This induces to the following process, start by sampling $\lambda \vdash n $ with probability $\rho(\lambda)$, then sample the rest of the partitions using the previous transition probability. 

Given positive reals $0<\alpha_1<\alpha_2<\dots<\alpha_s=1$, let $n_t=\lfloor \alpha_t n \rfloor$ and $\lambda^t \vdash n_t$. We define the probability measure 
\begin{equation}\label{JointDistribution}
p(\lambda^1,\lambda^2,\dots,\lambda^s)=\rho_n(\lambda^s)\prod_{t=1}^{s-1} p(\lambda^{t+1} \to \lambda^{t}). 
\end{equation}

These transition probabilities arise from the branching rules for the symmetric group, see \cite[Chapter 1]{BoO}. We will further explain this connection on section \ref{SectionTechnicalLemmasProofs}.

\begin{theorem}\label{TheoremMultilevelLLN}
Let $\lambda^1 \subsetneq \lambda^2 \subsetneq \cdots \subsetneq \lambda^s$ be a random increasing sequence of partitions sampled from equation (\ref{JointDistribution}) and let $X_k^{\alpha}=\frac{1}{\sqrt{n^k}}\int_\R x^k \,m_K[\lambda^\alpha](dx)$. If the sequence of measures $\rho_{n}$ is ${\bf{LLN-appropriate}}$, then there exists a collection of numbers $(a_k^{\alpha})_{k\in \N}$ such that 
\begin{enumerate}
    \item for each $k \in \N$, we have
    $$\lim_{n \to \infty} \E_{\rho_n}[X_k^{\alpha}] = a_k^\alpha.$$
    \item For each $r\geq 2$, $k_1,k_2,\dots, k_r \in \N$, we have
    $$\lim_{n \to \infty} \kappa(X_{k_1}^{\alpha_1},X_{k_2}^{\alpha_2},\dots,X_{k_r}^{\alpha_r}) = 0.$$
\end{enumerate}
The numbers $a_k^{\alpha}$ can be computed through
\begin{align*}
a_k^\alpha=[z^{-1}] \alpha^{-1}\frac{1}{k+1} \big(\alpha z^{-1}+zF_\rho(z)\big)^{k+1}.
\end{align*}
\end{theorem}

We further have that

\begin{theorem}\label{TheoremMultilevelCLT}
Let $\lambda^1 \subsetneq \lambda^2 \subsetneq \cdots \subsetneq \lambda^s$ be a random increasing sequence of partitions sampled from equation (\ref{JointDistribution}) and let $X_k^{\alpha}=\frac{1}{\sqrt{n^k}}\int_\R x^k \,m_K[\lambda^\alpha](dx)$. If the sequence of measures $\rho_{n}$ is ${\bf{CLT-appropriate}}$, then there exists two collection of numbers $(a_k^{\alpha})_{k\in \N}$ and $(b_{k,k'}^{\alpha,\alpha'})_{k,k'\in \N}$ such that 
\begin{enumerate}
    \item for each $k \in \N$, we have
    $$\lim_{n \to \infty} \E_{\rho_n}[X_k^{\alpha}] = a_k^\alpha.$$
    \item for each $k,k' \in \N$, we have
    $$\lim_{n \to \infty} n \kappa(X_k^{\alpha},X_{k'}^{\alpha'}) = b_{k,k'}^{\alpha,\alpha'}.$$
    \item For each $r\geq 3$, $k_1,k_2,\dots, k_r \in \N$, we have
    $$\lim_{n \to \infty} n^{r/2} \kappa(X_{k_1}^{\alpha_1},X_{k_2}^{\alpha_2},\dots,X_{k_r}^{\alpha_r}) = 0.$$
\end{enumerate}
 The numbers $a_k^{\alpha}$ can be computed as in Theorem \ref{TheoremMultilevelLLN} while $b_{k,k'}^{\alpha,\alpha'}$ can be computed through
\begin{align*}
b_{k,k'}^{\alpha,\alpha'} &= [z^{-1}w^{-1}] \bigg[\frac{1}{\alpha \alpha'}\Big(\alpha z^{-1}+zF_\rho(z)\Big)^{k}\Big(\alpha'w^{-1}+wF_\rho(w)\Big)^{k'}  
\Big(Q_\rho(z,w)\\
& \hspace{2cm}- zw \partial_z \partial_w\big( \frac{zw}{\max(\alpha,\alpha')}F_\rho(z)F_\rho(w)+\ln(1 -zw\frac{zF_\rho(z)-wF_\rho(w)}{\max(\alpha,\alpha')(z-w)})\big)\Big)\bigg].
\end{align*}
\end{theorem}

Given a sequence of random integer partitions $\lambda \in \mathbb{Y}_n$ with distribution $\rho_n$ which satisfies a LLN, for each $\lambda$ we can sample a new partition $\mu \subseteq \lambda$ with $\mu \vdash n_{\textup{sub}}$ and $n_{\textup{sub}} \leq n$ using the method described in equation (\ref{JointDistribution}). Theorem \ref{TheoremMultilevelLLN} describe the limit shape of the rescaled partition $\mu$ in terms of Kerov's transition measure only when $\lim_{n \to \infty} n_{\textup{sub}}/n \to \alpha$ for $\alpha \in (0,1)$, furthermore, if $\rho_n$ satisfies a CLT, then Theorem \ref{TheoremMultilevelCLT} describe the fluctuations of $\mu$. It is natural to wonder what happens if we still have $n_{\textup{sub}} \to \infty$ but $n_{\textup{sub}}/n\to 0$ as $n\to \infty$. Our methods allow to study this situation. The following example shows that, after properly rescaling, the limiting shape is ensured to be the Vershik--Kerov--Logan--Shepp curve.

\begin{example}[Sublinear random Young diagrams] \label{TheoremSublinearn}
Let $n_{\textup{sub}}$ be an increasing sequence of integers such that $n_{\textup{sub}}\to \infty$ and $n_{\textup{sub}}/n \to 0$ as $n\to \infty$. Let $\rho_n$ a probability measure on $\mathbb{Y}_n$ that satisfies a LLN and sample $(\mu,\lambda) \in \mathbb{Y}_{n_{\textup{sub}}} \times \mathbb{Y}_n$ with probability $\rho_n(\lambda)p(\lambda \to \mu)$. Let $X_k^\textup{sub}=\frac{1}{\sqrt{n_{\textup{sub}}^k}}\int_\R x^k \,m_K[\mu](dx)$, then for each $k \in \N$, 
\begin{equation*}
\lim_{n \to \infty} X_k^\textup{sub} = \begin{cases}
       \frac{1}{1+k/2} \binom{k}{k/2} & \text{ if } k \text{ is even,}\\
        0 & \text{ else,}
    \end{cases} \hspace{1cm} \text{in probability.}
\end{equation*}
Meaning that the rescaled random transition measure converges, in the sense of moments, in probability to the semicircle distribution. This is an application of Theorem \ref{TheoremMultilevelLLN}, noticing that the calculation of the moments corresponds to computing the limit 
\begin{equation*}
\lim_{\alpha \to 0} \alpha^{-k/2} a_k^\alpha= \lim_{\alpha \to 0}  [z^{-1}] \frac{\alpha^{-(k+1)/2}}{k+1} \big(\sqrt{\alpha} z^{-1}+\sqrt{\alpha}zF_\rho(\sqrt{\alpha}z)\big)^{k+1}=[z^{-1}] \frac{1}{k+1} (z^{-1}+z)^{k+1}.
\end{equation*} 
Similarly, an application of Theorem \ref{TheoremMultilevelCLT} shows that the collection of rescaled random variables $\big(\sqrt{n_{\textup{sub}}} X_k^\textup{sub}\big)_k$ jointly converge to a Gaussian process with covariance
\begin{equation*}
    \lim_{n \to \infty} n_{\textup{sub}} \textup{Cov}(X_k^\textup{sub},X_{k'}^\textup{sub})=[z^{-1}w^{-1}] \bigg[ \Big(z^{-1}+z\Big)^{k}\Big(w^{-1}+w\Big)^{k'}
\Big(-zw\partial_z \partial_w\ln(1 -zw\big)- zw\Big)\bigg].
\end{equation*}
Which coincides with the covariance computed in Example \ref{ExampleLLNandCLTPlancherel}.

\end{example}
Furthermore, Theorem \ref{TheoremMultilevelCLT} and Theorem \ref{TheoremCLT} combine very well together, as one can verify by other methods that a probability distribution on integer partitions satisfies a  CLT and then we automatically have that it satisfy a multilevel CLT. 

\begin{example} \label{ExampleGelfandDist}
The Gelfand distribution over integer partitions is given for $\lambda$ a partition of size $n$ by $\Prob(\lambda)=\frac{\dim(\lambda)}{I_n}$ where $I_n$ denotes the number of involutions in $S_n$, that is, the number of elements $\sigma \in S_n$ such that $\sigma^2=\text{id}_{S_n}$.  In \cite{Mel11, Mel10a} it is shown that the Gelfand distribution satisfies a CLT, then Theorem \ref{TheoremCLT} ensures that it is CLT-appropriate and hence Theorem \ref{TheoremMultilevelCLT} provides a multilevel CLT. See \cite{APR08} for further discussion regarding the combinatorial interpretation of this distribution and it's representation-theoretic meaning.
\end{example}

\section{Applications}\label{SectionApplications}

\subsection{Coordinate systems}\label{SubsectionCoordinateSystems}

There are multiple coordinate systems that allow an effective description of integer partitions. While we already introduced continuous Young diagrams and their associate transition measure, we briefly introduce an additional coordinate system, the co-transition measure. We will briefly explain how to translate our results between these different point of views. Understanding how to do this has a dual purpose. On one hand, the continuous Young diagram point of view is more appropriate to characterize the fluctuations of the height function. In fact, we will need to work directly with continuous Young diagrams rather than the transition measure to uncover the conditional Gaussian Free Field structure. On the other hand, in 2007 Pittel and Romik conjectured a CLT for co-transition measure (see \cite[ Section 6]{PiR}), our results allow us to obtain such a statement, here we briefly explain how this is achieved.

For each continuous Young diagram $\omega \in \mathcal{D}[a,b]$ we denote by $\sigma[\omega]$ the function defined by $\sigma[\omega](t)=\frac{\omega(t)-|t|}{2}$. Notice that it allows us to calculate the area of the corresponding continuous Young Diagram via $\mathcal{A}(\omega)=2\int_\R \sigma(t)\,dt$. The Markov--Krein correspondence can be used to define the transition measure of a continuous diagram, a similar transformation allows us to define the co-transition measure in a general setting. 

\begin{definition}[\cite{Rom2}]
Given a continuous Young diagram $\omega \in \mathcal{D}[a,b]$, we define its \textbf{co-transition measure} $m_{A}[\omega]$ to be the unique probability distribution satisfying
\begin{equation}\label{CotransitionMeasureDef}
    \mathcal{A}(\omega) \int_\R \frac{1}{z-t} m_A[\omega](dt) = z-\Big(\int_\R \frac{1}{z-t} m_{K}[\omega](dt)\Big)^{-1}.
\end{equation}
\end{definition}

Similarly to the transition measure, we have explicit formulas for the co-transition when the continuous diagram is a Young diagram. In fact, given a partition $\lambda$, the co-transition measure of $\lambda$ is given the following formula, 
\begin{equation*}
m_A[\lambda]=\sum_i \nu_i \delta_{y_i},
    \text{ where } \displaystyle \nu_i=\frac{\dim(\Lambda)}{\dim(\lambda)},
\end{equation*}
$\{x_i\}_i$, $\{y_i\}_i$ denote the collection of minima and, respectively, maxima of the Young diagram $\lambda$ and $\Lambda$ is the partition obtained from $\lambda$ when subtracting a square at the location of the maxima $y_i$. Notice that continuous diagrams of different shape may have the same co-transition measure, for example taking Young diagrams with square shape will always give a co-transition measure concentrated at the point $0$. Furthermore, Romik \cite[Theorem 6]{Rom2} proved that when restricted to the set of centered continuous diagrams with a fixed positive area, then equation (\ref{CotransitionMeasureDef}) defines a homeomorphism. 

The co-transition measure has a nice probabilistic interpretation. Given an integer partition $\lambda \vdash n$, when we choose a standard Young tableau of shape $\lambda$ uniformly at random, the co-transition measure $m_A[\lambda]$ is the probability distribution of the $x$th coordinate (In Russian notation) of the box containing $n$ (Also called the content of the box containing $n$).

While the relationship between continuous Young diagrams, transition measure and co-transition measure is given implicitly via Stieltjes transforms, these relations provides explicit formulas between the respective moments induced by each measure. Fortunately, the formulas are explicit enough that they allows us to prove both laws of large numbers and central limit theorems for each one of the measures when choosing LLN-appropriate and CLT-appropriate probability measures on partitions. A detailed account on how this is done is given in section \ref{PreliminariesSection7}. We now show how to linearize the Markov--Krein correspondence to obtain Gaussian fluctuations for the moments of random Young diagrams. 

Let $\rho_n$ be a CLT-appropriate sequence of probability measures on $\mathbb{Y}_n$, we will denote $X_k=\frac{1}{\sqrt{n^k}}\int_\R x^k \,m_K[\lambda](dx)$ and $Y_k=\frac{k}{\sqrt{n^k}}\int_\R x^{k-1} \,d\sigma[\lambda](x)$ the random rescaled moments of the transition measure and the continuous Young diagrams respectively. Then Theorem \ref{TheoremCLT} ensures that as $n\to \infty$, $X_k \approx X_k^{\infty}+\tfrac{1}{\sqrt{n}}\xi_k$ for $(\xi_k)_k$ a collection of Gaussian random variables. Similarly, we can write $Y_k \approx Y_k^{\infty}+\tfrac{1}{\sqrt{n}}\eta_k$. The Markov--Krein correspondence states that
\begin{equation*}
     \sum_{k=0}^{\infty} z^{-k}X_k^\infty+ \frac{1}{\sqrt{n}}\sum_{k=0}^{\infty} z^{-k}\xi_k=\exp\Big(-\sum_{k=0}^{\infty} z^{-k-1}Y_k^\infty\Big)\exp\Big(\frac{-1}{\sqrt{n}}\sum_{k=0}^{\infty} z^{-k-1}\eta_k\Big).
\end{equation*}
This induces the following relation between the random variables $\{\xi_k\}_k$ and $\{\eta_k\}_k$.
\begin{equation*}
    \sum_{k=0}^{\infty} z^{-k}\frac{\xi_k}{\sqrt{n}}=\frac{-1}{\sqrt{n}}\exp\Big(-\sum_{k=1}^{\infty}z^{-k-1}Y_k^{\infty} \Big) \sum_{k=1}^{\infty}z^{-k-1}\eta_k+O\Big(\frac{1}{n}\Big).
\end{equation*}
Since the limits $Y_k^{\infty}$ are deterministic, this induces a linear relation between the random variables  $(\xi_k)_k$ and $(\eta_k)_k$ in the limit $n\to \infty$. In particular, it proves Gaussian fluctuations for the random variables $Y_k$. 

Similarly, we can define $Z_k=\frac{1}{\sqrt{n^k}}\int_\R x^k \,m_A[\lambda](dx)$, the random rescaled moments of the co-transition measure, and write $Z_k \approx Z_k^{\infty}+\tfrac{1}{\sqrt{n}}\zeta_k$ as $n\to \infty$. Notice that when we rescale Young diagrams of size $n$ by $1/\sqrt{n}$ we obtain a continuous Young diagram with area $ \mathcal{A}(\omega)=1$. This ensures that equation (\ref{CotransitionMeasureDef}) is an homeomorphism over the space of rescaled Young diagrams. As we did for the Markov--Krein correspondence, we can linearize the characterization of the co-transition measure given in equation (\ref{CotransitionMeasureDef}), we obtain the relation
\begin{equation*}
   \frac{1}{\sqrt{n}} \sum_{k=0}^{\infty} z^{-k-1}\xi_k=\frac{-1}{\sqrt{n}\big(\sum_{k=0}^{\infty} z^{-k-1} Z_k^\infty\big)^2} \sum_{k=0}^{\infty} z^{-k-1}\zeta_k+O\Big(\frac{1}{n}\Big).
\end{equation*}
Again, since the limits $Z_k^{\infty}$ are deterministic, this induces a linear relation between the random variables  $(\xi_k)_k$ and $(\zeta_k)_k$ in the limit $n\to \infty$. That is, we obtain a CLT statement for the random moments of the co-transition measure. 

\begin{remark}
Notice that in the continuous Young diagram coordinate system there is a hidden measure playing an analogous role to the transition and co-transition measures. The \textbf{Rayleigh measure} of a partition $\lambda$ defined by $M_R[\lambda]=\sum \delta_{x_i} - \sum \delta_{y_i}$, where $\{x_i\}_i$, $\{y_i\}_i$ denotes the collection of minima and maxima of the Young diagram $\lambda$. We can relate the moments of $\sigma[\lambda]$ and the Rayleigh measure $M_R[\lambda]$ by integrating by parts. We will not examine this relationship in detail and the reader is referred to \cite{Ho2} for further details.      
\end{remark}

\subsection{Height functions and standard Young tableaux}

We briefly remind the notation introduced in the last few sections. Given two partitions $\mu$ and $\lambda$, denote $\mu \nearrow \lambda$ when $\lambda$ is obtained by adding a single square on a corner of $\mu$. Furthermore, each standard Young Tableau SYT of size $n$ is in correspondence with a sequence of partitions $\lambda^1 \nearrow \lambda^2 \nearrow \cdots \nearrow \lambda^n$ where $\lambda^i \vdash i$.  It follows that sampling an element of $\mathbb{Y}_{1}\times \mathbb{Y}_{2}\times\dots \times \mathbb{Y}_{n}$ with fixed $\lambda^n$ correspond to sampling a tableau $T$ of shape $\lambda^n \vdash n$ uniformly among all standard Young tableau with that shape. 

We want to associate to each SYT a surface description. As already discussed in section \ref{SubsectionCoordinateSystems}, there are different coordinate systems that can be appropriate to encode a SYT as a stepped surface. We will present two of them. One option is to simply fix the height to be the entry of each box in a SYT in Russian notation, that is, each tableau $T$ defines a function $T:\R^2\to\R$ via
\begin{equation*}
T(x,y):= \begin{cases}
m & \textup{ if } (x,y) \textup{ is contained in the box with entry } m,\\
0 & \textup{otherwise.}
\end{cases}
\end{equation*}

Another possibility is to consider the sequences of Young diagrams induced by the collection of increasing integer partitions forming the tableau, we further subtract $|x|$ and normalize to obtain a compactly supported function. That is, 
\begin{equation*}
\textup{H}(x,t):= \frac{\lambda^{\lfloor t \rfloor}(x)-|x|}{2}.
\end{equation*}

The function $\textup{H}$ denotes the length of the $x$th diagonal of the partition $\lambda^{\lfloor t \rfloor}$. Alternatively, by interpreting the sequence of partitions $(\lambda^n)_{n\in\N}$ as an infinite standard Young tableau $T$, then $\textup{H}$ is counting the number of boxes in the $x$th diagonal of $T$ with entry smaller or equal than $t$. Following the convention established in the literature concerning limiting shapes of random SYT, through the rest of this paper we will focus on \textup{H} and refer to it as the \textbf{height function} of the tableau $T$. See figure \ref{HeightFunction} for a slice of this height function, see figure \ref{Random&ExpectedSurface1000Plancherel} (A) for a sample of the random height function in the Plancherel growth process. 

\begin{figure}[h]
    \centering
    \includegraphics[scale=0.3]{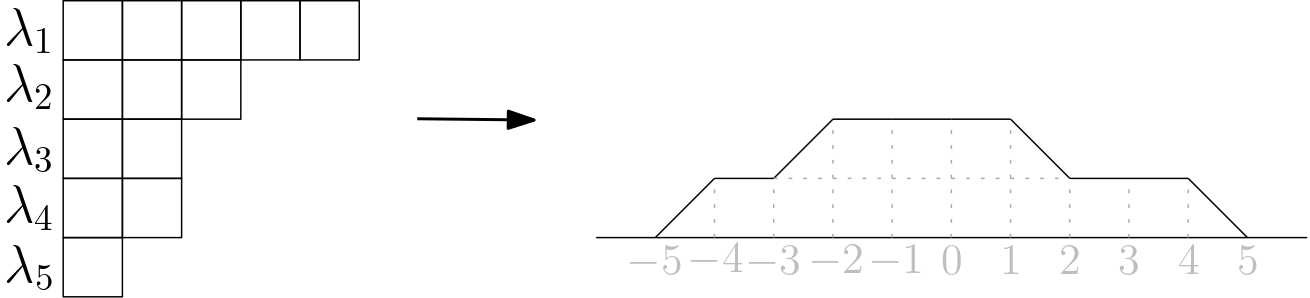}
    \caption{Graph of $x\to \textup{H}(x,13)$ when $\lambda^{13}=(5,3,2,2,1)$.}
    \label{HeightFunction}
\end{figure}

Notice that at the level of partitions the choice of coordinate system is inconsequential, that is, moving from one coordinate system to the other can be made without major difficulties, we have that
\begin{equation*}
    T(x,y)<t \hspace{5mm} \textup{ if and only if} \hspace{5mm} H(x,t)>\tfrac{1}{2}(y-|x|).
\end{equation*}

A consequence of Theorem \ref{TheoremMultilevelLLN} is that a limiting surface $\textup{H}^{\infty}(x,t)=\lim_{n\to \infty}\tfrac{1}{\sqrt{n}}\textup{H}(\sqrt{n}x, t n)$ exists for any LLN-appropriate sequence of $\rho_n$, see also \cite[Theorem 7.15]{Su} and \cite[Theorem 3]{BBFM} for other proofs. In fact the sections $\textup{H}^{\infty}(\cdot,t)$ can be computed explicitly by inverting the Stieltjes transform, in section in section \ref{PreliminariesSection7} we provide an explicit formula for the moments of continuous Young diagrams which further simplify this computation, see Example \ref{ExampleComputingLimitingShapeSquare}. In this paper we are interested in the fluctuations of the limiting surface obtained after rescaling the height function, that is, we describe the fluctuations of $\big[\textup{H}(\sqrt{n}x, t n)-\E\textup{H}(\sqrt{n}x, t n)\big]$, as $n\to \infty$. 

\begin{remark}
The relationship between these two descriptions of the SYT in the limiting surface is not straightforward. This relationship is further studied in \cite{BBFM} where it is shown that the limiting surface $T^\infty(x,y)=\lim_{n\to \infty} \frac{1}{n} T(\sqrt{n}x,\sqrt{n} y)$ is not continuous in general.
\end{remark}

\begin{remark}
Notice that our height function is associated with the measure $d\sigma[\lambda](x)$. We can similarly define height functions associated with the transition measure and the co-transition measure. We will study their corresponding two-dimensional fluctuations in a follow-up paper \cite{Ra25}. 
\end{remark}

\subsection{Two-dimensional Gaussian fields}\label{SubsectionGaussianFields}

Given an arbitrary set $I$, a centered Gaussian process on $I$ is a collection of Gaussian random variables $\{X_i\}_{i\in I}$ such that any finite linear combination of the variables $X_i$, $i\in I$, is
Gaussian with expectation $0$. Any Gaussian process arises a covariance kernel $K:I\times I \to \R$ defined by $K(i,j)=\textup{Cov}(X_i,X_j)$. Furthermore, for any $K:I\times I \to \R$ which is bilinear, symmetric and positive-definite there exists a Gaussian process on $I$ with covariance kernel $K$. Denote by $\mathbb{H}=\{z\in \C: \Im(z)>0\}$ the upper-half plane, we are interested on Gaussian processes defined over a \textit{good} space of functions over $\mathbb{H}$. Take for instance $I$ to be the space of
smooth real–valued compactly supported test functions on $\mathbb{H}$. Letting $f,g\in I$ we are interested in covariances kernel of the form
\begin{equation*}
    K(f,g)=\int_\mathbb{H} \int_\mathbb{H} f(z) g(w)G(z,w)\,dz\,d\bar{z}\,dw\,d\bar{w} 
\end{equation*}
for some function $G:\mathbb{H}\times \mathbb{H}\to \R$. We call such Gaussian processes to be \textbf{generalized Gaussian fields} on $\mathbb{H}$. For instance if $G(z,w)=-\frac{1}{2\pi} \ln \bigr|\frac{z-w}{z-\bar{w}} \bigl|$ then the corresponding Gaussian process is the well known Gaussian Free Field $\mathfrak{G}$ on $\mathbb{H}$ with zero boundary conditions. The field $\mathfrak{G}$ can't be defined as a random function on $\mathbb{H}$, however one can make sense of the integrals $\int f(z)\mathfrak{G}(z)\,dz$ over finite contours in $\mathbb{H}$ with continuous functions $f(z)$, see \cite{PW} for more details. In this paper we will be integrating the generalized Gaussian fields with respect to measures $\mu$ on $\mathbb{H}$ whose support is a smooth curve $\mathcal{C}$ and whose density
with respect to the natural (arc-length) Lebesgue measure on $\mathcal{C}$ is given by a smooth function $h(z)$ such that 
\begin{equation*}
    \int\int_{\mathcal{C}\times \mathcal{C}} h(z)h(w)G(z,w)\,dz\,dw<\infty.
\end{equation*}
This will ensure the existence of such Gaussian field with covariance given by integration with respect to $G$. In addition, all the fluctuation will be restricted to certain domains $D$ that we are able to map to the upper-half plane $\mathbb{H}$. That is, given a domain $D$, corresponding to the liquid region, a map $\Omega:D\to \mathbb{H}$ and a generalized Gaussian field $\mathfrak{C}$ on $\mathbb{H}$, we will consider the pullback $\mathfrak{C}\circ \Omega$. For each of the models treated in this paper we have a function $\Omega^{-1}(z)=\big(x(z),t(z)\big)$ and we will be integrating the Gaussian field over curves $\mathcal{C}_\alpha=\{z:t(z)=\alpha\}$ for $\alpha\geq0$. We are interested in the generalized Gaussian field with covariance induced by the function
\begin{equation*}
G(z,w)=\frac{-1}{2\pi}\ln\Biggr| \frac{z-w}{z-\bar{w}}\Biggr|
-\frac{\min\big(t(z),t(w)\big)}{\pi}\Im\Big(\frac{1}{z}\Big)\Im\Big(\frac{1}{w}\Big).
\end{equation*}
In which case we denote this generalized Gaussian field as $\mathfrak{C}$.

\begin{proposition}\label{PropConditionalGFF}
In each of the models treated in this paper, that is  the setting of sections \ref{SectionPlancherelgrowthprocess}, \ref{SubsectionExtremeCharacters}and \ref{subsectionFixedShape}, the generalized Gaussian field $\mathfrak{C}$ can be identified as a Gaussian free field $\mathfrak{G}$ conditioned to be $0$ when integrated over the curves $\mathcal{C}_\alpha=\{z:t(z)=\alpha\}$.
\end{proposition}

We will discuss in further detail the sense of the conditioning and the proof of the proposition in section \ref{SectionProofsOfApplications}. See \cite{LG} for further details regarding conditioning of Gaussian processes in an infinite dimensional setting. The main interest of the conditioned GFF is that it will play the role of the universal limit object for the two-dimensional fluctuations of our models. In this sense, we interpret it to be a higher-dimensional analogue of a Brownian motion.

Note that this conditioning is intrinsic to our models and is already visible at the level of the height function. For instance, as already noticed in section \ref{SubsectionCoordinateSystems}, we have that $\int_\R \textup{H}(x,t)\,dx=t$, the size of the partition, for any $t\geq 1$. This ensures that there are no fluctuations on the height function when integrated with respect to $x$, in other words, the sum of the height fluctuations over horizontal lines is equal to $0$. This conditioning appears to be visually mild in the sense that when comparing the simulations of the conditioned GFF with simulations of the GFF it seems futile to discern the conditioning. See for figures \ref{LimitFluctuations1000Plancherel}, \ref{LimitFluctuationsSchurWeyl}, \ref{OtherShapes} and \ref{LimitFluctuationsSquare} for simulations of the conditioned GFF. 

\begin{figure}[h]
    \centering
    \subfloat[\centering Random height function.]{{\includegraphics[scale=0.67]{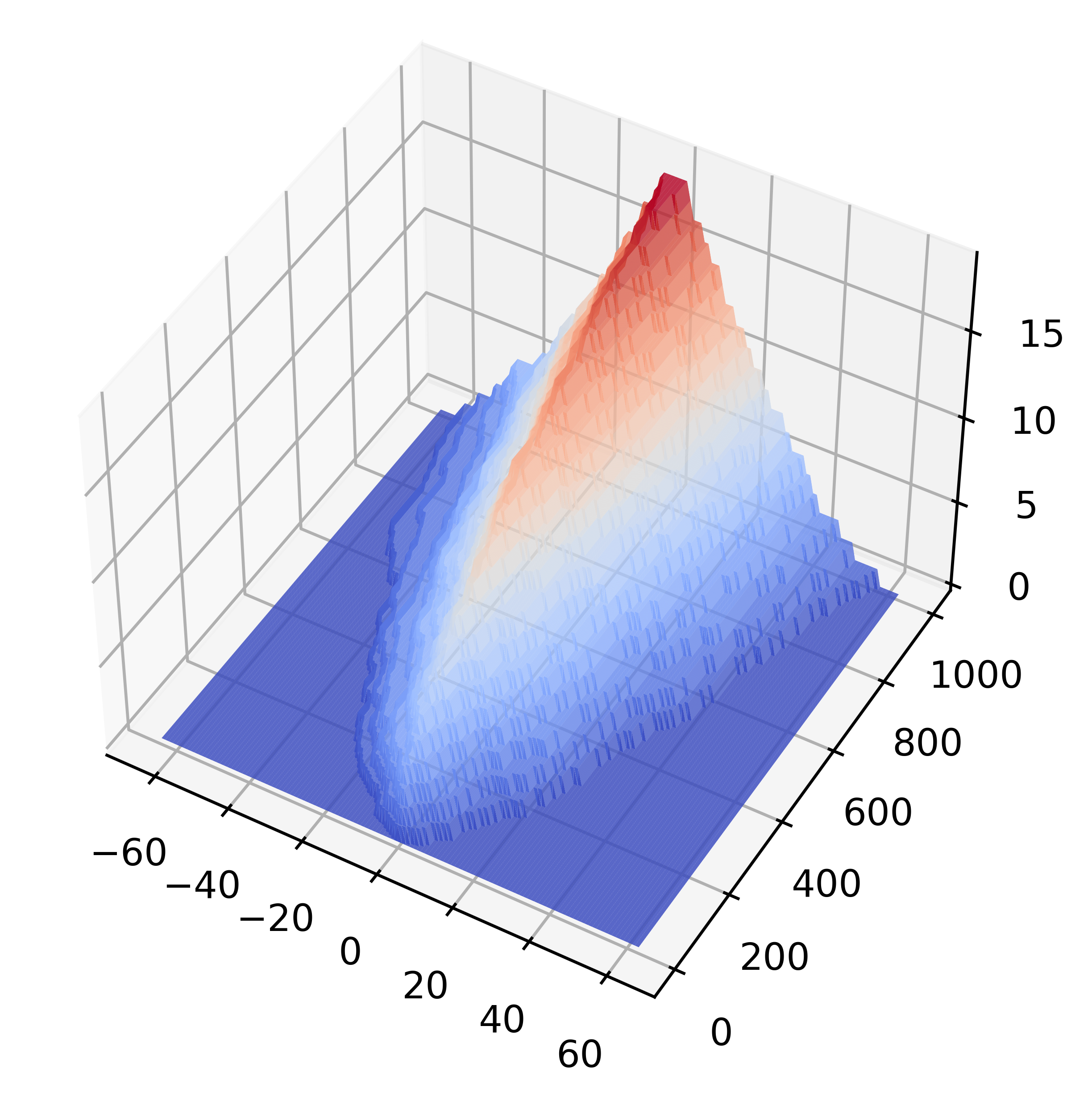} }}%
    \qquad
    \subfloat[\centering Expected height function.]{{\includegraphics[scale=0.67]{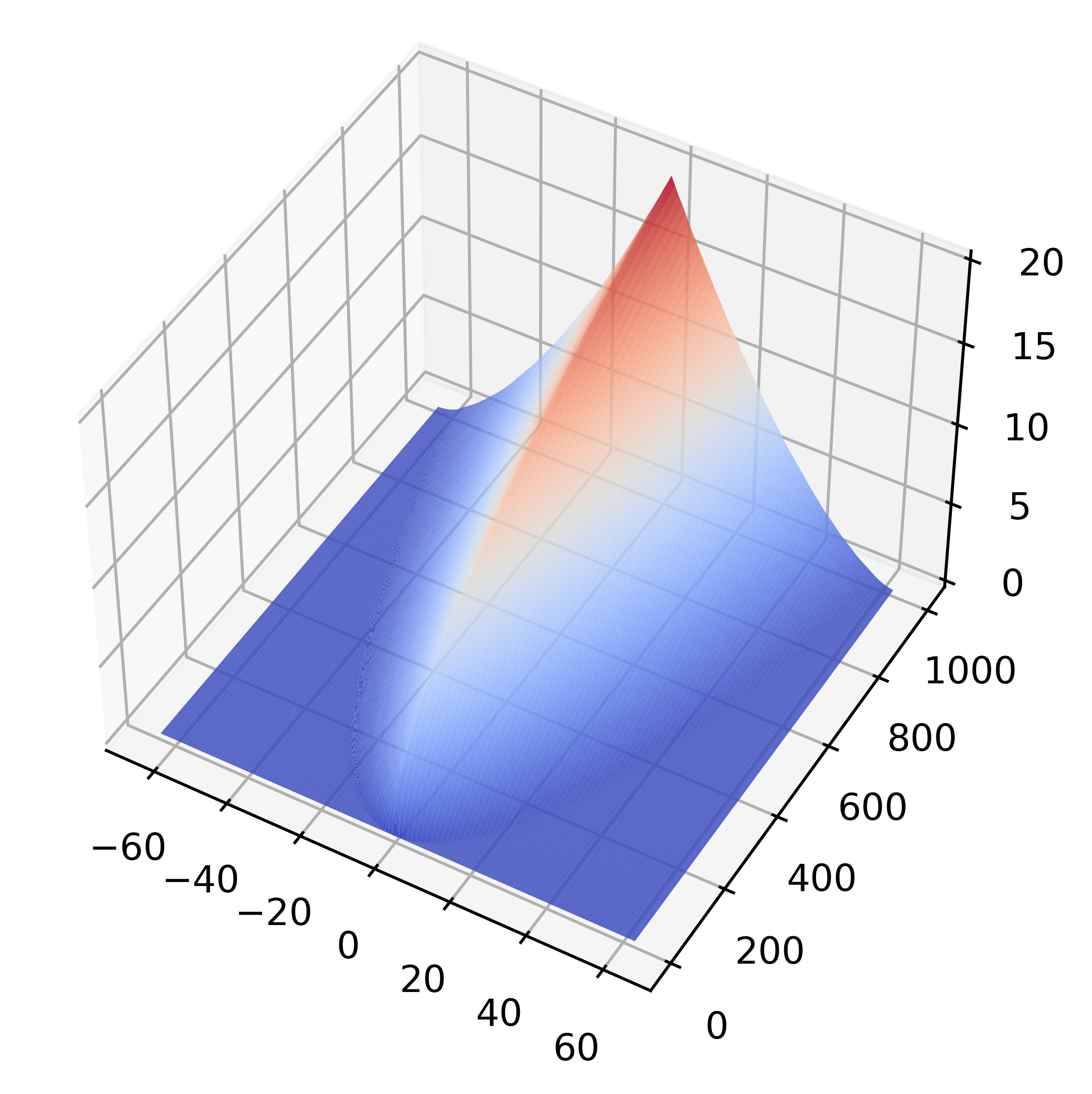} }}%
    \caption{Random height function $\textup{H}$ and expected height function $\mathbb{E}\textup{H}$ for $t\leq 1000$.}%
\label{Random&ExpectedSurface1000Plancherel}%
\end{figure}

\subsection{Plancherel growth process}\label{SectionPlancherelgrowthprocess}
We follow the discussion from the introduction with greater detail. Define the Young graph (Also called Young lattice, see \cite[Chapter 3]{BoO}) to be the graph with vertices $\mathbb{Y}$ and edges between $\mu,\lambda \in \mathbb{Y}$ if and only if $\mu \nearrow \lambda$. This is graded graph with levels $\mathbb{Y}_n$ for each non-negative integer $n$. From this perspective, we can understand the Plancherel growth process as a random walk on the Young graph with starting point the empty partition $\emptyset$ and transition probabilities given by
\begin{equation*}
    p(\mu,\lambda):=\begin{cases}
        \frac{\dim(\mu)}{|\mu|\dim(\lambda)} & \textup{ if } \mu \nearrow \lambda,\\\
        0 &\, \textup{otherwise.}
    \end{cases}
\end{equation*}

The induced probability measure on the $n$th level of the Young graph is precisely the Plancherel measure. Moreover, the reader may observe the similitude between the transition probabilities and the coefficients of the transition measure of a given partition $\lambda$, this gives a good reason for the name of $m_K[\lambda]$. Notice that in this setting we can interpret the second coordinate of the height function $\textup{H}$ as either the time or the number of steps in the Markov chain. 

We are interested in describing the joint fluctuations of paths of this Markov chain on the Young graph. To this end, first define the moments of the random height function as
\begin{equation*}
    \mathcal{M}_{\alpha,k}^{\textup{P}}=\sqrt{\pi}\int_{-\infty}^{+\infty} x^k \big[\textup{H}(\sqrt{n}x,n\alpha)-\E\textup{H}(\sqrt{n}x,n\alpha)\big]\,dx
\end{equation*}

Similarly, define the corresponding moment of the conditioned GFF $\mathfrak{C}$ as
\begin{equation*}
    \mathcal{M}_{\alpha,k,1}^{\textup{CGFF}}=\int\limits_{\substack{z\in \mathbb{H},\, |z|^2=\alpha}} x(z)^k \mathfrak{C}(z)\frac{dx(z)}{dz}\,dz, \hspace{5mm}\textup{where }\hspace{1mm} x(z)=\tfrac{1}{2}(z+\bar{z}).
\end{equation*}

Furthermore, we can identify the upper-half plane $\mathbb{H}$ with the interior of a parabola $\{(x,y)\in \R\times\R_{>0}:x^2\leq4y\}$ via the map
\begin{equation*}
    \Omega(x,y):= \frac{x}{2}+i\sqrt{y-\frac{x^2}{4}}.
\end{equation*}
Notice that the map $\Omega$ has inverse $\Omega^{-1}(z)=(2\Re(z),|z|^2)$, in particular it maps horizontal lines from the interior of the parabola to semicircles in the upper half plane. Using this map we can take the pullback of the conditioned GFF to confine the fluctuations within the parabola. 
\begin{theorem}\label{ThmCGFFforPlancherel}
    Let $\textup{H}$ be the random height function associated with the Plancherel growth process, then, as $n\to \infty$
    \begin{equation*}
    \sqrt{\pi}\big[\textup{H}(\sqrt{n}x,nt)-\E\textup{H}(\sqrt{n}x,nt)\big]\to \mathfrak{C}(x,t).
    \end{equation*}
    That is, as $n \to \infty$, the collection of random variables $\big(\mathcal{M}_{\alpha,k}^{\textup{P}}\big)_{\alpha,k}$ converges, in the sense of moments, to $\big(\mathcal{M}_{\alpha,k,1}^{\textup{CGFF}}\big)_{\alpha,k}$.
\end{theorem}

\begin{figure}[h]
    \centering
    \subfloat[\centering Contour lines.]{{\includegraphics[scale=0.6]{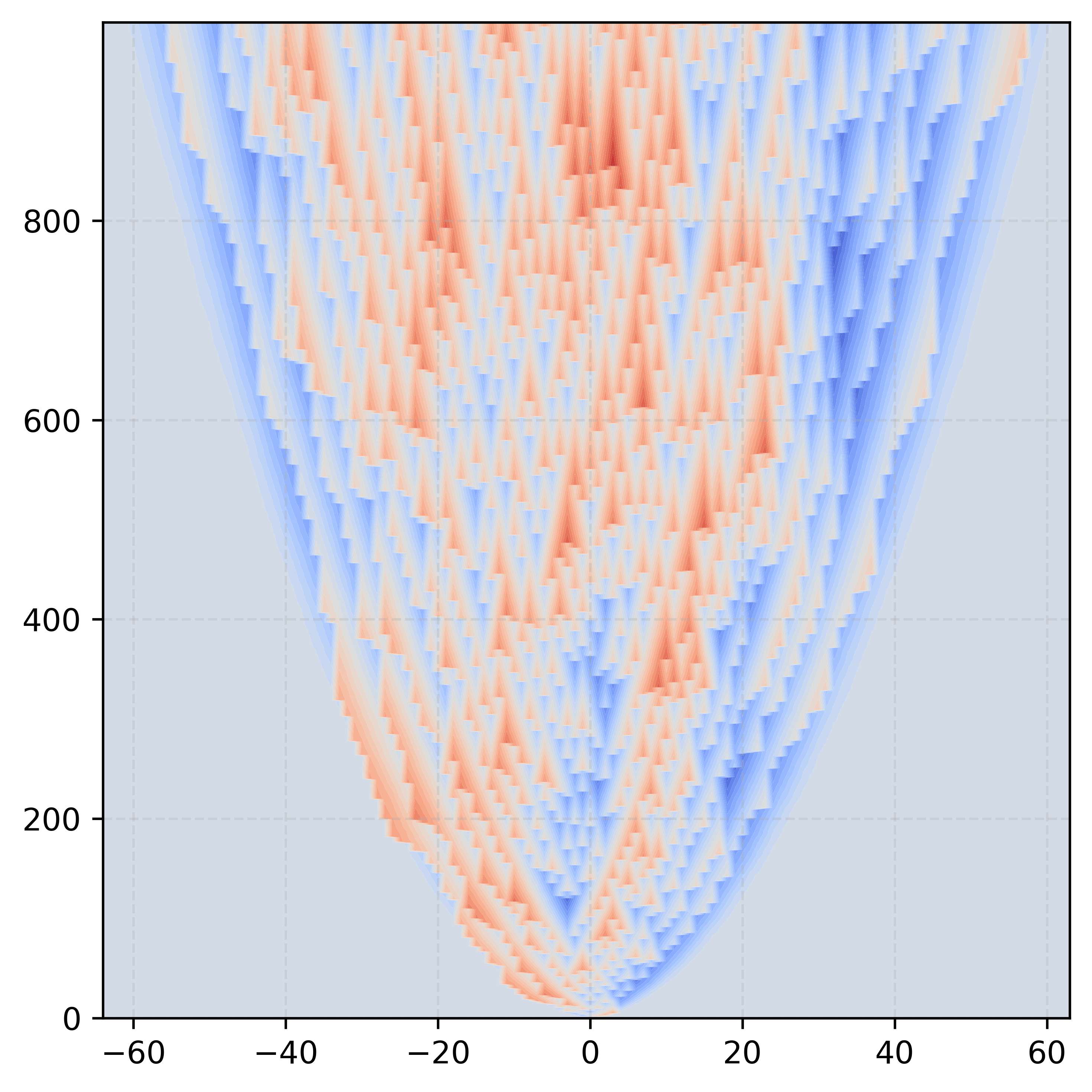} }}%
    \qquad
    \subfloat[\centering Surface fluctuations.]{{\includegraphics[scale=0.74]{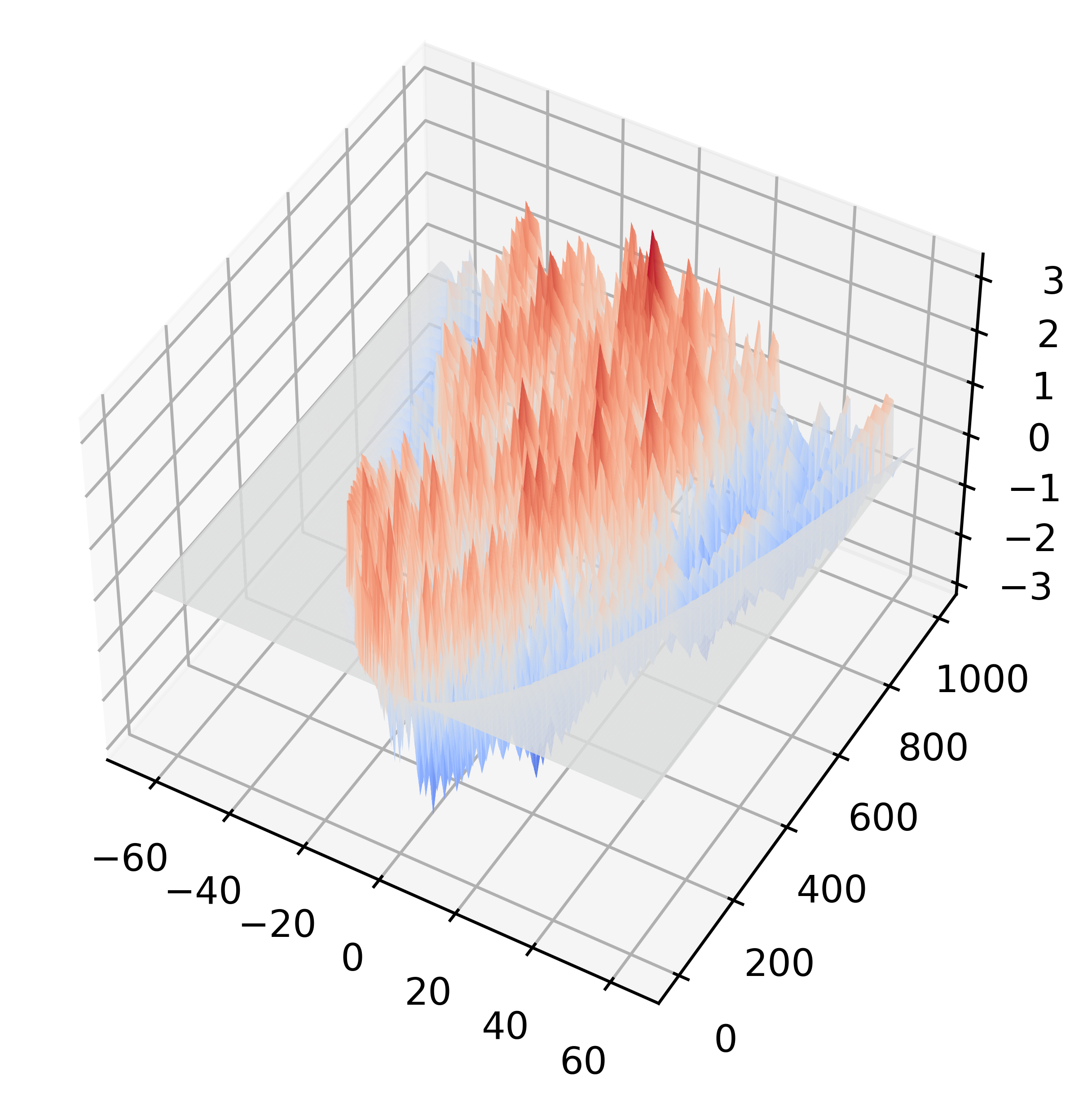} }}%
    \caption{Sample of $\sqrt{\pi}\big[\textup{H}(x,t)-\E\textup{H}(x,t)\big]$ with $t\leq 1000$ for Plancherel growth process.}
    \label{LimitFluctuations1000Plancherel}
\end{figure}

A different combinatorial perspective to understand the Plancherel growth process consist sampling an infinite sequence of  independently and identically distributed random variables uniformly from the interval $(0,1)$ and then apply the Robinson--Schensted--Knuth algorithm (RSK) to obtain a couple of SYT of same shape, the induced distribution on the shape is precisely the Plancherel growth process. In fact, we use this algorithm to produce the simulations in this section, that is, a simulation of a random height function and the limit shape can be found in figure \ref{Random&ExpectedSurface1000Plancherel}, while a simulation of the fluctuations, rescaled by $\sqrt{\pi}$, can be found in figure \ref{LimitFluctuations1000Plancherel}. We end this section by briefly mentioning how the previous theorem can be adapted to other representation-theoretic induced distributions with same asymptotic limit shape and fluctuations. 

\begin{example} \label{ExampleCGFFGelfand}
Following Example \ref{ExampleGelfandDist}, Méliot proved that the fluctuation induced by the Gelfand distributions are equal to the ones of the Plancherel distribution rescaled by a factor of $\sqrt{2}$. Hence Theorem \ref{ThmCGFFforPlancherel} implies that the fluctuations of the random height function associated with the Gelfand distribution will also converge to a conditioned GFF. 
\end{example}

\subsection{Fluctuations for characters of $S_\infty$}\label{SubsectionExtremeCharacters}

In this section we study a family of rescaled extreme characters from $S_\infty$. By a character of $S_\infty$ we mean a function $M:S_\infty \to \C$ such that
\begin{enumerate}
    \item It is normalized, that is, $M(e)=1$. 
    \item It is central, that is, for any $g$ and $h$ in $S_\infty $, $M(ghg^{-1})=M(h)$. 
    \item It is positive definite, that is, $\big[M(g_ig_j^{-1})\big]_{i,j=1}^n$ is an hermitian and positive-definite matrix for any $n\geq 1$ and $g_1,\dots,g_n\in S_\infty$.
\end{enumerate}
The space of characters of $S_\infty$ is clearly convex, its extreme points are the extreme characters of $S_\infty$ and  replace characters of irreducible representations in the infinite-dimensional setting. The classification of these extreme characters is given by Thoma's theorem \cite[Corollary 4.2]{BoO}. We have that these characters are parametrized by points on the set 
\begin{equation*}
    \mathcal{T}:=\big\{(\alpha,\beta)\in [0,1]^\N\times [0,1]^\N: \alpha_1\geq \alpha_2\geq \dots, \, \beta_1\geq \beta_2\geq \dots,\, \sum_{i=1}^\infty \alpha_i+\beta_i\leq 1\big\}.
\end{equation*}
For each pair $\upsilon=(\alpha,\beta)\in \mathcal{T}$ and $\sigma \in S_\infty$, the extreme characters of $S_\infty$ are given by
\begin{equation*}
M_\upsilon(\sigma)=\prod_{k=2}^\infty  \bigg(\sum_{i=1}^\infty \alpha_i^k +(-1)^{k-1} \sum_{i=1}^\infty \beta_i^k  \bigg)^{m_k}
\end{equation*}
where, for each $k\geq 2$, $m_k$ denotes the number of cycles of length $k$ in $\sigma$. Note that each character $M$ of $S_\infty$ induces a probability distribution $\rho_n$ on $\mathbb{Y}_n$ for every $n\geq 1$. Indeed, by taking
\begin{equation*}
    M\bigr|_{S_n}(\cdot)=\sum_{\lambda\in \mathbb{Y}_n} \rho_n(\lambda) \frac{\chi_\lambda(\cdot)}{\dim(\lambda)},
\end{equation*}
where $\chi_\lambda$ denotes the irreducible character of $S_n$ indexed by the partition $\lambda\in \mathbb{Y}_n$, then $\rho_n$ is a probability measure. 

Denoting $\gamma=1-\big(\sum_{i=1}^\infty \alpha_i+\beta_i\big)$, we can understand the triplet $(\alpha,\beta,\gamma)$ as the growth rates for an infinite SYT. For instance each $\alpha_i$ denotes the rate of growth for the $i$th row of a SYT,  $\beta_i$ denotes the rate of growth for the $i$-th column, while $\gamma$ denotes the rate for the Plancherel growth. In fact, this process can be formalized as follows. Consider a sequence of independently sampled random variables with distribution $\Prob(X=k)=\alpha_k$ for $k\geq 1$, $\Prob(X=-k)=\beta_k$ for $k\geq 1$ and $\Prob\big(X\in(a,b)\big)=\gamma(b-a)$ for $0\leq a< b \leq 1$, then we can apply a generalization of RSK \cite{VK} which induces the growth of SYT previously discussed. See also \cite{Sni4} for further details. For example, in the case $\gamma=1$ we recover the Plancherel growth process.  See figure \ref{SinftyGrowth} for an illustration of the rate growth induced by the triplet $(\alpha,\beta,\gamma)$.

Note that the previous model does not adapt directly to the setting in consideration of in this paper where the row and column growth of the collection of increasing partitions is of order $\sqrt{n}$. Rather than consider the distribution induced for fixed $(\alpha,\beta,\gamma)$ we will be interested in certain sequences of extreme characters of $S_\infty$ to ensure that both column and row growth is of order $\sqrt{n}$. In representation-theoretic terms, we will fix a sequence of characters in $S_\infty$ that converge to the regular character at speed $\sqrt{n}$.

Let $\upsilon(n)=\big(\alpha(n),\beta(n)\big)$ denote a sequence of elements in $\mathcal{T}$ such that
\begin{equation}\label{ConditionSinftyGFF}
    \frac{1}{n} \sum_{i=1}^\infty\frac{1}{\alpha_i(n)}\delta_{\sqrt{n}\alpha_i(n)}\to \mathcal{A}\, \hspace{3mm} \textup{ and } \hspace{3mm} \frac{1}{n} \sum_{i=1}^\infty \frac{1}{\beta_i(n)}\delta_{\sqrt{n}\beta_i(n)}\to \mathcal{B}
\end{equation}
where $\mathcal{A}$ and $\mathcal{B}$ are finite, compactly supported, measures on $\R$ and the converge occurs in a weak sense. Note that condition (\ref{ConditionSinftyGFF}) implies that $\gamma(n)=1-\sum_{i=1}^\infty \alpha_i(n)-\sum_{i=1}^\infty\beta_i(n)$ converges as $n\to \infty$ to $1-\big(\int_\R x^2 \mathcal{A}(dx)+\int_\R x^2 \mathcal{B}(dx)\big)$. Hence, for each $n\geq 1$, the extreme character $M_{\upsilon(n)}$ induces a probability distribution $\rho_n$ on $\mathbb{Y}_n$. We are interested in the sequence of probability distributions $\rho=(\rho_n)_{n\in \N}$. We will show, in section \ref{SectionProofsOfApplications}, that this is a sequence of CLT-appropriate probability distributions, in particular the formal power series $F_\rho$ is well defined. Furthermore, w need to identify the domain of fluctuation with the upper-half plane, this is achieved via the following proposition, whose proof is postponed to section \ref{SectionProofsOfApplications}. 

\begin{figure}[h]
    \centering
    \includegraphics[scale=0.25]{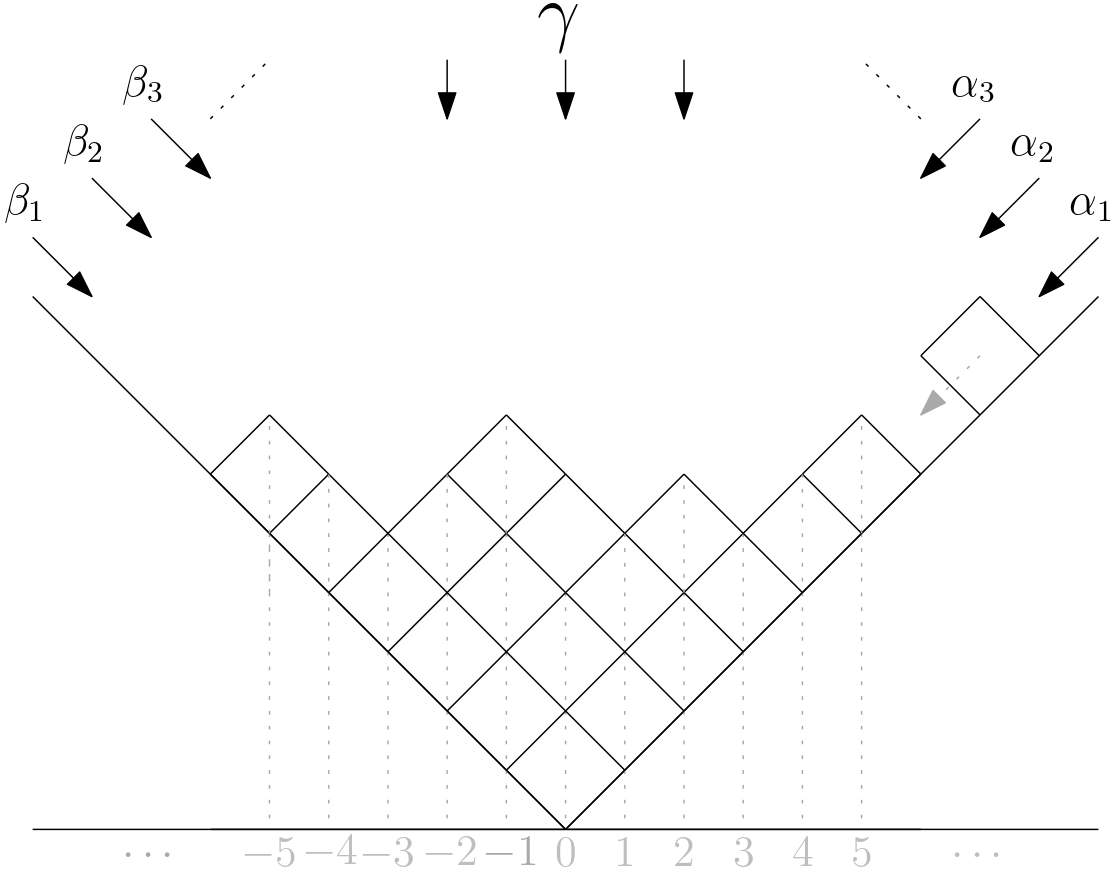}
    \caption{Illustration of stochastic process induced from $(\alpha,\beta,\gamma)$.}
    \label{SinftyGrowth}
\end{figure}

\begin{proposition}\label{PropInverseSinfty}
   Let $\rho=(\rho_n)_{n\in \N}$ as in equation (\ref{ConditionSinftyGFF}) and let $F(z)=F_\rho(z)$ be as in Definition \ref{defLLNappropriate}, then for any $y\in \R$ and $\alpha\in[0,1]$ the equation
   \begin{equation*}
       \frac{\alpha}{z}+zF(z)=y
   \end{equation*}
   has at most one root $z\in \mathbb{H}$. Let $D_F\subseteq \R\times [0,1]$ be the set of pairs $(y,\alpha)$ such that this root exists. Then the map $D_F \to \mathbb{H}$ from such a pair to such a root is a diffeomorphism. 
\end{proposition}

We denote $z\to (y_F,s_F)$ an inverse of the map produced in Proposition \ref{PropInverseSinfty}. Define the moments of the random height function as
\begin{equation*}
    \mathcal{M}_{\alpha,k}^{\textup{S}}=\sqrt{\pi}\int_{-\infty}^{+\infty} x^k \big[\textup{H}(\sqrt{n}x,n\alpha)-\E\textup{H}(\sqrt{n}x,n\alpha)\big]\,dx
\end{equation*}

Similarly, define the corresponding moment of the conditioned GFF $\mathfrak{C}$ as
\begin{equation*}
    \mathcal{M}_{\alpha,k,2}^{\textup{CGFF}}=\int\limits_{\substack{z\in \mathbb{H},\, s_F(z)=\alpha}} y_F(z)^k \mathfrak{C}(z)\frac{dy_F(z)}{dz}\,dz.
\end{equation*}

\begin{theorem}\label{ThmCGFFforSinfty}
    Let $\textup{H}$ be the random height function associated with $(\rho_n)_{n\in \N}$ defined above, then as $n\to \infty$
    \begin{equation*}
    \sqrt{\pi}\big[\textup{H}(\sqrt{n}x,nt)-\E\textup{H}(\sqrt{n}x,nt)\big]\to \mathfrak{C}(x,t).
    \end{equation*}
    That is, as $n \to \infty$, the collection of random variables $\big(\mathcal{M}_{\alpha,k}^{\textup{S}}\big)_{\alpha,k}$ converges, in the sense of moments, to $\big(\mathcal{M}_{\alpha,k,2}^{\textup{CGFF}}\big)_{\alpha,k}$.
\end{theorem}

Besides the Plancherel distribution, obtained when $\gamma=1$, we show another example of a classical distribution obtained from this model. 

\begin{example}\label{ExampleSchurWeylDist}
Let $(D_n)_{n\in \N}$ be a collection of integers and $c\geq 0$ a real number such that $n/D_n^2\to c^2$ as $n\to \infty$. Let $\alpha_i(n)=\tfrac{1}{D_n}$ for $i=1,\dots,D_n$, $0$ for any other value of $i$, $\beta_i=0$ for any $i$ and $\gamma=0$. This induces the Schur--Weyl distribution on integer partitions, Theorem \ref{ThmCGFFforSinfty} shows that the fluctuations of the height function are given by a conditioned GFF. It follows from the definition that
\begin{equation*}
    \textup{A}_{\rho_n}(\vec{x})\approx\prod_{k=1}^\infty \Big(1+\sum_{i=1}^\infty n^{\frac{i(k-1)}{2}} D_n^{i(1-k)}  \frac{x_k^{i}}{i!}\Big)
\end{equation*}

A direct computation shows that $\lim_{n \to \infty} \partial_i\ln\textup{A}_{\rho_n}(\vec{x})\bigr|_{\vec{x}=0} = c^{i-1}$, similarly we can check conditions (2) and (3) of Definition \ref{defCLTappropriate}. This verifies that the induced distribution is actually CLT-appropriate and hence Theorem \ref{TheoremMultilevelCLT} ensures that it satisfies a multilevel CLT.

\begin{figure}[h]
    \centering
    \subfloat[\centering Contour lines.]{{\includegraphics[scale=0.6]{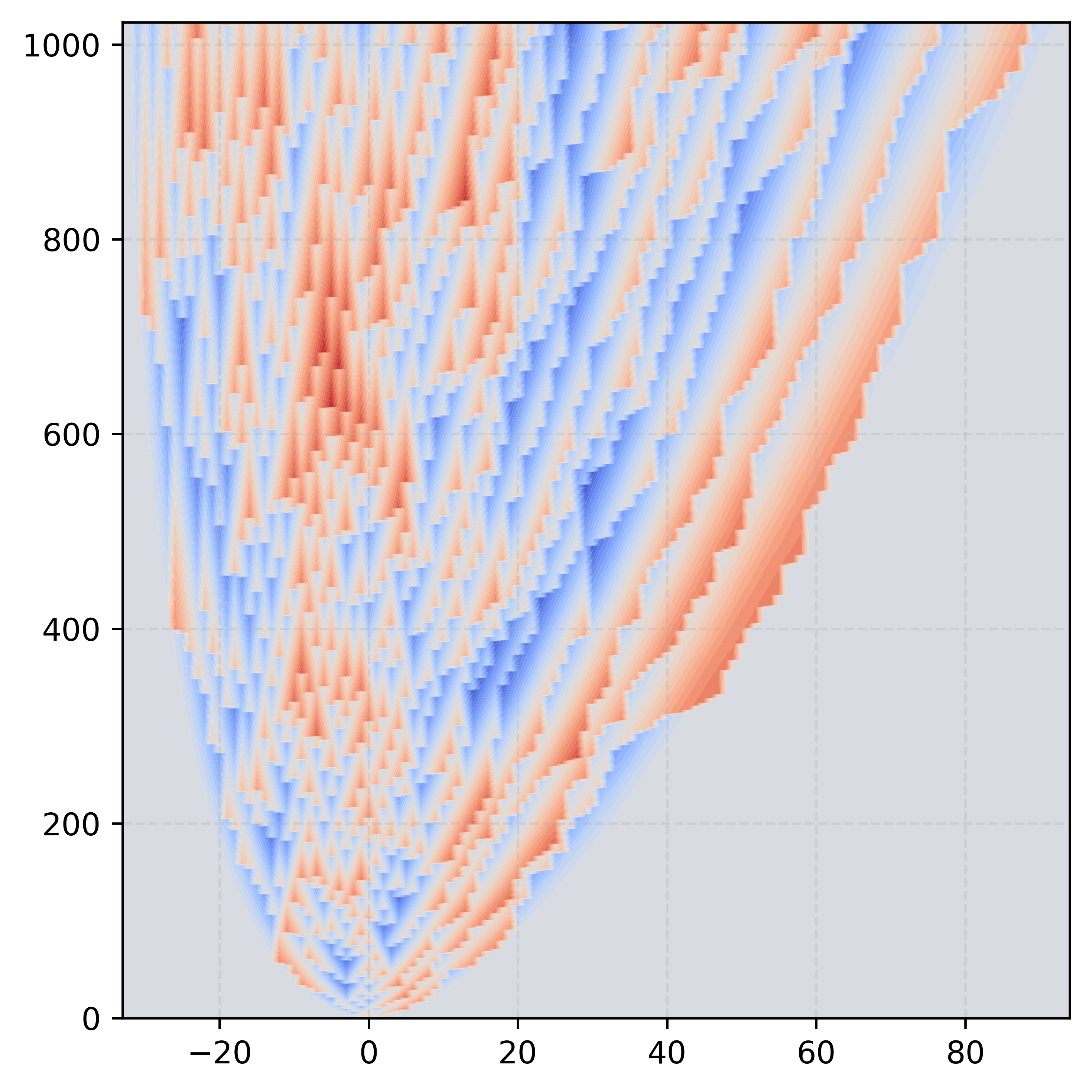} }}%
    \qquad
    \subfloat[\centering Surface fluctuations.]{{\includegraphics[scale=0.74]{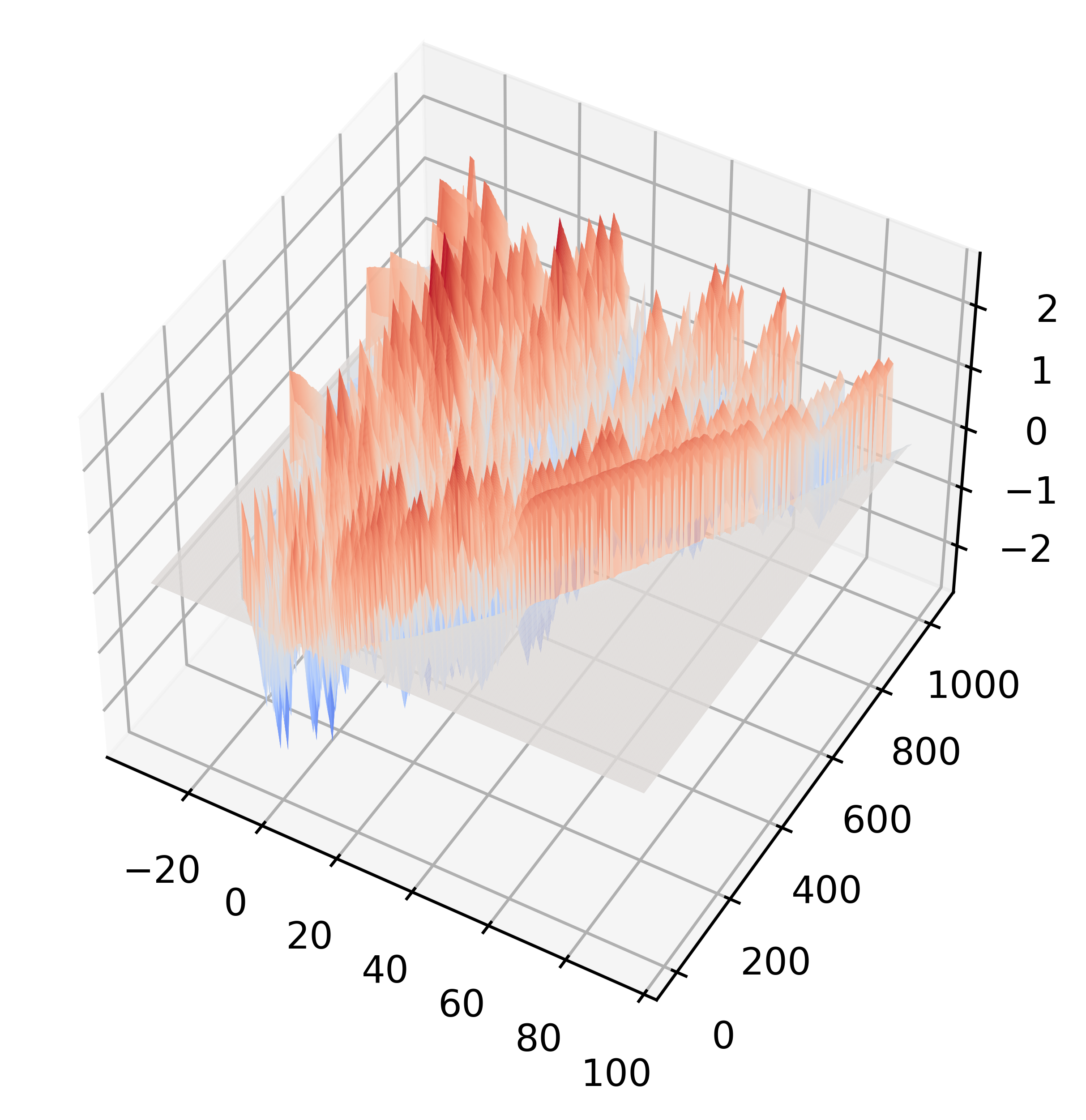} }}%
    \caption{Sample of $\sqrt{\pi}\big[\textup{H}(x,t)-\E\textup{H}(x,t)\big]$ with $t\leq 1024$ under Schur--Weyl distribution with parameter $c=1$.} 
    \label{LimitFluctuationsSchurWeyl}
\end{figure}

This stochastic process corresponds to the growth  of a Young diagram obtained by using RSK on random variables sampled independently and uniformly at random from $\{1,\dots,D_n\}$. If $c=0$ the rate of repeated sampled values is small enough that the process is closer to the Plancherel growth process, so the limit shape is the Vershik--Kerov--Logan--Shepp curve. See also \cite{Mel11, Mel10b} for the single-level CLT and further discussion regarding the representation-theoretic meaning of this distribution. See figure \ref{LimitFluctuationsSchurWeyl} for a simulation in the case $c=1$.
\end{example}

\subsection{Random standard Young tableaux of fixed shape}\label{subsectionFixedShape}

In this section we are interested in describing the fluctuations for the height function of standard Young tableaux of fixed shape. For instance, fix a deterministic sequence of integer partitions $(\lambda^n)_{n \in \N}$ that converges, in the space continuous Young diagrams, to some fixed shape $\omega$. This sequence trivially satisfies a CLT and hence Theorem \ref{TheoremCLT} ensures that it is CLT-appropriate. For each partition $\lambda^n$ we can sample uniformly at random a SYT of shape $\lambda^n$, that is, we obtain a random sequence of integer partitions which we can study via a height function. In fact, Theorem \ref{TheoremMultilevelLLN} ensures the existence of a limiting surface while Theorem \ref{TheoremMultilevelCLT} allows us to describe the fluctuations. 

Similarly to the previous sections, we start by identifying the domain of fluctuation with the upper-half plane using the following proposition whose proof is postponed to section \ref{SectionProofsOfApplications}.
\begin{proposition}\label{PropInverseFixedShape}
Denote $C(z)$ the Stieltjes transform of the transition measure of the diagram $\omega$ as above. Then, for any $y\in \R$ and $\alpha\in [0,1]$, the equation
\begin{equation*}
    \frac{1}{z}+\frac{\alpha-1}{C(1/z)}=y
\end{equation*}
has at most one root $z \in \mathbb{H}$. Let $D_F\subseteq \R^2$ be the set of pairs $(y,1-\frac{1}{\alpha})$ such that this root exists, then the map $D_F\to \mathbb{H}$ from such a pair to such a root is a diffeomorphism. 
\end{proposition}

\begin{figure}[h]
    \centering \hspace*{-3mm}\subfloat[\centering Heart shape, for \ensuremath{\lambda^0=[2,1]}.]{{\includegraphics[scale=0.7]{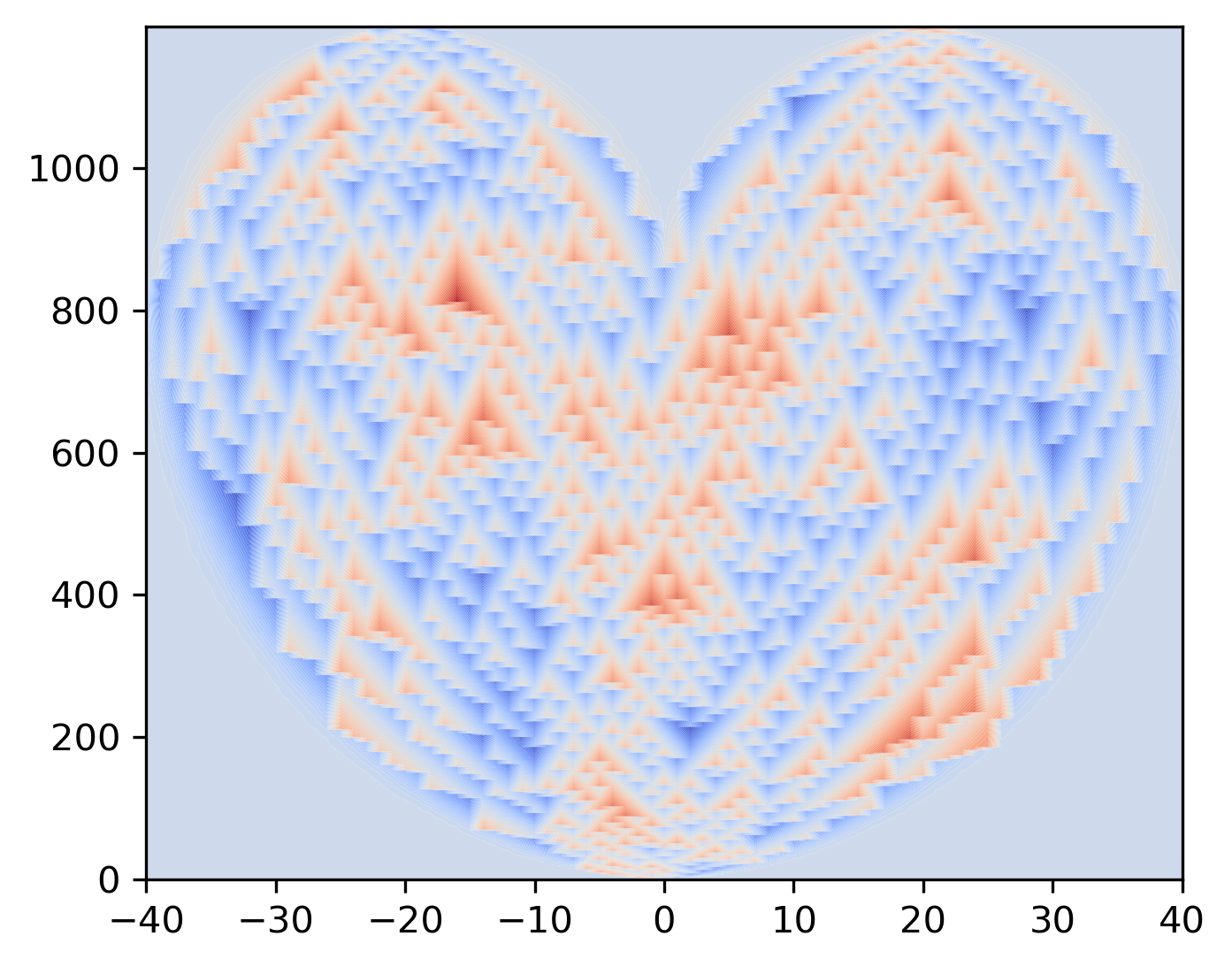}}}%
    \hspace*{-0mm}\subfloat[\centering Pipe shape, for \ensuremath{\lambda^0=[10,2]}.]{{\includegraphics[scale=0.7]{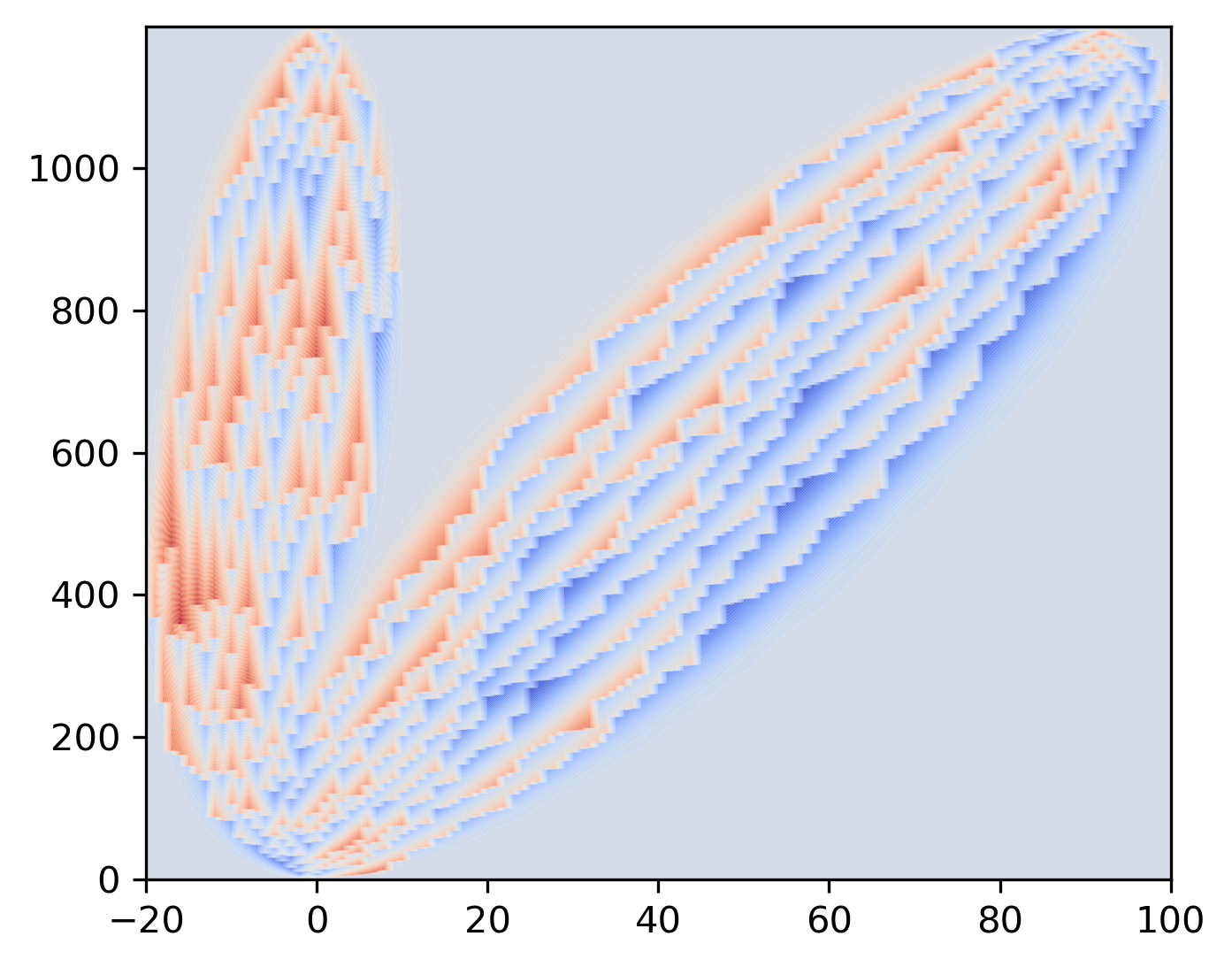}}}%
    \caption{Contour lines of height function fluctuations for random SYT of size $1200$.}%
    \label{OtherShapes}%
\end{figure}

Denote $z\to (y_F,\hat{s}_F)$ an inverse of the map produced in Proposition \ref{PropInverseFixedShape}. Define the moments of the random height function as
\begin{equation*}
    \mathcal{M}_{\alpha,k}^{\textup{Fix}}=\sqrt{\pi}\int_{-\infty}^{+\infty} x^k \big[\textup{H}(\sqrt{n}x,n\alpha)-\E\textup{H}(\sqrt{n}x,n\alpha)\big]\,dx
\end{equation*}

Similarly, define the corresponding moment of the conditioned GFF $\mathfrak{C}$ as
\begin{equation*}
    \mathcal{M}_{\alpha,k,3}^{\textup{CGFF}}=\int\limits_{\substack{z\in \mathbb{H},\, \hat{s}_F(z)=\tfrac{1}{1-\alpha}}} y_F(z)^k \mathfrak{C}(z)\frac{dy_F(z)}{dz}\,dz.
\end{equation*}

\begin{theorem}\label{ThmCGFFforFixedShape}
    Let $\textup{H}$ be the random height function associated with the random SYT of fixed shape, then, as $n\to \infty$
    \begin{equation*}
    \sqrt{\pi}\big[\textup{H}(\sqrt{n}x,nt)-\E\textup{H}(\sqrt{n}x,nt)\big]\to \mathfrak{C}(x,t).
    \end{equation*}
    That is, as $n \to \infty$, the collection of random variables $\big(\mathcal{M}_{\alpha,k}^{\textup{Fix}}\big)_{\alpha,k}$ converges, in the sense of moments, to $\big(\mathcal{M}_{\alpha,k,3}^{\textup{CGFF}}\big)_{\alpha,k}$.
\end{theorem}

\begin{figure}[h]
    \centering
    \subfloat[\centering Contour lines.]{{\includegraphics[scale=0.6]{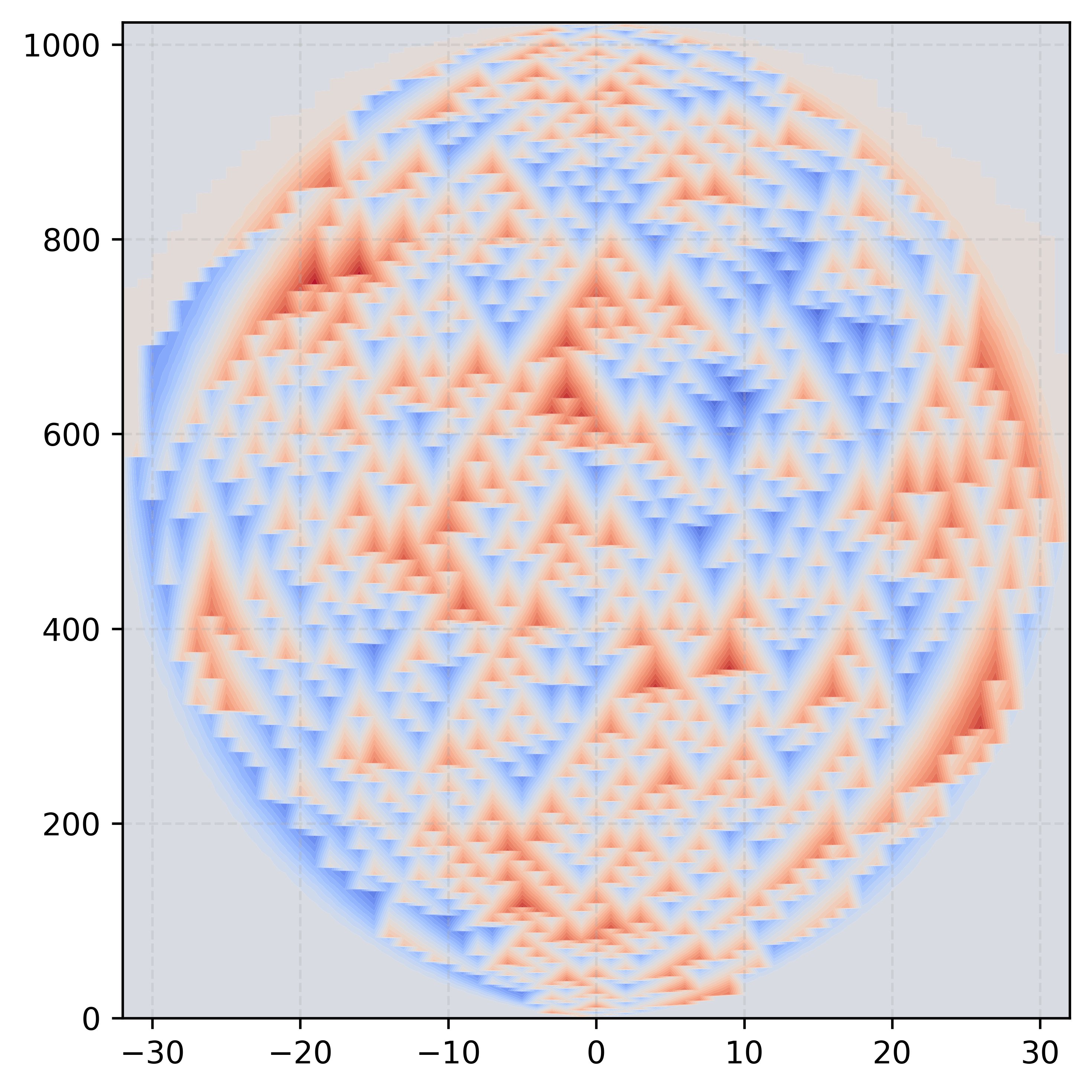} }}%
    \qquad
    \subfloat[\centering Surface fluctuations.]{{\includegraphics[scale=0.74]{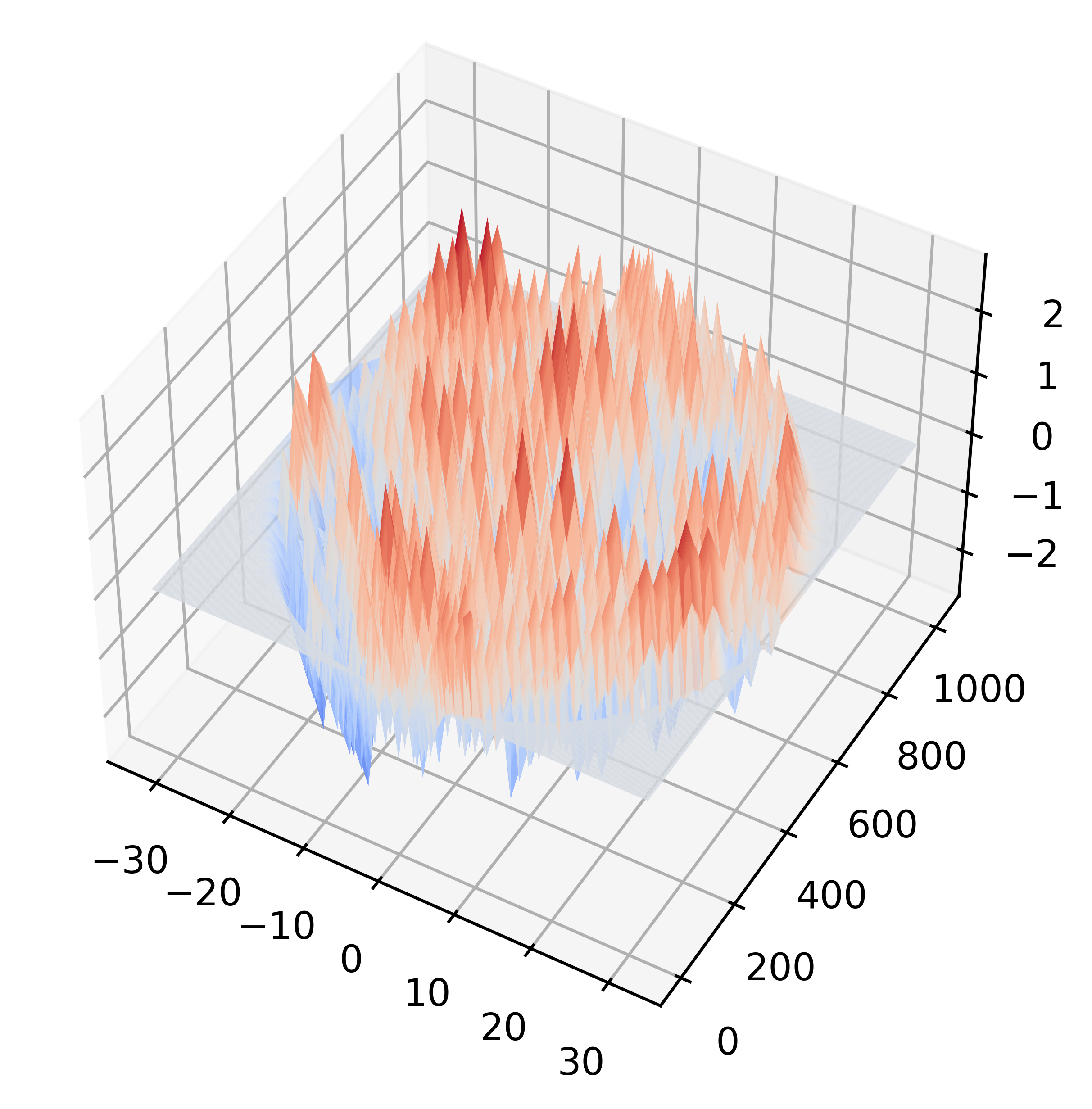} }}%
    \caption{Sample of $\sqrt{\pi}\big[\textup{H}(x,t)-\E\textup{H}(x,t)\big]$ for square of size $32 \times 32$.}
    \label{LimitFluctuationsSquare}
\end{figure}

Of particular interest is the situation where the shape of the sequences of partitions is fixed. For instance, fix a partition $\lambda^0$, we can produce a sequence of partitions whose rescaled shape is equal to $\lambda^0$ as follows. At time $1$ we take $\lambda^1=\lambda^0$ then at time $n$ we fix $\lambda^n$ to be equal to the partition obtained by dividing each box in $\lambda^0$ in $n^2$ boxes. This will generate a sequence of partitions whose rescaled shapes are all equal to $\lambda^0$. For instance at times $1$ and $2$ we have
\begin{equation*}
\lambda^1=(\lambda_1\geq\lambda_2\geq\dots\geq \lambda_r) \hspace{2mm} \textup{ and } \hspace{2mm}\lambda^2=(2\lambda_1\geq 2\lambda_1 \geq 2\lambda_2\geq 2\lambda_2\geq \dots \geq 2\lambda_r \geq 2\lambda_r).
\end{equation*}

For each $n$ the partition $\lambda^n$ has rescaled shape $\lambda^0$. We can then sample a SYT uniformly at random of rescaled shape $\lambda^n$. See figures \ref{OtherShapes} and \ref{LimitFluctuationsSquare} for some simulations of the fluctuations of random SYT of fixed shape, which are produced by using a probabilistic algorithm, the hook walk, developed by Greene--Nijenhuis--Wilf in \cite{GNW}.

\begin{remark}
While all the simulations shown in this section corresponds to the case where the rescaled shape of the sequence of Young diagrams is fixed, our theorem does not has such a restriction. For instance, the staircase partitions $\lambda^n=(\lfloor\sqrt{n}\rfloor,\lfloor\sqrt{n}\rfloor-1,\lfloor\sqrt{n}\rfloor-2,\dots,2,1)$ converges to a horizontal line $h(x)$ and Theorem \ref{ThmCGFFforFixedShape} still applies in this setting
\begin{equation*}
    h(x)=\begin{cases}
1 & \text{ if } |t|\leq \sqrt{2},\\
|t| & \text{ if } |t|\geq \sqrt{2}.\\
\end{cases}
\end{equation*}
\end{remark}

We end this section by explaining how to describe the limiting surfaces of the height function.

\begin{example} \label{ExampleComputingLimitingShapeSquare}
We will focus on the limiting shape and fluctuations of a random square SYT. This is indeed the focus of the paper of Pittel and Romik \cite{PiR} where the limiting shape is found by  solving a variational problem. Other formulas have been found by using determinantal point processes \cite[Remark 10]{BBFM}. While already various methods are available in the literature to compute these limiting surfaces, here we briefly describe Biane's original representation-theoretic method \cite{Bi} to describe these surfaces and provide an example.

One of the reasons that the computation of the limiting surface in the case of random SYT of fixed partition shape $\lambda^0$ is simpler, is that the computation of the Stieltjes transform of the Kerov's transition measure is trivial. For instance, in the square case it is a direct computation, using the identity given in Example \ref{ExamplesOfMarkovKrein}, that the Stieltjes transform of the transition measure in the square setting is
\begin{equation*}
    C(z)=\frac{1}{2(z-1)}+\frac{1}{2(z+1)}=\frac{z}{z^2-1}.
\end{equation*}
By inverting this function, using Theorem \ref{TheoremMultilevelLLN} and computing the inverse of again, we get that the Stieltjes transform of transition measure of the level lines of the limiting surface are given by 
\begin{equation*}
    C_\alpha(z)=\frac{(2\alpha-1)z+\sqrt{z^2+4\alpha^2-4\alpha}}{2\alpha(z^2-1)}\hspace{2mm}\textup{ for }\hspace{2mm} \alpha\in[0,1].
\end{equation*}

It is then enough to compute the inverse of the Markov--Krein transform, that is, multiply by $z$ and taking the logarithm, and finally compute the inverse Stieltjes transform to obtain the level lines of the limiting surface. See figure \ref{Random&ExpectedSurfaceSquare} for a sample of the height function and its limiting shape.

\begin{figure}[h]
    \centering
    \subfloat[\centering Random height function.]{{\includegraphics[scale=0.67]{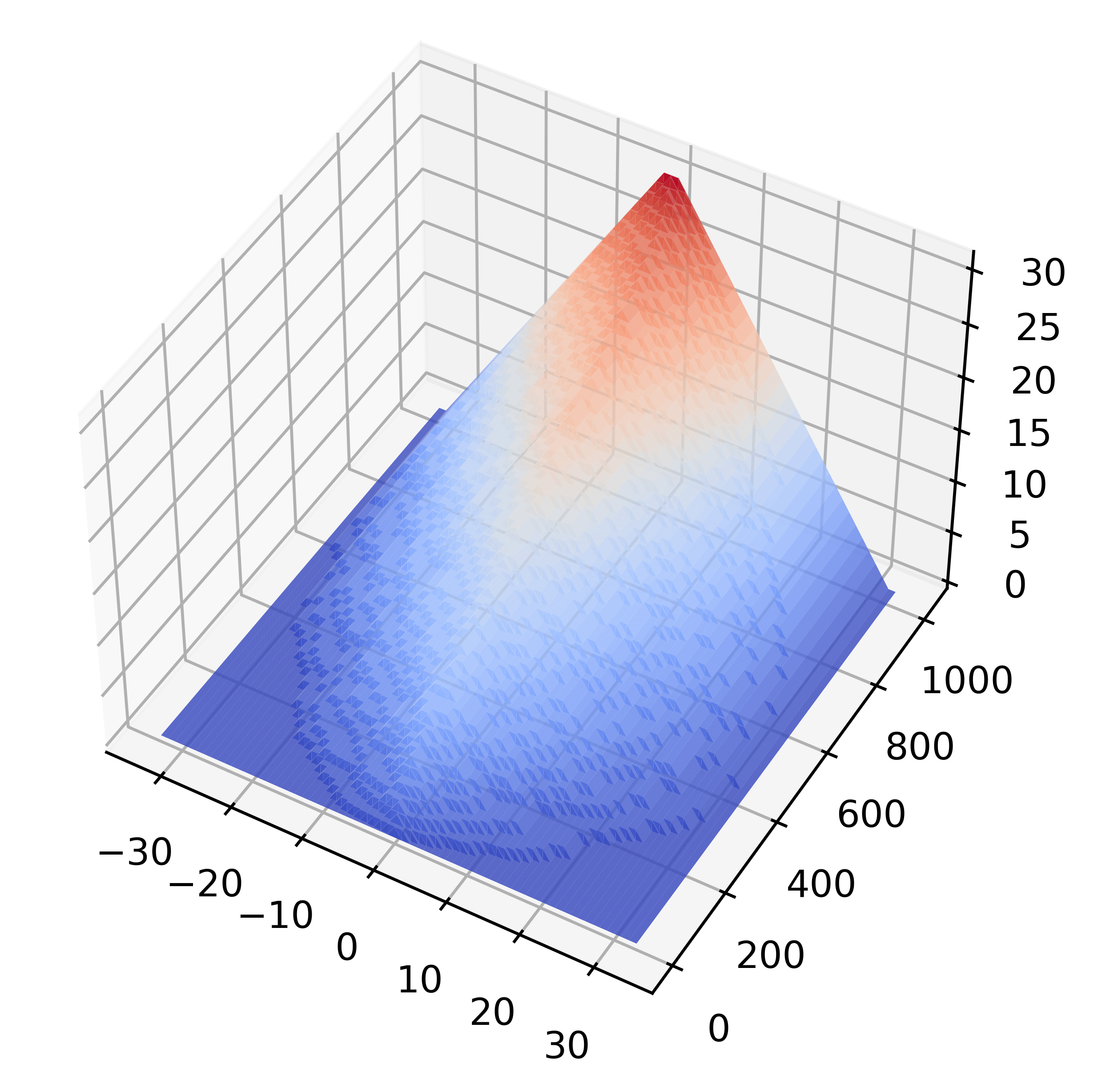} }}%
    \qquad
    \subfloat[\centering Expected height function.]{{\includegraphics[scale=0.67]{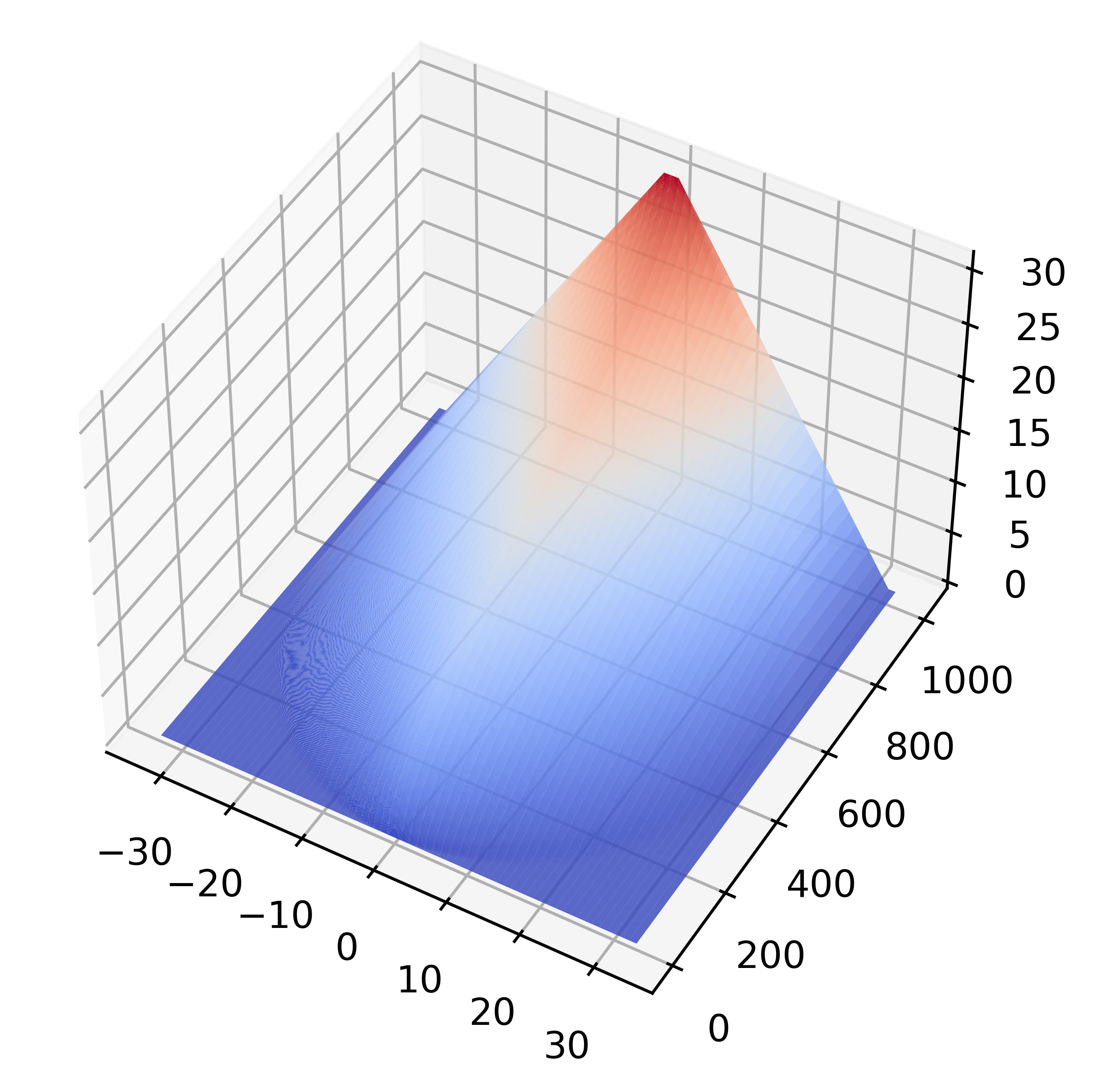} }}%
    \caption{Random height function $\textup{H}$ and expected height function $\mathbb{E}\textup{H}$ for square of size $32 \times 32$.}%
    \label{Random&ExpectedSurfaceSquare}%
\end{figure}

More generally, the following computation can be carried out to  describe the limiting surfaces in the general setting. By inverting the Stieltjes transform of the transition measure we recover the function $F_\rho$, at this point Theorem \ref{TheoremMultilevelLLN} describes the moments, hence the  Stieltjes transform, of the transition measures of the level lines of the limiting surface. Taking the inverse Markov--Krein transform followed by inverting the Stieltjes transform gives the description. Alternatively, if one prefers to avoid computing the inverse of the Markov--Krein transform is to use Lemma \ref{CovarianceNewCoordiantes} to directly obtain the moments of the derivatives of the level lines of the limiting surface. 
\end{example}

\section{Expansion of elements in the Gelfand–Tsetlin algebra} \label{SectionTechnicalLemmasProofs}

In this section, we develop asymptotic expansions for certain sequences of elements in the algebra $\C[S_n]$.
In the one-level setting, corresponding to Theorem \ref{TheoremLLN} and Theorem \ref{TheoremCLT}, we can restrict our attention to the center of the algebra $\C[S_n]$ which we denote $Z\big[\C[S_n]\big]$. 
On the other hand, for the multilevel setting, we will need to work with elements on $Z\big[\C[S_{n_i}]\big]$ for $1\leq i \leq s$.  We will need to develop asymptotics for sequences of elements in the Gelfand--Tsetlin subalegebra of $\C[S_n]$, denoted
\begin{equation*}
    \textup{GZ}_n:=\big\langle Z\big[\C[S_{n_1}]\big],\dots,Z\big[\C[S_{n_s}]\big]\big\rangle.
\end{equation*}
The Gelfand--Tsetlin subalgebra is the maximal commutative subalgebra of $\C[S_n]$, see \cite{OV} or \cite{CST} for an introduction. While a few results concerning the expansion of operators on $Z\big[\C[S_{n}]\big]$, the one-level setting, are already present in the literature, we will need to rework the theory with two objectives in mind:
\begin{enumerate}
    \item We need explicit leading coefficients in the expansion with the goal of computing explicit formulas for the mixed moments. In particular, with the objective of explicitly calculating the covariances of the Gaussian processes obtained in Theorem \ref{TheoremCLT} and Theorem \ref{TheoremMultilevelCLT}.
    \item We need to generalize the expansions to take into account the multilevel setting rather than the  single-level setting. That is, we will work on the Gelfand--Tsetlin subalgebra rather than the central algebras of $\C[S_n]$. 
\end{enumerate}
\subsection{The operator}

In his celebrated paper \cite{Bi} Biane, using the language of free probability, introduced an operator that allows to compute joint moments of transition measures of Young diagrams, he constructed the following elements of $Z\big[\C[S_n]\big]$.

\begin{definition} \label{DefinitionOfDk}
Let $\Gamma(n)$ be the $(n+1)\times(n+1)$ matrix with entries on $\C[S_n]$

$$\Gamma(n):=\begin{bmatrix}
    0 & 1 & 1 & 1 & \cdots & 1\\
    1 & 0 & (1 \,2) & (1 \,3) & \cdots & (1 \,n)\\
    1 & (2 \,1) & 0 & (2\, 3) & \cdots & (2\, n)\\
    1 & (3\, 1) & (3\, 2) & 0 & \cdots & (3\, n)\\
    \vdots & \vdots & \vdots & \vdots & \ddots & \\
    1 & (n\, 1) & (n\, 2) & (n\, 3) & \cdots & 0 
\end{bmatrix}.$$

We define $D_k$ to be the element on $Z\big[\C[S_n]\big]$ given by $D_k:=\frac{1}{n+1}\textup{tr}(\Gamma(n)^k)$, equivalently,

\begin{equation*}
D_k=\frac{1}{n+1} \sum_{0\leq i_1\neq i_2\neq \dots \neq i_k\neq i_1\leq n} \big((i_1\, i_2)(i_2\, i_3)\dots (i_k\, i_1)\big).
\end{equation*}
\end{definition}

The purpose of the element $D_k$ becomes clear with the following lemma.  

\begin{lemma}[\cite{Bi}] \label{BasicTraceLemma}
Consider the operator $\mathcal{D}_k$ acting on characters $\chi$ of the symmetric group via $\mathcal{D}_k\chi(g) = \chi(D_k g)$, then
\begin{equation*}
\E_{\rho}\bigg[\int_\R x^k m_K[\lambda](dx)\bigg]=\big[\mathcal{D}_k M_{\rho}(g)\big] \bigr|_{g=e}.
\end{equation*}
\end{lemma}

The previous lemma is the fundamental tool we use to study the random distributions $m_K[\lambda]$. From a free probability point of view the elements of $\C[S_n]$ are non-commutative random variables while $M_\rho$ is a non-commutative expectation, see \cite[Example 1.4 (4)]{NS}. The main interest of the operator $\mathcal{D}_k$ is that it provides explicit combinatorial formulas to compute joint moments of the random variables $\int_\R x^k m_K[\lambda](dx)$, in fact, a direct application of Lemma \ref{BasicTraceLemma} joint with the observation that the elements $D_k$ are in the center of the group algebra is that

\begin{corollary}
Let $r\geq 1$ and $k_1,\dots,k_r$ a sequence of integers, then
\begin{equation*}
\E_{\rho}\bigg[\prod_{i=1}^r\int_\R x^{k_i} m_K[\lambda](dx)\bigg]=\Big[\prod_{i=1}^r\mathcal{D}_{k_i} M_{\rho}(g)\Big] \biggr|_{g=e}.
\end{equation*}
\end{corollary}

The previous corollary allows to compute joint moments of the random variables
\begin{equation*}
  \big(X_k\big)_{k\in \N} = \Big(\int_\R x^k m_K[\lambda](dx)\Big)_{k\in \N}.
\end{equation*}
In particular, we can also compute cumulants of these random variables, see Lemma \ref{LemmaCumulantPermCumulant}. For the interested reader, the construction of the element $D_k$ on Definition \ref{DefinitionOfDk} is based on the Jucys--Murphy elements of the group algebra of the symmetric group. An alternative expression for $D_k$ can be given using these elements, we provide a short proof.
\begin{lemma}
Let $X_n=(1\,n)+(2\,n)+\dots+(n-1\,n) \in \C[S_n]$ be the $n$th Jucys--Murphy element and let $\E_{n-1}:\C[S_n]\to \C[S_{n-1}]$ defined by $\E_{n-1}[g]=g$ if $g\in S_{n-1}$  and $0$ if $g\in S_n\backslash S_{n-1}$, then
\begin{equation*}
D_k=\E_{n}[X_{n+1}^k].
\end{equation*}
\end{lemma}

\begin{proof}
Certainly both elements $D_k$ and $\E_{n}[X_{n+1}^k]$ are in the center of the group algebra of the symmetric group, hence by Schur's lemma both elements act as constants on the irreducible representations. Furthermore, it is shown in \cite[Theorem 9.23]{HoO} that the characters of these elements coincide over irreducible representations and hence over any representation. By choosing the regular representation, which is faithful, we conclude that $D_k=\E_{n}[X_{n+1}^k]$.
\end{proof}

\begin{remark}
Another approach to understand the operator $\mathcal{D}_k$ is the following. By defining the maps $\iota_e:V\to S_{n+1} \otimes V$, $\iota_e(v)=e\otimes v$ and $\pi_e:S_{n+1}\otimes V\to V$, $\pi_e(u\otimes v)=\E_{n}[u]\cdot v$. And using that the action of $X_{n+1}^k$ is equal to the action of $\Gamma(n)^k$ in $S_{n+1}\otimes V$, see \cite[Page 147]{Bi}. We can decompose the action of the operator $E_{n}[X_{n+1}^k]$ as $E_{n}[X_{n+1}^k]\cdot v=\pi_e\big[X_{n+1}^k\cdot\iota_e(v)\big]=\pi_e\big[\Gamma(n)^k\cdot\iota_e(v)\big]=\big[\Gamma(n)^k\big]_{1,1}\cdot v$.   Which further shows that the $(1,1)$-matrix entry of the $k$th power of the matrix $\Gamma(n)$ coincides with $D_k$. In fact, we believe that all diagonal elements of $\Gamma(n)^k$ are equal for any $k\geq 0$. 
\end{remark}

For further references regarding Jucys--Murphy elements, the construction of the operator $\mathcal{D}_k$ and the proof of Lemma \ref{BasicTraceLemma}, see \cite[Sections 8 and 9]{Me} or \cite[Section 9]{HoO}.

We are also interested in sampling multiple partitions simultaneously, for instance using the probability distribution introduced in equation (\ref{JointDistribution}). For each $n'\leq n$, consider the element $D_{k}\bigr|_{n'}=\textup{tr}(\Gamma(n')^{k})$ and $\mathcal{D}_{k}\bigr|_{n'}$ the operator acting via $D_{k}|_{n'}$. This operator can be understood as a restriction of the original operator into $S_{n'}$. More explicitly, we have that
\begin{equation*}
D_{k}\bigr|_{n'}=\frac{1}{n'+1} \sum_{0\leq i_1\neq i_2\neq \dots \neq i_k\neq i_1\leq n'} \big((i_1\, i_2)(i_2\, i_3)\dots (i_k\, i_1)\big).
\end{equation*}

The purpose of these new operators is that they will allows us to compute joint moments,  hence cumulants, in the multilevel setting, as shown in the following corollary.  
\begin{corollary}
Let $\lambda^1\subsetneq\lambda^2\subsetneq\cdots\subsetneq\lambda^s$ sampled from equation (\ref{JointDistribution}), let $r\geq 1$ and $k_1,\dots,k_r$ a sequence of integers, then
\begin{equation*}
\E_{\rho}\bigg[\prod_{i=1}^r\int_\R x^{k_i} m_K[\lambda^i](dx)\bigg]=\Big[\prod_{i=1}^r\mathcal{D}_{k_i}\bigr|_{n_i} M_{\rho}(g)\Big] \biggr|_{g=e}.
\end{equation*}
\end{corollary}
\begin{proof}
The probability distribution introduced in equation (\ref{JointDistribution}) is given by the branching rule for the symmetric group, more precisely, given an irreducible representation $V_\lambda$ of $S_n$ and $n'\leq n$ we have the following decomposition for the restriction of $V_\lambda$ into $S_{n'}$ (see \cite[Chapter 1]{BoO}). 
\begin{equation}\label{EquationBranchingRuleSn}
V_\lambda|_{S_{n'}}=\bigoplus_{\mu \vdash n'} \dim(\lambda \backslash \mu) V_\mu.    
\end{equation}

Where $\dim(\lambda \backslash \mu)$ is defined as in section \ref{SectionMultilevelTheorems}. By taking the normalized trace of the representations on both sides of equation (\ref{EquationBranchingRuleSn}) we obtain that for any $n'\leq n$ integers, 
\begin{equation*}\label{Equation2BranchingRuleSn}
\frac{\chi_\lambda}{\dim(\lambda)}\bigr|_{S_{n'}}=\sum_{\mu \vdash n'}p(\lambda \to \mu) \frac{\chi_\mu}{\dim(\mu)}.
\end{equation*}
Where the appearance of the term $p(\lambda \to \mu)$ follows from its definition, see equation (\ref{JointDistribution}). The conclusion follows by acting on both sides of equation (\ref{Equation2BranchingRuleSn}) for $n'=n_1,\dots,n_r$. At the right hand side we recover the joint moments of the random variables $\big(\int_\R x^{k_i} m_K[\lambda^i](dx)\big)_i$ while at the left hand side the action is given by the operators $\mathcal{D}_{k_i}\bigr|_{n_i}$.
\end{proof}

This gives a starting point to compute the joint moments of the collection of random variables
\begin{equation*}
    \big(X_k^{\alpha}\big)_{k\in \N, 0\leq \alpha\leq 1}=\Big(\frac{1}{\sqrt{n^k}}\int_\R x^k \,m_K[\lambda^\alpha](dx)\Big)_{k\in \N, 0\leq \alpha\leq 1}
\end{equation*}

To actually achieve this computation we will need a simplified explicit description of both the operators and also products of these operators. In the next few sections we introduce the tools that will allow us to do the required simplifications. 

\subsection{Expansion of operators}

We are interested in describing the asymptotic behavior of sequences of elements $x=(x_n)_{n\in \N}$ for $x_n \in \C[S_n]$ while $n\to \infty$, to facilitate this task, rather than indexing conjugacy classes with partitions $\mathbb{Y}$ we will employ a modified version of them.

\begin{definition}
Denote $\bar{\mathbb{Y}}_k$ be the set of partitions $\lambda$ of $k$ with $\lambda_i \geq 2$ and $ \bar{\mathbb{Y}}=\bigcup_{k\geq 2} \bar{\mathbb{Y}}_k$.
\end{definition}

Through this section,  we will restrict our attention to partitions in $\bar{\mathbb{Y}}$. In addition, to simplify notation we will omit the index $n$ unless needed.  Given a sequence of elements $x=(x_n)_{n\in \N}$ with $x_n \in \C[S_n]$, we will often be interested in counting the number of permutations of a given type in $x_n$. Since characters are invariant under conjugation, rather than needing a full expansion of product of elements it is enough to know the number of permutations with a given partition shape in the expansion. This motivates the need for a function that quantifies the weight of permutations of type $\lambda \in \bar{\mathbb{Y}}$ on the sequence $x$. The following notation will be useful, for $\lambda\in\overline{\mathbb Y}$ and $n\in \N$, we let $\Per(\lambda,n)$ denote the set of all permutations in $S_n$ whose cycles of length at least two are counted by $\lambda$. For instance, $\Per(\lambda,n)$ is empty if $n<\sum_i \lambda_i$. 

\begin{definition} Given $\lambda\in\overline{\mathbb Y}$ and $x=(x_n)_{n\in \N}$ with $x_n \in \C[S_n]$. We can decompose $x_n$ into a linear combination of permutations as $x_n= \sum_{\sigma\in S_n} c_{\sigma, n} \cdot \sigma $. We define
$$
 N_\lambda[x](n)= \sum_{\sigma\in \mathrm{Per}(\lambda,n)} c_{\sigma,n}.
$$
\end{definition}

The elements $(x_n)_{n\in\N}$ for which we plan to evaluate $N_\lambda[x](n)$ are going to be in Gelfand--Tsetlin algebra $\textup{GZ}_n$ or the center of the group algebra of the symmetric group $Z\big[\C[S_n]\big]$. In the second case it will be convenient to provide a basis of the central algebra. We will introduce some notation, for $\lambda=(\lambda_1\ge \lambda_2\ge \dots \ge \lambda_r)\in \overline{\mathbb Y}$ with $|\lambda|\le n$, denote
\begin{equation}\label{EqDefSigmalambda}
\Sigma_{\lambda,n}=\sum (i_{1,1}\cdots i_{1,\lambda_1})\cdots(i_{r,1}\cdots i_{r,\lambda_{r}})
\end{equation}
where the sum runs over all distinct indices $1 \leq i_{a,b} \leq n$ for $1 \leq b \leq \lambda_r$, $1 \leq a\leq r$. We are going to omit $n$ in $\Sigma _{\lambda,n}$, as in the last formula, when it is clear from the context. The total number of terms in the sum can be expressed as a falling factorial. Introducing the notation
\begin{equation*}
(y\ff k)=\begin{cases}
    y(y-1)\ldots(y-k+1),& \text{if $k=1,2,\ldots$,}\\
1,& \text{if $k=0$,}
\end{cases}
\end{equation*}

the number of terms in (\ref{EqDefSigmalambda}) is $N_\lambda[\Sigma_\lambda](n)=(n\ff |\lambda|)$. It is also convenient to introduce normalized sums over conjugacy classes as
\begin{equation*}
\hat{\Sigma}_{\lambda,n}= \frac{ \Sigma_{\lambda,n} } { (n\ff |\lambda|) }.
\end{equation*}

Note that any element from the center of the group algebra, $x_n\in Z\big[\mathbb \C[S_n]\big]$, can be decomposed as:
\begin{equation*}
x_n= \sum\limits_{\substack{ \ \lambda\in\overline{\mathbb Y}, \\ |\lambda| \le n }} N_\lambda[x](n) \hat \Sigma_{\lambda,n}
\end{equation*}

In practice, we will only be interested in knowing the leading term of $N_\lambda[x]$. In most of our results we will be interested in knowing the values of  $\alpha$ and $\beta$ real numbers such that $N_\lambda[x](n)=\alpha n^\beta\big(1+o(1)\big)$. It will be very convenient to have a more explicit description of such sequences of $N_\lambda[x]$ for a given $x$. 

When working at the center of the symmetric group it will be convenient to switch between the normalized situation and the non normalized one. For instance, when doing explicit computations the non normalized notation is more convenient while for computing asymptotics the normalized notation will be more appropriate. Our interest being that 
\begin{equation*}
    M_\rho(\hat{\Sigma}_\lambda)=M_\rho(\sigma[\lambda])
\end{equation*}
where $\sigma[\lambda]$ denotes any permutation with cycle decomposition given by $\lambda$. We will simplify the notations when $\lambda$ has a unique row of length $k$, that is $\lambda=(k)$. We will omit the parenthesis and write $k$ instead of $\lambda$. For example, $\sigma[k]$ refers to a cycle of length $k$. We will be interested in understanding the asymptotic of $N_\lambda[x](n)$ for different $x$ and $\lambda$ while $n \to \infty$. When the sequence $x=(x_n)_{n\in \N}$ is clear form the context, we additionally shorten $N_\lambda(n)$ for $N_\lambda[x](n)$.

\begin{example}
$\Per(2,n)$ denote the set of all transpositions in $S_n$ and $\Sigma_2=\sum_{1 \leq a \neq b \leq n} (a \,b)$. We get $N_\lambda[\Sigma_2](n)=n(n-1)$ when $\lambda=2$, while $N_\lambda[\Sigma_2](n)=0$ for each other partition $\lambda$. Additionally, $\hat{\Sigma}_2 = \frac{2}{n(n-1)}\sum_{1 \leq a < b \leq n}(a \,b)$. 
\end{example}

\begin{remark}
Fix $\lambda \in \bar{\mathbb{Y}}$, notice that $\Sigma_\lambda$ is not the sum over all permutations in the conjugacy class of $\lambda$ but these elements of $\C[S_n]$ are equal up to a constant depending only on $\lambda$. Denote $s_i=|\{j:\lambda_j=i\}|$, since $|\Per(\lambda,n)|=\frac{(n\ff |\lambda|)}{\prod_{i=2}^ri^{s_i}\cdot s_i! }$ we have that 
\begin{equation*}
\Sigma_{\lambda,n}=\big(\prod_{i=2}^ri^{s_i}\cdot s_i!\big) \sum_{\sigma \in \textup{Per}(\lambda,n)} \sigma.
\end{equation*}
Since the constant does not depends on $n$, this shows that we can restate all our results in terms of sums of over elements in the conjugacy classes $\textup{Per}(\lambda,n)$ for each $\lambda \in \bar{\mathbb{Y}}$.
\end{remark}

We are now in better shape to correctly state the expansion for the operator $\mathcal{D}_k$.

\begin{definition}
For each $k\geq 2$, let $\textup{Rem}_k: \bar{\mathbb{Y}}_k \to \bar{\mathbb{Y}}$ be defined as follows, given $\lambda \in \bar{\mathbb{Y}}_k$ let $\textup{Rem}_k(\lambda)$ be the partition obtained after removing all rows of length $2$ and the first column.
\end{definition}

\begin{example}
The image of $\lambda=(5,4,3,2,2) \in \bar{\mathbb{Y}}_{16}$ under $\textup{Rem}_{16}$ is $(4,3,2) \in \bar{\mathbb{Y}}$.  In figure \ref{ExampleRemk} we illustrate this evaluation by marking with $\times$ the deleted boxes. 

\begin{figure}[h]
    \centering
    \includegraphics[scale=0.35]{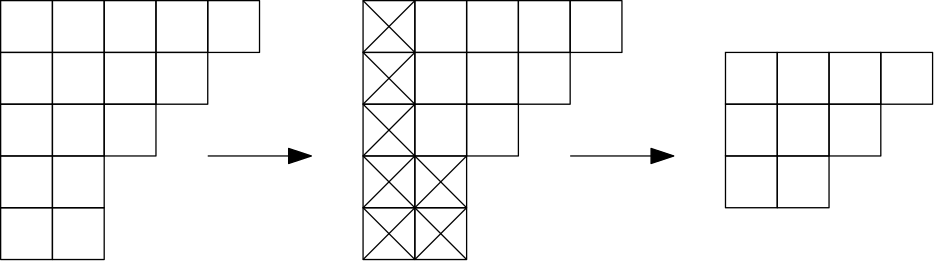}
    \caption{Computation of $\textup{Rem}_{16}(5,4,3,2,2)$.}
    \label{ExampleRemk}
\end{figure}\end{example}

It is immediate from the definition of $\textup{Rem}_k$ that

\begin{proposition}\label{PropositionRemk}
For each $k\geq 2$, 
\begin{enumerate}
    \item $\textup{Rem}_k$ is injective.
    \item Given $\mu\in \bar{\mathbb{Y}}$, there exist $\lambda \in \bar{\mathbb{Y}}_k$ such that $\textup{Rem}_k(\lambda)=\mu$ if and only if there exist $j\geq 0$ such that $|\mu|+\ell(\mu)+2j=k$.
\end{enumerate}
\end{proposition}

For a given partition $\lambda \in \bar{\mathbb{Y}}$, denote $|\lambda|$ the size of the partition, $\ell(\lambda)$ the length of the partition and $l(\lambda)=|\lambda|-\ell(\lambda)$. The main interest of the quantity $l(\lambda)$ is that it coincides with the Cayley distance of any permutation with conjugacy class $\lambda$. For instance, for a permutation $\sigma$ denote $|\sigma|$ the distance from the identity to $\sigma$ in the Cayley graph generated by transpositions, equivalently the minimum number of transpositions such that their product is equal to $\sigma$, then $l(\lambda) =|\sigma[\lambda]|$. 

\begin{example}
Let $\lambda=(5,3,2,2,1)$ then $|\lambda|=13$, $\ell(\lambda)=5$ and $l(\lambda)=8$.
\end{example}

As already stated in Definition \ref{DefinitionOfDk}, the element $D_k$ is invariant under conjugations and hence $D_k \in Z\big[\C[S_n]\big]$. In particular, we have that $D_k$ is a linear combination of $\hat{\Sigma}_\lambda$ for some collection of partitions. We will use the following result about non-crossing partitions. 

\begin{lemma} \label{Kreweras}
To each partition $\lambda$ we associate the numbers $s_i=|\{j:\lambda_j=i\}|$, $\ell(\lambda)=h$ and $q=\lambda_1$. We denote $\text{NC}(\lambda)$ the set of non-crossing set partitions of $m$ with $h$ classes and $s_i$ classes of size $i$ for $i=1,\dots,q$. Then
\begin{equation*}
|\text{NC}(\lambda)| = \frac{\big(|\lambda|\ff h-1\big)}{s_1!s_2!\dots s_{q}!}
\end{equation*}
\end{lemma}
\begin{proof}
    This is theorem 4 in \cite[Page 340]{Kr}.
\end{proof}

\begin{remark}
Since $s_1+s_2+\dots+s_q=h$ and $s_1+2s_2+3s_3+\dots+qs_q=|\lambda|$, knowing the sequence of values $(s_1,\dots,s_q)$ is enough to determine $|\text{NC}(\lambda)|$. 
\end{remark}

The following theorem provides an explicit description of the collection of partitions and asymptotics for the coefficients $N_\lambda[D_k]$. The proof is postponed to the next subsection. 

\begin{theorem}  \label{TheoremExpansionofDk}

Let $D_k$ be given in Definition \ref{DefinitionOfDk}. The coefficients $N_\lambda[D_k](n)$ of decomposition for $D_k$ into linear combinations of $\hat{\Sigma}_\lambda$ satisfy the following properties:
\begin{itemize}
    \item $N_\lambda[D_k](n)=0$ unless $\lambda \in \textup{Rem}_k(\bar{\mathbb{Y}}_k)$.
    \item For fixed $k$ and $\lambda$, as $n\to \infty$, we have $N_\lambda[D_k](n)=|\text{NC}(\mu)| n^{\frac{l(\lambda)+k}{2}} \big(1+O(n^{-1})\big)$, where $\textup{Rem}_k(\mu)=\lambda$.
\end{itemize}
\end{theorem}

Note that we can calculate the leading coefficients $|\text{NC}(\mu)|$ using Lemma \ref{Kreweras}.

At this point it will be important to work on $\textup{GZ}_n$ rather than $Z\big[\C[S_n]\big]$. Let $n_1\leq n$, denote
\begin{equation*}
\Sigma_\lambda\bigr|_{n_1}=\sum (i_{1,1}\cdots i_{1,\lambda_1})\cdots(i_{r,1}\cdots i_{r,\lambda_{r}})
\end{equation*}
where the sum run over all distinct indices $1 \leq i_{a,b} \leq n_1$ for $1 \leq b \leq \lambda_r$, $1 \leq a\leq r$.

Let  $r \geq 1$, for $\vec{\lambda}=(\lambda^1,\dots,\lambda^r)$ a vector of $r$ partitions in $\bar{\mathbb{Y}}$. Since we will be considering products of the operator $\mathcal{D}_k$ it will be convenient to develop expansion formulas for products of the form $\prod_{j=1}^r \Sigma_{\lambda^j}\bigr|_{n_j}$. Notice that if $n_1=\dots=n_r=n$, then $\prod_{j=1}^r \Sigma_{\lambda^j}$ is invariant under conjugation and then it is an element of the central algebra  $Z\big[\C[S_n]\big]$, in particular it is a linear combination of $\Sigma_\lambda$. While we will work in a more general setting where our expansion will not be that explicit, all our result can be projected into the central algebra setting where expansions are given explicitly in a similarly fashion to Theorem \ref{TheoremExpansionofDk}. 

We will deepen our notation by denoting $\bar{\Sigma}_\lambda$ some normalized sum of permutations in $\C[S_n]$ with conjugacy class $\lambda$, that is we have that $N_\lambda[\bar{\Sigma}_\lambda]=1$ and $N_\mu[\bar{\Sigma}_\lambda]=0$ for $\mu \neq \lambda$. We can then write
\begin{equation*}
x_n= \sum N_\lambda[x](n) \bar{\Sigma}_\lambda
\end{equation*}
where the sun runs over all partitions in $\{\lambda[\sigma] \in \bar{\mathbb{Y}}: \sigma \in S_n\}$. Notice that we still have $M_{\rho}(\bar{\Sigma}_\lambda)=M_\rho(\sigma[\lambda])$.

Furthermore, we can interpret $\lambda$ as multisets of non-negative integers, from this point of view, we can write  $i \in \lambda$ to denote $i$ an element in the multi set $\{\lambda_j:j=1,\dots,r\}$. Furthermore, given $i \in \lambda$, $\lambda\backslash\{i\}$ denotes the partition to which we removed the element $i$. For $\lambda$ and $\mu$ are partitions, denote $\lambda \cup \mu$ the partition constructed by taking the union of all elements in both $\lambda$ and $\mu$. 

In the case in which $r=2$ we are interested in finding an expansion with explicit coefficients in the leading terms. This will allow us to compute the covariance between the random variables $X^{\alpha}_k$ and $X^{\alpha'}_{k'}$.

\begin{theorem}\label{lemmaproductoftwopartitions2} \label{lemmaproductoftwopartitions}
Let $\lambda$ and $\lambda'$ be partitions and $n_1\leq n_2\leq n$, then
\begin{equation*}
    \Sigma_{\lambda}\bigr|_{n_1} \Sigma_{\lambda'} \bigr|_{n_2} =(n_1\ff |\lambda|)(n_2-|\lambda|\ff|\lambda'|)  \bar{\Sigma}_{\lambda \cup \lambda'} +Q_{\lambda,\lambda'}+R_{\lambda,\lambda'}
\end{equation*}

where 
\begin{equation*}
\begin{split}
    Q_{\lambda,\lambda'} = \sum_{\substack{i \in \lambda\\ j \in \lambda'}} \sum_{r\geq 1} \sum_{\substack{s_1,\dots,s_r\geq 1\\ s_1+\dots+s_r=i}}\sum_{\substack{t_1,\dots,t_r\geq 1\\ t_1+\dots+t_r=j}} \dfrac{ij}{r} (n_1\ff |\lambda|)(n_2-|\lambda|\ff|\lambda'|-r)  \hspace{1cm}\\
    \cdot \bar{\Sigma}_{\lambda\backslash\{i\}\cup \lambda'\backslash\{j\}\cup_{m=1}^r (s_m+t_m-1)},    
\end{split}
\end{equation*}
if any $s_m+t_m-1$ are equal to $1$ we delete them to ensure that we obtain a partition in $\bar{\mathbb{Y}}$, and there exists a set of partitions $\mathcal{R} \subset \bar{\mathbb{Y}}$ such that 
\begin{equation*}
R_{\lambda,\lambda'}=\sum_{\mu \in \mathcal{R}} N_\mu[R](n) \bar{\Sigma}_\mu
\end{equation*}
with  $N_\mu[R](n)=O(n^\beta)$ with $\beta=\frac{l(\mu)+|\lambda|+|\lambda'|+\ell(\lambda)+\ell(\lambda')}{2}-2$.
\end{theorem}

If $n_1=n_2=n$ we can ensure that all involved terms in the sum are in $Z\big[\C[S_n]\big]$, hence Theorem \ref{lemmaproductoftwopartitions2} can be restated in a more explicit form where the terms $\bar{\Sigma}$ can be replaced by $\hat{\Sigma}$. 

To correctly state the expansion formula for the product $\prod_{j=1}^r \Sigma_{\lambda^j}\bigr|_{n_j}$ will require some technical preparation. Our objective is to provide an asymptotic expansion for a product of permutations. It will be convenient to use the language of set partitions. 

Note that $\Theta_r$ of set partitions of $\{1,\dots,r\}$ has a lattice structure with the refinement order. It has two operations, the join of $a,b\in \Theta_r$, denoted  $a\vee b$, which is the least upper bound of $a$ and $b$, and the meet $a\wedge b$, which is the greatest lower bound of $a$ and $b$. The partition lattice has a maximal and minimal element, $\hat{1}$ denotes the maximal partition $\big\{\{1,\dots,r\}\big\}$ and $\hat{0}$ denotes the minimal partition $\big\{\{1\},\dots,\{r\}\big\}$. Additionally, for $\pi \in \Theta_r$ we call an element $B\in \pi$ a block of the partition and denote by $|\pi|$ the number of blocks of $\pi$. For instance $|\hat{1}|=1$ and $|\hat{0}|=r$. Furthermore, for any set $B$ denote $\Theta_B$ the set of set partitions on $B$. Fix $\eta=\{B_1,\dots,B_s\}\in \Theta_r$, the set of partitions $\pi \in \Theta_r$ with $\eta \leq \pi$ is a sub-lattice of $\Theta_r$ and is isomorphic to $\Theta_s$. Denote this isomorphism by $\textup{Proj}_\eta:\{\pi \in \Theta_r: \eta \leq \pi\} \to \Theta_s$, the isomorphism consists of collapsing each block $B_i$ into a singleton $\{i\}$ for $i=1,\dots,s$. 

We are now interested in describing the conjugacy class of products of permutations. It is convenient to understand what is the data required to described this conjugacy class. Since we are only interested in the conjugacy class, it is enough to know the conjugacy class of each permutation and to know which entries in the description of are equal to each other. That is, to each collection of permutations we associate a set partition describing the entries of the permutation that are equal. In the next technical definition we describe how this is achieved. 

We will slightly generalize our notation to take into account multiple partitions simultaneously. For $\vec{\lambda}=(\lambda^1,\dots,\lambda^r)$ a collection of $r$ partitions, we denote $|\vec{\lambda}|=\sum_{i=1}^r |\lambda^i|$, $\ell(\vec{\lambda})=\sum_{i=1}^r \ell(\lambda^i)$ and $l(\vec{\lambda})=\sum_{i=1}^r l(\lambda^i)$.

\begin{definition}\label{WeirdIBdefinition}
Let $\vec{\lambda}=(\lambda^1,\dots,\lambda^r)$ be a vector of $r$ partitions. Denote the set of permutations with conjugacy class $\vec{\lambda}$ by
\begin{equation*}
    \mathcal{A}[\vec{\lambda}]:=\Big\{(\sigma_1,\dots,\sigma_r)\in \prod_{j=1}^r S_\infty: \textup{The conjugacy class of }\, \sigma_j \textup{ is } \, \lambda^j,\,\,  j=1,\dots,r \Big\}. 
\end{equation*}

Given $(\sigma_1,\dots,\sigma_r) \in \mathcal{A}[\vec{\lambda}]$, for $\lambda^j=(\lambda^j_1\geq\dots \geq \lambda^j_{s_j})$ we have that $\sigma_j=\sigma_{j,1}\cdots\sigma_{j,s_j}$ where $\sigma_{j,t}$ is a cycle of length $\lambda^j_t$ for $1 \leq t \leq s_j$ and $1\leq j \leq r$. We can describe each such cycle as
\begin{equation*}
    \sigma_{j,t}=\bigg(i\Big[|\lambda^{1}|+\dots+|\lambda^{j-1}|+\lambda^j_1+\dots+\lambda^j_{t-1}\Big]\,\,\cdots \, \, i\Big[|\lambda^{1}|+\dots+|\lambda^{j-1}|+\lambda^j_1+\dots+\lambda^j_{t}\Big]\bigg). 
\end{equation*}
where $i[a]\in \N$ for each $1\leq a \leq |\vec{\lambda}|$. We define the \textbf{associated partition} of $(\sigma_1,\dots,\sigma_r)$ to be the set partition $\theta \in \Theta_{|\vec{\lambda}|}$ such that $1\leq a,b \leq |\vec{\lambda}|$ are in the same block of $\theta$ if and only if $i[a]=i[b]$. 
\end{definition}

The interest in the previous definition is that for a given collection of permutations with associated partition $\theta$, then the the conjugacy class of $\prod_{j=1}^r \sigma_j$ only depends on $\theta$.

\begin{example}\label{ExamplewWeirdDef1}
Let $\vec{\lambda}=\big((2\geq2),(2),(2)\big)$, hence $|\vec{\lambda}|=8$. Let $\sigma_1=(1\,2)(3\,4)$, $\sigma_2=(1\,5)$ and $\sigma_3=(3\,6)$, then $(\sigma_1,\sigma_2,\sigma_3)\in \mathcal{A}[\vec{\lambda}]$. Additionally, its associated partition $\theta$ is given by 
\begin{equation*}
    \theta=\big\{\{1,5\},\{2\},\{3,7\},\{4\},\{6\},\{8\}\big\} \in \Theta_8.
\end{equation*} 
\end{example}

\begin{definition}\label{WeirdDefinitionSECOND}
Following the notation of Definition \ref{WeirdIBdefinition}, denote $\eta$ the set partition in $ \Theta_{|\vec{\lambda}|}$ with blocks 
\begin{equation*}
    B_j=\Big\{|\lambda^{1}|+\dots+|\lambda^{j-1}|+1,\dots,|\lambda^{1}|+\dots+|\lambda^{j}| \Big\} \hspace{2mm} \text{ for } \hspace{2mm} j=1,\dots,r.
\end{equation*}

We will say that $(\sigma_1,\dots,\sigma_r)$ \textbf{jointly intersects} if $\theta \vee \eta=\hat{1}$. Finally, for each $\pi\in \Theta_r$ we define $ \Xi_\pi[\vec{\lambda}]$ to be the set of pairs $(\mu,\theta)\in \bar{\mathbb{Y}}\times \Theta_{|\vec{\lambda}|}$ such that $\mu$ is the conjugacy class of $\prod_{j=1}^r \sigma_j$, $\theta$ is the associated partition of $(\sigma_1,\dots,\sigma_r)$ and $\pi=\textup{Proj}_{\eta}(\theta \vee \eta)$.
\end{definition}

Using the notation introduced in the last definition, we notice that terms of the form $n^{|\theta|}\big(1+O(n^{-1})\big)\bar{\Sigma}_\mu$ for $(\mu,\theta) \in \Xi[\vec{\lambda}]$ will appear in the expansion of $\prod_{j=1}^r \Sigma_{\lambda^j}\bigr|_{n_j}$. The main interest in considering jointly intersecting collection of permutations is that for these collections we are able to produce sharp bounds relating $|\theta|$ and $l(\mu)$. Moreover, the partition $\pi=\textup{Proj}_{\eta}(\theta \vee \eta)$, where $\textup{Proj}_{\eta}$ is the isomorphism above, describe the maximal sets of permutations that jointly intersect among $\sigma_1,\dots,\sigma_r$, that is, for each $B\in \pi$, the set of permutations $\{\sigma_j\}_{j\in B}$ jointly intersects. 

\begin{example}
Following example \ref{ExamplewWeirdDef1}, we have $ \eta=\big\{\{1,2,3,4\},\{5,6\},\{7,8\}\big\} \in \Theta_8$. Since $\eta\vee \theta =\hat{1}$, we have that the permutations $\sigma_1,\sigma_2,\sigma_3$ jointly intersect. 
\end{example}

\begin{remark}
An equivalent definition of jointly intersection of permutations $\sigma_1,\dots,\sigma_r$ where we denote the support of $\sigma_j$ by $B_j$ is that for each $1\leq i,j \leq r$ there exist a finite collection of indices $1\leq i_1,\dots,i_s \leq r$ with $i_1=i$ and $i_s=j$ such that for each $k=1,\dots,s-1$, we have that $B_{k}\cap B_{k+1}\neq \emptyset$.
\end{remark}

Finally, to precisely describe the leading terms in the expansion of products we will need a generalized notion of falling factorial.

\begin{definition}\label{Generalizedfallingfactorial}
Let $s \geq 1$, given  $\vec{n}=(n_1\leq n_2 \leq \dots \leq n_s)$ an increasing collection of integers and $\vec{k}=(k_1,\dots,k_s)$ a collection of non-negative integers, define the generalized $\vec{k}$th falling factorial of $\vec{n}$ to be
\begin{equation*}
    (\vec{n}\ff \vec{k}):=\prod_{i=1}^s(n_i-\sum_{j=1}^{i-1} k_j\ff k_i)
\end{equation*}
\end{definition}
Notice that this generalizes the usual falling factorial in multiple ways. We can (1) take $s=1$, (2) take $n_1=n_2=\dots=n_s$ or (3) take $k_j=0$ for $j\geq 2$, in any of these three situations $(\vec{n}\ff \vec{k})=(n \ff k)$. In what follows we define the sum of vectors of integers by coordinate-wise addition, furthermore, we will use the notation $|\vec{k}|=k_1+\dots+k_s$.
\begin{remark}
There is a combinatorial interpretation for the generalized falling factorial. In fact, for $i=1,\dots,s$ denote $B_i=\{1,\dots,k_i\}$ and $B=\sqcup_{i=1}^s B_i$ the disjoint union of the sets $B_s$. Denote $\mathcal{F}[\vec{k},\vec{n}]$ the set of injective functions $f:B \to \{1,\dots,n_s\}$ such that $f(B_i)\subseteq\{1,\dots,n_i\}$, then $(\vec{n}\ff \vec{k})$ is the cardinality of $\mathcal{F}[\vec{k},\vec{n}]$. 
\end{remark}

We finally introduce the following notations. For $B\subseteq \{1,\dots,r\}$, we denote $\vec{\lambda}_B=(\lambda^i)_{i \in B}$. Additionally, by interpreting $\vec{\lambda}$ as the multi-set given by $\cup_{j=1}^r \lambda^j$, we get that $\bar{\Sigma}_{\vec{\lambda}}=\bar{\Sigma}_{\cup_{j=1}^r \lambda^j}$.

\begin{theorem} \label{Product1stOrderExpansionUPGRADEDMultilevel}  \label{Product1stOrderExpansionUPGRADED}
Let $r\geq 1$, $\vec{n}=(n_1\leq n_2 \leq \dots \leq n_r)$ be integers and $\vec{\lambda}=(\lambda^1,\dots,\lambda^r)$ be a vector of $r$ partitions in $\bar{\mathbb{Y}}$, then
\begin{equation*}
\prod_{j=1}^r \Sigma_{\lambda^j}\bigr|_{n_j} = \sum_{\pi \in \Theta_r} \Delta_{\vec{\lambda}}^\pi  \hspace{2mm} \textup{ where } \hspace{2mm} \Delta_{\vec{\lambda}}^\pi = \sum_{(\mu,\theta) \in \Xi_\pi[\vec{\lambda}]} (\vec{n}\ff \vec{k}_\theta)  \bar{\Sigma}_{\mu}.
\end{equation*}
Where,
\begin{enumerate}
    \item $|\vec{k}_\theta|\leq \frac{l(\mu)+ |\vec{\lambda}|+\ell(\vec{\lambda})}{2}+|\pi|-r$.
    \item If $\pi = \hat{0}$ then $\Delta_{\vec{\lambda}}^\pi = \big(\vec{n}\ff (|\lambda^1|,\dots,|\lambda^r|) \big) \bar{\Sigma}_{\vec{\lambda}}$. 
    \item If $\pi \neq \hat{0}$ and $(\mu,\theta) \in \Xi_\pi[\vec{\lambda}]$, then either $l(\mu)<l(\vec{\lambda})$ or $l(\mu)=l(\vec{\lambda})$ and $\ell(\mu) < \ell(\vec{\lambda})$.
\end{enumerate}
Additionally, for each $\pi\in \Theta_r$ and $B\in \pi$, denote by $\hat{1}_B$ the maximal element in $\Theta_B$, then there exists a collection $\big\{(\mu_B,\theta_B)\big\}_{B \in \pi}$ with $(\mu_B,\theta_B)\in \Xi_{\hat{1}_B}[\vec{\lambda_B}]$ such that $\mu=\cup_{B \in \pi} \mu_B$ and $\vec{k}_{\theta}=\sum_{B \in \pi} \vec{k}_{\theta_B}$.
\end{theorem}

Although Theorem \ref{Product1stOrderExpansionUPGRADEDMultilevel} doesn't state explicitly how to compute the integer partitions $\mu$ appearing on the first coordinate of elements of $\Xi_\pi[\vec{\lambda}]$ in the next subsection we will provide the tools enabling that calculation.  We end this subsection by providing a simplified version of the previous theorem which is more adequate for the proof of Theorem \ref{TheoremLLN}.

\begin{corollary}\label{Product1stOrderExpansion}
Let $r\geq 1$ and $\vec{\lambda}=(\lambda^1,\dots,\lambda^r)$ be $r$ partitions, then the decomposition 
\begin{equation*}
\prod_{j=1}^r \Sigma_{\lambda^j} = \Sigma_{\vec{\lambda}} +R_{\vec{\lambda}}      \hspace{2mm} \textup{ and } \hspace{2mm} R_{\vec{\lambda}} = \sum_{\mu} N_\mu[R](n)\Sigma_\mu
\end{equation*}
satisfies $N_\mu[R](n)=O(n^\beta)$ with $\beta=\frac{l(\mu)+|\vec{\lambda}|+\ell(\vec{\lambda})}{2}-|\mu|-1$. Moreover, if we either have that $l(\mu)>l(\vec{\lambda})$ or $l(\mu)=l(\vec{\lambda})$ and $\ell(\mu) \geq\ell(\vec{\lambda})$, then  $N_\mu[R](n)=0$. 
\end{corollary}

\subsection{Proofs of Theorems \ref{TheoremExpansionofDk}, \ref{lemmaproductoftwopartitions2} and \ref{Product1stOrderExpansionUPGRADEDMultilevel}}   

For each totally order set $B$ we consider $\textup{NC}(B)$ to be the set of non-crossing partitions over $B$. Similarly to $\Theta_B$, $\textup{NC}(B)$ is a lattice, a partially ordered set with the refinement order and two operations $\vee$ and $\wedge$ where $a\vee b$ denotes the least upper bound of the elements $a$ and $b$, while $a\wedge b$ denotes the greatest lower bound of the elements $a$ and $b$. Note that while the $\wedge$ operation coincide on both $\textup{NC}(B)$ and $\Theta_B$, the $\vee$ operation is different, in particular $\textup{NC}(B)$ is not in general a sublattice of $\Theta_B$. We will now construct a map that allows us compute the expansion of the operator $\mathcal{D}_k$. 

For each $A \subseteq B$ there is a natural embedding 
$\iota:\textup{NC}(A) \to \textup{NC}(B)$ defined by
\begin{equation*}
    \iota(\pi):=\pi\cup\bigcup_{x \in B\backslash A}\Big\{\{x\}\Big\} \in \textup{NC}(B) \hspace{2mm} \text{ for each }\hspace{2mm} \pi \in \textup{NC}(A)
\end{equation*}

Denote by $\textup{NC}(k)$ the lattice of non-crossing partitions over $\{1,\dots,k\}$. Let
\begin{equation*}
E_k:=\{2,4,\dots,2k\}\hspace{2mm} \text{ and }\hspace{2mm} O_k:=\{1,3,\dots,2k-1\}.
\end{equation*}
Notice that $\textup{NC}(k)$ is isomorphic to both $\textup{NC}(E_k)$ and $\textup{NC}(O_k)$. Denote by $\iota_{\textup{even}}$ and by $\iota_{\textup{odd}}$ the embeddings from $\textup{NC}(k)$ into $\textup{NC}(2k)$ via $\textup{NC}(E_k)$ and $\textup{NC}(O_k)$ respectively. We will say that $\iota_{\textup{even}}\big(\textup{NC}(k)\big)$ is the complement lattice of $\iota_{\textup{odd}}\big(\textup{NC}(k)\big)$. The following lemma will be fundamental to calculate the conjugacy class of a product of permutations given their associated set partition. Informally the lemma states that given a non-crossing partition $\theta$ there exists a unique maximal non-crossing partition $\pi$ constructed over the complement that does not intersect with $\theta$, moreover, this partition $\pi$ uniquely determines $\theta$.  

\begin{lemma}
Denoting $\textup{NC}(k)$, $\iota_{\textup{odd}}$ and $\iota_{\textup{even}}$ as before. We have that
\begin{itemize}
    \item For each $\theta \in \textup{NC}(k)$, there exists a unique maximal non-crossing set partition $\pi \in \textup{NC}(k)$ such that $\iota_{\textup{odd}}(\pi) \wedge \iota_{\textup{even}}(\theta) \in \textup{NC}(2k)$.
    \item The map $K: \textup{NC}(k) \to \textup{NC}(k)$ defined by
\begin{equation*}
K(\theta):=\sup\big\{\pi \in \textup{NC}(k):\iota_{\textup{odd}}(\pi) \wedge \iota_{\textup{even}}(\theta) \in \textup{NC}(2k)\big\}
\end{equation*}
is a bijection.
\end{itemize}
\end{lemma}

We call $K(\theta)$ the \textbf{Kreweras complement} of $\theta$. This map was first introduced by Kreweras \cite{Kr}, it is an anti-automorphism on $\textup{NC}(k)$ and a fundamental piece in the construction of the group of skew-automorphism of the lattice of non-crossing partitions, see \cite{Bi3}. The Kreweras complement has a beautiful geometric interpretation which is illustrated with the following example in figure \ref{FigureExampleKrewerasMap}. 
\begin{example}\label{ExampleKrewerasComplement}
For $k=8$, the partition $\theta=\big\{\{1,3,5\},\{2\},\{4\},\{6,8\},\{7\}\big\}$ (in red) with its Kreweras complement  (in blue) $K(\theta)=\big\{\{1,2\},\{3,4\},\{5,8\},\{6,7\}\big\}$. See figure \ref{FigureExampleKrewerasMap}.

\begin{figure}[h]
    \centering
    \includegraphics[scale=0.3]{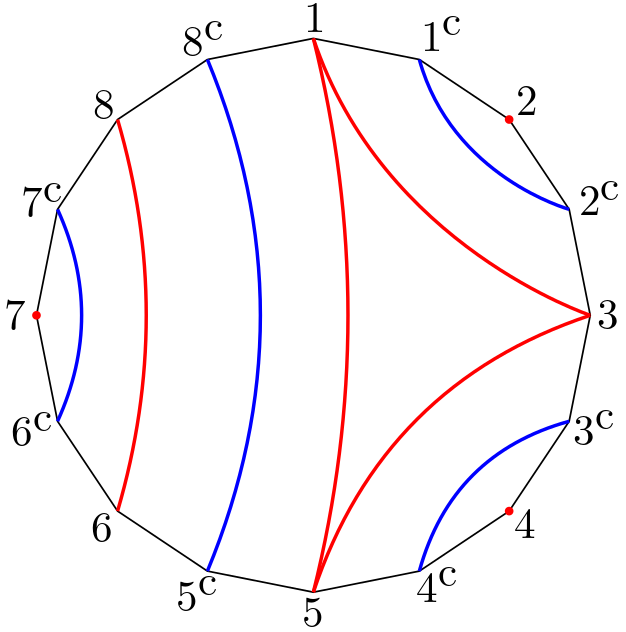}
    \caption{Computation of $K\Big(\big\{\{1,3,5\},\{2\},\{4\},\{6,8\},\{7\}\big\}\Big)$.}
    \label{FigureExampleKrewerasMap}
\end{figure}\end{example}

To perform our computations we will be interested in restricting the Kreweras complement map to some specific subsets of non-crossing partitions.

\begin{definition}
We will say that a set partition $\pi \in \Theta_k$  is \textbf{admissible} if for each $B\in \pi$ and $x\in B$, $1\leq x<k$, then $x+1\not \in B$. Additionally, if $x=k$ we require $1\not \in B$. Denote the set of admissible set partitions of $k$ by $\mathcal{P}_a(k)$. Furthermore, denote by $\textup{NC}_a(k):=\mathcal{P}_a(k)\cap \textup{NC}(k)$ the set of admissible non-crossing partitions. 
\end{definition}

\begin{proposition} \label{Propositionbijection}
Let $\textup{NC}_{\geq 2}(k)$ be the set of non-crossing partitions of $k$ where each block has size at least $2$. Then the Kreweras complement gives a bijection between $\textup{NC}_a(k)$ and $\textup{NC}_{\geq 2}(k)$. 
\end{proposition}
\begin{proof}
Since $\textup{NC}_a(k)$ is a subset of $\textup{NC}(k)$, Lemma \ref{LemmaBijection} ensures that $K$ is injective. Similarly, since $\textup{NC}_{\geq 2}(k)$ is a subset of $\textup{NC}(k)$, Lemma \ref{LemmaBijection} ensures that there exist a unique $\eta \in \textup{NC}(k)$ such that $K(\eta)=\pi$. Suppose $\eta \not \in \textup{NC}_{\geq 2}(k)$, then there exists some $x \in \{1,\dots,k\}$ such that $x$ and $x+1$, modulo $k$, are in the same block of $\eta$, hence by construction of the Kreweras map, $\{2k+1\}$ is a singleton in $K(\eta)$, which contradicts $K(\eta)\in \textup{NC}_{\geq 2}(k)$.
\end{proof}

Note that Example \ref{ExampleKrewerasComplement} illustrates the calculation of the Kreweras complement of an admissible non-crossing partition. The interest in considering admissible partitions is that they allow us to characterize the conjugacy class of elements in the sum of Theorem \ref{TheoremExpansionofDk}. In fact, the leading term in the expansion provided in Theorem \ref{TheoremExpansionofDk} comes from non-crossing partitions. This will allow us to find the expansion of other operators as shown in the next section. To formalize this idea we introduce the following notation. 

\begin{definition}
Let $\pi \in \mathcal{P}_a(k)$ and let $Z$ be either the empty set or a block $Z\in\pi$. Denote $I^\pi_Z$ to be a set of \textbf{indices} $\vec{i}=(i_1,\dots,i_k)$ such that
\begin{enumerate}
    \item $0 \leq i_j\leq n$ for $j=1,\dots,k$.
    \item $i_j=i_{j'}$ if and only if there exists $B\in \pi$ such that $j,j'\in B$.
    \item $i_j=0$ if and only if $j\in Z$. 
\end{enumerate}

Additionally, define $\tau: I^\pi_Z \to S_n$ by
\begin{equation*}
    \tau(\vec{i})=(i_1\,i_2)(i_2\,i_3)\cdots(i_{k-1}\,i_k)(i_k\,i_1),
\end{equation*}
where the transposition $(a\,b)$ is the identity element of the symmetric group when $a$ or $b$ are $0$.
\end{definition}

It follows from the definition of $D_k$ that
\begin{equation} \label{Eq1PCL}
    D_k=\frac{1}{n+1} \sum_{\vec{i} \in I(k)} \tau(\vec{i}) = \frac{1}{n+1} \sum_{\pi \in \mathcal{P}_a(k)} \sum_{Z\in \pi\cup\{\emptyset\} } \sum_{\vec{i} \in I_Z^\pi(k)} \tau(\vec{i}).
\end{equation}
Notice that for each $\vec{i}\in I_Z^\pi(k)$, the conjugacy class of $\tau(\vec{i})$ is uniquely determined by the pair $(Z,\pi)$. In particular the summand $\sum_{\vec{i} \in I_Z^\pi(k)} \tau(\vec{i})$ is an element of $Z\big[\C[S_n]\big]$ and there is a partition $\lambda \in \bar{\mathbb{Y}}$  and a constant $C_\lambda(n)$ that only depends on $(Z,\pi)$, such that
\begin{equation*}
    \sum_{\vec{i} \in I_Z^\pi(k)} \tau(\vec{i}) = C_\lambda(n)\Sigma_\lambda.
\end{equation*}

It is then natural to wonder how to compute $\lambda$ and $C_\lambda(n)$ from $\pi$ and $Z$. We will answer that question with Lemma \ref{DkLeadingTerms}. The following elementary proposition will be useful. 

\begin{proposition}\label{ElementaryPropositionCycles}
Let $\sigma$ be any permutation in $S_n$, let $1 \leq r \neq x \leq n$ with $x$ not contained in the support of $\sigma$, or equivalently $\sigma(x)=x$, then
    \begin{equation*}
        (x\, \sigma(r)) \cdot \sigma \cdot (r\, x) = \sigma.
    \end{equation*}
    
Moreover, if all elements $i_j$ for $1\leq j \leq k$ are distinct, then
\begin{equation*}
    \tau(\vec{i})=(i_1\,i_2)(i_2\,i_3)\cdots(i_{k-1}\,i_k)(i_k\,i_1) = (i_2\,i_3\,i_4\,\dots\, i_{k-1}\,i_k). \end{equation*}
\end{proposition}

We will further introduce some language to study products of transpositions, for $j \geq 1$, let $\sigma_j=(i_1\, i_1')(i_2\, i_2')\cdots(i_{j-1}\, i_{j-1}')(i_j\, i_j')$ be with the property that $i_j \in \{i_1,\dots,i_{j-1},i_1',\dots,i_{j-1}'\}$ for each $j=2,\dots,k$. Furthermore, for each $j=2,\dots,k$ we will say that there is an 
\begin{itemize}
    \item \textbf{Addition}. If $i_{j}'\not \in \{i_1,\dots,i_{j-1},i_1',\dots,i_{j-1}'\}$. Let $s\geq 0$ and denote $(a_1\,\cdots a_s\, i_j)$ the cycle of $\sigma_j$ containing $i_j$. We have that
    \begin{equation*}
        (a_1\,\cdots a_s\, i_j)(i_j\,i_{j}')=(a_1\,\cdots a_s\, i_j\, i_{j}').
    \end{equation*}
    \item \textbf{Splitting}. If  $i_{j}' \in \{i_1,\dots,i_{j-1},i_1',\dots,i_{j-1}'\}$ and both $i_{j}$ and $i_{j}'$ are in the support of the same cycle on $\sigma_j$. Let $s,s'\geq 0$ and denote $(a_1\,\cdots a_s\, i_{j}\, b_1\,\cdots\,b_{s'}\, i_{j}')$ the cycle of $\sigma_j$ containing both $i_j$ and $i_{j}'$. Then
    \begin{equation*}
        (a_1\,\cdots a_s\, i_j\, b_1\,\cdots\,b_{s'}\, i_{j}')(i_j\,i_{j}')=(a_1\,\cdots a_s\, i_j)(b_1\,\cdots b_{s'}\,i_{j}').
    \end{equation*}
    \item \textbf{Merging}. If $i_{j}' \in \{i_1,\dots,i_{j-1},i_1',\dots,i_{j-1}'\}$ but $i_j$ and $i_{j}'$ are in the support of distinct, possibly trivial, cycles of $\sigma_j$. Let $s,s'\geq 0$ and denote $(a_1\,\cdots a_s\, i_j)$ and $(b_1\,\cdots\,b_{s'}\, i_{j}')$ the disjoint cycles of $\sigma_j$ containing $i_j$ and $i_{j}'$ respectively. Then
    \begin{equation*}
        (a_1\,\cdots a_s\, i_j)(b_1\,\cdots\,b_{s'} \,i_{j}')(i_j\,i_{j}')=(a_1\,\cdots a_s\, i_j\, b_1\,\cdots\,b_{s'} i_{j}').
    \end{equation*}
\end{itemize}

We can now answer the question on how to obtain the leading terms in the sum given on equation (\ref{Eq1PCL}).  In fact these are obtained when $Z=\emptyset$ and $\pi\in\textup{NC}_a(k)$. 

\begin{lemma} \label{DkLeadingTerms}
Let $\pi \in \mathcal{P}_a(k)$, $\vec{i} \in I^\pi_\emptyset$ and denote by $\lambda \in \bar{\mathbb{Y}}$ the conjugacy class of $\tau(\vec{i})$. Then $\lambda \in \textup{Rem}_k\big(\bar{\mathbb{Y}}_k\big)$, moreover 
\begin{enumerate}
    \item If $\pi \in \textup{NC}_a(k)$ with $K(\pi) \in \textup{NC}(\mu)$, then $\lambda=\textup{Rem}_k(\mu)$ and $|\pi|= \frac{1}{2}\big(l(\lambda)+k\big)+1$.
    \item If $\pi \not \in \textup{NC}_a(k)$, then $|\pi|\leq \frac{1}{2}\big(l(\lambda)+k\big)$.
\end{enumerate}
\end{lemma}

We describe here a short unified proof of all statements. Before starting with the proof we will give a short example on how to use Lemma \ref{DkLeadingTerms}.

\begin{example}
For $k=8$, notice that the partition $\pi=\big\{\{1,3,5\},\{2\},\{4\},\{6,8\},\{7\}\big\}\in \textup{NC}_a(k)$. Moreover, Example \ref{ExampleKrewerasComplement} showed that $K(\pi)=\big\{\{1,2\},\{3,4\},\{5,8\},\{6,7\}\big\}$, hence $K(\pi)\in \textup{NC}(\mu)$ for $\mu=(2,2,2,2)\in \bar{\mathbb{Y}}_8$. Finally,  $\textup{Rem}_k(\mu)=\emptyset$. This means that $\tau(\vec{i})=e$ the identity element in the symmetric group. We have that $|\pi|=5$ while $l(\lambda)=0$ so the equality in statement \textit{(1)} of the lemma above is satisfied.
\end{example}

\begin{proof}[Proof of Lemma \ref{DkLeadingTerms}]
For $j=2,\dots,k$, let $\sigma_j=(i_1\,i_2)(i_2\,i_3)\cdots(i_{j-1}\,i_j) (i_j\,i_1)$ and let $\lambda^j$ be the conjugacy class of $\sigma_j$.  It follows from Proposition \ref{ElementaryPropositionCycles} that $\sigma_{j+1}=\sigma_j (i_j\,i_{j+1})$. Notice that $l(\lambda^j)=|\sigma_j|$. We are interested in studying the quantities
\begin{equation*}
    M_j=|\lambda^j|+\ell(\lambda^j)\hspace{2mm}\textup{ and }\hspace{2mm}N_j=l(\lambda^j)+j-2|\pi_j|+2,
\end{equation*}
where $\pi_j$ denotes the partition $\pi$ restricted to the set $\{1,\dots,j\}$. Notice that for $j=2$ we have $l(\lambda^2)=0$, $|\lambda^2|=0$, $\ell(\lambda^2)=0$ and $|\pi_2|=2$. Hence $M_2=N_2=0$. When computing $\sigma_{j+1}=\sigma_j (i_j\,i_{j+1})$ we can use the classification between addition, splitting and merging to describe the evolution of $M_j$ and $N_j$. Following the notation above, $s$ and $s'$ denote the number of elements in the cycle where $i_j$ and $i_{j+1}$ appears in the case of an addition, a splitting or a merging. 
\begin{enumerate}
    \item \textbf{Addition}. In this setting, we get $l(\lambda^{j+1})=l(\lambda^{j})+1$ and $|\pi_{j+1}|=|\pi_j|+1$, hence $N_{j+1}=N_j$. 
    Similarly, if $s\geq 1$ then $|\lambda^{j+1}|=|\lambda^j|+1$ and $\ell(\lambda^{j+1})=\ell(\lambda^j)$ , if $s=0$, then $|\lambda^{j+1}|=|\lambda^j|+2$ and $\ell(\lambda^{j+1})=\ell(\lambda^j)+1$, considering both cases we either have $M_{j+1}=M_j+1$ or $M_{j+1}=M_j+3$. Notice that $M_{j+1}=M_j+3$ only occurs only if $s=0$, that is, only if $\sigma_j$ has a fixed point. 
    \item \textbf{Splitting}. We get $l(\lambda^{j+1})=l(\lambda^{j})-1$ and $|\pi_{j+1}|=|\pi_j|$, hence $N_{j+1}=N_j$. By considering all $4$ cases depending on $s=0$ or $s\geq 1$ and $s'=0$ or $s'\geq1$, we get that $M_{j+1}=M_j-3$, $M_{j+1}=M_j-1$ or $M_{j+1}=M_j+1$.
    \item \textbf{Merging}. We get $l(\lambda^{j+1})=l(\lambda^{j})+1$ and $|\pi_{j+1}|=|\pi_j|$, hence $N_{j+1}=N_j+1$. Similarly, either $M_{j+1}=M_j-1$, $M_{j+1}=M_j+1$ or $M_{j+1}=M_j+3$. Notice that $M_{j+1}=M_j+3$ occurs only if $s=s'=0$, that is, only if $\sigma_j$ has a fixed point. 
\end{enumerate}
Since $M_j=j$ implies that $\sigma_j$ has no fixed points, we necessarily have that $M_{j+1}\leq M_j+1$ for each $j=1,\dots,k$, hence $|\lambda^{k}|+\ell(\lambda^{k})\leq k$. Moreover $M_{j+1}=M_j+1\mod{2}$, hence $|\lambda^{k}|+\ell(\lambda^{k})=k \mod{2}$ and Proposition \ref{PropositionRemk} ensures that $\lambda^{k}\in \textup{Rem}_k\big(\bar{\mathbb{Y}}_k\big)$.

We will now prove that $N_{k+1}=0$ if $\pi$ is a non-crossing partition, while $N_{k+1}\geq 1$ if $\pi$ has a crossing. Notice that $N_{k+1}=0$ unless for some $j$ there is a merging. We prove that if $\pi$ is a non-crossing partition then there is never a merging and that if $\pi$ has a crossing then there must be a merging. Note that the cases $k=2$ and $k=|\pi|$ are verified by proposition \ref{ElementaryPropositionCycles}. 
\begin{enumerate}
    \item Without loss of generality let $m=\min\{s\geq 2: i_s=i_1\}$, we will show that if $\pi$ is a non-crossing partition then for each $j=2,\dots,m-1$ we have that $i_j$ and $i_1$ are on the same cycle. At each step, we are either adding $i_{j+1}$, hence $i_j$ and $i_1$ are on the same cycle, or the cycle is of the form $(i_1\, b_1\,\cdots\,b_{s}\, i_{j})$ and because $\pi$ is a non-crossing partition $i_{j+1}\in\{b_1,\dots,b_s\}$ by computing $(i_1\, b_1\,\cdots\,b_{s}\, i_{j})(i_j\, i_{j+1})$ we verify that $i_1$ and $i_j$ are still on the same cycle. This ensures that $i_1$ and $i_{m}$ are on the same cycle and hence each step corresponds to an adding or a splitting. 
    \item Let $j_1$, $j_2$, $j_3$ and $j_4$ such that $i_{j_1}=i_{j_3}\neq i_{j_2}=i_{j_4}$ and $j_4$ is minimal such that $\pi_{j_4}$ has a crossing. It follows by the previous item that $\pi_{j_4-1}$ is a non-crossing partition and hence at each step we have a split. In particular, we have that $i_{j_1}$ and $i_{j_2}$ are on different cycles while $i_{j_4-1}$ is on the same cycle as $i_{j_1}$. Hence when multiplying with $(i_{j_4-1}\, i_{j_4})$ we necessarily get a merge. 
\end{enumerate}

Finally, we prove by induction on $j$, starting with $j=2$ that if $K(\pi_j) \in \textup{NC}(\mu^j)$, then $\lambda^j=\textup{Rem}_j(\mu^j)$. The case $j=2$ being a direct calculation, suppose that the hypothesis is true for $j\geq2$. There is two cases, if $i_{j+1}\not \in\{i_1,\dots,i_j\}$ then we are in an addition situation. It is direct from the definition of the map $K$ that $K(\pi_{j+1}) \in \textup{NC}(\mu^{j+1})$ where $\mu^{j+1}$ is equal to $\mu^j$ by adding a box in a row of the partition. Similarly, in the addition setting, we necessarily have that $\lambda^{j+1}=\textup{Rem}_{j+1}(\mu^{j+1})$. Now suppose that $i_{j+1} \in\{i_1,\dots,i_j\}$, then the block containing $i_{j+1}$ in $\pi_{j+1}$ separates $K(\pi_{j+1})$ in at least two disjoint blocks of size smaller than $j$, this allows to use the induction hypothesis from which we conclude. 
\end{proof}

\begin{remark}\label{Remarktrickforproducts}
Lemma \ref{DkLeadingTerms} allows us to compute the conjugacy class of products of consecutive transpositions based on the associated set partitions, we need to keep track of the additions, splittings and merging at each step. However, we can exploit Proposition \ref{ElementaryPropositionCycles} to compute the conjugacy class $\lambda$ of any product of permutations based on their associated partitions. As already stated in the proposition, we can describe any cycle as a product of consecutive transpositions, now suppose that some family of permutations $\sigma_1,\dots,\sigma_r$ can be written as a product of transpositions, for each permutation $\sigma_j$ denote $i_j$ and $i_j'$ the first and last index in the description as product of transpositions, Proposition \ref{ElementaryPropositionCycles} gives 
\begin{equation*}
   \sigma_1\sigma_2\cdots\sigma_r=(\ast\, i_1) \sigma_1(i_1'\,\ast)(\ast\,i_2)\sigma_2\cdots(\ast\,i_r)\sigma_r(i_r'\,\ast)
\end{equation*}
where $\ast$ is any indeterminate element. While this approach would allow to compute the conjugacy classes of the products we will not need such a detailed description. 
\end{remark}

We can now use Lemma \ref{DkLeadingTerms} on equation (\ref{Eq1PCL}) to find the expansion of $D_k$.

\begin{proof}[Proof of Theorem \ref{TheoremExpansionofDk}]
 For each $\pi \in \mathcal{P}_a(k)$ and $\vec{i} \in I^\pi_Z$ we consider three different cases.
\begin{enumerate}
    \item If $Z\neq \emptyset$, take any $\vec{i}' \in I^\pi_\emptyset$.  It follows from Proposition \ref{ElementaryPropositionCycles} that  $\tau(\vec{i})$ and $\tau(\vec{i}')$ are on the same conjugacy class, which we index with the partition $\lambda$. Since $|I^\pi_Z|=n^{|\pi|-1}\big(1+O(n^{-1})\big)$, then
    \begin{equation*}
        \sum_{\vec{i} \in I^\pi_{Z}(k)} \tau(\vec{i}) = n^{|\pi|-1}\big(1+O(n^{-1})\big) \hat{\Sigma}_\lambda.
    \end{equation*}
    It follows from Lemma \ref{DkLeadingTerms} items \textit{(2)} that $|\pi|\leq \frac{1}{2}\big(l(\lambda)+k\big)+1$, hence
    \begin{equation} \label{EQ1TheoremExpansionofDkproof}
        \dfrac{1}{n+1}\sum_{\vec{i} \in I^\pi_{Z}(k)} \tau(\vec{i}) = O(n^{\frac{1}{2}\big(l(\lambda)+k\big)-1}) \hat{\Sigma}_\lambda.
    \end{equation}
    
    \item If $\pi \not \in \textup{NC}_a(k)$ and $Z=\emptyset$ then Lemma \ref{DkLeadingTerms} item \textit{(2)} guarantees that $|\pi|\leq \frac{1}{2}\big(l(\lambda)+k\big)$. Since $|I^\pi_Z|=n^{|\pi|}\big(1+O(n^{-1})\big)$, then
    \begin{equation} \label{EQ2TheoremExpansionofDkproof}
        \dfrac{1}{n+1}\sum_{\vec{i} \in I^\pi_{\emptyset}(k)} \tau(\vec{i}) = O(n^{\frac{1}{2}\big(l(\lambda)+k\big)-1}) \hat{\Sigma}_\lambda.
    \end{equation}
    
    \item Notice that $\{\textup{NC}(\mu)\}_{\mu \in \bar{\mathbb{Y}}_k}$ is a set partition of $\textup{NC}_{\geq 2}(k)$, hence  $\big\{K^{-1}\big(\textup{NC}(\mu)\big)\big\}_{\mu\in  \bar{\mathbb{Y}}_k}$ is also a set partition of $\textup{NC}_a(k)$. Additionally, Proposition \ref{Propositionbijection} guarantees that $K:\textup{NC}_g(k)\to\textup{NC}_{\geq 2}(k)$ is a bijection, hence $|K^{-1}\big(\textup{NC}(\mu)\big)|=|\textup{NC}(\mu)|$. This observation jointly with $|I^\pi_Z|=n^{|\pi|}\big(1+O(n^{-1})\big)$ gives
    \begin{align*}
        \sum_{\pi \in \textup{NC}_a(k)} \sum_{\vec{i} \in I^\pi_\emptyset} \tau(\vec{i}) &= \sum_{\mu \in \bar{\mathbb{Y}}_k} \sum_{\pi \in K^{-1}\big(\textup{NC}(\mu)\big)} \sum_{\vec{i} \in I^\pi_\emptyset} \tau(\vec{i})\\
        &=\sum_{\mu \in \bar{\mathbb{Y}}_k} \sum_{\pi \in K^{-1}\big(\textup{NC}(\mu)\big)} n^{|\pi|}\big(1+O(n^{-1})\big) \hat{\Sigma}_{\textup{Rem}_k(\mu)}\\
        &=\sum_{\mu \in \bar{\mathbb{Y}}_k} |\textup{NC}(\mu)| n^{|\pi|}\big(1+O(n^{-1})\big) \hat{\Sigma}_{\textup{Rem}_k(\mu)}
    \end{align*}
    Lemma \ref{DkLeadingTerms} item \textit{(1)} ensures that $|\pi|=\frac{1}{2}\big(l(\lambda)+k\big)+1$. By renaming $\lambda=\textup{Rem}_k(\mu)$ and $\mathcal{S}_k=\textup{Rem}_k(\bar{\mathbb{Y}}_k)$ we get
    \begin{equation} \label{EQ3TheoremExpansionofDkproof}
    \dfrac{1}{n+1}\sum_{\pi \in \textup{NC}_a(k)} \sum_{\vec{i} \in I^\pi_\emptyset} \tau(\vec{i}) =\sum_{\lambda \in \mathcal{S}_k} |\textup{NC}(\mu)| n^{\frac{1}{2}\big(l(\lambda)+k\big)}\big(1+O(n^{-1})\big) \hat{\Sigma}_{\lambda}.
    \end{equation}
\end{enumerate}

Putting together equations (\ref{EQ1TheoremExpansionofDkproof}), (\ref{EQ2TheoremExpansionofDkproof}) and (\ref{EQ3TheoremExpansionofDkproof}) together with equation (\ref{Eq1PCL}) gives the conclusion. 
\end{proof}

\begin{remark}
Note that another proof of Theorem \ref{TheoremExpansionofDk} can be found in \cite{Ho} while a partial proof can be derived from \cite[section 4]{Bi}. We simplify these approaches via Lemma \ref{DkLeadingTerms}.
\end{remark}

We now prove a simplified version of Theorem \ref{lemmaproductoftwopartitions2} where rather than arbitrary permutations we simply have cycles. This is a preliminary step that will allow us to prove the general case. 

\begin{lemma}\label{lemmaproductofcycles}
Let $k_1, k_2$ and $n_1\leq n_2 \leq n$ be integers, we have the following expansion,
\begin{equation*}
\Sigma_{k_1}\bigr|_{n_1}\Sigma_{k_2}\bigr|_{n_2}= (n_1\ff k_1)(n_2-k_1\ff k_2)\bar{\Sigma}_{k_1,k_2}+Q_{k_1,k_2}+R_{k_1,k_2}
\end{equation*} 
where 

\begin{equation*}
Q_{k_1,k_2}=\sum_{r\geq 1} \sum_{\substack{s_1,\dots,s_r\geq 1\\ s_1+\dots+s_r=k_1}}\sum_{\substack{t_1,\dots,t_r\geq 1\\ t_1+\dots+t_r=k_2}} \dfrac{k_1 k_2}{r} (n_1\ff k_1)(n_2-k_1\ff k_2-r) \bar{\Sigma}_{\cup_{i=1}^r (s_i+t_i-1)},
\end{equation*} 
if any $s_i+t_i-1$ are equal to $1$ we delete them to ensure that we obtain a partition in $\bar{\mathbb{Y}}$, and there exists a set of partitions $\mathcal{R} \subset \bar{\mathbb{Y}}$ such that 
\begin{equation*}
R_{k_1,k_2}=\sum_{\mu \in \mathcal{R}} N_\mu[R](n) \bar{\Sigma}_\mu
\end{equation*}
and $N_\mu[R]=O(n^\beta)$ with $\beta=\frac{l(\mu)+k_1+k_2}{2}-1$. 
\end{lemma}

\begin{proof}
The idea of this proof is that if the intersection of the support of the cycles is non-empty, then we can do an algebraic manipulation to rewrite the product of the cycles as $\tau(\vec{i})$, this allows us to use Lemma \ref{DkLeadingTerms} to do computations. For $m=1$ and $m=2$ denote
\begin{equation*}
    I'_{m}=\big\{(i_1,\dots,i_{k_m})\in \{1,\dots,n_m\}^{k_m}: \textup{All indices } (i_m)_m \textup{ are distinct} \big\}. 
\end{equation*}
Additionally, let $\tau':I'_m\to S_{n_m}$ be defined by
\begin{equation*}
    \tau'(\vec{i})=(i_1\,i_2\,\cdots\,i_{k_m})
\end{equation*}
We then have that
\begin{equation*}
\Sigma_{k_1}\bigr|_{n_1}\Sigma_{k_2}\bigr|_{n_2}= \sum_{(\vec{i}_1,\vec{i}_2)\in I'_{1}\times I'_{2}} \tau'(\vec{i}_1)\tau'(\vec{i}_2),
\end{equation*} 
where $\vec{i}_1=(i_1,\dots,i_{k_1})$ and $\vec{i}_2=(i_{k_1+1},\dots,i_{k_1+k_2})$. Let $\pi\in \Theta_{k_1+k_2}$, denote $J^\pi$ the set of indices $(\vec{i}_1,\vec{i}_2) \in I'_{1}\times I'_{2} $ such that $i_{a}=i_b$ if and only if $a,b$ are in the same block of $\theta$. The conjugacy class of $\tau'(\vec{i}_1)\tau'(\vec{i}_2)$ only depends on $\pi$. Since $ \tau'(\vec{i}_1)$ and $ \tau'(\vec{i}_2)$ are cycles of length $k_1$ and $k_2$ respectively, then $J^\pi$ is empty unless the following two conditions are satisfied. For each $\pi \in \Theta_{k_1+k_2}$ we have
\begin{enumerate}
    \item All blocks have size $2$ or $1$. 
    \item Each block of size $2$ contains exactly one element in $\{1,\dots,k_1\}$ and in $\{k_1+1,\dots,k_1+k_2\}$.
\end{enumerate}
Denote $\mathcal{F}_r$ the set of such partitions with $r$ blocks of size $2$, that is we have that
\begin{equation*}
    \Sigma_{k_1}\bigr|_{n_1}\Sigma_{k_2}\bigr|_{n_2} = \sum_{r\geq 0} \sum_{\pi \in \mathcal{F}_r} \sum_{(\vec{i}_1,\vec{i}_2)\in J^\pi} \tau'(\vec{i}_1)\tau'(\vec{i}_2).
\end{equation*}

We start by describing the set $\mathcal{F}_0$. In this situation all blocks in $\pi$ have size $1$, that is, all indices $(i_1,\dots,i_{k_1+k_2})$ are distinct. In this situation $J^\pi$ contains exactly $(n_1\ff k_1)(n_2-k_1\ff k_2)$ elements and $\tau'(\vec{i}_1)\tau'(\vec{i}_2)$ is a product of two disjoint cycles of length $k_1$ and $k_2$, so
\begin{equation}\label{EQ1LemmaProdCyc}
\sum_{\pi \in \mathcal{F}_0} \sum_{(\vec{i}_1,\vec{i}_2)\in J^\pi} \tau'(\vec{i}_1)\tau'(\vec{i}_2)=(n_1\ff k_1)(n_2-k_1\ff k_2)\bar{\Sigma}_{k_1,k_2}.
\end{equation}

Now let $r\geq 1$, since $\tau'$ is invariant under cyclic permutations on the variables $\vec{i}_1$ and $\vec{i}_2$ we can write
\begin{equation*}
    \tau'(\vec{i}_1)\tau'(\vec{i}_2)=(i_1\,i_2\,\cdots\,i_{k_1}) (i_{k_1+k_2}\,i_{k_1+1}\cdots\,i_{k_1+k_2-1})
\end{equation*}
such that $i_{k_1}=i_{k_1+k_2}$. There are $k_1k_2/r$ terms of this form.  Proposition \ref{ElementaryPropositionCycles} gives 
\begin{equation*}
    \tau'(\vec{i})\tau'(\vec{j})=(0\,i_1)(i_1\,i_2)(i_2\,i_3)\cdots(i_{k_1-1}\,i_{k_1}) (i_{k_1}\,i_{k_1+1})\cdots (i_{k_1+k_2-2} \, i_{k_1+k_2-1}) (i_{k_1+k_2-1} \, 0).
\end{equation*}

By shifting the indices, we conclude that there is a bijection between $\mathcal{F}_r$ and $\mathcal{E}_r \subseteq \Theta_{k_1+k_2}$ where $ \pi \in \mathcal{E}_r$ is such that
\begin{enumerate}
    \item $\pi$ has $r-1$ blocks of size $2$ and all other blocks have size $1$.
    \item Each block of size $2$ contains exactly one element in $\{2,\dots,k_1\}$ and in $\{k_1+2,\dots,k_1+k_2\}$.
    \item Both $1$ and $k_1+1$ are contained in blocks of size $1$. 
\end{enumerate}

Denote $I^{\pi}_\emptyset(n_1,n_2)$ the set of $\vec{i} \in I^{\pi}_{\{1\}}$ such that $1\leq i_j\leq n_1$ if $2 \leq j \leq k_1+1$, $1\leq i_j\leq n_2$  if $k_1+1 \leq j \leq k_1+k_2$, we just proved that
\begin{equation*}
     \sum_{\pi \in \mathcal{F}_r} \sum_{(\vec{i}_1,\vec{i}_2)\in J^\pi} \tau'(\vec{i}_1)\tau'(\vec{i}_2) = \frac{k_1k_2}{r} \sum_{\pi \in \mathcal{E}_r} \sum_{\vec{i} \in I^{\pi}_{\{1\}}(n_1,n_2)} \tau(\vec{i}).
\end{equation*}
At this point we are in good shape to apply Lemma \ref{DkLeadingTerms}. Let  $\pi \in \mathcal{E}_{r}\cap \textup{NC}_g(k_1+k_2)$, then the partition is uniquely determined by the positions of the indices contained in blocks of size $2$. For each $s_1,\dots,s_r,t_1,\dots,t_r\geq 1$ such that $ s_1+\dots+s_r=k_1$ and $ t_1+\dots+t_r=k_2$. Let $\pi$ be the partition formed uniquely by blocks of size $1$ and $2$,  with blocks of size $2$ being 
\begin{equation*}
\{2+s_1+\dots+s_j-j,k_1+k_2-(t_1+\dots+t_j)+j\}    
\end{equation*}
for $j=1,\dots,r-1$. See Figure \ref{FigureExampleCalculationProductCycles} for an illustration on how this partition looks like and a computation of its Kreweras complement. Similarly to the case $r=0$, we have that $|I^{\pi}_{\{1\}}(n_1,n_2)|=(n_1\ff k_1)(n_2-k_1\ff k_2-r)$, We now use Lemma \ref{DkLeadingTerms} item \textit{(1)} to compute the conjugacy class corresponding to such partitions. Kreweras map gives that $K(\pi) \in \textup{NC}(\cup_{i=1}^r (s_i+t_i))$ and hence $\textup{Rem}_{k_1+k_2}(\cup_{i=1}^r (s_i+t_i))=\cup_{i=1}^r (s_i+t_i-1)$ where we identify the entries of length $1$ with the empty partition. We conclude that
\begin{equation}\label{EQ2LemmaProdCyc}
\begin{split}
   \sum_{\pi \in \mathcal{E}_{r}\cap \textup{NC}_a(k_1+k_2)} \sum_{\vec{i} \in I^{\pi}_{\{1\}}(n_1,n_2)} \tau(\vec{i}) = \sum_{\substack{s_1,\dots,s_r\geq 1\\ s_1+\dots+s_r=k_1}}\sum_{\substack{t_1,\dots,t_r\geq 1\\ t_1+\dots+t_r=k_2}}(n_1\ff k_1)(n_2-k_1\ff k_2-r) \\
   \cdot \bar{\Sigma}_{\cup_{i=1}^r (s_i+t_i-1)}.
\end{split}
\end{equation}
\begin{figure}[h]
\vspace{-0mm}
    \centering
    \includegraphics[scale=0.5]{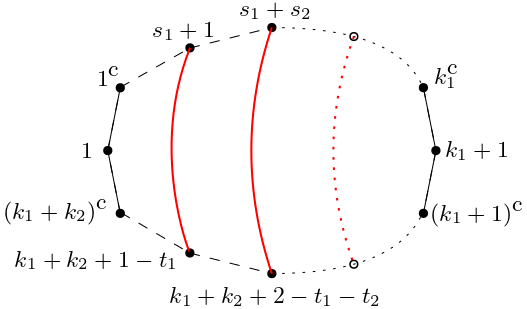}
    \caption{Illustration of computation of Kreweras complement.} \label{FigureExampleCalculationProductCycles}
\end{figure}

We finally define
\begin{equation*}
R_{k_1,k_2}= \sum_{\pi \in \mathcal{E}_{r}\backslash \textup{NC}_a(k_1+k_2)} \sum_{\vec{i} \in I^{\pi}_{\{1\}}(n_1,n_2)} \tau(\vec{i}).
\end{equation*}
Notice that $|I^{\pi}_{\{1\}}(n_1,n_2)|=n^{|\pi|-1}\big(1+O(n^{-1})\big)$ then Lemma \ref{DkLeadingTerms} item \textit{(2)} ensures that
\begin{equation}\label{EQ3LemmaProdCyc}
    \sum_{\vec{i} \in I^{\pi}_{\{1\}}(n_1,n_2)} \tau(\vec{i}) =n^{|\pi|-1}\big(1+O(n^{-1})\big) \bar{\Sigma}_\mu
\end{equation}
where  $|\pi|-1\leq \frac{l(\mu)+k_1+k_2}{2}-1$. We conclude from equations (\ref{EQ1LemmaProdCyc}), (\ref{EQ2LemmaProdCyc}) and (\ref{EQ3LemmaProdCyc}). 
\end{proof}

Furthermore, we will need to develop some bounds for the size of the associated set partition of a collection of permutations. We start by considering the case in which all partitions are cycles. 

\begin{lemma}\label{lemmaJointlyintersectionofcycles}
Following the notation of Definition \ref{WeirdIBdefinition}. Let $r\geq 1$ and $k_1,\dots, k_r\geq 2$ a collection of integers. Let $(\sigma_1,\dots,\sigma_r)\in \mathcal{A}[(k_1,\dots,k_r)]$ be jointly intersecting with associated partition $\theta\in \Theta_{k_1+\dots k_r}$. Let $\mu\in \bar{\mathbb{Y}}$ be the conjugacy class of $\prod_{j=1}^r \sigma_j$. Then
\begin{equation*}
    |\theta|\leq \dfrac{l(\mu)-r+\sum_{j=1}^r k_j}{2}+1
\end{equation*}
\end{lemma}
\begin{proof}
The proof is based on both the proofs of Lemma \ref{DkLeadingTerms}  and Lemma \ref{lemmaproductofcycles}. Denote the cycles to be $\sigma_j=(i^j_1\,i^j_2\,\cdots\,i^j_{k_j})$. We start by showing that without loss of generalization, we can assume that we have $i^j_1\in\{i^1_1,\dots,i^1_{k_1},i^2_1,\dots,i^{j-1}_{k_{j-1}}\}$ for each $j=2,\dots,r$, we call this the good property.

We produce an algorithm to show that there exist a new sequence of cycles $(\sigma_1',\dots,\sigma_r')$ that has the the good property and such that the following three properties hold
\begin{enumerate}
    \item The sizes of the associated set partitions of $(\sigma_1',\dots,\sigma_r')$ and $(\sigma_1,\dots,\sigma_r)$ are equal.
    \item The sum of lengths of cycles in $(\sigma_1',\dots,\sigma_r')$ and $(\sigma_1,\dots,\sigma_r)$ are equal.
    \item The conjugacy classes of $\prod_{j=1}^r \sigma_j'$ and $\prod_{j=1}^r \sigma_j$ are the same.
\end{enumerate}
We achieve this by modifying the list $(\sigma_1,\dots,\sigma_r)$ one index at the time. Start by choosing $\sigma_1'=\sigma_1$, we trivially have (1), (2) and (3) for $(\sigma_1',\sigma_2\dots,\sigma_r)$.  Now suppose that we have chosen $\sigma_1',\dots,\sigma_j'$. If we have the good property with $\sigma_1',\dots,\sigma_j',\sigma_{j+1}$ we choose $\sigma'_{j+1}=\sigma_{j+1}$ and  (1), (2) and (3) hold for $(\sigma_1',\sigma_2'\dots,\sigma'_{j+1},\dots,\sigma_r)$. Otherwise, the support of $\sigma_{j+1}$ is disjoint with the support of $\sigma_1',\dots,\sigma'_j$, hence the list $(\sigma_{j+1},\sigma_1',\dots,\sigma_j',\sigma_{j+2},\dots,\sigma_r)$ still satisfies (1), (2) and (3). In addition, since we are only interested in the conjugacy class, the list $(\sigma_1',\dots,\sigma_j',\sigma_{j+2},\dots,\sigma_r,\sigma_{j+1}^{-1})$ also satisfies (1), (2) and (3). The since the sequence of permutations is jointly intersecting we are guaranteed that the algorithm is correct, that is, eventually the algorithm stops and we obtain the desired list $(\sigma_1',\dots,\sigma_r')$.

This ensures that we have the identity
\begin{equation*}
\prod_{j=1}^r \sigma_j = (i^1_1\, i^1_2)\cdots (i^1_{k_1-1}\, i^1_{k_1})(i^2_1\,i^2_2)\cdots(i^2_{k_2-1}\,i^2_{k_2})\cdots(i^r_1\,i^r_2)\cdots (i^r_{k_r-1}\,i^r_{k_r}).
\end{equation*}
and that $i^j_1\in\{i^1_1,\dots,i^1_{k_1},i^2_1,\dots,i^{j-1}_{k_{j-1}}\}$ for each $j=2,\dots,r$. That is, $\prod_{j=1}^r \sigma_j$ is a product of $k=k_1+\dots+k_r-r$ transpositions which we further rewrite as  
\begin{equation*}
    \prod_{j=1}^r \sigma_j=(i_1\, i_1')(i_2\, i_2')\cdots(i_{k-1}\, i_{k-1}')(i_k\, i_k')
\end{equation*}
with the property that $i_j \in \{i_1,\dots,i_{j-1},i_1',\dots,i_{j-1}'\}$ for each $j=2,\dots,k$. Furthermore, denote the partial products of these transpositions as $\textup{Part}_j=(i_1\, i_1')(i_2\, i_2')\cdots(i_{j}\, i_{j}')$, we clearly have that $\textup{Part}_{j+1}=\textup{Part}_j (i_{j+1}\,i_{j+1}')$. Denote $\lambda^j\in \bar{\mathbb{Y}}$ the conjugacy class of $\sigma_j$. Let $D_j=|\{i_1,\dots,i_{j},i_1',\dots,i_{j}'\}|$ we want to study the quantity $N_j=j+l(\lambda^j)-2D_j+2$, notice that $N_1=0$. Similarly to the proof of Lemma \ref{DkLeadingTerms}, at each step we are then guaranteed that each step we either have an addition, a merging or a splitting, hence $N_j\leq 0 $ for $j=1,\dots,k$. Noticing that $D_k=|\theta|$, the inequality $N_k\leq 0$ for $k=k_1+\dots+k_r-r$ gives the conclusion. 
\end{proof}

\begin{remark}
Notice that in the proofs of Lemma \ref{lemmaproductofcycles} and Lemma \ref{lemmaJointlyintersectionofcycles} we did not use the method proposed on Remark \ref{Remarktrickforproducts} by introducing an indeterminate element. While this approach is possible, a direct application of Lemma \ref{DkLeadingTerms} will not provide us with the desired bound, rather in a similar fashion to the proof of Lemma \ref{DkLeadingTerms} we would need to count the number of times a merging occurs to obtain the desired bound. Our approach bypasses this calculation. 
\end{remark}

We now generalize the previous lemma to arbitrary permutations rather than cycles. 
\begin{lemma}\label{lemmaJointlyintersectionofpermutations}
Following the notation of Definition \ref{WeirdIBdefinition}. Let $r \geq 1$, $\vec{\lambda}=(\lambda^1,\dots,\lambda^r)$ a vector of $r$ partitions and let $(\sigma_1,\dots,\sigma_r)\in \mathcal{A}[\vec{\lambda}]$  be jointly intersecting with associated partition $\theta\in \Theta_{|\vec{\lambda}|}$. Let $\mu \in \bar{\mathbb{Y}}$ be the conjugacy class of $\prod_{j=1}^r \sigma_j$. Then
\begin{equation*}
    |\theta|\leq \dfrac{l(\mu)+|\vec{\lambda}|+\ell(\vec{\lambda})}{2}+1-r
\end{equation*}
\end{lemma}
\begin{proof}
The proof is based on considering intersection of cycles on the permutations, which is handled by Lemma \ref{lemmaJointlyintersectionofcycles} and considering the cycles that are disjoint to the rest of the products. Consider the set of cycles $\sigma(\lambda^j_i)$ for $j=1,\dots,r$ and $i=1,\dots,\ell(\lambda^j)$. By putting together maximal set of cycles that jointly intersect we define a partition $\{D_1,\dots,D_k\}\in \Theta_{|\vec{\lambda}|}$. Denote $\theta_s$ the restriction of $\theta$ into each $D_s$. Similarly consider $\mu^s$ the partition induced from the conjugacy class of the product of cycles that induced $D_s$ and $r_s$ to be the number of cycles forming $D_s$. Lemma \ref{lemmaJointlyintersectionofcycles} gives
\begin{equation*}
    |\theta_s| \leq \dfrac{l(\mu^s)-r_s+|D_s|}{2}+1\, \textup{ for each } s=1,\dots,k.
\end{equation*}
Moreover, we have $|\theta|=\sum_{s=1}^k |\theta_s|$,  $l(\mu)=\sum_{s=1}^k l(\mu^s)$, $|\vec{\lambda}|=\sum_{s=1}^k |D_s|$ and that $\ell(\vec{\lambda})=\sum_{s=1}^k r_s$. From which we get that
\begin{equation*}
    |\theta| =\sum_{s=1}^k |\theta_s|
    \leq \sum_{s=1}^k \Big(\dfrac{l(\mu^s)-r_s+|D_s|}{2}+1\Big) = \dfrac{l(\mu)+|\vec{\lambda}|+\ell(\vec{\lambda})}{2}+k-\sum_{s=1}^k r_s.
\end{equation*}
Finally notice that for each $s=1,\dots,k$, $r_s\geq 1$ and since the permutations $\sigma_1,\dots,\sigma_r$ are jointly intersecting, there is at least one $s$ such that $r_s\geq r$. This ensures that $k-\sum_{s=1}^k r_s \leq 1-r$ from which the conclusion follows.
\end{proof}

We can finally prove our main results concerning the expansions of operators. The proof of Theorem \ref{lemmaproductoftwopartitions2} is in practical terms the same as the one of Lemma \ref{lemmaproductofcycles} with the added difficulty that we need to correctly bound the order of magnitude of some terms. Fortunately, this is handled by Lemma \ref{lemmaJointlyintersectionofpermutations}.

\begin{proof}[Proof of Theorem \ref{lemmaproductoftwopartitions2}]
This is a sum of product of cycles with possible common support. There are three cases to consider. (1) All supports are disjoint. (2)  Only two cycles jointly intersect. (3) Either three cycles jointly intersect or two pairs of cycles jointly intersect. 

In case (1), there are $(n_1\ff |\lambda|)(n_2-|\lambda|\ff|\lambda'|) $ such permutations, Additionally, the conjugacy class of the product of two permutations with conjugacy classes $\lambda$ and $\lambda'$ is $\lambda\cup\lambda'$. In case (2), select the two cycles that jointly intersect to have lengths $i\in \lambda$ and $j \in \lambda'$. Lemma \ref{lemmaproductofcycles} ensures that the leading term in the product between these two cycles can be expanded as a summation of permutations with conjugacy classes $\cup_{m=1}^r (s_m+t_m-1)$ for $s_1,\dots,s_r,t_1,\dots,t_r\geq 1$ such that $ s_1+\dots+s_r=i$ and $ t_1+\dots+t_r=j$, we are additionally multiplying with two disjoint permutations with conjugacy classes given by $\lambda\backslash\{i\}$ and $\lambda'\backslash\{j\}$, this gives a permutation with conjugacy class $\lambda\backslash\{i\}\cup \lambda'\backslash\{j\}\cup_{m=1}^r (s_m+t_m-1)$. There are $(n_1\ff |\lambda|)(n_2|\lambda|\ff|\lambda'|-r) $ ways to obtain such a permutation. In case (3), Lemma \ref{lemmaJointlyintersectionofpermutations} ensures that the remaining terms will have order at most $O(n^\beta)$ with $\beta=\frac{l(\mu)+|\lambda|+|\lambda'|+\ell(\lambda)+\ell(\lambda')}{2}-2$.
\end{proof}

\begin{proof}[Proof of Theorem \ref{Product1stOrderExpansionUPGRADEDMultilevel}]
Fix $r\geq 1$, $\vec{n}=(n_1,\dots,n_r)$ and $\vec{\lambda}=(\lambda^1,\dots,\lambda^r)$. Denote
\begin{equation*}
I=\Big\{\big(i[1],\dots,i[\bigr|\vec{\lambda}\bigr|]\big): 1\leq i[a] \leq n_j\, \textup{ for }|\lambda^{1}|+\dots+|\lambda^{j-1}| \leq a \leq |\lambda^{1}|+\dots+|\lambda^{j}|,\, j=1,\dots,r\Big\}.    
\end{equation*}
Additionally, for $\lambda^j=(\lambda^j_1\geq \dots \geq \lambda^j_s)$ and each $\vec{i} \in I$ consider the cycles of length $\lambda^j_{t}$
\begin{equation*}
    \tau_{j,t}[\vec{i}]:=\bigg(i\Big[|\lambda^{1}|+\dots+|\lambda^{j-1}|+\lambda^j_1+\dots+\lambda^j_{t-1}\Big]\,\cdots \, i\Big[|\lambda^{1}|+\dots+|\lambda^{j-1}|+\lambda^j_1+\dots+\lambda^j_{t}\Big]\bigg),
\end{equation*}
for each $1\leq t\leq s$. Furthermore, we consider the permutations 
\begin{equation*}
    \tau_j[\vec{i}]:=\tau_{j,1}[\vec{i}]\cdots\tau_{j,s}[\vec{i}] \hspace{2mm}\textup{ for }\hspace{1mm}1\leq j \leq r, \hspace{1mm}\textup{ and }\hspace{2mm} \tau[\vec{i}]:=\tau_1[\vec{i}]\cdots \tau_r[\vec{i}].
\end{equation*}

For each $\theta \in \Theta_{|\vec{\lambda}|}$ denote $I^{\theta}$ to be the set of indices $\vec{i}\in I$ such that $i[a]=i[b]$ if and only if $a$ and $b$ are in the same block of $\theta$. We can now define  
\begin{equation*}
\Delta_{\vec{\lambda}}^\pi:=\sum_{(\vec{\mu},\theta) \in \Xi_\pi[\vec{\lambda}]} \sum_{\vec{i} \in I^\theta} \tau[\vec{i}].    
\end{equation*}

We clearly have that $\prod_{j=1}^r \Sigma_{\lambda^j}\bigr|_{n_j} = \sum_{\pi \in \Theta_r} \Delta_{\vec{\lambda}}^\pi$. Furthermore, for each  $\vec{i}$ the conjugacy class of $\tau[\vec{i}]$ only depends on $\theta$ and is given by $\mu$. We are left to estimate $|I^\theta|$. The combinatorial interpretation ensures that there exists a vector of integers $\vec{k}_\theta$ such that $|I^\theta|=(\vec{n}\ff \vec{k}_\theta)$ with $|\vec{k}_\theta|=|\theta|$. Fix $\pi \in \Theta_r$, $B\in \pi$ and $\vec{i}\in I^\theta$. It follows from Definition \ref{WeirdDefinitionSECOND} of the set $\Xi_\pi[\vec{\lambda}]$ that the collection of permutations $\big(\tau_j[\vec{i}]\big)_{j\in B}$ jointly intersects. Introduce the set
\begin{equation*}
D_B:=\bigcup_{j \in B} \Big\{|\lambda^{1}|+\dots+|\lambda^{j-1}|+1,\dots, |\lambda^{1}|+\dots+|\lambda^{j-1}|+|\lambda^j|\Big\}.
\end{equation*}

Then $\theta$ can be partitioned into $\theta_B \in \Theta_{D_B}$ for each $B \in \pi$. Denote by $\mu_B$ the conjugacy class of $\prod_{j \in B}\tau_j[\vec{i}]$. Note that $|\theta|=\sum_{B \in \pi} |\theta_B|$ and that $\mu=\cup_{B \in \pi} \mu_B$, hence Lemma \ref{lemmaJointlyintersectionofpermutations} ensures that
\begin{equation*}
|\theta_B| \leq \frac{l(\mu_B)+|\vec{\lambda}_B|+\ell(\vec{\lambda}_B)}{2}+1-|B|.
\end{equation*}
Since $\sum_{B \in \pi} 1-|B|= |\pi|-r$, we conclude item (1) of the lemma. When taking $\pi=\hat{0}$ we clearly have that $|I^\pi|=(\vec{n}\ff(|\lambda^1|,\dots,|\lambda^r|))$. Additionally, if all permutations are disjoint then the conjugacy class of the product is given by $\cup_{j=1}^r \lambda^j$, hence $\Delta_{\vec{\lambda}}^\pi = \Sigma_{\vec{\lambda}}$ and conclude item (2). Finally, Lemma \ref{lemmaproductofcycles} ensures that each time we have a product of two jointly intersecting cycles of length $k_1$ and $k_2$, denoting by $\mu$ the conjugacy class of their product, we have that either $l(\mu)<k_1+k_2-2$ or $l(\mu)=k_1+k_2-2$ and $\ell(\mu)=1$. By iterating this argument over the multiple cycles in $\vec{\lambda}$ we get that for $\mu \in \Xi_\pi[\vec{\lambda}]$ and $\pi \neq \hat{0}$, then $l(\mu)<l(\vec{\lambda})$ or $l(\mu)=l(\vec{\lambda})$ and $\ell(\mu) < \ell(\vec{\lambda})$.

Finally, let $B\subseteq\{1,\dots,r\}$, by choosing $(\mu_B,\theta_B)$ as before, we clearly have that $\vec{k}_{\theta}=\sum_{B \in \pi} \vec{k}_{\theta_B}$, from which the final part of the lemma is proven. Note that this term also appears on the sum $\Delta_{\vec{\lambda}_B}^{\hat{1}_B}=\sum_{(\mu_B,\theta_B)\in \Xi_{\hat{1}_B}[\vec{\lambda}_B]} (\vec{n}\ff \vec{k}_{\theta_B}) \bar{\Sigma}_{\mu_B}$.
\end{proof}

\section{LLN: Proof of Theorem \ref{TheoremLLN}}\label{SectionProofsLLN}

\subsection{Law of large numbers}

We start by showing that if $\rho_n$ is LLN-appropriate then it satisfies a LLN as $n\to \infty$. It will be convenient to introduce a notion of cumulant adapted to permutations. 

\begin{definition}
Given $r \in \N$, a collection of permutations $\sigma_1,\dots,\sigma_r \in S_{\infty}$ with disjoint support and let $M_\rho$ to be the associated character of a probability distribution $\rho$ on $\mathbb{Y}_n$. The $r$th order \textbf{permutation-cumulant} of $\sigma_1,\dots,\sigma_r$ is defined through
$$\kappa_{\rho,r}(\sigma_1,\sigma_2,\dots) := \sum_{\pi \in \Theta_r} (-1)^{|\pi|-1} (|\pi|-1)! \prod_{B \in \pi} M_\rho\big( \prod_{j\in B} \sigma_{j}\big).$$
\end{definition}

Notice that the permutation-cumulant is nothing else but a classical cumulant adapted to take values on the symmetric group. This can be defined in more general settings by making use of the tools of free probability. In fact, the $1$st order and $2$nd order permutation-cumulants are very manageable and coincide with the expectation and covariance in the classical case; we have that
\begin{equation*}
\kappa_{\rho,1}(\sigma)=M_\rho(\sigma)   \hspace{2mm}\textup{ and }\hspace{2mm} \kappa_{\rho,2}(\sigma_1,\sigma_2)=M_\rho(\sigma_1\sigma_2)-M_\rho(\sigma_1)M_\rho(\sigma_1).
\end{equation*}
Similarly as to how $M_{\rho}$ can be extended linearly to be defined on $\C[S_\infty]$, we can do the same for the permutation-cumulant. In fact, many properties are shared with the classical cumulants as shown by the next few results. 

\begin{proposition}\label{PropMultiLinePermCumul}
Given $r \in \N$ and  $\rho$ a probability distribution on $\mathbb{Y}_n$. Consider the $r$th order permutation-cumulant $\kappa_{\rho,r}:\prod_{j=1}^r\R[S_\infty]\to \R$, then
\begin{enumerate}
    \item The $r$th order permutation-cumulant $\kappa_{\rho,r}$ is multilinear. 
    \item If $x_1,\dots,x_r \in \R[S_\infty]$ commute with each other, then the $r$th order permutation-cumulant is invariant under permutations, that is for any permutation $\tau \in S_r$, 
    \begin{equation*}
    \kappa_{\rho,r}(x_1,\dots,x_r)=\kappa_{\rho,r}(x_{\tau(1)},\dots,x_{\tau(r)}).
    \end{equation*}
\end{enumerate}
\end{proposition}
\begin{proof}
Both statements are straightforward calculations.  
\end{proof}

In the rest of the paper we will always work with elements on the Gelfand--Tsetlin algebra $\textup{GZ}_n$, since this algebra is commutative, item \textit{(2)} of proposition \ref{PropMultiLinePermCumul} is always guaranteed. The main interest of defining the permutation-cumulant is that we can relate them to the cumulants of the random variables $\{X_i\}_{i=1}^\infty$. For instance,

\begin{lemma} \label{LemmaCumulantPermCumulant}
Let $k_1,\dots,k_r \in \N$ be a collection of integers, then
\begin{equation*}
\kappa(X_{k_1},X_{k_2},\dots,X_{k_r}) = n^{-\frac{k_1+k_2+\dots+k_r}{2}} \kappa_{\rho_n,r}(D_{k_1},\dots,D_{k_2}).
\end{equation*}
\end{lemma}
\begin{proof}
Let $r\geq 2$ and $B\subseteq \{1,\dots,r\}$, Lemma \ref{BasicTraceLemma} gives 
\begin{equation*}
\E_{\rho_n}\Big[ \prod_{j\in B} X_{k_j}\Big] = n^{- \frac{1}{2}\sum_{j \in B} k_j } M_{\rho_n}\Big(\prod_{j \in B} D_{k_j}\Big).    
\end{equation*}
Let $k_1,\dots,k_r \in \N$ be a collection of integers, 
\begin{align*}
    \kappa(X_{k_1},X_{k_2},\dots,X_{k_r}) & = \sum_{\pi \in \Theta_r} (-1)^{|\pi|-1} (|\pi|-1)! \prod_{B \in \pi} \E_{\rho_n}\Big[ \prod_{j\in B} X_{k_j}\Big]\\
    &=\sum_{\pi \in \Theta_r} (-1)^{|\pi|-1} (|\pi|-1)!n^{- \frac{1}{2}\sum_{j=1}^r k_j } \prod_{B\in \pi} M_{\rho_n}\Big(\prod_{j \in B} D_{k_j}\Big)\\
    &= n^{-\frac{k_1+k_2+\dots+k_r}{2}} \kappa_{\rho_n,r}(D_{k_1},\dots,D_{k_2}). \qedhere
\end{align*}
\end{proof}

\begin{lemma}\label{LemmaRelatingcumulantGenwithLTIP}
For each $r\geq 1$ and $k_1,\dots,k_r$ integers we denote $\textup{Rep}[k]$ the number of times the value $k$ is repeated in the $r$-tuple $(k_1,\dots,k_r)$. We define the permutation-cumulant generating function to be the formal power series given by 
\begin{equation*}
\textup{B}_\rho(x_1,x_2,\dots):=\sum_{r=1}^\infty \sum_{k_1,\dots,k_r\geq 0}    \kappa_{\rho_n,r}(\sigma[k_1],\dots,\sigma[k_r]) \prod_{j=1}^r  \frac{x_{k_j}^{\textup{Rep}[k_j]}}{\textup{Rep}[k_j]!} \in \R[\vec{x}]
\end{equation*}
Where for each $r$-tuple $(k_1,\dots,k_r)$ the cycles $\sigma[k_1],\dots,\sigma[k_r]$ have disjoint support. Then,
\begin{equation*}
    \textup{B}_\rho(x_1,x_2,\dots) = \ln\big(A_\rho(x_1,x_2,\dots)\big)
\end{equation*}
\end{lemma}
\begin{proof}
It follows by the Möbius inversion formula on the partition lattice, see \cite[Page 154, 4.19 (iv)]{Ai} that for any partition $\pi \in \Theta_r$,
\begin{equation}\label{EqLLNappropriateprop}
    \prod_{D \in \pi} M_{\rho_n}\big( \prod_{j\in D} \sigma_{j}\big) = \sum_{\theta \leq \pi} \prod_{\{j_1,\dots,j_s\}\in \theta} \kappa_{\rho,s}(\sigma_{j_1},\dots,\sigma_{j_s}).
\end{equation}
Notice that $\textup{B}_\rho(x_1,x_2,\dots)\bigr|_{\vec{x}=0}=0$, hence $\exp\big(\textup{B}_\rho(x_1,x_2,\dots)\big) \in \R[\vec{x}]$ is well defined. At this point, we can use equation (\ref{EqLLNappropriateprop}) to compare each coefficient of $\textup{B}_\rho$ with $\textup{A}_\rho$ to get
\begin{equation*}
\exp\big(\textup{B}_\rho(x_1,x_2,\dots)\big)=\textup{A}_\rho(x_1,x_2,\dots).
\end{equation*}
This is a classical computation that can be found on \cite{Spe}. Also see \cite[Section 3.2]{PT} for a proof in the classical real-valued random variable case. Finally, notice that the logarithm $\ln$ is the inverse of the exponential $\exp$ over $\R[\vec{x}]$, from which the conclusion follows. 
\end{proof}

\begin{proposition}\label{LLNappropriateprop}
$\rho_n$ is LLN-appropriate if and only if

\begin{enumerate}
    \item for each cycle $\sigma \in S_\infty$ of length $|\sigma|+1=k$, $\lim_{n \to \infty} n^{\frac{|\sigma|}{2}} M_{\rho_n}(\sigma) = c_k$.
    \item for each $r\geq 2$ and disjoint cycles $\sigma_1,\dots,\sigma_r \in S_\infty$, 
    $$\lim_{n \to \infty} n^{\frac{|\sigma_1|+|\sigma_2|+\dots+|\sigma_r|}{2}}\kappa_{\rho_n,r}(\sigma_1,\dots,\sigma_r) = 0.$$
\end{enumerate}
\end{proposition}
\begin{proof}
The proof is a direct consequence of Lemma \ref{LemmaRelatingcumulantGenwithLTIP} by computing the partial derivatives of $\textup{B}_{\rho_n}$.
\end{proof}

Before starting with the proof of Theorem \ref{TheoremLLN} we will show that when correctly rescaled, $M_{\rho_n}$ commutes with product of operators as $n\to \infty$ (Lemma \ref{Productlimit1}). In fact, the following lemma states that this commuting relation happens at the level of cycles if and only if some families of rescaled cumulants converge to $0$. W will further verify that this will always the case when $\rho_n$ is LLN-appropriate in Corollary \ref{CorollaryMrhoismorph}.  

\begin{lemma} \label{lemmaproduct}
Take $r\geq 2$, disjoint cycles $\sigma_1, \dots, \sigma_r \in S_{\infty}$ and let $\rho_n$ be a sequence of probability measures on $\mathbb{Y}_n$. Assume that the limits $\lim_{n\to \infty} n^{\frac{|\sigma_i|}{2}} M_{\rho_n}(\sigma_i)$ exist. Then the following are equivalent
\begin{enumerate}
    \item $\displaystyle  \lim_{n \to \infty }n^{\frac{\sum_{i=1}^{r'}|\tau_i|}{2}}\Big(M_{\rho_n}\big(\prod_{i=1}^{r'} \tau_i\big)-\prod_{i=1}^{r'}M_{\rho_n}( \tau_i)\Big)= 0$ for all $r'\geq 1$, $ \{\tau_1,\dots,\tau_{r'}\} \subseteq \{\sigma_1,\dots,\sigma_r\}$.
    \item $\displaystyle \lim_{n \to \infty} n^{\frac{\sum_{i=1}^{r'}|\tau_i|}{2}}  \kappa_{\rho_n,r'}(\tau_1,\dots,\tau_{r'})=0$ for all $r'\geq 2$, $ \{\tau_1,\dots,\tau_{r'}\} \subseteq \{\sigma_1,\dots,\sigma_r\}$.
\end{enumerate}
\end{lemma}

\begin{proof}

We first prove that $\textit{(1)}$ implies $\textit{(2)}$. Start by writing 
\begin{equation}\label{EQlemmaLLN2First}
    \kappa_{\rho_n,r'}(\tau_1,\tau_2,\dots,\tau_{r'}) = \sum_{\pi \in \Theta_{r'}} (-1)^{|\pi|-1} (|\pi|-1)! \prod_{B \in \pi} M_{\rho_n}\big( \prod_{j\in B} \tau_{j}\big)
\end{equation}

Moreover, it follows from the Möbius inversion formula on the partition lattice, see \cite[Page 154, 4.19 (iv)]{Ai}, that for $r' \geq 2$, $\sum_{\pi \in \Theta_{r'}} (-1)^{|\pi|-1} (|\pi|-1)! = 0$. In particular, 
\begin{equation}\label{EQlemmaLLN2Second}
\sum_{\pi \in \Theta_{r'}} (-1)^{|\pi|-1} (|\pi|-1)! \prod_{j=1}^{r'} M_\rho\big( \tau_{j}\big) = 0,
\end{equation}
which can also be interpreted as vanishing of cumulants for the constant random variables $M_{\rho_n}(\tau_j)$. Subtracting equation (\ref{EQlemmaLLN2Second}) from equation (\ref{EQlemmaLLN2First}) gives
\begin{equation}\label{EQlemmaLLN}
\begin{split}
    \kappa_{\rho_n,r'}(\tau_1,\tau_2,\dots,\tau_{r'}) = \sum_{\pi \in \Theta_{r'}} (-1)^{|\pi|-1} (|\pi|-1)! \Big(\prod_{B \in \pi}  M_{\rho_n}\big( \prod_{j\in B} \tau_{j}\big) -\prod_{j=1}^{r'}M_{\rho_n}\big( \tau_{j}\big)\Big).
\end{split}
\end{equation}

Since for each $B\in \pi$, $|B|\leq r$, the hypothesis gives
\begin{equation*}
    \lim_{n \to \infty} n^{\frac{\sum_{j \in B} |\tau_j|}{2}}  \Big(\prod_{B \in \pi}  M_{\rho_n}\big( \prod_{j\in B} \tau_{j}\big) -\prod_{j=1}^{r'}M_{\rho_n}\big( \tau_{j}\big)\Big)=0.
\end{equation*}
Additionally, $|\Theta_{r'}|$ is finite, so the sum at the right hand side of equation (\ref{EQlemmaLLN}) is finite and we get  
\begin{equation*}
    \lim_{n \to \infty} n^{\frac{\sum_{i=1}^{r'}|\tau_i|}{2}}  \kappa_{\rho_n,r'}(\tau_1,\dots,\tau_{r'})=0.
\end{equation*}

Assume that $\textit{(2)}$ holds. We prove that $\textit{(1)}$ by induction in $r'$. Set $\hat{1}=\big\{\{1,\dots,r'\}\big\}$ to be the maximal set partition of $\Theta_{r'}$, we can rewrite equation (\ref{EQlemmaLLN}) as 
\begin{equation}\label{EQlemmaLLN2}
\begin{split}
    M_{\rho_n}(\prod_{i=1}^{r'} \tau_i)-\prod_{i=1}^{r'}M_{\rho_n}( \tau_i) =\kappa_{\rho_n,r'}(\tau_1,\tau_2,\dots,\tau_{r'})
    -\bigg[ \sum_{\substack{\pi \in \Theta_{r'}\\ \pi \neq \hat{1}}} (-1)^{|\pi|-1} (|\pi|-1)!\\ \Big(\prod_{B \in \pi}  M_{\rho_n}\big( \prod_{j\in B} \tau_{j}\big) -\prod_{j=1}^{r'}M_{\rho_n}\big( \tau_{j}\big)\Big)\bigg]
\end{split}
\end{equation}

Furthermore, we know from the hypothesis that
\begin{equation}\label{EQlemmaLLN21}
    \lim_{n \to \infty} n^{\frac{\sum_{i=1}^{r'}|\tau_i|}{2}}\kappa_{\rho_n,r'}(\tau_1,\tau_2,\dots,\tau_{r'})=0.
\end{equation}

Since $|B|<r'$ for each $B\in \pi$, $\pi \neq \hat{1}$, the inductive hypothesis gives
\begin{equation}\label{EQlemmaLLN22}
    \lim_{n \to \infty} n^{\frac{\sum_{j \in B} |\tau_j|}{2}} \prod_{B \in \pi} \Big(M_{\rho_n}\big( \prod_{j\in B} \tau_{j}\big)-\prod_{j\in B}M_{\rho_n}\big( \tau_{j}\big)\Big)=0.
\end{equation}

Putting together equations (\ref{EQlemmaLLN21}) and (\ref{EQlemmaLLN22}) at the right hand side of equation (\ref{EQlemmaLLN2}), we conclude that 
\begin{equation*}
     \lim_{n \to \infty }n^{\frac{\sum_{i=1}^{r'}|\tau_i|}{2}}\Big(M_{\rho_n}\big(\prod_{i=1}^{r'} \tau_i\big)-\prod_{i=1}^{r'}M_{\rho_n}( \tau_i)\Big)= 0. \qedhere
\end{equation*}
\end{proof}

\begin{remark}
Note that the Möbius inversion formula on the partition lattice 
\begin{equation*}
    \sum_{\pi \in \Theta_{r}} (-1)^{|\pi|-1} (|\pi|-1)! = 0\hspace{2mm} \textup{ for } \hspace{1mm} r \geq 2,
\end{equation*}
can be given a probabilistic interpretation in terms of the cumulants. This identity is stating that the $r$th order cumulant, for $r\geq 2$, of a family of constant random variables is $0$. The identity fails for $r=1$ since the $1$st order cumulants coincides with the expectation. 
\end{remark}

\begin{corollary}\label{CorollaryMrhoismorph}
Let $\rho_n$ be a LLN-appropriate sequence of probability measures on $\mathbb{Y}_n$ and $\sigma \in S_\infty$ any permutation with conjugacy class indexed by $\lambda$, then
\begin{equation*}
\lim_{n \to \infty }n^{\frac{|\sigma|}{2}}M_{\rho_n}(\sigma)=\prod_{i \in \lambda} c_{i}.
\end{equation*}
\end{corollary}
\begin{proof}
Start by decomposing the permutation into a product of disjoint cycles $\sigma=\sigma_1\cdots \sigma_r$, condition \textit{(1)} of proposition \ref{LLNappropriateprop} guarantees that $\lim_{n\to\infty} n^{\frac{|\sigma_i|}{2}} M_{\rho_n}(\sigma_i)= c_{|\sigma_i|+1}$. Condition \textit{(2)} of proposition \ref{LLNappropriateprop} joint with Lemma \ref{lemmaproduct} gives
\begin{equation*}
\lim_{n \to \infty }n^{\frac{|\sigma|}{2}}M_{\rho_n}(\sigma) = \lim_{n \to \infty }n^{\frac{\sum_{i=1}^r|\sigma_i|}{2}}M_{\rho_n}\Big(\prod_{i=1}^r\sigma_i\Big)=\lim_{n \to \infty }\prod_{i=1}^r n^{\frac{|\sigma_i|}{2}}M_{\rho_n}(\sigma_i)=\prod_{i \in \lambda} c_{i}. \qedhere
\end{equation*}
\end{proof}

\begin{lemma}\label{Productlimit1}
Let $\rho_n$ be a sequence of LLN-appropriate probability measures on $\mathbb{Y}_n$. For each integers $r$ and $k_1,\dots,k_r$,
\begin{equation*}
\lim_{n\to \infty} n^{-\frac{1}{2} \sum_{j=1}^r k_j} M_{\rho_n}\Big(\prod_{j=1}^r D_{k_j}\Big) = \lim_{n\to \infty}  n^{-\frac{1}{2} \sum_{j=1}^r k_j}  \prod_{j=1}^r  M_{\rho_n}( D_{k_j}).
\end{equation*}
\end{lemma}
\begin{proof}
We compute both $ M_{\rho_n}\big(\prod_{j=1}^r D_{k_j}\big)$ and $\prod_{j=1}^r M_{\rho_n}( D_{k_j})$ and verify they coincide. Theorem \ref{TheoremExpansionofDk} gives an expansion for $D_{k_j}$, we have 
\begin{equation*}
D_{k_j}=\sum_{\lambda \in \mathcal{S}_{k_j}} N_\lambda[k_j](n) \hat{\Sigma}_{\lambda}, \hspace{2mm} \textup{ where } \hspace{2mm} N_{\lambda^j}[k_j](n)=C_{\lambda^j,k_j}n^{l(\lambda^j)/2+k_j/2}\big(1+O(n^{-1})\big)
\end{equation*}
and $C_{\lambda^j,k_j}$ is a constant only depending on $\lambda^j$ and $k_j$. Notice that although Theorem \ref{TheoremExpansionofDk} gives the value of that constant, we will not need it for the proof. 

We restate the previous identity as $n^{-k_j/2}D_{k_j}=\sum_{\lambda \ \in \mathcal{S}_{k_j}} N'_\lambda[k_j](n) \hat{\Sigma}_{\lambda}$ where  $N'_{\lambda^j}[k_j](n)=C_{\lambda^j,k_j}n^{l(\lambda^j)/2}(1+O(n^{-1}))$. Using the linearity of $M_{\rho_n}$ we can expand the products $\prod_{j=1}^r D_{k_j}$ and $\prod_{j=1}^r M_{\rho_n}( D_{k_j})$ to get
\begin{equation}\label{EQlemmaLLN3}
n^{-\frac{1}{2} \sum_{j=1}^r k_j} M_{\rho_n}\Big(\prod_{j=1}^r D_{k_j}\Big) = \sum_{\substack{\lambda^1,\dots,\lambda^r, \\ \lambda^j \in\mathcal{S}_{k_j}}} \prod_{j=1}^r   N'_{\lambda^j}[k_j](n) M_{\rho_n}\Big(\prod_{j=1}^r\hat{\Sigma}_{\lambda^j}\Big).  
\end{equation}

Similarly,
\begin{equation}\label{EQlemmaLLN4}
n^{-\frac{1}{2} \sum_{j=1}^r k_j} \prod_{j=1}^rM_{\rho_n}( D_{k_j}) = \sum_{\substack{\lambda^1,\dots,\lambda^r, \\ \lambda^j \in\mathcal{S}_{k_j}}} \prod_{j=1}^r   N'_{\lambda^j}[k_j](n) \prod_{j=1}^rM_{\rho_n}(\hat{\Sigma}_{\lambda^j}).  
\end{equation}
 
Notice that
\begin{equation*}
    \prod_{j=1}^r   N'_{\lambda^j}[k_j](n) =\Big(\prod_{j=1}^r C_{\lambda^j,k_j} \Big) n^{\frac{l(\vec{\lambda})}{2}}(1+O(n^{-1})).
\end{equation*}

Hence, by comparing equations (\ref{EQlemmaLLN3}) and (\ref{EQlemmaLLN4}), it will be enough to verify that 
\begin{equation}\label{EQlemmaLLNProductCentral}
    \lim_{n\to \infty} n^{\frac{l(\vec{\lambda})}{2}} M_{\rho_n}\Big(\prod_{j=1}^r\hat{\Sigma}_{\lambda^j}\Big)=\lim_{n\to \infty}n^{\frac{l(\vec{\lambda})}{2}} \prod_{j=1}^rM_{\rho_n}(\hat{\Sigma}_{\lambda^j}).
\end{equation}

At the right hand side of equation (\ref{EQlemmaLLNProductCentral}) we have that
\begin{equation*}
    \lim_{n\to \infty} n^{\frac{l(\vec{\lambda})}{2}} \prod_{j=1}^rM_{\rho_n}(\hat{\Sigma}_{\lambda^j}) = \lim_{n\to \infty} \prod_{j=1}^r n^{\frac{ l(\lambda^j)}{2}} M_{\rho_n}(\hat{\Sigma}_{\lambda^j})=  \prod_{j=1}^r  \lim_{n\to \infty} n^{\frac{ |\sigma[\lambda^j]|}{2}} M_{\rho_n}(\sigma[\lambda^j])= \prod_{i \in \vec{\lambda}} c_{i}.
\end{equation*} 
Where we used Corollary \ref{CorollaryMrhoismorph} in the last equality. On the other hand, Corollary \ref{Product1stOrderExpansion} gives
\begin{align*}
     \prod_{j=1}^r\hat{\Sigma}_{\lambda^j} &=  \frac{1}{\prod_{j=1}^r(n \ff |\lambda^j|)} \prod_{j=1}^r\Sigma_{\lambda^j}= \frac{1}{\prod_{j=1}^r(n \ff |\lambda^j|)} \Sigma_{\vec{\lambda}} + \sum_\mu o(n^{\frac{l(\mu)+|\vec{\lambda}|+\ell(\lambda)}{2}-|\mu|-|\vec{\lambda}|}) \cdot \Sigma_\mu\\
     &= \hat{\Sigma}_{\vec{\lambda}} + O(n^{-1}) \cdot \hat{\Sigma}_{\vec{\lambda}} + \sum_\mu o(n^{\frac{l(\mu)-|\vec{\lambda}|+\ell(\lambda)}{2}}) \cdot \hat{\Sigma}_\mu
\end{align*}

Hence, at the left hand side of equation (\ref{EQlemmaLLNProductCentral}) we get 
\begin{equation*}
n^{\frac{l(\vec{\lambda})}{2}} M_{\rho_n}\Big(\prod_{j=1}^r\hat{\Sigma}_{\lambda^j}\Big) = n^{\frac{l(\vec{\lambda})}{2}} (1+O(n^{-1}))M_{\rho_n}(\hat{\Sigma}_{\vec{\lambda}})+\sum_\mu o(n^{\frac{l(\mu)}{2}}) M_{\rho_n}(\hat{\Sigma}_\mu).    
\end{equation*}

Additionally, since $\lim_{n \to \infty } o(n^{\frac{l(\mu)}{2}}) M_{\rho_n}(\hat{\Sigma}_\mu) =0$,  we get
\begin{equation*}
     \lim_{n\to \infty} n^{\frac{l(\vec{\lambda})}{2}} M_{\rho_n}\Big(\prod_{j=1}^r\hat{\Sigma}_{\lambda^j}\Big) =  \lim_{n\to \infty} n^{\frac{l(\vec{\lambda})}{2}} M_{\rho_n}(\hat{\Sigma}_{\vec{\lambda}})=  \lim_{n\to \infty} n^{\frac{|\sigma[\vec{\lambda}]|}{2}} M_{\rho_n}(\sigma[\vec{\lambda}])=\prod_{i \in \vec{\lambda}} c_{i}.
\end{equation*}
Where we used Corollary \ref{CorollaryMrhoismorph} in the last equality. This verifies that equation (\ref{EQlemmaLLNProductCentral}) holds. \qedhere
\end{proof}

\begin{proof}[Proof of Theorem \ref{TheoremLLN}, first implication]
We start by computing $a_k$. On one hand Lemma \ref{BasicTraceLemma} gives $\E_{\rho_n}[X_k] = n^{-k/2} M_{\rho_n}(D_k)$. Using the expansion provided by Theorem \ref{TheoremExpansionofDk} on the operator $D_k$ and taking the limit as $n\to \infty$ gives
\begin{equation*}
    a_k  = \lim_{n \to \infty }n^{-k/2} M_{\rho_n}(D_k) = \sum_{\lambda \in \mathcal{S}_k} |\text{NC}(\mu)| \lim_{n \to \infty} n^{\frac{l(\lambda)}{2}} M_{\rho_n}(\hat{\Sigma}_\lambda) = \sum_{\lambda \in \mathcal{S}_k} |\text{NC}(\mu)| \prod_{i \in \lambda} c_{i}.
\end{equation*}
Here $\mu$ denotes the unique partition of size $k$ such that $\textup{Rem}_k(\mu)=\lambda$. 

We plug the numbers $ |\text{NC}(\mu)|$ using Kreweras formula (Lemma \ref{Kreweras}). For each partition $\mu \in \bar{\mathbb{Y}}_k$, denote $q=\mu_1$ the maximal length of a row of $\mu$ and, for $j=2,\dots,q$, denote $s_j$ the number of rows of length $j$.
\begin{align*}
    a_k & =  \sum_{\lambda \in \mathcal{S}_k} |\text{NC}(\mu)| \prod_{i \in \lambda} c_{i}\\
    & = \sum_{\substack{s_2,s_3,\dots,s_k\geq 0,\\ 2s_2+3s_3+\dots+ks_k=k.}} \frac{(k\ff s_2+\dots+s_k-1)}{s_2!s_3!\dots s_k!} \prod_{i \geq 2} c_i^{s_{i+1}}\\
    &= \sum_{\substack{s_2,s_3,\dots,\, s_k\geq 0,\\ 2s_2+3s_3+\dots+ks_k=k.}} \frac{1}{k+1} \binom{k+1}{1+s_2+2s_3+\dots+(k-1)s_k,s_2,s_3,\dots,s_q} \prod_{i \geq 2} c_i^{s_{i+1}}
\end{align*}

Via the change of variables $s_{i+1}=t_i$ for $1 \leq i\leq k-1$ and $1+s_2+2s_3+\dots+(k-1)s_k=t_{-1}$ we can rewrite the last equation as
\begin{equation*}
    a_k = \sum_{\substack{t_{-1},t_1,\dots, t_{k-1}\geq 0,\\ t_{-1}+t_1+\dots+t_{k-1}=k+1,\\ t_1+2t_2+\dots+(k-1)t_{k-1}=t_{-1}-1.}} \frac{1}{k+1} \binom{k+1}{t_{-1},t_1,t_2,\dots,t_{k-1}} \prod_{i \geq 2} c_i^{t_{i}}.
\end{equation*}

Additionally, we also have that
\begin{align*}
\big(z^{-1}+zF_\rho(z)\big)^{k+1} &= (z^{-1}+z+c_2z^2+\dots)^{k+1}\\
&= \sum_{\substack{t_{-1},t_1,\dots, t_{k-1}\geq 0,\\ t_{-1}+t_1+\dots+t_{k-1}=k+1}}  \binom{k+1}{t_{-1},t_1,t_2,\dots,t_{k-1}} z^{-t_{-1}}z^{t_1}\prod_{i \geq 2} (c_iz^i)^{t_{i}}\\
&= \sum_{\substack{t_{-1},t_1,\dots, t_{k-1}\geq 0,\\ t_{-1}+t_1+\dots+t_{k-1}=k+1}}  \binom{k+1}{t_{-1},t_1,t_2,\dots,t_{k-1}} \prod_{i \geq 2} (c_i)^{t_{i}}z^{\sum_{j=1}^{k-1}jt_j-t_{-1}}.
\end{align*}
Here, for each $j=-1,1,2,\dots,k-1$, we denote $t_j$ denotes the number of times we choose the term $z^j$. Hence $a_k$ is the coefficient of the monomial $z^{-1}$ in the expansion of $\big(z^{-1}+zF_\rho(z)\big)^{k+1}$.
\begin{equation*}
a_k = [z^{-1}] \frac{1}{k+1} \big(z^{-1}+zF_\rho(z)\big)^{k+1}.
\end{equation*}

We still need to verify item $(2)$ of Definition \ref{defSatiesfiesLLN}. Let $r\geq 2$ and $k_1,\dots,k_r$ any collection of positive integers. Lemma \ref{LemmaCumulantPermCumulant} allows to express the cumulants into 
\begin{equation*}
    \kappa(X_{k_1},X_{k_2},\dots,X_{k_r}) =\sum_{\pi \in \Theta_r} (-1)^{|\pi|-1} (|\pi|-1)!n^{- \frac{1}{2}\sum_{j=1}^r k_j } \prod_{B\in \pi} M_{\rho_n}\Big(\prod_{j \in B} D_{k_j}\Big).
\end{equation*}

For each $\pi$ set partition of $\{1,\dots,r\}$, we can use Lemma \ref{Productlimit1} to get
\begin{equation*}
    \lim_{n \to \infty } n^{- \frac{1}{2}\sum_{j=1}^r k_j } \prod_{B\in \pi} M_{\rho_n}\Big(\prod_{j \in B} D_{k_j}\Big)= \lim_{n\to \infty}  \prod_{j=1}^r n^{-\frac{1}{2} k_j} M_{\rho_n}( D_{k_j}) =  \prod_{j=1}^r a_{k_j}.
\end{equation*}

Notice that the limit does not depend on $\pi$, hence 
\begin{align*}
    \lim_{n\to \infty} \kappa(X_{k_1},X_{k_2},\dots,X_{k_r}) &=\sum_{\pi \in \Theta_r} (-1)^{|\pi|-1} (|\pi|-1)! \lim_{n\to \infty} n^{- \frac{1}{2}\sum_{j=1}^r k_j } \prod_{B \in \pi } M_{\rho_n}\Big(\prod_{j \in B} D_{k_j}\Big)\\
    &=  \prod_{j=1}^r a_{k_j}\sum_{\pi \in \Theta_r} (-1)^{|\pi|-1} (|\pi|-1)! =0.
\end{align*}
where we used the Möbius inversion formula on the partition lattice  \cite[Page 154]{Ai}. 
\end{proof}

\subsection{Reverse law of large numbers}
Denote $\mathcal{C}_{t,r}$ the set of all $r$-tuples of non-trivial cycles with disjoint support in $S_{\infty}$ such that their total Cayley distance to the identity is equal to $t$. That is, 
\begin{equation*}
\mathcal{C}_{t,r}=\Big\{(\sigma_1,\dots,\sigma_r)\in \prod_{j=1}^r S_\infty:\{\sigma_j\}_{j=1}^r \text{ are non-trivial disjoint cycles and } \bigr|\prod_{j=1}^r \sigma_j\bigr|=t \Big\}.
\end{equation*}
Since all the cycles in such a $r$-tuple are non-trivial, we have that $|\prod_{j=1}^r \sigma_j|=\sum_{j=1}^r |\sigma_j|\geq r$. Hence $\mathcal{C}_{t,r}$ is empty if $r>t$. For example $\mathcal{C}_{1,1}$ consists of $2$-cycles, $\mathcal{C}_{2,1}$ consists of $3$-cycles and $\mathcal{C}_{2,2}$ consists of pairs of $2$-cycles. 

\begin{proof}[Proof of Theorem \ref{TheoremLLN}, second implication.]

We will prove by induction on the lexicographic order over the coordinates $(t,r)$ that 
\begin{enumerate}
    \item There exists a constant $c_{t+1}$ such that $\lim_{n\to \infty} n^{t/2}M_{\rho_n}(\sigma)= c_{t+1}$ for each $\sigma \in \mathcal{C}_{t,1}$.
    \item $\lim_{n\to \infty} n^{t/2}\kappa_{\rho_n,r}(\vec{\sigma})= 0$ for each $\vec{\sigma} \in \mathcal{C}_{t,r}$ for $r\geq 2$. 
\end{enumerate}

For instance the first three cases are $(1,1)$, $(2,1)$, $(2,2)$. In the proof we threat first two cases explicitly. It follows from Theorem \ref{TheoremExpansionofDk} that we can write the expansion of $D_k$ in the following manner,
\begin{equation}\label{EQexpansionDkforReverseLLN}
    D_{t+2}=n^{t+1}(1+O(n^{-1})) \hat{\Sigma}_{t+1}+\sum_{\lambda \in \mathcal{S}_{t+2}^\ast} n^{\frac{l(\lambda)+t+2}{2}}(C_\lambda+O(n^{-1}))\hat{\Sigma}_{\lambda},
\end{equation}
where $\mathcal{S}^\ast_{t+2}=\mathcal{S}_{t+2}\backslash\{(t+1)\}$ and $C_\lambda$ is a nonzero constant. Notice that all $\lambda \in \mathcal{S}_{t+2}^\ast$ satisfy $l(\lambda) \leq t-2$. Equation (\ref{EQexpansionDkforReverseLLN}) is the basis of our argument since it allow us to separate the operator into a higher order term and lower order terms, this is the foundation of our inductive argument. Lemma \ref{BasicTraceLemma} gives $\E_{\rho_n}[X_{t+2}] = n^{-(t+2)/2} M_{\rho_n}(D_{t+2})$ and the linearity of $M_{\rho_n}$ on equation (\ref{EQexpansionDkforReverseLLN}) gives
\begin{equation}\label{EQexpansionDk2}
    \E_{\rho_n}[X_{t+2}]=n^{t/2}(1+O(n^{-1})) M_{\rho_n}(\sigma[t+1])+\sum_{\lambda \in \mathcal{S}_{t+2}^\ast} n^{|\sigma[\lambda]|/2}(C_\lambda+O(n^{-1})) M_{\rho_n}(\sigma[\lambda]).
\end{equation}

Using the previous equation we will prove the base case of our induction, corresponding to $(t,r)=(1,1)$ and $(t,r)=(2,1)$. In fact, equation (\ref{EQexpansionDk2}) for $t=1$ gives $\E_{\rho_n}[X_{3}]=n^{1/2}M_{\rho_n}(\sigma[2])$ and for $t=2$ it gives $\E_{\rho_n}[X_{4}]=nM_{\rho_n}(\sigma[3])+O(1)$. In both cases the convergence of the left hand side implies the convergence of the right hand side without any additional assumption.

We now prove the inductive step corresponding to item (1), that is, we are in the case $(t,r)$ with $r=1$. Since $\rho_n$ satisfies a LLN we know that the left hand side in equation (\ref{EQexpansionDk2}) converges. We need to show converges for the terms in the sum of the right hand side.

Given $\lambda \in \mathcal{S}_k^\ast$, start by decomposing each $\sigma[\lambda]$ into a product of disjoint cycles $\sigma[\lambda]=\sigma[\lambda_1]\cdots\sigma[{\lambda_{\ell(\lambda)}}]$. Since $|\sigma[\lambda]|<t$, the inductive hypothesis ensures both that for each $j=1,\dots,\ell(\lambda)$,  the limit $\lim_{n \to \infty }n^{|\sigma[\lambda_j]|/2}M_{\rho_n}(\sigma[\lambda_j])$ exists and that for each subset $\{j_1,\dots,j_{\ell'}\} \subseteq \{\lambda_1,\dots,\lambda_{\ell(\lambda)}\}$, the rescaled permutation-cumulants converge to $0$, that is, 
\begin{equation*}
\lim_{n \to \infty}n^{\frac{l(\lambda_{j_1})+\dots+l(\lambda_{j_r})}{2}}\kappa_{\rho_n}(\sigma[\lambda_{j_1}],\dots,\sigma[\lambda_{j_r}]) = 0.    
\end{equation*}

We now use Lemma \ref{lemmaproduct}, with $\sigma_j=\sigma[\lambda_j]$ for $1\leq j \leq \ell(\lambda)$ to obtain that for each $\lambda \in \mathcal{S}_k^\ast$,
\begin{align*}
\lim_{n \to \infty}n^{|\sigma[\lambda]|/2} M_{\rho_n}(\sigma[\lambda]) &= \lim_{n \to \infty}\prod_{i=1}^{\ell(\lambda)} n^{|\sigma[\lambda_i]|/2} M_{\rho_n}(\sigma[\lambda_i]) =\prod_{i=1}^{\ell(\lambda)}  c_{\lambda_i}.
\end{align*}
Putting this into equation (\ref{EQexpansionDk2}) ensures that the limit $\lim_{n \to \infty}n^{t/2} M_{\rho_n}(\sigma[t+1])$ must also exist, we denote this limit $c_{t+1}$. We can now prove item (2), that is $(t,r)$ is such that $r\geq 2$. Suppose we want to prove that our predicate holds for some $\vec{\sigma}=(\sigma_1,\dots,\sigma_r) \in \mathcal{C}_{t,r}$, $r\geq 2$ and denote $t_i=|\sigma_i|$. Notice that $\sum_{j=1}^r t_j=t$.

We start by expressing the standard cumulant in terms of the permutation-cumulant. Using Lemma \ref{LemmaCumulantPermCumulant} with $k_j=t_j+2$ for $1\leq j \leq r$, gives
\begin{equation*}
\kappa(X_{t_1+2},X_{t_2+2},\dots,X_{t_r+2}) = n^{-\frac{t_1+\dots+t_r+2r}{2}} \kappa_{\rho_n,r}(D_{t_1+2},\dots,D_{t_r+2}).
\end{equation*}

We now expand the operators $D_{t_j+2}$ using Theorem \ref{TheoremExpansionofDk} and use the linearity of the permutation-cumulants, that is proposition \ref{PropMultiLinePermCumul}, to obtain
\begin{equation}\label{EQinReverseLLN}
\begin{split}
\kappa(X_{t_1+2},X_{t_2+2},\dots,X_{t_r+2}) = n^{t/2}(1+O(n^{-1}))\kappa_{\rho_n,r}(\hat{\Sigma}_{t_1+1},\dots,\hat{\Sigma}_{t_r+1})\hspace{1cm}\\
+ \sum_{\vec{\lambda}\in \mathcal{S}_{\vec{t}}} C_{\vec{\lambda}} n^{\frac{l(\vec{\lambda})}{2}}(1+O(n^{-1}))\kappa_{\rho_n,r}(\hat{\Sigma}_{\lambda^1},\dots,\hat{\Sigma}_{\lambda^r}).
\end{split}
\end{equation}
Here $\mathcal{S}_{\vec{t}}$ is a finite set of vectors of partitions $\vec{\lambda}$ with $|\vec{\lambda}|<t$. Notice that equation (\ref{EQinReverseLLN}) is a higher order version of equation (\ref{EQexpansionDk2}). The exact partitions contained on the set $\mathcal{S}_{\vec{t}}$ is unimportant for the proof. We first analyze the cumulant given by the leading term, 
\begin{equation*}
   n^{t/2} \kappa_{\rho_n,r}(\hat{\Sigma}_{t_1+1},\dots,\hat{\Sigma}_{t_r+1}) = \sum_{\pi \in \Theta_r} (-1)^{|\pi|-1} (|\pi|-1)!\prod_{B\in \pi} n^{\frac{\sum_{j\in B} t_j}{2}}M_{\rho_n}\Big(\prod_{j \in B} \hat{\Sigma}_{t_j+1}\Big).
\end{equation*}

Now fix $\pi \in \Theta_r$ and $B\in \pi$, it follows from Corollary \ref{Product1stOrderExpansion} that 
\begin{align*}
    n^{\frac{\sum_{j\in B} t_j}{2}}M_{\rho_n}\Big(\prod_{j \in B} \hat{\Sigma}_{t_j+1}\Big) & = n^{\frac{\sum_{j\in B} t_j}{2}}M_{\rho_n}( \hat{\Sigma}_{\lambda_B})+ \sum_{\lambda \in \mathcal{R}_B}  o(n^{\frac{l(\lambda)}{2}}) M_{\rho_n}( \hat{\Sigma}_{\lambda})\\
    &=n^{\frac{\sum_{j\in B} |\sigma_j|}{2}}M_{\rho_n}\Big(\prod_{j\in B} \sigma_j\Big)+ \sum_{\lambda \in \mathcal{R}_B}  o(n^{\frac{|\sigma[\lambda]|}{2}}) M_{\rho_n}( \sigma([\lambda])
\end{align*}
where $\lambda_B$ is the partition constructed by taking the union of all $t_j+1$ for $j\in B$ and $\mathcal{R}_B$ is a set of partitions such that $(|\lambda|,\ell(\lambda))$ is smaller than $(t,r)$ in the lexicographic order. Notice that this jointly with the inductive hypothesis ensure that the limits $\lim_{n\to \infty } n^{\frac{|\sigma[\lambda]|}{2}} M_{\rho_n}( \sigma[\lambda])$ exists for each $\lambda \in \mathcal{R}_B$. We get that $\lim_{n \to \infty} o(n^{\frac{|\sigma[\lambda]|}{2}}) M_{\rho_n}( \sigma([\lambda]) =0$ for each $\lambda \in \mathcal{R}_B$.

\begin{remark}
In fact, for all $\pi \neq \hat{1}$, we have $|B|<r$ for $B \in \pi$, hence by inductive hypothesis, $n^{\frac{\sum_{j\in B} |\sigma_j|}{2}}M_{\rho_n}(\prod_{j\in B} \sigma_j)$ will also converge. We don't use this fact in our proof. 
\end{remark}

Putting this information together we obtain
\begin{equation}\label{EQconclideReverseLLN1}
    n^{t/2} \kappa_{\rho_n,r}(\hat{\Sigma}_{t_1+1},\dots,\hat{\Sigma}_{t_r+1})= n^{t/2} \kappa_{\rho_n,r}(\sigma_1,\dots,\sigma_r)+o(1).
\end{equation}

We now analyze the cumulants in the sum of equation (\ref{EQinReverseLLN}) in a similar manner. 
\begin{equation}\label{EQconclideReverseLLN3}
    n^{\frac{l(\vec{\lambda})}{2}}\kappa_{\rho_n,r}(\hat{\Sigma}_{\lambda^1},\dots,\hat{\Sigma}_{\lambda^r}) = \sum_{\pi \in \Theta_r} (-1)^{|\pi|-1} (|\pi|-1)! \prod_{B\in \pi} n^{\frac{\sum_{j\in B} l(\lambda^j)}{2}}M_{\rho_n}\Big(\prod_{j \in B} \hat{\Sigma}_{\lambda^j}\Big).
\end{equation}

Again, we fix $\pi \in \Theta_r$ and $B\in \pi$, it follows from Corollary \ref{Product1stOrderExpansion} that 
\begin{align*}
    n^{\frac{\sum_{j\in B} l(\lambda^j)}{2}}M_{\rho_n}\Big(\prod_{j \in B} \hat{\Sigma}_{\lambda^j}\Big) & = n^{\frac{\sum_{j\in B} l(\lambda^j)}{2}} M_{\rho_n}( \hat{\Sigma}_{\mu_B})+ \sum_{\lambda \in \mathcal{R}_B}  o(n^{\frac{l(\lambda)}{2}}) M_{\rho_n}( \hat{\Sigma}_{\lambda})\\
    &=n^{\frac{|\sigma(\mu_B)|}{2}}M_{\rho_n}(\sigma[\mu_B])+ \sum_{\lambda \in \mathcal{R}_B}  o(n^{\frac{|\sigma(\lambda)|}{2}}) M_{\rho_n}( \sigma[\lambda])
\end{align*}

Where $\mu_B$ is the partition constructed by taking the union of all $\lambda^j$ for $j \in B$ and $\mathcal{R}_B$ is such that $|\lambda|<t$ for each $\lambda \in \mathcal{R}_B$. Additionally, notice that $|\mu_B|<t$, hence the inductive hypothesis guarantees that $\lim_{n \to \infty} n^{\frac{|\sigma(\mu_B)|}{2}}M_{\rho_n}(\sigma[\mu_B])=\prod_{i\in \mu_B} c_{i}$ and that $\lim_{n \to \infty } o(n^{\frac{|\sigma(\lambda)|}{2}}) M_{\rho_n}( \sigma[\lambda])=0$. Hence,
\begin{equation}\label{EQconclideReverseLLN4}
\lim_{n \to \infty}\prod_{B\in \pi} n^{\frac{\sum_{j\in B} l(\lambda^j)}{2}}M_{\rho_n}\Big(\prod_{j \in B} \hat{\Sigma}_{\lambda^j}\Big) =\prod_{i \in \vec{\lambda}} c_i    
\end{equation}
where we notice that as multi-sets, $\vec{\lambda}=\cup_{B \in \pi} \mu_B$ for each $\pi \in \Theta_r$. Since this limit does not depend on $\pi$, putting together equations (\ref{EQconclideReverseLLN3}), (\ref{EQconclideReverseLLN4}) and using the Möbius inversion formula gives
\begin{equation}\label{EQconclideReverseLLN2}
    \lim_{n\to \infty} n^{\frac{l(\lambda^1)+\dots+l(\lambda^r)}{2}}\kappa_{\rho_n,r}(\hat{\Sigma}_{\lambda^1},\dots,\hat{\Sigma}_{\lambda^r})=0.
\end{equation}

We conclude that $ \lim_{n \to \infty }n^{t/2} \kappa_{\rho_n,r}(\sigma_1,\dots,\sigma_r)=0$ by putting together equations (\ref{EQinReverseLLN}), (\ref{EQconclideReverseLLN1}) and (\ref{EQconclideReverseLLN2}).
\end{proof}

\section{CLT: Proof of Theorem \ref{TheoremCLT}} \label{SectionProofsCLT}

The main difference when comparing the proofs of Theorem \ref{TheoremLLN} with Theorem \ref{TheoremCLT} is that now we will need to take into account lower order terms, for this reason we start by stating a series of cumulant expansion formulas that will allow us to carry out the computation. We state these formulas in section \ref{subsectionExpansionPermutationCumulants} with the exception of Lemma \ref{ExpansionofDefFcumulants2} whose proof is postponed to section \ref{CombRosas} which contains combinatorial formulas of independent interest. We will directly prove Theorem \ref{TheoremMultilevelCLT} from which the first implication of Theorem \ref{TheoremCLT} follows in section \ref{SectionProofMultilevelCLT} while we finish its proof in section \ref{SectionProofReverseCLT}. 

\subsection{Expansion of permutation-cumulants}\label{subsectionExpansionPermutationCumulants}

Similarly to the proof of Theorem \ref{TheoremLLN}, it will be convenient to give an alternative definition for a sequence of distributions being CLT-appropriate in terms of permutation-cumulants. 

\begin{proposition}\label{CLTappropriateprop}
$\rho_n$ is CLT-appropriate if and only if

\begin{enumerate}
    \item for each cycle $\sigma$ of length $k=|\sigma|+1$, $\lim_{n \to \infty} n^{\frac{|\sigma|}{2}} M_{\rho_n}(\sigma) = c_k$.
    \item for each two disjoint cycles $\sigma_1$, $\sigma_2$ of lengths $k_1=|\sigma_1|+1$ and $k_2=|\sigma_2|+1$, $$\lim_{n \to \infty} n^{\frac{|\sigma_1|+|\sigma_2|} {2}+1} \kappa_{\rho_n,2}(\sigma_1,\sigma_2) = d_{k_1,k_2}.$$
    \item for each $r\geq 3$ and disjoint cycles $\sigma_1,\dots,\sigma_r$, 
    $$\lim_{n \to \infty} n^{\frac{|\sigma_1|+|\sigma_2|+\dots+|\sigma_r|+r}{2}}\kappa_{\rho_n,r}(\sigma_1,\dots,\sigma_r) = 0.$$
\end{enumerate}
\end{proposition}
\begin{proof}
The proof follows from Lemma \ref{LemmaRelatingcumulantGenwithLTIP} by computing the partial derivatives of $\textup{B}_{\rho_n}$.
\end{proof}

In what follows we will present various results that allow us to decompose cumulants in a more manageable form. We first explain how to decompose cumulants of central elements. It will be convenient to introduce a generalization of the permutation-cumulant. 

\begin{definition}\label{DefFcumulants2} \label{DefFcumulants}
Fix $s\geq 1$ and $\vec{n}=(n_1\leq \dots \leq n_s)$. Let $r \geq 1$, $\sigma_1,\dots,\sigma_r \in S_\infty$ permutations and $\vec{k}=(\vec{k}^1,\dots,\vec{k}^r)$ a collection of $r$ vectors of $s$ integers. We define the $\vec{n}$-\textbf{falling cumulant} to be

\begin{equation*}
\kappa^{F[\vec{k},\vec{n}]}_{\rho,r}(\sigma_1,\dots,\sigma_r):=\sum_{\pi \in \Theta_r} (-1)^{|\pi|-1} (|\pi|-1)! \prod_{B \in \pi} (\vec{n}\ff \sum_{j \in B} \vec{k}^j) M_\rho\big( \prod_{j\in B} \sigma_{j}\big).
\end{equation*}
\end{definition}

While we are interested in the case in which the entries are permutations, the previous definition does makes sense in the general setting in which  we have non-commutative random variables and in particular in the case in which we have standard real valued random variables. 

\begin{lemma}\label{lemmaExpansionCumulantCentralElements} \label{lemmaExpansionCumulantCentralElements2}
Ler $r\geq 1$, $\vec{\lambda}=(\lambda^1,\dots,\lambda^r)$ a vector of $r$ partitions, $\pi \in \Theta_r$ and let $(\mu,\theta)\in \Xi_\pi[\vec{\lambda}]$ be as in Definition \ref{WeirdIBdefinition}, then there exists $\vec{k}_\theta=(\vec{k}^1,\dots,\vec{k}^r)$ a collection of $r$ vectors of $s$ integers, depending only on $\theta$ such that
\begin{equation*}
\kappa_{\rho,r}(\Sigma_{\lambda^1}\bigr|_{n_1},\dots,\Sigma_{\lambda^r}\bigr|_{n_r}) = \sum_{\pi\in \Theta_r} \sum_{(\vec{\mu},\theta) \in \Xi_\pi[\vec{\lambda}]} \kappa^{F[\vec{k}_\theta,\vec{n}]}_{\rho,|\pi|}(\sigma[\mu^1],\dots,\sigma[\mu^{|\pi|}]).
\end{equation*}

Additionally, $\sum_{j=1}^r|\vec{k}^j|\leq\frac{l(\vec{\mu})+|\vec{\lambda}|+\ell(\vec{\lambda})}{2}+|\pi|-r$.
\end{lemma}
\begin{proof}

Start by expanding the permutation-cumulant using the definition,
\begin{equation*}
    \kappa_{\rho,r}(\Sigma_{\lambda^1},\dots,\Sigma_{\lambda^r}) =\sum_{\pi \in \Theta_r} (-1)^{|\pi|-1} (|\pi|-1)! \prod_{B\in \pi} M_{\rho}\big(\prod_{j \in B}\Sigma_{\lambda^j}\bigr|_{n_j}\big).
\end{equation*}

Theorem \ref{Product1stOrderExpansionUPGRADED} allows us to expand the products $\prod_{j \in B} \Sigma_{\lambda^j}|_{n_j}$, we have that
\begin{align*}
     \prod_{B\in \eta} M_{\rho}(\prod_{j \in B}\Sigma_{\lambda^j}\bigr|_{n_j})& = \prod_{B\in \eta} \Big( \sum_{\pi_B \in \Theta_{B}} \sum_{(\mu_B,\theta_B) \in \Xi_{\pi_B}[\vec{\lambda}_B]} (\vec{n}\ff \vec{k}_{\theta_B}) M_{\rho}(\bar{\Sigma}_{\mu_B})\Big)\\
     &=\sum_{\substack{\pi_B \in \Theta_B,\\ B \in \eta }} \sum_{(\mu_B,\theta_B) \in \Xi_{\pi_B}[\vec{\lambda}_B]}\prod_{B\in \eta} (\vec{n}\ff \vec{k}_{\theta_B}) M_{\rho}(\bar{\Sigma}_{\mu_B})\\ 
     &=\sum_{\pi\leq \eta} \sum_{(\vec{\mu},\theta) \in \Xi_{\pi}[\vec{\lambda}]}\prod_{B\in \eta} (\vec{n}\ff \vec{k}_{\theta_B}) M_{\rho}(\bar{\Sigma}_{\mu_B})
\end{align*}
Here we used the fact that a sum of set partitions $\pi_B$ over $B$ for $B \in \eta$ induces a sum over the set partitions $\pi \leq \eta$. We use this expression to switch the order of the summation
\begin{align*}
 \kappa_{\rho,r}(\Sigma_{\lambda^1},\dots,\Sigma_{\lambda^r}) &=\sum_{\eta \in \Theta_r} \sum_{\pi\leq \eta} \sum_{(\vec{\mu},\theta) \in \Xi_{\pi}[\vec{\lambda}]} (-1)^{|\eta|-1} (|\eta|-1)!\prod_{B\in \eta} (\vec{n}\ff \vec{k}_{\theta_B}) M_{\rho}(\bar{\Sigma}_{\mu_B})\\
 &=\sum_{\pi\in \Theta_r} \sum_{(\vec{\mu},\theta) \in \Xi_{\pi}[\vec{\lambda}]} \sum_{\pi \leq \eta }\prod_{B\in \eta}   (-1)^{|\eta|-1} (|\eta|-1)! \prod_{B \in \eta} (\vec{n}\ff \vec{k}_{\theta_B}) M_{\rho}(\bar{\Sigma}_{\mu_B})
\end{align*}
Observing that for each set partition $\pi$, the sub-lattice of set partitions $\eta \in \Theta_r$ $\eta \geq \pi$ is isomorphic to the lattice of set partitions $\Theta_{|\pi|}$. Moreover, for each $\pi \in \Theta_r$  we re-index $(\vec{\mu},\theta)=(\mu_D,\theta_D)_{D\in \pi}$ as $(\mu^j,\theta^j)_{1\leq j\leq |\pi|}$.  Note that Theorem \ref{Product1stOrderExpansionUPGRADEDMultilevel} guarantees that for each $\eta \in \Theta_{|\pi|}$ and $B\in \eta$ then $\vec{k}_{\theta_B}=\sum_{j\in B}\vec{k}_{\theta^j}$ and $\mu_B=\cup_{j\in B} \mu^j$, hence
\begin{align*}
\kappa_{\rho,r}(\Sigma_{\lambda^1},\dots,\Sigma_{\lambda^r})  &=\sum_{\pi\in \Theta_r} \sum_{(\vec{\mu},\theta) \in \Xi_{\pi}[\vec{\lambda}]}  \sum_{\eta \in \Theta_{|\pi|}}   (-1)^{|\eta|-1} (|\eta|-1)! \prod_{B\in \eta} (\vec{n}\ff \sum_{j\in B}\vec{k}_{\theta^j}) M_{\rho}( \bar{\Sigma}_{\cup_{j\in B}\mu^j})\\
 &=\sum_{\theta\in \Theta_r} \sum_{(\vec{\mu},\theta) \in \Xi_{\pi}[\vec{\lambda}]} \sum_{\eta \in \Theta_{|\pi|}} (-1)^{|\eta|-1} (|\eta|-1)! \prod_{B\in \eta} (\vec{n}\ff \sum_{j\in B}\vec{k}_{\theta^j}) M_{\rho}( \sigma[\cup_{j\in B}\mu^j])\\
 &= \sum_{\pi\in \Theta_r} \sum_{(\vec{\mu},\theta) \in \Xi_\pi[\vec{\lambda}]} \kappa^{F[\vec{k}_\theta,\vec{n}]}_{\rho,|\pi|}(\sigma[\mu^1],\dots,\sigma[\mu^{|\pi|}]).
\end{align*}
Finally, Theorem \ref{Product1stOrderExpansionUPGRADEDMultilevel} ensures that $|\sum_{j=1}^{|\pi|}\vec{k}_{\theta^j}| \leq \frac{l(\vec{\mu})+|\vec{\lambda}|+\ell(\vec{\lambda})}{2}+|\pi|-r$.
\end{proof}

The following lemma ensures that every falling cumulant can be written as a linear combination of the classical cumulants. Its proof is purely combinatorial and involves a careful analysis of the generalized falling factorials. We postpone its proof to the end of the section. 

\begin{lemma}\label{ExpansionofDefFcumulants2}\label{ExpansionofDefFcumulants}
Following Definition \ref{DefFcumulants2} with $n_1\leq \dots\leq n_s\leq n$, then  

\begin{equation*}
\kappa^{F[\vec{k},\vec{n}]}_{\rho,r}(\sigma_1,\dots,\sigma_r)=\sum_{\pi \in \Theta_r} f_\pi(n) \kappa_{\rho,|\pi|}\Big(\prod_{j \in B_1} \sigma_j,\prod_{j \in B_2} \sigma_j,\dots,\prod_{j \in B_{|\pi|}} \sigma_j\Big),
\end{equation*}
where $\pi=\{B_1,\dots,B_{|\pi|}\}$ and $f_\pi(n)=O(n^{\beta})$ with $\beta=\sum_{j=1}^r \sum_{i=1}^s k^j_i+|\pi|-r$. 
\end{lemma}

Finally, the following lemma allow us to decompose cumulants of permutations into cumulants of cycles. This is the last cumulant expansion formula we will need for the proof of Theorem \ref{TheoremCLT}.  

\begin{lemma}\label{Cumulantsofproducts}
Given $r\geq 1$, $\lambda=(\lambda_1\geq\dots\geq \lambda_r)$ a partition and $\sigma_1,\dots,\sigma_r$ permutations that can be further decomposed in $\lambda_i$ cycles as $\sigma_i=\sigma_{\lambda_1+\dots+\lambda_{i-1}+1}\cdots \sigma_{\lambda_1+\dots+\lambda_i}$. Additionally, let $\pi_\lambda$ be the set partition of $|\lambda|$ formed by blocks $\{\lambda_1+\dots+\lambda_{i-1}+1,\dots, \lambda_1+\dots+\lambda_i\}$ for $i=1,\dots,r$ and denote by $\hat{1}$ the maximal partition of $\Theta_{|\lambda|}$, then
$$\kappa_{\rho,r}(\sigma_1,\dots,\sigma_r)=\sum_{\substack{\pi \in \Theta_{|\lambda|},\\ \pi \vee \pi_\lambda =\hat{1}}} \prod_{\{j_1,\dots,j_s\} \in \pi} \kappa_{\rho,s}(\sigma_{j_1},\dots,\sigma_{j_s}).$$
\end{lemma}
\begin{proof}
This is proposition 3.4 in \cite{Le}.
\end{proof}

\begin{example}
Let $\lambda=(1 \geq 1)$ and $\sigma_1$, $\sigma_2$ two cycles with disjoint support. Then $\pi_\lambda\in \Theta_2$ is the maximal permutation $\hat{1}$, so the sum on Lemma \ref{Cumulantsofproducts} runs over all partitions in  $\Theta_2$, that is, $\big\{\{1,2\}\big\}$ and $\big\{\{1\},\{2\}\big\}$. We get 
\begin{equation*}
    \kappa_{\rho,1}(\sigma_1\sigma_2)=\kappa_{\rho,2}(\sigma_1,\sigma_2)+\kappa_{\rho,1}(\sigma_1)\kappa_{\rho,1}(\sigma_2).
\end{equation*}    
\end{example}

The previous formula is known as \textit{Malyshev’s formula}, see \cite[Proposition 3.2.1]{PT} for a proof on the classical case. 

\begin{remark}
We note that Theorem \ref{TheoremExpansionofDk} jointly with Lemmas \ref{lemmaExpansionCumulantCentralElements}, \ref{ExpansionofDefFcumulants} and \ref{Cumulantsofproducts} allows to simplify cumulants of elements in the Gelfand–Tsetlin algebra into cumulants of disjoint cycles. By specializing this expansion to the single level setting, that is when all the elements are in the center of the algebra of the Symmetric group, it is simple to adapt our strategy to prove the factorization property \cite[Theorem 3.1]{Sni}. This argument is developed in more detail in the next few sections where we prove Theorem \ref{TheoremCLT}. 
\end{remark}

\subsection{(Multilevel) Central limit theorem} \label{SectionProofMultilevelCLT}

We will directly prove the multilevel statement, that is Theorem \ref{TheoremMultilevelCLT}. As a preliminary step we go over the main computational step of Theorem \ref{TheoremMultilevelLLN}, that is, computing the values $\big(a_k^\alpha\big)_k$.

\begin{proof}[Proof of Theorem \ref{TheoremMultilevelLLN}]
Take $0 < \alpha <1$ and $n_1=\lfloor \alpha n \rfloor$, Lemma \ref{BasicTraceLemma} ensures that $\E_{\rho_n}[X_k^{\alpha}] = n^{-k/2} M_{\rho_n}(D_k|_{n_1})$. Additionally, Theorem \ref{TheoremExpansionofDk} gives the expansion
\begin{align*}
    a_k^{\alpha} & = \lim_{n \to \infty }n^{-k/2} M_{\rho_n}(D_k\bigr|_{n_1})\\
    & = \sum_{\lambda \in \mathcal{S}_k} |\text{NC}(\mu)| \lim_{n \to \infty} \frac{n_1^{\frac{l(\lambda)+k}{2}}}{n^{k/2}} M_{\rho_n}(\hat{\Sigma}_\lambda) = \sum_{\lambda \in \mathcal{S}_k} |\text{NC}(\mu)| \alpha^{\frac{l(\lambda)+k}{2}} \prod_{i \in \lambda} c_{i}.
\end{align*}
Here $\mu$ denotes the unique partition of size $k$ such that $\textup{Rem}_k(\mu)=\lambda$. Moreover, it follows from the definition of $\textup{Rem}_k$ jointly with lemma \ref{Kreweras} that $(l(\lambda)+k)/2=k-\ell(\mu)=s_2+2s_3+\dots+(k-1)s_k$. Using the change of variables from the proof of theorem \ref{TheoremLLN} we get $(l(\lambda)+k)/2=t_{-1}-1$. This induces
\begin{equation*}
    a_k^{\alpha} = \sum_{\substack{t_{-1},t_1,\dots,t_{k-1}\geq 1\\ t_{-1}+t_1+\dots+t_{k-1}=k+1,\\ t_1+2t_2+\dots+(k-1)t_{k-1}=t_{-1}-1.}} \frac{1}{k+1} \binom{k+1}{t_{-1},t_1,t_2,\dots,t_{k-1}} \alpha^{t_{-1}-1} \prod_{i \geq 2} c_i^{t_{i}}.
\end{equation*}
From which we get the expression
\begin{equation*}
a_k^{\alpha} = [z^{-1}] \alpha^{-1}\frac{1}{k+1} \big(\alpha z^{-1}+zF_\rho(z)\big)^{k+1}.\qedhere
\end{equation*}
\end{proof}

We now focus on the proof of Theorem  \ref{TheoremMultilevelCLT}, we will first calculate the covariance, that is, we compute the values $\big(b_{k,k'}^{\alpha,\alpha'}\big)_{k,k'}$. While this computation is similar to the one we did for the expectations, it will be much more technical and make use of Theorem \ref{lemmaproductoftwopartitions2}. Afterwards, we are going to focus on proving Gaussianity, or equivalently, proving that higher order cumulants vanish when $n \to \infty$. This is the most challenging part of the proof, the main idea will be to successively reduce higher horder cumulants until we obtain permutation-cumulant of cycles. We summarize the main steps in our proof as follows.

\begin{enumerate}
    \item Use Theorem \ref{TheoremExpansionofDk} to obtain cumulants of central elements.
    \item Use Lemma \ref{lemmaExpansionCumulantCentralElements} to obtain falling cumulants. 
    \item Use Lemma \ref{ExpansionofDefFcumulants} to obtain cumulants of disjoint permutations.
    \item Use Lemma \ref{Cumulantsofproducts} to obtain cumulants of disjoint cycles.
\end{enumerate}

\begin{proof}[Proof of Theorem \ref{TheoremCLT}, first implication, and Theorem \ref{TheoremMultilevelCLT}]

Fix $0<\alpha,\alpha'<1$,  $n_1=\lfloor \alpha n \rfloor$ and  $n_2=\lfloor \alpha' n \rfloor$. We already proved item \textit{(1)} of Theorem \ref{TheoremMultilevelCLT} above. We now prove item \textit{(2)}, that is, $n \kappa(X_k^{\alpha},X_{k'}^{\alpha'})$ converges to some value $b_{k,k'}^{\alpha,\alpha'}$ that we compute explicitly. Lemma \ref{LemmaCumulantPermCumulant} gives
\begin{equation*}
n \kappa(X_k^{\alpha},X_{k'}^{\alpha'})=n^{1-\frac{k+k'}{2}}\kappa_{\rho_n,2}(D_k\bigr|_{n_1},D_{k'}\bigr|_{n_2})   
\end{equation*}
We now expand $D_k\bigr|_{n_1}$ and $D_{k'}\bigr|_{n_2}$ using Theorem \ref{TheoremExpansionofDk}.
\begin{equation}\label{EQ1CLT}
  n^{1-\frac{k+k'}{2}}\kappa_{\rho_n,2}(D_k\bigr|_{n_1},D_{k'}\bigr|_{n_2}) = \sum_{\lambda \in \mathcal{S}_k} \sum_{\lambda' \in \mathcal{S}_{k'}} n^{1-\frac{k+k'}{2}}N_{\lambda}[k](n_1)N_{\lambda'}[k'](n_2)\kappa_{\rho_n,2}(\hat{\Sigma}_\lambda\bigr|_{n_1},\hat{\Sigma}_{\lambda'}\bigr|_{n_2})\
\end{equation}

Furthermore, Theorem \ref{TheoremExpansionofDk} gives the following asymptotic,
\begin{equation}\label{EQ2CLT}
\begin{split}
n^{1-\frac{k+k'}{2}}  N_{\lambda}[k](n_1) N_{\lambda'}[k'](n_2) = C_{\lambda,\lambda'}n^{1-\frac{k+k'}{2}} \cdot n_1^{\frac{l(\lambda)+k}{2}}\cdot n_2^{\frac{l(\lambda')+k'}{2}}
+C_{\lambda,\lambda'}'n^{1-\frac{k+k'}{2}} \cdot n_1^{\frac{l(\lambda)+k}{2}-1}\hspace{0cm}\\ \cdot n_2^{\frac{l(\lambda')+k'}{2}}(1+O(n^{-1}))
+C_{\lambda,\lambda'}''n^{1-\frac{k+k'}{2}} \cdot n_1^{\frac{l(\lambda)+k}{2}}\cdot n_2^{\frac{l(\lambda')+k'}{2}-1}(1+O(n^{-1})) 
\end{split}   
\end{equation}

where  $ C_{\lambda,\lambda'}$, $ C_{\lambda,\lambda'}'$ and $ C_{\lambda,\lambda'}''$ are constants, specifically we know that $ C_{\lambda,\lambda'}=|\textup{NC}(\mu)||\textup{NC}(\mu)'|$ for $\mu \vdash k$ and $\mu'\vdash k'$ partitions such that $\textup{Rem}_k(\mu)=\lambda$ and $\textup{Rem}_{k'}(\mu')=\lambda'$. As we will see shortly, the value of the constants $ C_{\lambda,\lambda'}'$ and $ C_{\lambda,\lambda'}''$ is not important for the proof. 

Since $\rho_n$ is LLN-appropriate, Corollary \ref{CorollaryMrhoismorph} ensures that 
\begin{equation*}
\lim_{n\to \infty} n^{\frac{l(\lambda)+l(\lambda')}{2}}M_{\rho_n}(\hat{\Sigma}_{\lambda}\bigr|_{n_1} \hat{\Sigma}_{\lambda'}\bigr|_{n_2})=\lim_{n\to \infty} n^{\frac{l(\lambda)+l(\lambda')}{2}}M_{\rho_n}(\hat{\Sigma}_{\lambda}\bigr|_{n_1})M_{\rho_n}(\hat{\Sigma}_{\lambda'}\bigr|_{n_2})    
\end{equation*}
In  particular, we have that
\begin{equation}\label{EQ3CLT}
\begin{split}
    \lim_{n\to \infty } n^{1-\frac{k+k'}{2}}  n_1^{\frac{l(\lambda)+k}{2}-1}  n_2^{\frac{l(\lambda')+k'}{2}}  \kappa_{\rho_n,2}(\hat{\Sigma}_\lambda\bigr|_{n_1},\hat{\Sigma}_{\lambda'}\bigr|_{n_2}) =0\, \textup{ and }\, \\
    \lim_{n\to \infty } n^{1-\frac{k+k'}{2}}  n_1^{\frac{l(\lambda)+k}{2}}  n_2^{\frac{l(\lambda')+k'}{2}-1}  \kappa_{\rho_n,2}(\hat{\Sigma}_\lambda\bigr|_{n_1},\hat{\Sigma}_{\lambda'}\bigr|_{n_2}) =0.
\end{split}
\end{equation}

Putting together equations (\ref{EQ2CLT}) and (\ref{EQ3CLT}) we conclude that to compute the limit of the right hand side of equation (\ref{EQ1CLT}) it will be enough to compute
\begin{equation*}
\begin{split}
     \lim_{n \to \infty} n^{1-\frac{k+k'}{2}} n_1^{\frac{l(\lambda)+k}{2}} n_2^{\frac{l(\lambda')+k'}{2}} \kappa_{\rho_n,2}(\hat{\Sigma}_\lambda|_{n_1},\hat{\Sigma}_{\lambda'}|_{n_2})=\lim_{n \to \infty} \frac{n^{1-\frac{k+k'}{2}}  n_1^{\frac{l(\lambda)+k}{2}} n_2^{\frac{l(\lambda')+k'}{2}}}{(n_1\ff |\lambda|)(n_2\ff |\lambda|')}\hspace{2cm}\\ \cdot \big(M_{\rho_n}(\Sigma_{\lambda}|_{n_1} \Sigma_{\lambda'}|_{n_2})
     -M_{\rho_n}(\Sigma_{\lambda}|_{n_1})M_{\rho_n}(\Sigma_{\lambda'}|_{n_2})\big).    
\end{split}
\end{equation*}

We focus on that task now. Since
\begin{align*}
    M_{\rho_n}(\Sigma_{\lambda}|_{n_1}) &= (n_1\ff |\lambda|)M_{\rho_n}(\sigma[\lambda])= \big(n_1^{|\lambda|}-\binom{|\lambda|}{2}n_1^{|\lambda|-1}+O(n^{|\lambda|-2})\big) M_{\rho_n}(\sigma[\lambda]),
\end{align*}
then
\begin{equation}\label{EQ4CLT}
   \begin{split}\frac{^{1-\frac{k+k'}{2}}  n_1^{\frac{l(\lambda)+k}{2}} n_2^{\frac{l(\lambda')+k'}{2}}}{(n_1\ff |\lambda|)(n_2\ff |\lambda'|)} M_{\rho_n}(\Sigma_{\lambda}|_{n_1})M_{\rho_n}(\Sigma_{\lambda'}|_{n_2}) 
   = \Big( n^{1-\frac{k+k'}{2}}  n_1^{\frac{l(\lambda)+k}{2}} n_2^{\frac{l(\lambda')+k'}{2}} -\big(\frac{n}{n_1}\binom{|\lambda|}{2} \hspace{1cm}\\+\frac{n}{n_2}\binom{|\lambda'|}{2}\big) n^{-\frac{k+k'}{2}}  n_1^{\frac{l(\lambda)+k}{2}} n_2^{\frac{l(\lambda')+k'}{2}} +O(n^{\frac{l(\lambda)+l(\lambda')}{2}-1})\Big) M_{\rho_n}(\sigma[\lambda])M_{\rho_n}(\sigma[\lambda']).
   \end{split}
\end{equation}

We need to expand the product $\Sigma_{\lambda}\bigr|_{n_1} \Sigma_{\lambda'}\bigr|_{n_2}$, this is given by Theorem \ref{lemmaproductoftwopartitions2}.
\begin{equation*}
\Sigma_{\lambda}|_{n_1} \Sigma_{\lambda'} |_{n_2} = (n_1)_{|\lambda|}(n_2-|\lambda|)_{|\lambda'|} \bar{\Sigma}_{\lambda \cup \lambda'} +Q_{\lambda,\lambda'}+R_{\lambda,\lambda'}\end{equation*}

with the term $Q_{\lambda,\lambda'}$ given by
\begin{equation}\label{EQ4.1QCLT}
\begin{split}
    Q_{\lambda,\lambda'} = \sum_{\substack{i \in \lambda\\ j \in \lambda'}} \sum_{r\geq 1} \sum_{\substack{s_1,\dots,s_r\geq 1\\ s_1+\dots+s_r=i}}\sum_{\substack{t_1,\dots,t_r\geq 1\\ t_1+\dots+t_r=j}} \dfrac{ij}{r} (n_1\ff |\lambda|)(n_2-|\lambda|\ff|\lambda'|-r)  \hspace{1cm}\\
    \cdot \bar{\Sigma}_{\lambda\backslash\{i\}\cup \lambda'\backslash\{j\}\cup_{m=1}^r (s_m+t_m-1)}    
\end{split}
\end{equation}

and the term $R_{\lambda,\lambda'}$ given by
\begin{equation}\label{EQ4.1RCLT}
R_{\lambda,\lambda'}=\sum_{\mu \in \mathcal{R}} N_\mu[R](n) \bar{\Sigma}_\mu,
\end{equation}
where  $N_\mu[R](n)=O(n^\beta)$ with $\beta=\frac{l(\mu)+|\lambda|+|\lambda'|+\ell(\lambda)+\ell(\lambda')}{2}-2$. We then have 

\begin{equation}\label{EQ5CLT}
\begin{split}
    \frac{n^{1-\frac{k+k'}{2}}  n_1^{\frac{l(\lambda)+k}{2}} n_2^{\frac{l(\lambda')+k'}{2}}(n_1\ff|\lambda|)(n_2-|\lambda|\ff|\lambda'|)}{(n_1\ff |\lambda|)(n_2\ff |\lambda|')}M_{\rho_n}(\bar{\Sigma}_{\lambda \cup \lambda'}) = \Big(n^{1-\frac{k+k'}{2}}  n_1^{\frac{l(\lambda)+k}{2}} n_2^{\frac{l(\lambda')+k'}{2}}\\
    -\Big(\frac{n}{n_1}\binom{|\lambda|}{2}+\frac{n}{n_2}\binom{|\lambda'|}{2}+ \frac{n}{\max(n_1,n_2)}|\lambda||\lambda'|\Big)n^{-\frac{k+k'}{2}}  n_1^{\frac{l(\lambda)+k}{2}} n_2^{\frac{l(\lambda')+k'}{2}}\\
    +O(n^{\frac{l(\lambda)+l(\lambda')}{2}-1})\Big) M_{\rho_n}(\sigma[\lambda\cup\lambda']).
\end{split}
\end{equation}

Notice that for $\sum_{m=1}^r s_m=i$ and $\sum_{m=1}^r t_m=j$ we have
\begin{equation*}
\begin{split}
\frac{l\big(\lambda\backslash\{i\}\cup \lambda'\backslash\{j\}\cup_{m=1}^r (s_m+t_m-1)\big)}{2} 
&=\frac{l(\lambda)+l(\lambda')}{2}+1-r.
\end{split}
\end{equation*}

Hence, using this last identity on equation (\ref{EQ4.1QCLT}) gives
\begin{equation}\label{EQ6CLT}
\begin{split}
    \frac{n^{1-\frac{k+k'}{2}}  n_1^{\frac{l(\lambda)+k}{2}} n_2^{\frac{l(\lambda')+k'}{2}}}{(n_1\ff |\lambda|)(n_2\ff |\lambda|')}M_{\rho_n}(Q_{\lambda,\lambda'}) =  \sum_{\substack{i \in \lambda\\ j \in \lambda'}} \sum_{r\geq 1} \sum_{\substack{s_1,\dots,s_r\geq 1\\ s_1+\dots+s_r=i}}\sum_{\substack{t_1,\dots,t_r\geq 1\\ t_1+\dots+t_r=j}} \dfrac{ij}{r}
    \Big(\frac{n}{\max(n_1,n_2)}\Big)^r\\ \cdot \Big(\frac{n_1}{n}\Big)^{\frac{l(\lambda)+k}{2}} \Big(\frac{n_2}{n}\Big)^{\frac{l(\lambda')+k'}{2}} n^{\frac{l(\lambda\backslash\{i\}\cup \lambda'\backslash\{j\}\cup_{m=1}^r (s_m+t_m-1))}{2}}(1+O(n^{-1}))\\ \cdot M_{\rho_n}(\sigma[\lambda\backslash\{i\}\cup \lambda'\backslash\{j\}\cup_{m=1}^r (s_m+t_m-1)])
\end{split}
\end{equation}

and similarly, equation (\ref{EQ4.1RCLT}) gives
\begin{equation}\label{EQ7CLT}
   \frac{n^{1-\frac{k+k'}{2}}  n_1^{\frac{l(\lambda)+k}{2}} n_2^{\frac{l(\lambda')+k'}{2}}}{(n_1\ff |\lambda|)(n_2\ff |\lambda|')} M_{\rho_n}(R_{\lambda,\lambda'}) = \sum_{\mu \in \mathcal{R}} O(n^{l(\mu)/2-1})M_{\rho_n}(\sigma[\mu]).
\end{equation}

At this point we are in good shape to start computing limits. It follows from equation (\ref{EQ7CLT}) that $\lim_{n \to \infty}  \frac{n^{1-\frac{k+k'}{2}}  n_1^{\frac{l(\lambda)+k}{2}} n_2^{\frac{l(\lambda')+k'}{2}}}{(n_1\ff |\lambda|)(n_2\ff |\lambda|')}M_{\rho_n}(R_{\lambda,\lambda'})=0$. Similarly, from equation (\ref{EQ6CLT}) we get
\begin{equation}\label{EQ9CLT}
\begin{split}
\lim_{n \to \infty} \frac{n^{1-\frac{k+k'}{2}}  n_1^{\frac{l(\lambda)+k}{2}} n_2^{\frac{l(\lambda')+k'}{2}}}{(n_1\ff |\lambda|)(n_2\ff |\lambda|')}M_{\rho_n}(Q_{\lambda,\lambda'}) = \alpha^{\frac{l(\lambda)+k}{2}}  \alpha'^{\frac{l(\lambda')+k'}{2}} \prod_{\substack{i \in \lambda\\j \in \lambda'}} c_{i}c_{j} \sum_{\substack{i \in \lambda\\j \in \lambda'}} \frac{1}{c_{i}c_{j}} \sum_{r\geq 1} \hspace{1cm}\\ \sum_{\substack{s_1,\dots,s_r,\geq 1,\\t_1,\dots,t_r\geq 1,\\ s_1+\dots+s_r=i,\\t_1+\dots+t_r=j}} \dfrac{ij}{r\max(\alpha,\alpha')^{r}}  \prod_{i=1}^r c_{s_i+t_i-1}.
\end{split}
\end{equation}

Additionally, putting together equations (\ref{EQ4CLT}) and (\ref{EQ5CLT}) gives
\begin{equation*}
\begin{split}
        \frac{n^{1-\frac{k+k'}{2}}  n_1^{\frac{l(\lambda)+k}{2}} n_2^{\frac{l(\lambda')+k'}{2}}}{(n_1\ff |\lambda|)(n_2\ff |\lambda'|)} \Big((n_1)_{|\lambda|}(n_2-|\lambda|)_{|\lambda'|}  M_{\rho_n}(\bar{\Sigma}_{\lambda \cup \lambda'})-M_{\rho_n}(\Sigma_{\lambda}\bigr|_{n_1})M_{\rho_n}(\Sigma_{\lambda'}\bigr|_{n_2})\Big) 
   = \\
  \Big(\frac{n_1}{n}\Big)^{\frac{l(\lambda)+k}{2}} \Big(\frac{n_2}{n}\Big)^{\frac{l(\lambda')+k'}{2}} \Big[n^{\frac{l(\lambda)+l(\lambda')}{2}+1}\big(M_{\rho_n}(\sigma[\lambda\cup\lambda'])-M_{\rho_n}(\sigma[\lambda])M_{\rho_n}(\sigma[\lambda'])\big)
  \\+\big(\frac{n}{n_1}\binom{|\lambda|}{2}+\frac{n}{n_2}\binom{|\lambda'|}{2}\big) n^{\frac{l(\lambda)+l(\lambda')}{2}}M_{\rho_n}(\sigma[\lambda])M_{\rho_n}(\sigma[\lambda'])
  \\ -\big(\frac{n}{n_1}\binom{|\lambda|}{2}+\frac{n}{n_2}\binom{|\lambda'|}{2}+\frac{n}{\max(n_1,n_2)}|\lambda||\lambda'|\big)n^{\frac{l(\lambda)+l(\lambda')}{2}}M_{\rho_n}(\sigma[\lambda\cup\lambda'])\Big]\\
   +O(n^{\frac{l(\lambda)+l(\lambda')}{2}-1})\big(M_{\rho_n}(\sigma[\lambda\cup\lambda'])-M_{\rho_n}(\sigma[\lambda])M_{\rho_n}(\sigma[\lambda'])\big).
\end{split}
\end{equation*}

We certainly have that
\begin{equation*}
    \lim_{n \to \infty } O(n^{\frac{l(\lambda)+l(\lambda')}{2}-1}) \big(M_{\rho_n}(\sigma[\lambda\cup\lambda'])-M_{\rho_n}(\sigma[\lambda])M_{\rho_n}(\sigma[\lambda'])\big)=0
\end{equation*}

while 
\begin{equation}\label{EQ10CLT}
\begin{split}
    \lim_{n \to \infty } \big(\frac{n}{n_1}\binom{|\lambda|}{2}+\frac{n}{n_2}\binom{|\lambda'|}{2}\big) n^{\frac{l(\lambda)+l(\lambda')}{2}}M_{\rho_n}(\sigma[\lambda])M_{\rho_n}(\sigma[\lambda']) -\big(\frac{n}{n_1}\binom{|\lambda|}{2}+\frac{n}{n_2}\binom{|\lambda'|}{2}\\+\frac{n}{\max(n_1,n_2)}|\lambda||\lambda'|\big)n^{\frac{l(\lambda)+l(\lambda')}{2}}M_{\rho_n}(\sigma[\lambda\cup\lambda']) = \max(\alpha,\alpha')^{-1}  |\lambda||\lambda'| \prod_{\substack{i \in \lambda\\ j \in \lambda'}} c_{i}c_{j}
\end{split}
\end{equation}

We are left to compute the limit of
\begin{equation*}
     n^{\frac{l(\lambda)+l(\lambda')}{2}+1}\big(M_{\rho_n}(\sigma[\lambda\cup\lambda'])-M_{\rho_n}(\sigma[\lambda])M_{\rho_n}(\sigma[\lambda'])\big) =  n^{\frac{l(\lambda)+l(\lambda')}{2}+1}\kappa_{\rho_n,2}(\sigma[\lambda],\sigma[\lambda']).
\end{equation*}

We can decompose each permutation $\sigma[\lambda]$ and $\sigma[\lambda']$ in a disjoint product of cycles, $\sigma[\lambda]=\sigma[\lambda_1]\cdots \sigma[\lambda_g]$ and $\sigma[\lambda']=\sigma[\lambda_1']\cdots \sigma[\lambda_{h}']$ and now use Lemma \ref{Cumulantsofproducts} to get 
\begin{equation*}
    \kappa_{\rho_n,2}(\sigma[\lambda],\sigma[\lambda']) = \sum_{\pi} \prod_{B \in \pi} \kappa_{\rho_n,|B|}(\sigma_j:j\in B)
\end{equation*}

where the sum run over set partitions $\pi \in \Theta_{s+t}$ such that $\pi \vee \big\{\{1,\dots,g\},\{g+1,\dots,g+h\}\big\}$ is the maximal partition and $\sigma_j=\sigma[\lambda_j]$ for $1\leq j \leq g$ and   $\sigma_j=\sigma[\lambda_{j-g}']$ for $g+1\leq j \leq g+h$. Notice that
\begin{equation}\label{EQ10.5CLT}
\begin{split}
n^{\frac{l(\lambda)+l(\lambda')}{2}+\frac{1}{2} \sum_{\substack{B \in \pi\\ |B|\geq 2}}|B|} \prod_{B \in \pi} \kappa_{\rho_n,|B|}(\sigma_j:j\in B) =\prod_{\substack{B \in \pi\\ |B|=1}} n^{\frac{\sum_{j \in B} |\sigma_j|}{2}} \kappa_{\rho_n,|B|}(\sigma_j:j\in B) \\ \cdot \prod_{\substack{B \in \pi\\ |B|\geq 2}} n^{\frac{\sum_{j \in B} |\sigma_j|+|B|}{2}} \kappa_{\rho_n,|B|}(\sigma_j:j\in B).
\end{split}
\end{equation}

Since $\rho_n$ is CLT-appropriate, we are guaranteed that the right hand side of equation (\ref{EQ10.5CLT}) converges. Moreover, if $|\pi| \leq g+h-2$, we immediately get that $\sum_{\substack{B \in \pi\\ |B|\geq 2}}|B| \geq 3$. This ensures that
\begin{equation*}
\lim_{n \to \infty } n^{\frac{l(\lambda)+l(\lambda')}{2}+1} \prod_{B \in \pi} \kappa_{\rho_n,|B|}(\sigma_j:j\in B) =0.
\end{equation*}

Hence the only set partitions $\pi$ such that the limit of the products of the cumulants is non zero are the ones that contain a single doubleton and everything else are singletons. That is, there is a single $2$nd order cumulant of the form $\kappa_{\rho_n,2}(\sigma[\lambda_i],\sigma[\lambda_j'])$ for some $1\leq i \leq g$ and $1\leq j \leq h$. We conclude that
\begin{equation}\label{EQ8CLT}
    \lim_{n \to \infty }n^{\frac{l(\lambda)+l(\lambda')}{2}+1}\kappa_{\rho_n,2}(\sigma[\lambda],\sigma[\lambda']) = \prod_{\substack{i \in \lambda\\ j \in \lambda'}} c_{i}c_{j} \sum_{\substack{i \in \lambda\\ j \in \lambda'}} \frac{d_{i,j}}{c_{i}c_{j}}.
\end{equation}

Using the following identity on equation (\ref{EQ10CLT}) 
\begin{equation*}
    |\lambda||\lambda'| \prod_{\substack{i \in \lambda\\ j \in \lambda'}} c_{i}c_{j}  =  \prod_{\substack{i \in \lambda\\ j \in \lambda'}} c_{i}c_{j}  \sum_{\substack{i \in \lambda\\ j \in \lambda'}} ij,
\end{equation*}
we can now put together the limits computed in equations (\ref{EQ8CLT}), (\ref{EQ9CLT}) and (\ref{EQ10CLT}) to obtain
\begin{equation}\label{EQ11CLT}
\begin{split}
\lim_{n \to \infty}  n^{1-\frac{k+k'}{2}} n_1^{\frac{l(\lambda)+k}{2}} n_2^{\frac{l(\lambda')+k'}{2}}  \kappa_{\rho_n,2}(\hat{\Sigma}_\lambda,\hat{\Sigma}_{\lambda'}) =\alpha^{\frac{l(\lambda)+k}{2}} \alpha'^{\frac{l(\lambda')+k'}{2}}
\Big(\prod_{\substack{i \in \lambda\\ j \in \lambda'}} c_i c_j  \Big)\sum_{\substack{i \in \lambda\\ j \in \lambda'}}\bigg[ \frac{1}{c_i c_j} \Big(d_{i,j}\\-\frac{ijc_{i}c_{j}}{\max(\alpha,\alpha')} 
+ \sum_{r\geq 1} \sum_{\substack{s_1,\dots,s_r\geq 1\\ s_1+\dots+s_r=i}} \sum_{\substack{t_1,\dots,t_r\geq 1\\ t_1+\dots+t_r=j}} \dfrac{ij}{r\max(\alpha,\alpha')^r}\prod_{m=1}^r c_{s_m+t_m-1}\Big)\bigg].
\end{split} 
\end{equation}

Finally, we use the limit computed in equation (\ref{EQ11CLT}) on equation (\ref{EQ1CLT}). For each partitions $\mu\in \bar{\mathbb{Y}}_k$ and $\mu'\in \bar{\mathbb{Y}}_{k'}$ with $\textup{Rem}_k(\mu)=\lambda$ and $\textup{Rem}_k(\mu')=\lambda'$. Remember that we denote for each $k\geq 1$, $\mathcal{S}_k=\textup{Rem}_k(\bar{\mathbb{Y}}_k)$. Furthermore, for $j=2,\dots,k$, $j=2,\dots,k'$ we denote $q_j$ and $q'_{j'}$ the number of rows of length $j$ and $j'$ in $\mu$ and $\mu'$ respectively. Denote the limit in equation (\ref{EQ11CLT}) by  $f_{\lambda,\lambda'}$, we have that
\begin{align*}
    b_{k,k'} &= \sum_{\lambda \in \mathcal{S}_k} \sum_{\lambda' \in \mathcal{S}_{k'}} |\textup{NC}(\mu)| \cdot |\textup{NC}(\mu')|f_{\lambda,\lambda'}\\
    &= \sum_{\substack{q_2,\dots, q_k\geq 0 \\2q_2+3q_3+\dots+kq_k=k.}} \sum_{\substack{q_2',\dots, q_k'\geq 0 \\2q_2'+3q_3'+\dots+k'q_{k'}'=k'.}} \binom{k}{q_2+2q_3+\dots+(k-1)q_k,q_2,\dots,q_k}\\
    &\hspace{6.1cm} \cdot\binom{k'}{q_2'+2q_3'+\dots+(k'-1)q_{k'}',q_2',\dots,q_{k'}'}f_{\lambda,\lambda'}.
\end{align*}
Similarly to the proof of theorem \ref{TheoremLLN} we would like to do a change of variables. We have $q_{i+1}=h_i$ for $1\leq i \leq k-1$, $q_{i+1}'=h_i'$ for $1\leq i \leq k'-1$, $1+q_2+2q_3+\dots(k-1)q_k=h_{-1}$ and $1+q_2'+2q_3'+\dots(k'-1)q_{k'}'=h_{-1}'$. Note that for each $i\geq 1$, $h_i$ is the number of rows of length $i$ in $\lambda$. We get
\begin{equation}\label{EQ11.1CLT}
\begin{split}
b_{k,k'}= \sum_{\substack{h_{-1},\dots,h_{k-1}\geq 0,\\ h_{-1}+h_1+\dots+h_{k-1}=k,\\h_1+2h_2+\dots+(k-1)h_{k-1}=h_{-1}-1.}}     \sum_{\substack{h_{-1}',\dots,h_{k'-1}'\geq 0,\\ h_{-1}'+h_1'+\dots+h_{k'-1}'=k',\\h_1'+2h_2'+\dots+(k'-1)h_{k'-1}'=h_{-1}'-1.}} \binom{k+1}{h_{-1},h_1,\dots,h_{k-1}}\hspace{3mm}\\ \times\binom{k'+1}{h_{-1}',h_1',\dots,h_{k'-1}'}\frac{f_{\lambda,\lambda'}}{(k+1)(k'+1)}.
\end{split}
\end{equation}

Before fully expanding $f_{\lambda,\lambda'}$, for each $i\in \lambda$ and $j\in \lambda'$ denote
\begin{equation*}
    g_{i,j}=d_{i,j}-\frac{ijc_{i}c_{j}}{\max(\alpha,\alpha')} 
+ \sum_{r\geq 1} \sum_{\substack{s_1,\dots,s_r\geq 1\\ s_1+\dots+s_r=i}} \sum_{\substack{t_1,\dots,t_r\geq 1\\ t_1+\dots+t_r=j}}\dfrac{ij}{r\max(\alpha,\alpha')^r}\prod_{m=1}^r c_{s_m+t_m-1},
\end{equation*}
so that we have
$f_{\lambda,\lambda'}=\alpha^{\frac{l(\lambda)+k}{2}} \alpha'^{\frac{l(\lambda')+k'}{2}}
\prod_{\substack{i \in \lambda\\ j \in \lambda'}} c_i c_j  \sum_{\substack{i \in \lambda\\ j \in \lambda'}} \frac{1}{c_i c_j}g_{i,j}$. Using this identity in equation (\ref{EQ11.1CLT}) gives
\begin{equation*}
\begin{split}
 b_{k,k'} =\sum_{\substack{h_{-1},\dots,h_{k-1}\geq 0,\\ h_{-1}+h_1+\dots+h_{k-1}=k+1,\\h_1+2h_2+\dots+(k-1)h_{k-1}=h_{-1}-1.}}\sum_{\substack{h_{-1}',\dots,h_{k'-1}'\geq 0,\\ h_{-1}'+h_1'+\dots+h_{k'-1}'=k'+1,\\h_1'+2h_2'+\dots+(k'-1)h_{k'-1}'=h_{-1}'-1.}} \hspace{-3mm}\frac{\alpha^{h_{-1}-1} \alpha'^{h_{-1}'-1}}{(k+1)(k'+1)}\binom{k+1}{h_{-1},h_1,\dots,h_{k-1}}\\   \cdot \binom{k'+1}{h_{-1}',k_1',\dots,h_{k'-1}'} \Big(\prod_{\substack{2\leq a \leq k\\ 2\leq b \leq k'}} c_a^{h_a} c_b^{h_b'}\Big) \sum_{\substack{2\leq i \leq k,\\ 2\leq j \leq k'}} \frac{h_i h_j}{c_i c_j} g_{i,j}.\\
\end{split}
\end{equation*}

We will switch the order of summation to take into account the term $\frac{h_ih_j}{c_i c_j}$ and how it interact with the previous terms in the sum. We get that
\begin{equation*}
\begin{split}
 b_{k,k'} = \sum_{\substack{2\leq i \leq k\\ 2\leq j \leq k'}} \sum_{\substack{h_{-1},\dots,h_{k-1}\geq 0,\\ h_{-1}+h_1+\dots+h_{k-1}=k+1,\\h_1+2h_2+\dots+(k-1)h_{k-1}=h_{-1}-1.}}\sum_{\substack{h_{-1}',\dots,h_{k'-1}'\geq 0,\\ h_{-1}'+h_1'+\dots+h_{k'-1}'=k'+1,\\h_1'+2h_2'+\dots+(k'-1)h_{k'-1}'=h_{-1}'-1.}} \hspace{-5mm} \binom{k}{h_{-1},h_1,\dots,h_i-1,\dots,h_{k-1}}\\
 \cdot \binom{k'}{h_{-1}',h_1',\dots,h_j'-1,\dots,h_{k'-1}'} \frac{\alpha^{h_{-1}-1} \alpha'^{h_{-1}'-1}}{c_ic_j}\Big(\prod_{\substack{2\leq a \leq k\\ 2\leq b \leq k'}} c_a^{h_a} c_b^{h_b'} \Big)g_{i,j}.\\
\end{split}
\end{equation*}
Now for each $2\leq i\leq q$ and $2\leq j \leq q'$ we change variables $h_i-1\to \hat{h}_i$ and $h_j-1\to \hat{h}_j$. This gives
\begin{equation*}
\begin{split}
 b_{k,k'} = \frac{1}{\alpha \alpha'} \sum_{\substack{2\leq i \leq k\\ 2\leq j \leq k'}} \sum_{\substack{h_{-1},\dots,h_{k-1}\geq 0,\\ h_{-1}+h_1+\dots+h_{k-1}=k+1,\\h_1+2h_2+\dots+(k-1)h_{k-1}=h_{-1}-1.}}\sum_{\substack{h_{-1}',\dots,h_{k'-1}'\geq 0,\\ h_{-1}'+h_1'+\dots+h_{k'-1}'=k'+1,\\h_1'+2h_2'+\dots+(k'-1)h_{k'-1}'=h_{-1}'-1.}} \hspace{-5mm} \binom{k}{h_{-1},h_1,\dots,h_{k-1}}\\
 \cdot \binom{k'}{h_{-1}',h_1',\dots,h_{k'-1}'} \alpha^{h_{-1}} \alpha'^{h_{-1}'}\Big(\prod_{\substack{2\leq a \leq k\\ 2\leq b \leq k'}} c_a^{h_a} c_b^{h_b'}\Big) g_{i,j}.\\
\end{split}
\end{equation*}

Similarly to what we did to compute the moments of the Law of Large numbers, we can reformulate the equation as
\begin{equation}\label{EQ11.5CLT}
\begin{split}
    b_{k,k'} =[z^{-1}w^{-1}]\frac{1}{\alpha \alpha'}\big(\alpha z^{-1}+zF_\rho(z)\big)^{k}\big(\alpha' w^{-1}+wF_\rho(w)\big)^{k'} 
    \Big( \sum_{i,j\geq 2 } g_{i,j} z^iw^j \Big)\\
\end{split}
\end{equation}

Finally we simplify this formula.  Denoting $H(z,w) = \sum_{i,j\geq 2 } g_{i,j}z^iw^j$, we have 
\begin{align*}
H(z,w) &= \sum_{i,j\geq 2 }z^iw^j \bigg[d_{i,j}-\frac{ij c_i c_j}{\max(\alpha,\alpha')}  +\sum_{r\geq 1} \frac{ij}{r\max(\alpha,\alpha')^r} \sum_{\substack{s_1,\dots,s_r\geq 1\\ s_1+\dots+s_r=i}} \sum_{\substack{t_1,\dots,t_r\geq 1\\ t_1+\dots+t_r=j}} \prod_{m=1}^r c_{s_m+t_m-1}\bigg]\\
&=\sum_{i,j\geq 2 }  (d_{i,j}- \frac{i j c_i c_j}{\max(\alpha,\alpha')}) z^iw^j  + zw \partial_z\partial_w \big(\sum_{r=1}^\infty \frac{1}{r\max(\alpha,\alpha')^r} \big(\sum_{i,j\geq 1} c_{i+j-1}z^iw^j\big)^r\big)\\
&=\sum_{i,j\geq 2 }  (d_{i,j}- \frac{i j c_i c_j}{\max(\alpha,\alpha')}) z^iw^j  + zw \partial_z\partial_w\ln\big(1-\frac{\sum_{i,j\geq 1} c_{i+j-1}z^iw^j}{\max(\alpha,\alpha')}\big)\\
&=Q_\rho(z,w) - zw \partial_z \partial_w \Big(\frac{zwF_\rho(z)F_\rho(w)}{\max(\alpha,\alpha')}+ \ln\big(1-\frac{\sum_{m\geq 2} \sum_{j=1}^{m-1} c_{m-1}z^{m}(w/z)^j}{\max(\alpha,\alpha')}\big)\Big)\\
&=Q_\rho(z,w) - zw \partial_z \partial_w \Big(\frac{zwF_\rho(z)F_\rho(w)}{\max(\alpha,\alpha')}+ \ln\big(1- \frac{\sum_{m\geq 2} c_{m-1}z^{m} \frac{(w/z)-(w/z)^m}{1-w/z}}{\max(\alpha,\alpha')}\big)\Big)\\
&=Q_\rho(z,w) - zw \partial_z \partial_w \Big(\frac{zw}{\max(\alpha,\alpha')}F_\rho(z)F_\rho(w)+ \ln\big(1-zw\frac{zF_\rho(z)-wF_\rho(w)}{\max(\alpha,\alpha')(z-w)}\big)\Big).\\
\end{align*}
Finally, putting all back into equation (\ref{EQ11.5CLT}) we obtain
\begin{equation*}
\begin{split}
b_{k,k'}^{\alpha,\alpha'} = [z^{-1}w^{-1}] \frac{1}{\alpha \alpha'} \Big(\alpha z^{-1}+zF_\rho(z)\big)^{k}\big(\alpha'w^{-1}+wF_\rho(w)\big)^{k'}  
 \big(Q_\rho(z,w) \hspace{3cm} \\- zw \partial_z \partial_w\big( \frac{zw}{\max(\alpha,\alpha')}F_\rho(z)F_\rho(w)+\ln(1 -zw\frac{zF_\rho(z)-wF_\rho(w)}{\max(\alpha,\alpha')(z-w)})\big)\Big).
\end{split}
\end{equation*}

We now turn on to prove item \textit{(3)} of Definition \ref{defSatiesfiesCLT}. The proof is similar to the calculation of the covariance, with the advantage that we don't need to track coefficients or precise expansions for products of central elements and the disadvantage that we need to compute lower order terms. Lemma \ref{LemmaCumulantPermCumulant}, Theorem \ref{TheoremExpansionofDk} and Proposition \ref{PropMultiLinePermCumul} give    
\begin{equation*}
    n^{r/2} \kappa(X_{k_1},X_{k_2},\dots,X_{k_r}) = \sum_{\substack{\lambda^i \in \mathcal{S}_{k_i},\\ i=1,\dots,r}}^r  N_{\vec{\lambda}}(n) \kappa_{\rho_n,r}(\Sigma_{\lambda^1},\dots,\Sigma_{\lambda^r})
\end{equation*}

where 
\begin{equation*}
N_{\vec{\lambda}}(n)= n^{r/2} \prod_{i=1}^r\frac{N_{\lambda^i}[k_i](n)}{n^{\frac{k_i}{2}} (n\ff |\lambda^i|)}=C_{\vec{\lambda}} n^{\frac{r+ l(\vec{\lambda})-2|\vec{\lambda}|}{2}}+O( n^{\frac{r+ l(\vec{\lambda})-2|\vec{\lambda}|}{2}-1})
\end{equation*}
and $C_{\vec{\lambda}} $ is a constant. It follows that it will be enough to prove that
\begin{equation*}
\lim_{n \to \infty} n^{\frac{r+ l(\vec{\lambda})-2|\vec{\lambda}|}{2}} \kappa_{\rho_n,r}(\Sigma_{\lambda^1},\dots,\Sigma_{\lambda^r}) =0.
\end{equation*}

Fortunately Lemma \ref{lemmaExpansionCumulantCentralElements2} allow us to expand this cumulant as a finite sum of falling cumulants. We need to keep track of the powers of $n$ accompanying each term. We have that
\begin{equation*}
\kappa_{\rho_n,r}(\Sigma_{\lambda^1},\dots,\Sigma_{\lambda^r}) = \sum_{\pi\in \Theta_r} \sum_{(\vec{\mu},\theta) \in \Xi_\pi[\vec{\lambda}]} \kappa^{F[\vec{k}_\theta,\vec{n}]}_{\rho_n,|\pi|}(\sigma[\mu^1],\dots,\sigma[\mu^{|\pi|}]).
\end{equation*}

Additionally, $|\vec{k}_\theta|=\sum_{j=1}^{|\pi|} |\vec{k}^j|\leq \frac{l(\vec{\mu})+|\vec{\lambda}|+\ell(\vec{\lambda})}{2}+|\pi|-r$.
We can now fix $\pi \in \Theta_r$ and $(\vec{\mu},\theta) \in \Xi_\pi[\vec{\lambda}]$, it will be enough to prove that
\begin{equation*}
\lim_{n \to \infty} n^{\frac{r+ l(\vec{\lambda})-2|\vec{\lambda}|}{2}}\kappa^{F[\vec{k}_\theta,\vec{n}]}_{\rho_n,|\pi|}(\sigma[\mu^1],\dots,\sigma[\mu^{|\pi|}]) =0.
\end{equation*}

We now use Lemma \ref{ExpansionofDefFcumulants} to expand the falling cumulant into a sum of permutation-cumulants, we have that
\begin{equation*}
\kappa^{F[\vec{k}_\theta,\vec{n}]}_{\rho_n,|\pi|}(\sigma[\mu^1],\dots,\sigma[\mu^{|\pi|}])=\sum_{\theta \in \Theta_r} f_\theta(n) \kappa_{\rho,|\theta|}\Big(\prod_{j \in B_1} \sigma[\mu^j],\prod_{j \in B_2} \sigma[\mu^j],\dots,\prod_{j \in B_{|\theta|}} \sigma[\mu^j]\Big),
\end{equation*}
where $\theta=\{B_1,\dots,B_{|\theta|}\}$ and $f_\theta(n)=O(n^{\beta})$ with $\beta=\sum_{j=1}^r |\vec{k}^j|+|\theta|-|\pi|$.  Once again, this reduces the problem to prove that
\begin{equation}\label{EQ12CLT}
\lim_{n \to \infty } n^{\frac{l(\vec{\mu})-r}{2}+|\theta|}\kappa_{\rho,|\theta|}\Big(\prod_{j \in B_1} \sigma[\mu^j],\prod_{j \in B_2} \sigma[\mu^j],\dots,\prod_{j \in B_{|\theta|}} \sigma[\mu^j]\Big)=0
\end{equation}

Fix $s=|\theta|$ and for $i=1,\dots,s$ rename the permutations $\sigma_i=\prod_{j \in B} \sigma[\mu^j]$, hence $l(\vec{\mu})=\sum_{i=1}^{s} |\sigma_i|$ and equation (\ref{EQ12CLT}) becomes
\begin{equation}\label{EQ13CLT}
\lim_{n \to \infty } n^{\frac{\sum_{i=1}^{s} |\sigma_i|}{2}+s-r/2} \kappa_{\rho_n,s}(\sigma_1,\dots,\sigma_s)=0
\end{equation}

We subdivide in three cases to verify that equation  (\ref{EQ13CLT}) is satisfied. If $s=1$, we know from Lemma \ref{lemmaproduct} that $\lim_{n \to \infty} n^{\frac{|\sigma_1|}{2}}M_{\rho_n}(\sigma_1)$ exists, since $r\geq 3$, equation (\ref{EQ13CLT}) is satisfied. If $s=2$, we know from the computation of the covariance that the limit $\lim_{n \to \infty} n^{\frac{|\sigma_1|+|\sigma_2|+1}{2}} \kappa_{\rho_n,2}(\sigma_1,\sigma_2)$ exists, again, since $r\geq 3$, equation (\ref{EQ13CLT}) is satisfied. We now turn on to the case $s\geq 3$, we will further decompose the partitions $\sigma_i=\sigma_{i,1}\cdots\sigma_{i,t_i}$ for $i=1,\dots,s$ where each $\sigma_{i,j}$ is a cycle. Since $s\leq r$, it is enough to prove that  
\begin{equation*}
\lim_{n \to \infty } n^{\frac{\sum_{i=1}^{s} |\sigma_i|+s}{2}} \kappa_{\rho_n,s}(\sigma_1,\dots,\sigma_s)=0
\end{equation*}

Lemma \ref{Cumulantsofproducts} gives
\begin{equation*}
\kappa_{\rho_n,s}(\sigma_1,\dots,\sigma_s) =    \sum_{\nu} \prod_{D \in \nu} \kappa_{\rho_n,|D|}(\sigma_{i,j}:(i,j)\in D)
\end{equation*}

The sum runs among partitions $\nu$ of the set $\{(1,1),\dots,(1,t_1),(2,1),\dots,(s,1),\dots,(t,s)\}$, such that for $\eta=$ the partitions with blocks $\{(i,1),\dots,(i,t_i))\}$ we have $\nu \vee \eta=\hat{1}$. It will be enough to verify that
\begin{equation} \label{EQ16CLT}
    \lim_{n \to \infty} n^{\frac{\sum_{i=1}^{s} |\sigma_i|+s}{2}}   \prod_{D \in \nu} \kappa_{\rho_n,|D|}(\sigma_j:j\in D)=0
\end{equation}

to conclude that equation (\ref{EQ13CLT}) is true for $s \geq 3$. Similarly to equation (\ref{EQ10.5CLT}) we have that
\begin{equation}\label{EQ14CLT}
\begin{split}
n^{\frac{\sum_{i=1}^{s} |\sigma_i|}{2}+\frac{1}{2} \sum_{\substack{D \in \nu\\ |D|\geq 2}}|D|}   \prod_{D \in \nu} \kappa_{\rho_n,|D|}(\sigma_{i,j}:(i,j)\in D)  =\prod_{\substack{D \in \nu\\ |D|=1}} n^{\frac{\sum_{(i,j) \in D} |\sigma_{i,j}|}{2}} \hspace{2cm} \\ \cdot \kappa_{\rho_n,|D|}(\sigma_{i,j}:(i,j)\in D) \prod_{\substack{D \in \nu\\ |D|\geq 2}} n^{\frac{\sum_{(i,j)\in D} |\sigma_{i,j}|+|D|}{2}} \kappa_{\rho_n,|D|}(\sigma_{i,j}:(i,j)\in D).
\end{split}
\end{equation}

On one hand, the right hand side of equation (\ref{EQ14CLT}) converges. Notice that 
\begin{align*}
     \sum_{\substack{D \in \nu\\ |D|\geq 2}}|D| &=  \sum_{\substack{D \in \nu}}|D| - \sum_{\substack{D \in \nu\\ |D|=1}}|D| = \sum_{i=1}^s t_i - \sum_{\substack{D \in \nu\\ |D|=1}}1.
\end{align*}
Additionally, condition $\nu \vee \eta=\hat{1}$ implies that $\sum_{\substack{D \in \nu\\ |D|=1}}1\leq \sum_{i=1}^s t_i-1$ with equality only if we have a single non singleton $D\in \nu$ with $|D|=s$. Hence 
\begin{equation*}
s \leq \sum_{\substack{D \in \nu\\ |D|\geq 2}}|D|,    
\end{equation*}
with equality only if there is some $D \in \nu$ with $|D|=s$. We consider both situations, on one hand if $s < \sum_{\substack{D \in \nu\\ |D|\geq 2}}|D|$, the convergence of equation (\ref{EQ14CLT}) ensures that (\ref{EQ16CLT}) is satisfied. On the other hand, if $s = \sum_{\substack{D \in \nu\\ |D|\geq 2}}|D|$, then for some $D \in \eta$, $|D|=s\geq 3$. In particular, for such $D$ we have that 
\begin{equation*}
\lim_{n \to \infty } n^{\frac{\sum_{(i,j)\in D} |\sigma_{i,j}|+|D|}{2}} \kappa_{\rho_n,|D|}(\sigma_{i,j}:(i,j)\in D) =0
\end{equation*}
Again, the convergence of (\ref{EQ14CLT}) ensures that (\ref{EQ16CLT}) is satisfied. \qedhere
\end{proof}

\subsection{Reverse central limit theorem} \label{SectionProofReverseCLT}
The proof of the reverse CLT closely follows both the proof of the reverse LLN and the proof of Gaussianity for the CLT. Lets remind that
\begin{equation*}
\mathcal{C}_{t,r}=\Big\{(\sigma_1,\dots,\sigma_r)\in \prod_{j=1}^r S_\infty:\{\sigma_j\}_{j=1}^r \text{ are non-trivial disjoint cycles and } \bigr|\prod_{j=1}^r \sigma_j\bigr|=t \Big\}.
\end{equation*}

\begin{proof}[Proof of Theorem \ref{TheoremCLT}, second implication.]
    We will prove by induction on the lexicographic order over the coordinates $(t,r)$ that 
\begin{enumerate}
    \item There exists a constant $c_{t+1}$ such that $\lim_{n \to \infty} n^{t/2}M_{\rho_n}(\sigma) = c_{t+1}$ for each $\sigma \in \mathcal{C}_{t,1}$.
    \item There exists a constant $d_{t_1+1,t_2+1}$ such that $\lim_{n \to \infty} n^{\frac{t_1+t_2}{2}+1}\kappa_{\rho_n,2}(\sigma_1,\sigma_2)=d_{t_1+1,t_2+1}$ for each $(\sigma_1,\sigma_2) \in \mathcal{C}_{t,2}$, $|\sigma_1|=t_1$, $|\sigma_2|=t_2$. 
    \item $\lim_{n \to \infty} n^{\frac{t+r}{2}}\kappa_{\rho_n,r}(\vec{\sigma}) = 0$ for each $\vec{\sigma} \in \mathcal{C}_{t,r}$ for $r\geq 3$. 
\end{enumerate}

The proof for the base cases $(1,1)$ and $(2,1)$ as well as the proof for the case $(t,1)$ (item (1)) for each $t$ is exactly the same as in the reverse LLN situation. We focus on the inductive step in case (2), that is, we have two cycles $(\sigma_1,\sigma_2) \in \mathcal{C}_{t,2}$ with $|\sigma_1|=t_1$, $|\sigma_2|=t_2$. Lemma \ref{LemmaCumulantPermCumulant} gives
\begin{equation*}
n \kappa(X_{t_1+2},X_{t_2+2})=n^{1-\frac{t_1+t_2+4}{2}}\kappa_{\rho_n,2}(D_{t_1+2},D_{t_2+2})    
\end{equation*}

Expand $D_{t_1+2}$ and $D_{t_2+2}$ using Theorem \ref{TheoremExpansionofDk} and Proposition \ref{PropMultiLinePermCumul} gives
\begin{equation}\label{EQ1RCLT}
\begin{split}
n \kappa(X_{t_1+2},X_{t_2+2}) = n^{\frac{t_1+t_2}{2}+1}\kappa_{\rho_n,2}(\hat{\Sigma}_{t_1+1},\hat{\Sigma}_{t_2+1})+ \sum_{(\lambda,\lambda') \in \mathcal{S}_{t_1+2,t_2+2}^\ast} n^{\frac{l(\lambda)+l(\lambda')}{2}+1}\\ \cdot (1+O(n^{-1}))\kappa_{\rho_n,2}(\hat{\Sigma}_\lambda,\hat{\Sigma}_{\lambda'}).
\end{split}
\end{equation}

Here $\mathcal{S}_{t_1+2,t_2+2}^\ast$ is the set of partitions $\{(\lambda,\lambda')\in \mathcal{S}_{t_1+2}\times\mathcal{S}_{t_2+2}$ such that $ \lambda \neq (t_1+2)$ or $\lambda' \neq (t_2+2)$, crucially we have that $l(\lambda)+l(\lambda')  \leq t-2$. 
We will study each part of equation (\ref{EQ1RCLT}) separately. We know that the left hand side converges, so it is enough to prove that the limits 
\begin{equation*}
   \lim_{n \to \infty} n^{\frac{l(\lambda)+l(\lambda')}{2}+1}\kappa_{\rho_n,2}(\hat{\Sigma}_\lambda,\hat{\Sigma}_{\lambda'})
\end{equation*}
exists to conclude that 
\begin{equation*}
\lim_{n \to \infty} n^{\frac{t_1+t_2}{2}+1}\kappa_{\rho_n,2}(\hat{\Sigma}_{t_1+1},\hat{\Sigma}_{t_2+1})
\end{equation*}
exists. Notice that 
\begin{equation*}
    n^{\frac{l(\lambda)+l(\lambda')}{2}+1}\kappa_{\rho_n,2}(\hat{\Sigma}_\lambda,\hat{\Sigma}_{\lambda'}) = \frac{n^{\frac{l(\lambda)+l(\lambda')}{2}+1}}{(n\ff |\lambda|)(n\ff |\lambda'|)}M_{\rho_n}(\Sigma_{\lambda} \Sigma_{\lambda'})-n^{\frac{l(\lambda)+l(\lambda')}{2}+1}M_{\rho_n}(\sigma[\lambda])M_{\rho_n}(\sigma[\lambda']),
\end{equation*}
and expanding the product $\Sigma_{\lambda} \Sigma_{\lambda'}$ we get 
\begin{equation*}
\frac{n^{\frac{l(\lambda)+l(\lambda')}{2}+1}}{(n\ff |\lambda|)(n\ff |\lambda'|)}M_{\rho_n}(\Sigma_{\lambda} \Sigma_{\lambda'}) = n^{\frac{l(\lambda)+l(\lambda')}{2}+1} M_{\rho_n}(\sigma[\lambda\cup\lambda'])+ \sum_\mu n^{\frac{l(\mu)}{2}}(1+O(n^{-1})) M_{\rho_n}(\sigma[\mu]),
\end{equation*}
where the sum run over partitions $\mu$ such that $l(\mu)<t$. Putting these two equations together we obtain
\begin{equation}\label{EQ2RCLT}
\begin{split}
n^{\frac{l(\lambda)+l(\lambda')}{2}+1}\kappa_{\rho_n,2}(\hat{\Sigma}_\lambda,\hat{\Sigma}_{\lambda'}) = n^{\frac{l(\lambda)+l(\lambda')}{2}+1} \kappa_{\rho_n,2}(\sigma[\lambda],\sigma[\lambda'])+ \sum_\mu n^{\frac{l(\mu)}{2}}(1+O(n^{-1})) M_{\rho_n}(\sigma[\mu])
\end{split}
\end{equation}
where $\sigma[\lambda]$ and $\sigma[\lambda']$ are disjoint permutations. Since $l(\mu)<t$, Lemma \ref{lemmaproduct} jointly with the inductive hypothesis ensure that $\lim_{n \to \infty}n^{\frac{l(\mu)}{2}} M_{\rho_n}(\sigma[\mu])$ exists. Similarly, since $l(\lambda)+l(\lambda')<t$, Lemma \ref{Cumulantsofproducts} jointly with the inductive hypothesis ensure that $\lim_{n \to \infty} n^{\frac{l(\lambda)+l(\lambda')}{2}+1} \kappa_{\rho_n,2}(\sigma[\lambda],\sigma[\lambda'])$ exists. These two statements guarantees that left hand side of equation (\ref{EQ2RCLT}) converges. 

We now expand the cumulant $\kappa_{\rho_n,2}(\hat{\Sigma}_{t_1+1},\hat{\Sigma}_{t_2+1})$. Similarly to what is obtained with equation (\ref{EQ2RCLT}) we have that  
\begin{equation*}
\begin{split}
    n^{\frac{t_1+t_2}{2}+1}\kappa_{\rho_n,2}(\hat{\Sigma}_{t_1+1},\hat{\Sigma}_{t_2+1}) =  n^{\frac{t_1+t_2}{2}+1} \kappa_{\rho_n,2}(\sigma[t_1+1],\sigma[t_2+1]) \hspace{1cm} \\ \hspace{1cm}+ \sum_\mu n^{\frac{l(\mu)}{2}}(1+O(n^{-1})) M_{\rho_n}(\sigma[\mu]).
\end{split}
\end{equation*}
Lemma \ref{lemmaproductofcycles} ensures that the sum run over partitions $\mu$ with $l(\mu)<t$ with the exception of $\mu^1 = (t_1+1,t_2+1)$ and $\mu^2=(t_1+t_2+1)$, we then have 
\begin{equation}\label{EQ3RCLT}
\begin{split}
    n^{\frac{t_1+t_2}{2}+1}\kappa_{\rho_n,2}(\hat{\Sigma}_{t_1+1},\hat{\Sigma}_{t_2+1}) =  n^{\frac{t_1+t_2}{2}+1} \kappa_{\rho_n,2}(\sigma[t_1+1],\sigma[t_2+1])\hspace{2cm}\\ +  n^{\frac{l(\mu^1)}{2}}(1+O(n^{-1})) M_{\rho_n}(\sigma[\mu^1])+ n^{\frac{l(\mu^2)}{2}}(1+O(n^{-1})) M_{\rho_n}(\sigma[\mu^2])\\+ \sum_\mu n^{\frac{l(\mu)}{2}}(1+O(n^{-1})) M_{\rho_n}(\sigma[\mu])
\end{split}
\end{equation}
with the sum running over partitions $\mu$ with $l(\mu)<t$. The inductive hypothesis ensures that each limit $\lim_{n \to \infty}n^{\frac{l(\mu)}{2}} M_{\rho_n}(\sigma[\mu])$ exists. Similarly, since $\sigma[t_1+t_2+1] \in \mathcal{C}_{t,1}$, the inductive hypothesis ensures that the limit $ \lim_{n \to \infty} n^{\frac{l(\mu^2)}{2}} M_{\rho_n}(\sigma[\mu^2])$ exists. Finally, we have that
\begin{equation}\label{EQ3.5RCLT}
    M_{\rho_n}(\sigma[\mu^1]) = \kappa_{\rho_n,2}(\sigma[t_1+1],\sigma[t_2+1]) +M_{\rho_n}(\sigma[t_1+1])M_{\rho_n}(\sigma[t_2+1])
\end{equation}
again the inductive hypothesis ensures that $\lim_{n \to \infty} n^{\frac{l(\mu^1)}{2}} M_{\rho_n}(\sigma[t_1+1])M_{\rho_n}(\sigma[t_2+1])$ exists and regrouping the remaining terms on equation (\ref{EQ3RCLT}) we additionally conclude that the limit $\lim_{n \to \infty} n^{\frac{t_1+t_2}{2}+1} \kappa_{\rho_n,2}(\sigma[t_1+1],\sigma[t_2+1])$ exists. We now work the inductive step of case (3). 

Let $(\sigma_1,\dots,\sigma_r) \in \mathcal{C}_{t,r}$ with $r\geq 3$ and let $t_i=|\sigma_i|$. Lemma \ref{LemmaCumulantPermCumulant}, Theorem \ref{TheoremExpansionofDk} and proposition \ref{PropMultiLinePermCumul} give
\begin{equation}\label{EQ4RCLT}
\begin{split}
n^{r/2}\kappa(X_{t_1+2},X_{t_2+2},\dots,X_{t_r+2}) = n^{\frac{t+r}{2}}(1+O(n^{-1}))\kappa_{\rho_n,r}(\hat{\Sigma}_{t_1+1},\dots,\hat{\Sigma}_{t_r+1}) \hspace{1cm}\\
+ \sum_{\vec{\lambda}} C_{\vec{\lambda}} n^{\frac{l(\vec{\lambda})+r}{2}}(1+O(n^{-1}))\kappa_{\rho_n,r}(\hat{\Sigma}_{\lambda^1},\dots,\hat{\Sigma}_{\lambda^r}).
\end{split}
\end{equation}
Here the sum runs over $\vec{\lambda}$ with $l(\vec{\lambda})<t$. We already know that the left hand side on equation (\ref{EQ4RCLT}) converges to $0$, we will now prove that
\begin{equation*}
\lim_{n \to \infty} n^{\frac{l(\vec{\lambda})+r}{2}}\kappa_{\rho_n,r}(\hat{\Sigma}_{\lambda^1},\dots,\hat{\Sigma}_{\lambda^r}) =0.
\end{equation*} 
We follow the steps we used to prove item (3) in CLT, that is, we use Lemma \ref{lemmaExpansionCumulantCentralElements} to get 
\begin{equation*}
    n^{\frac{l(\vec{\lambda})+r}{2}}\kappa_{\rho_n,r}(\hat{\Sigma}_{\lambda^1},\dots,\hat{\Sigma}_{\lambda^r}) = \sum_{\pi\in \Theta_r} \sum_{(\vec{\mu},\theta) \in \Xi_\pi[\vec{\lambda}]} n^{\frac{l(\vec{\lambda})+r}{2}} \kappa^{F[\vec{k}_\theta,\vec{n}]}_{\rho_n,|\pi|}(\sigma[\mu^1],\dots,\sigma[\mu^{|\pi|}]).
\end{equation*}
Here $\sum_{j=1}^r |\vec{k}^j|\leq \dfrac{l(\vec{\mu})+|\vec{\lambda}|+\ell(\vec{\lambda})}{2}+|\pi|-r$. We further decompose the falling cumulants into an usual sum of cumulants using Lemma \ref{ExpansionofDefFcumulants}.
\begin{equation*}
n^{\frac{l(\vec{\lambda})+r}{2}} \kappa^{F[\vec{k}_\theta,\vec{n}]}_{\rho_n,|\pi|}(\sigma[\mu^1],\dots,\sigma[\mu^{|\pi|}])=\sum_{\theta \in \Theta_{|\theta|}} f_\theta(n) \kappa_{\rho_n,|\theta|}(\prod_{j \in B_1}\sigma[\mu^j],\dots,\prod_{j \in B_{|\theta|}}\sigma[\mu^j]),
\end{equation*}
where $\pi=\{B_1,\dots,B_{|\theta|}\}$. Furthermore, $f_\theta(n)=O(n^{\beta})$ with $\beta=\frac{l(\vec{\mu})-r}{2}+|\theta|$. Fix $s=|\theta|$ and for $i=1,\dots,s$ renaming the permutations $\sigma_i=\prod_{j \in B} \sigma[\mu^j]$, hence $l(\vec{\mu})=\sum_{i=1}^{s} |\sigma_i|$. We are now reduced to prove that
\begin{equation*}
\lim_{n \to \infty } n^{\frac{\sum_{i=1}^{s} |\sigma_i|}{2}+s-r/2} \kappa_{\rho_n,s}(\sigma_1,\dots,\sigma_s)=0
\end{equation*}
At this point the proof mimics what we did to prove the CLT, with the additional inductive observation that $\sum_{i=1}^s|\sigma_i| = l(\mu)<t$ that ensures that we can use the inductive hypothesis and hence all limits converge $0$. This proves that the remaining term of equation (\ref{EQ4RCLT}) converges to $0$, that is, 
\begin{equation*}
\lim_{n \to \infty}n^{\frac{t+r}{2}}\kappa_{\rho_n,r}(\hat{\Sigma}_{t_1+1},\dots,\hat{\Sigma}_{t_r+1})=0.
\end{equation*}
We which to expand this cumulant to conclude our induction. Lemma \ref{lemmaExpansionCumulantCentralElements} gives 

\begin{align*}
n^{\frac{t+r}{2}}\kappa_{\rho_n,r}(\hat{\Sigma}_{t_1+1},\dots,\hat{\Sigma}_{t_r+1}) &= \sum_{\pi\in \Theta_r} \sum_{(\vec{\mu},\theta) \in \Xi_\pi[\vec{t}]} n^{\frac{t+r}{2}-|\vec{t}|} \big(1+O(n^{-1})\big)\kappa^{F[\vec{k}_\theta]}_{\rho_n,|\pi|}(\sigma[\mu^1],\dots,\sigma[\mu^{|\pi|}])\\
&= \sum_{(\vec{\mu},\theta) \in \Xi_{\hat{0}}[\vec{t}]} n^{\frac{t+r}{2}-|\vec{t}|} \big(1+O(n^{-1})\big) \kappa^{F[\vec{k}_\theta]}_{\rho_n,|\hat{0}|}(\sigma[\mu^1],\dots,\sigma[\mu^{|\hat{0}|}])\\
&\hspace{-7mm} +  \sum_{\pi\in \Theta_r\backslash\{\hat{0}\}} \sum_{(\vec{\mu},\theta) \in \Xi_\pi[\vec{t}]} n^{\frac{t+r}{2}-|\vec{t}|} \big(1+O(n^{-1})\big) \kappa^{F[\vec{k}_\theta]}_{\rho_n,|\pi|}(\sigma[\mu^1],\dots,\sigma[\mu^{|\pi|}])\\
&= n^{\frac{t+r}{2}-|\vec{t}|} \big(1+O(n^{-1})\big) \kappa^{F[\vec{t}]}_{\rho_n,r}(\sigma[t_1+1],\dots,\sigma[t_r+1])\\
&\hspace{-7mm} +  \sum_{\pi\in \Theta_r\backslash\{\hat{0}\}} \sum_{(\vec{\mu},\theta) \in \Xi_\pi[\vec{t}]} n^{\frac{t+r}{2}-|\vec{t}|} \big(1+O(n^{-1})\big) \kappa^{F[\vec{k}_\theta]}_{\rho_n,|\pi|}(\sigma[\mu^1],\dots,\sigma[\mu^{|\pi|}])
\end{align*}
where $\vec{t}=(t_1+1,t_2+1,\dots,t_r+1)$. Theorem \ref{Product1stOrderExpansionUPGRADEDMultilevel} ensures that if $\pi \neq \hat{0}$ then for each $(\vec{\mu},\theta) \in \Xi_\pi[\vec{t}]$ we either have $l(\vec{\mu})<t$ or $l(\vec{\mu})<t$ and $\ell(\vec{\mu})<r$. similarly to what we did for previously, we directly conclude from the inductive hypothesis that for each $\pi \neq \hat{0}$ and $(\vec{\mu},\theta) \in \Xi_\pi[\vec{t}]$. Additionally, since $|\vec{k}_\theta|\leq \dfrac{l(\vec{\mu})+|\vec{t}|+\ell(\vec{t})}{2}+|\pi|-r$, we get that
\begin{equation*}
\lim_{n \to \infty} n^{\frac{t+r}{2}-|\vec{t}|} \kappa^{F[\vec{k}_\theta,\vec{n}]}_{\rho_n,|\pi|}(\sigma[\mu^1],\dots,\sigma[\mu^{|\pi|}]) =0.
\end{equation*}

We are then left to analyze the remaining falling cumulant. The main idea is to repeat what we did in case (2) of this proof, see equation (\ref{EQ3.5RCLT}), with more generality using Lemma \ref{Cumulantsofproducts}. First, Lemma \ref{ExpansionofDefFcumulants} gives
\begin{equation*}
    n^{\frac{t+r}{2}-|\vec{t}|} \kappa^{F[\vec{t},\vec{n}]}_{\rho_n,r}(\sigma[t_1+1],\dots,\sigma[t_r+1]) = \sum_{\pi \in \Theta_r} f_\pi(n) \kappa_{\rho_n,|\pi|}(\prod_{j \in B_1} \sigma[t_j+1],\dots,\prod_{j \in B_{|\pi|}} \sigma[t_j+1]),
\end{equation*}
where $\pi=\{B_1,\dots,B_{\pi}\}$ and $f_\pi(n)=O(n^\beta)$ with $\beta=|\pi|-r/2$. Additionally, Lemma \ref{Cumulantsofproducts} allows us to write
\begin{align*}
n^{t+|\pi|-r/2} \kappa_{\rho_n,|\pi|}(\prod_{j \in B} \sigma[t_j+1]:B \in \pi) &= n^{t+|\pi|-r/2}\sum_{\substack{\nu \in \Theta_r,\\ \nu \vee \nu_\pi =\hat{1}.}} \prod_{B \in \nu} \kappa_{\rho_n,|B|}(\sigma[t_j+1]:j\in B)\\
&= n^{t+|\pi|-r/2}\kappa_{\rho_n,r}(\sigma[t_1+1],\dots,\sigma[t_r+1])\\
& \hspace{5mm}+ \sum_{\substack{\nu \in \Theta_r,\\ \nu \vee \nu_\pi =\hat{1},\\ \nu \neq \hat{1}.}} n^{t+|\pi|-r/2} \prod_{B \in \nu} \kappa_{\rho_n,|B|}(\sigma[t_j+1]:j\in B)
\end{align*}
where $\nu_\pi$ is the set partition from Lemma \ref{Cumulantsofproducts} that depends on $\pi$. Notice that for each $\nu \neq \hat{1}$, $|\nu|\geq 2$ and hence all $B\in \nu$ satisfies $|B|<r$, moreover $\sum_{j \in B} |\sigma[t_j+1]|=\sum_{j \in B}  t_j<t$, this jointly with the inductive hypothesis ensures that for $|B|\geq 3$
\begin{equation*}
\lim_{n \to \infty} n^{\frac{\sum_{j \in B}  t_j+|B|\cdot 1(|B| \neq 1)}{2}} \kappa_{\rho_n,|B|}(\sigma[t_j+1]:j\in B)=0,
\end{equation*}
while for $|B|=1$ and $|B|=2$ we will have convergence to some limit. Hence for all $\nu \neq \hat{1}$ we have
\begin{equation*}
    \lim_{n \to \infty }  n^{t+|\pi|-r/2} \prod_{B \in \nu} \kappa_{\rho_n,|B|}(\sigma[t_j+1]:j\in B) =0.
\end{equation*}

The only terms left are $n^{t+|\pi|-r/2}\kappa_{\rho_n,r}(\sigma[t_1+1],\dots,\sigma[t_r+1])$ for each $\pi \in \theta_r$, adding over all this terms we get that
\begin{equation*}
\lim_{n \to \infty } n^{t+r/2}\big(1+o(1)\big)\kappa_{\rho_n,r}(\sigma[t_1+1],\dots,\sigma[t_r+1])=0
\end{equation*}
which concludes the induction step. 
\end{proof}

\subsection{Proof of Lemma \ref{ExpansionofDefFcumulants2}} \label{CombRosas}

The proof of Lemma  \ref{ExpansionofDefFcumulants} is purely combinatorial. The idea is to interpret the leading coefficients in the sum of Definition \ref{DefFcumulants2} as weights on vertices of the set partition lattice. We successively extract the weights corresponding to the smaller elements in the partition lattice and construct permutation-cumulants with these terms. This process is always possible, as such, the decomposition of the falling cumulant into permutation-cumulants is not difficult to prove, rather, we will need to verify that the leading coefficients are small. 

Before starting with the proof we introduce a few combinatorial lemmas. The following lemma will allow to decompose the products of generalized falling factorials. 

\begin{lemma}\label{LemmaBijection}
Let $s \geq 1$, given  $\vec{n}=(n_1\leq n_2 \leq \dots \leq n_s)$ an increasing collection of integers and $\vec{k}^j=(k_1^j,\dots,k_s^j)$ for $j=1,\dots,r$ a collection of vectors of integers. Let $B_j= \{1,2,\dots,k^j_1+\dots+k^j_s\}$ for each $j=1,\dots,r$, consider $K = \sqcup_{j=1}^r B_j$ the set of formal disjoint unions of the sets $B_j$. Let $\pi$ the set partition of $\Theta_{K}$ with blocks $B_j$ and let $\hat{0}$ be the minimal element of $\Theta_K$. Then we have the following decomposition for a product of generalized $\vec{k}$th falling factorials,
\begin{equation*}
(\vec{n}\ff \vec{k}^1)(\vec{n}\ff \vec{k}^2)\cdots (\vec{n}\ff \vec{k}^r) = \sum_{\substack{\eta \in \Theta_K,\, \eta \wedge \pi = \hat{0}}} \big(\vec{n}\ff \vec{\kappa}(\eta)\big),
\end{equation*}
where $\vec{\kappa}(\eta)$ is defined as follows. For each $D\in \eta$ and $1\leq i \leq s$, let
\begin{equation*}
t_D=\inf\{s:x\leq k^j_1+\dots+k^j_s \text{ for }\, x\in B_j\cap D\} \hspace{2mm} \textup{ and } \hspace{2mm} \kappa_i(\eta)=|\{D\in \eta: t_D=i\}|,
\end{equation*}
then $\vec{\kappa}(\eta)=(\kappa_1(\eta),\dots,\kappa_s(\eta))$.
\end{lemma}
\begin{proof}
We give a combinatorial proof. Given $n_1\leq\cdots\leq n_s$ integers and $\vec{k}^1,\dots,\vec{k}^r$ vectors of $s$ integers. Let $C_{\vec{k}}^{\vec{n}}$ the number of functions $f:K\to \{1,\dots,n_s\}$ such that $f$ restricted to $B_j$ is injective and for each $1\leq j \leq r$ and $1\leq t \leq s$, $f(\{1,\dots,k_1^j+\dots +k^j_t\}) \subseteq \{1,\dots,n_t\}$.

On one hand, we count the number of sets $f(B_j)$ for each $1 \leq j \leq r$. Successively for each $1\leq t \leq s$ we choose $k^j_t$ distinct elements on $\{1,\dots,n_t\}$, this is given by the product
\begin{equation*}
\big(n_t-(k_1^j+\dots+k_{t-1}^j)\big)  \cdots \big(n_t-(k_1^j+\dots+k_{t}^j-1)\big)   
\end{equation*}
Furthermore, by multiplying these values for each $1\leq t \leq s$ we obtain $(\vec{n}\ff \vec{k}^j)$. Finally, by multiplying for each $1\leq j \leq r$ we get that
\begin{equation*}
    C_{\vec{k}}^{\vec{n}} = (\vec{n}\ff \vec{k}^1)(\vec{n}\ff \vec{k}^2)\cdots (\vec{n}\ff \vec{k}^r).
\end{equation*}

On the other hand, we can associate to each function $f$ a partition $\eta \in \Theta_K$ where $x$ and $y$ are in the same block of $\eta$ when $f(x)=f(y)$. We know that $f(x)\neq f(y)$ if $x$ and $y$ are both on $B_j$ for some $1\leq j \leq r$, hence $\eta \wedge \pi = \hat{0}$. Now fix a partition $\eta \in \Theta_K$ and we count the number of such function with associated partition $\eta$. For each $D\in \eta$, $f(D)\subset \{1,\dots,n_{t_D}\}$ for $t_D=\inf\{s:x\leq k^j_1+\dots+k^j_s \text{ for }\, x\in B_j\cap D\}$. We now count the number of blocks whose image has the restriction to be contained in $\{1,\dots,n_{i}\}$, there are exactly $\kappa_i(\eta)=|\{D\in \eta: t_D=i\}|$ such elements for $1\leq i \leq s$. We can now repeat the previous combinatorial argument to get
\begin{equation*}
    C_{\vec{k}}^{\vec{n}} = \sum_{\substack{\eta \in \Theta_K,\, \eta \wedge \pi = \hat{0}}} \big(\vec{n}\ff \vec{\kappa}(\eta)\big). \qedhere
\end{equation*}
\end{proof}

The following lemma, of independent interest, will allow us to obtain of the main combinatorial identities we will use in the proof of Lemma \ref{ExpansionofDefFcumulants2}. Following the notation introduced for the falling factorial, we will denote the rising factorial
\begin{equation*}
(x\rf k)=\begin{cases}
    x(x+1)\ldots(x+k-1),& \text{if $k=1,2,\ldots$,}\\
1,& \text{if $k=0$.}
\end{cases}
\end{equation*}

\begin{lemma}\label{LemmaRosas}
Let $s\geq 1$ and $\pi \in \Theta_s$, let $\textup{Disj}(\pi,m)$ be the number of set partitions $\eta \in \Theta_s$ such that $\pi \wedge \eta =\hat{0}$ and $|\eta|=m$. Let $r\geq 1$ and $\lambda=(\lambda_1\geq\dots\geq \lambda_r)$ be a partition of $s$. Let $\pi_\lambda$ be the set partitions with blocks $\{\lambda_1+\dots+\lambda_{i-1}+1,\dots,\lambda_1+\dots+\lambda_{i}\}$ for $i=1,\dots,r$. Then
\begin{equation}\label{EQ1Comb}
(x\ff \lambda_1)(x\ff \lambda_2)\cdots (x\ff \lambda_r) = \sum_{m \geq 1}\textup{Disj}(\pi_\lambda,m) (x\ff m).
\end{equation}
and
\begin{equation}\label{EQ2Comb}
(x\rf \lambda_1)(x\rf \lambda_2)\cdots (x\rf \lambda_r) = \sum_{m \geq 1} (-1)^{s-m} \textup{Disj}(\pi_\lambda,m) (x\rf m).
\end{equation}
\end{lemma}
\begin{proof}
We will use Lemma \ref{LemmaBijection} with $s=1$, $n_1=n$, $\lambda_j=k^j_1$ and $B_j=\{1,\dots,\lambda_j\}$. Instead of taking a disjoint union it is enough to shift each $B_j$ into $\{\lambda_1+\dots+\lambda_{j-1}+1,\dots,\lambda_1+\dots+\lambda_{j}\}$ for $j=1,\dots,r$. In this setting, we have $(\vec{n} \ff \vec{k}^j)=(n \ff \lambda_j)$ and for $\eta \in \Theta_K$, the corresponding value of $\kappa_1=|\eta|$, hence $g_\eta(\vec{n})=(n\ff|\eta|)$. We obtain identity (\ref{EQ1Comb}) for integer values of $x$,
\begin{equation*}
        (n\ff \lambda_1)\cdots (n\ff \lambda_r) = \sum_{\substack{\eta \in \Theta_K,\\ \eta \wedge \pi = \hat{0}}} (n\ff|\eta|)=\sum_{m \geq 1} \sum_{\substack{\eta \in \Theta_K,\\ \eta \wedge \pi = \hat{0}, |\eta|=m}} (n\ff m)= \sum_{m \geq 1}\textup{Disj}(\pi_\lambda,k) (n \ff m).
\end{equation*}
Since the polynomials $\textup{pol}_1(x)=(x\ff \lambda_1)\cdots (x\ff \lambda_r)$ and $\text{pol}_2(x)=\sum_{k \geq 0}\textup{Disj}(\pi_\lambda,k) (x \ff k)$ coincide over the integer, we conclude that $\textup{pol}_1(x)=\textup{pol}_2(x)$, which proves identity (\ref{EQ1Comb}). Additionally, since $(-x\ff n)=(-1)^{n} (x \rf n)$, then identity (\ref{EQ1Comb}) implies identity (\ref{EQ2Comb}). 
\end{proof}

\begin{remark}
This result was already present in the literature, in fact, it can also be proven using the principal specialization on some linear relations between MacMahon symmetric function as shown in \cite{Ros}, Lemma \ref{LemmaBijection} provides a short combinatorial proof.
\end{remark}

\begin{remark}
We can understand these combinatorial identities as basis expansions for products of polynomials in the falling factorial and rising factorial basis in the ring of polynomials. Furthermore, we can repeat the idea used in the proof of Lemma \ref{LemmaRosas} to generalize the identity from Lemma \ref{LemmaBijection} to a multivariate polynomial decomposition. 
\end{remark}

\begin{corollary} \label{CorollaryCombinatoricsCumulant}
With the same notations as in Lemma \ref{LemmaRosas} with the additional hypothesis that $r \geq 2$, we have that
\begin{equation*}
\sum_{m \geq 1} (-1)^{m-1}(m-1)!  \textup{Disj}(\pi_\lambda,m) =0.
\end{equation*}
\end{corollary}
\begin{proof}
First divide by $x$ on both sides of identity (\ref{EQ2Comb}) to obtain
\begin{equation*}
(x+1)\cdots(x+\lambda_1-1)(x\rf \lambda_2)\cdots (x\rf \lambda_r) = \sum_{m \geq 1} (-1)^{s-m} \textup{Disj}(\pi_\lambda,m) (x+1)\cdots(x+m-1).
\end{equation*}
Now set $x=0$, since $r \geq 2$ the left-hand side is $0$. 
\end{proof}

\begin{remark}
The previous identity with $\pi_\lambda = \hat{0}$ corresponds to the formula $\sum_{\pi \in \Theta_r} (-1)^{|\pi|-1}(|\pi|-1)!=0$ for $r\geq 2$, which is directly obtained from Möbius inversion on the partition lattice, we can interpret our statement as a generalization of this identity. 
\end{remark}

\begin{proof}[Proof of Lemma \ref{ExpansionofDefFcumulants2}]
We follow the strategy sketched at the beginning of this subsection, that is, given some integer $r$, for each $\pi \in \Theta_r$ we construct a sequence of functions $f_{\pi}$ inductively as follows. 
\begin{equation*}
f_{\hat{0}}(\vec{n})= \prod_{j=1}^r (\vec{n}\ff\vec{k}^j) \hspace{2mm} \textup{and} \hspace{2mm} f_\pi(\vec{n})= \prod_{B \in \pi} (\vec{n}\ff\sum_{j \in B} \vec{k}^j)-\sum_{\theta<\pi} f_\theta(\vec{n}).
\end{equation*}
In particular we have $\prod_{B \in \pi} (\vec{n}\ff\sum_{j \in B} \vec{k}^j)=\sum_{\theta\leq\pi} f_\theta(\vec{n})$ for all $\pi \in \Theta_r$. We will use this last identity to expand the falling cumulant into a sum of classical cumulants. We have that
\begin{align*}
\kappa^{F[\vec{k},\vec{n}]}_{\rho,r}(\sigma_1,\dots,\sigma_r) &=\sum_{\pi \in \Theta_r} (-1)^{|\pi|-1} (|\pi|-1)! \prod_{B \in \pi} (\vec{n}\ff\sum_{j \in B} \vec{k}^j)M_\rho\big( \prod_{j\in B} \sigma_{j}\big)\\
&=\sum_{\pi \in \Theta_r} (-1)^{|\pi|-1} (|\pi|-1)! \Big( \sum_{\theta\leq\pi} f_\theta(\vec{n}) M_\rho\big( \prod_{j\in B} \sigma_{j}\big) \Big)\\
&=\sum_{\theta \in \Theta_r}  f_\theta(\vec{n}) \sum_{\theta\leq\pi} (-1)^{|\pi|-1} (|\pi|-1)! M_\rho\big( \prod_{j\in B} \sigma_{j}\big)\\
&=\sum_{\theta \in \Theta_r} f_\theta(n) \kappa_{\rho,|\theta|}(\prod_{j \in B} \sigma_j:B \in \theta).
\end{align*}
We still need to verify that if $n_i\leq n$ for $i=1,\dots,s$, then $f_\pi(\vec{n})=O(n^{\beta})$ with $\beta=\sum_{j=1}^r \sum_{i=1}^s k^j_i +|\pi|-r$. Start by using Möbius inversion on the partition lattice to obtain a closed formula for $f_\pi$. We get that
\begin{equation*}
f_\pi(\vec{n})=\sum_{\theta\leq \pi} \mu(\theta,\pi) \prod_{B \in \theta} (\vec{n}\ff\sum_{j \in B} \vec{k}^j).
\end{equation*}
Here $\mu(\theta,\pi)$ denotes the Möbius function on the interval $[\theta,\pi]$. We now use Lemma \ref{LemmaBijection} to expand the products on the right-hand side. Define $K$, $B_j$ and $\Theta_K$ as in Lemma \ref{LemmaBijection} and for each $\theta=\{D_1,\dots,D_{|\theta|}\} \in \Theta_r$ define the set partition $\pi_\theta \in \Theta_K$ with blocks $\sqcup_{j \in D_i} B_j$ for $i=1,\dots,|\theta|$. This gives the expansion 
\begin{equation*}
    f_\pi(\vec{n})=\sum_{\theta\leq \pi}\mu(\theta,\pi)  \Big( \sum_{\substack{\eta \in \Theta_K,\\ \eta \wedge \pi_\theta=\hat{0}}} g_\eta(\vec{n}) \Big)=\sum_{\eta \in \Theta_K} g_\eta(\vec{n})  \Big(\sum_{\substack{\theta\leq \pi,\\ \eta \wedge \pi_\theta=\hat{0}}} \mu(\theta,\pi)  \Big).
\end{equation*}

Notice that $g_\eta(\vec{n})=O(n^{|\eta|})$. Define the map $\textup{Proj}: \Theta_K \to \Theta_r$ as follows, for each $\eta \in \Theta_K$, we will have $x$ and $x'$ are in the same block in the set partition $\textup{Proj}(\eta)$ if and only if there is some $y \in B_x$ and $y' \in B_{x'}$ such that $y$ and $y'$ are the in the same block in the set partition $\eta$. Crucially, we have that $|\eta| \leq \sum_{j=1}^r \sum_{i=1}^s k^j_i+|\textup{Proj}(\eta)|-r$ and that $\pi_\theta \wedge \eta =\hat{0}$ if and only if $\theta \wedge \textup{Proj}(\eta)=\hat{0}$. This allows us to rewrite the last sum as
\begin{equation}\label{EQcentralCumulant}
f_\pi(\vec{n}) =\sum_{\phi \in \Theta_r} \Big( \sum_{\eta \in \textup{Proj}^{-1}(\phi)}g_\eta(\vec{n}) \Big) \sum_{\substack{\theta\leq \pi,\\ \phi \wedge \theta=\hat{0}}} \mu(\theta,\pi). 
\end{equation}
At this point we need to compute the Möbius function $\mu(\theta,\pi)$. Notice that for $\pi=\{D_1,\dots,D_{|\pi|}\}$ the lattice of partitions $\{\theta:\theta \leq \pi\}$ is isomorphic to $\prod_{i=1}^{|\pi|} \Theta_{|D_i|}$. Denote $\phi_i \in \Theta_{|D_i|}$ the partitions with blocks $D_i\cap B$ for each $B \in \phi$. For each $\theta\leq \pi$ denote $(\theta_1,\dots,\theta_{|\pi|})$ it's image in $\prod_{i=1}^{|\pi|} \Theta_{|D_i|}$. At this point we will assume that $|\phi|>|\pi|$, in this situation we have that $\phi \wedge \theta=\hat{0}$ if and only if $\phi_i \wedge \theta_i = \hat{0}$ for each $1\leq i \leq |\pi|$. This gives
\begin{align*}
    \sum_{\substack{\theta\leq \pi,\\ \phi \wedge \theta=\hat{0}}} \mu(\theta,\pi)&=\sum_{\substack{\theta\leq \pi,\\ \phi \wedge \theta=\hat{0}}} \prod_{i=1}^{|\pi|}\Big((-1)^{|\theta_i|-1}(|\theta_i|-1)!\Big)\\ 
    &=\prod_{i=1}^{|\pi|} \Big( \sum_{\substack{\theta \in \Theta_{|D_i|},\\ \phi_i \wedge \theta = \hat{0} }}   (-1)^{|\theta|-1} (|\theta|-1)!\Big)=\prod_{i=1}^{|\pi|} \Big( \sum_{m \geq 1} (-1)^{m-1} (m-1)! \textup{Disj}(\phi_i,m)\Big).
\end{align*}
Moreover, since $|\phi|\leq \sum_{i=1}^{|\pi|} |\phi_i|$, if $|\phi|>|\pi|$ then there is one $1\leq i \leq |\pi|$ such that $|\phi_i|\geq 2$. Hence Corollary \ref{CorollaryCombinatoricsCumulant} applied to our last identity implies that if $|\phi|>|\pi|$, then  
\begin{equation} \label{EQcentralCumulant2}
    \sum_{\substack{\theta\leq \pi,\\ \phi \wedge \theta=\hat{0}}} \mu(\theta,\pi) =0.
\end{equation}
Since $\sum_{\eta \in \textup{Proj}^{-1}(\phi)}g_\eta(\vec{n}) =O(n^{\beta})$ for $\beta=\sum_{j=1}^r \sum_{i=1}^s k^j_i+|\phi|-r$, equation (\ref{EQcentralCumulant2}) jointly with  equation (\ref{EQcentralCumulant}) ensure that $f_\pi(\vec{n})=O(n^{\beta})$ for $\beta=\sum_{j=1}^r \sum_{i=1}^s k^j_i+|\pi|-r$.
\end{proof}

\section{Proofs of applications}\label{SectionProofsOfApplications}

\subsection{Computation of covariance}\label{PreliminariesSection7}

We continue the analysis of the relationship between continuous Young diagrams and their transition measures started in section \ref{SubsectionCoordinateSystems}. Here we compute the expectations and covariances between the random moments of the respective measures. Our purpose being to obtain trackable formulas which will allow to uncover the conditional Gaussian free field structure. We start by stating an explicit relation between the random variables of interest. 

\begin{proposition}\label{PropPowerHomogeneous}
Denote by $X_k=\int_\R x^k dm_K[\mu](x)$ the random variables corresponding to moments of the transition measure and by $Y_k=\int_\R x^{k-1}d\sigma[\mu](x)$ the random variables corresponding to moments of the continuous Young Diagrams. We have the following relations.
\begin{equation*}
    Y_k=-\sum_{\substack{s_1,\dots,s_k\geq 0\\s_1+2s_2+\dots+ks_k=k}} \frac{(s_1+s_2+\dots+s_k-1)!}{s_1!s_2!\cdots s_k!}\prod_{i=1}^k(-X_i)^{s_i}.
\end{equation*}
\end{proposition}
\begin{proof}
This identity follows from the expansion formula for power sums polynomials in terms of complete homogeneous symmetric polynomials. This relationship is established in \cite[section 2.5]{Ke}.
\end{proof}

Proposition \ref{PropPowerHomogeneous} highlights the non-linear relation between these random variables. While a direct treatment of the computation of both expectations and covariances are possible using that statement, we will avoid dealing explicitly with those combinatorics by using Lagrange inversion. We will make use of the following lemma.

\begin{lemma}\label{LemmaUnicityAuxFunction}
Let $F(z)\in \R[[z]]$ and $\textup{Ext}(z,w)\in \R[[z,w]]$. Suppose that for all $k,\,k'\geq 0$ we have
\begin{equation*}
    [z^{-1}w^{-1}]\bigg[\Big(z^{-1}+zF(z)\Big)^{k}\Big(w^{-1}+wF(w)\Big)^{k'} \textup{Ext}(z,w)\bigg]=0.
\end{equation*}
Then $\textup{Ext}(z,w)=0$.
\end{lemma}
\begin{proof}
Denote $\textup{Ext}(z,w)=\sum_{i,j\geq 0} g_{i,j} z^iw^j$, we prove by induction on $m=i+j$ that $g_{i,j}=0$. If $m=0$ it is enough to notice that
\begin{equation*}
    [z^{-1}w^{-1}]\bigg[\Big(z^{-1}+zF(z)\Big)\Big(w^{-1}+wF(w)\Big) \textup{Ext}(z,w)\bigg]=g_{0,0}=0.
\end{equation*}
Now suppose that we have proven that $g_{i,j}=0$ for all $i+j\leq m$. Let $i+j=m+1$, we want to prove that $g_{i,j}=0$. We note that for any $k\geq0$ we have that 
\begin{equation*}
    \Big(z^{-1}+zF(z)\Big)^{k}=\sum_{r=0}^\infty C_{F,k,r}z^{-k+2r}
\end{equation*}
where $C_{F,k,r}$ are real coefficients such that $C_{F,k,0}\neq 0$. By choosing $k=i+1$ and $k'=j+1$ the hypothesis gives
\begin{equation*}
    [z^{-1}w^{-1}]\bigg[\Big(z^{-1}+zF(z)\Big)^{i+1}\Big(w^{-1}+wF(w)\Big)^{j+1} \textup{Ext}(z,w)\bigg]=\sum_{\substack{r\geq 0\\ r'\geq 0}}^\infty C_{F,i+1,r} C_{F,j+1,r'} d_{i-2r,j-2r'}=0.
\end{equation*}
A direct consequence of the inductive hypothesis is that $d_{i,j}=0$, which concludes the proof. 
\end{proof}

\begin{lemma}\label{CovarianceNewCoordiantes}
Let $\rho_n$ be a sequence of LLN-appropriate probability measures on $\mathbb{Y}_n$. Let $Y^{\alpha}_{k}=\frac{1}{\sqrt{n}}\int_\R t^k d\sigma[\lambda^\alpha](t)$, then 
\begin{equation*}
    \lim_{n \to \infty} \E[Y^{\alpha}_{k}] = [z^{-1}]\frac{-1}{(k+1)z}\Big(\frac{\alpha}{z}+zF(z)\Big)^{k+1}.
\end{equation*}
Furthermore, if $\rho_n$ is CLT-appropriate, then
\begin{align*}
\lim_{n \to \infty} n\textup{Cov}(Y^{\alpha}_{k},Y^{\alpha'}_{k'}) &= [z^{-1}w^{-1}] \bigg[\Big(\alpha z^{-1}+zF_\rho(z)\Big)^{k}\Big(\alpha'w^{-1}+wF_\rho(w)\Big)^{k'}  
\Big(\frac{1}{zw}Q_\rho(z,w)\\
& \hspace{0cm}- \partial_z \partial_w\big( \frac{zw}{\max(\alpha,\alpha')}F_\rho(z)F_\rho(w)+\ln(1 -zw\frac{zF_\rho(z)-wF_\rho(w)}{\max(\alpha,\alpha')(z-w)})\big)\Big)\bigg].
\end{align*}
\end{lemma}
\begin{proof}
Start by proving the first identity, which we use to prove the second identity. We restate \ref{TheoremMultilevelLLN} as 
\begin{equation}\label{EQrestatementMultilevelLLN}
    a_k^\alpha=[z^{-1}] \frac{1}{k+1} \big( z^{-1}+\alpha zF_\rho(\alpha z)\big)^{k+1}.
\end{equation}
    Denoting $R(z)=\alpha zF_\rho(\alpha z)$ and $S(z)=\lim_{n \to \infty} \int_\R \frac{1}{z-t} m_K[\lambda^\alpha](dx)=\sum_{k=0}^\infty a_k^{\alpha} z^{k+1}$, equation (\ref{EQrestatementMultilevelLLN}) ensures that $R(z)$ is the Voiculescu $R$-transform of the limiting measure $\lim_{n\to \infty} m_K[\lambda^\alpha]$, in particular we have the relation $S^{-1}(z)=1/z+R(z)$. 
Now denote $f(z)=S(1/z)$ and $G(z)=1+zR(z)$, we then have that $f(z)=zG(f(z))$. Furthermore, for $y^{\alpha}_k=\lim_{n \to \infty} \E[Y^{\alpha}_{k}]$, we can restate the Markov--Krein correspondence as $\sum_{k=0}^\infty -y^{\alpha}_k z^{k+1}=\ln(f(z)/z)$.
At this point we use a limit case of the Lagrange--Bürmann formula \cite[Equation (2.2.9)]{Gessel16}, which ensures that $(n+1)[z^{n+1}]\ln(f(z)/z)=[z^{n}]H'(z)G(z)^n$, it follows that
\begin{equation*}
   -y^{\alpha}_k=[z^{-1}]\frac{1}{(k+1)z}\Big(\frac{1}{z}+z\alpha F_\rho(\alpha z)\Big)^{k+1}=[z^{-1}]\frac{1}{(k+1)z}\Big(\frac{\alpha}{z}+z F_\rho( z)\Big)^{k+1}.
\end{equation*}
Which conclude the proof of the first identity. For the second identity we will need to repeat this argument by describing the covariance in two different manners. First denote  $S_n(z)= \int_\R \frac{1}{z-t} m_K[\lambda^\alpha](dx)=\sum_{k=0}^\infty X_k^{\alpha} z^{k+1}$ and $zF_n^\alpha(z)=S^{-1}_n(z)-1/z$, so that $\lim_{n\to \infty} z F_n(z)=\alpha z F_\rho( \alpha z)$. The standard Lagrange inversion formula \cite[5.4.3 Corollary]{Stanley} ensures that
\begin{equation}\label{EQ2ProofRestMultil}
X_k^{\alpha}=[z^{-1}]\frac{1}{k+1} \big( z^{-1}+ zF_n^\alpha(z)\big)^{k+1}.
\end{equation}
It is direct from the definition of the covariance that
\begin{equation}\label{EQ3ProofRestMultil}
\begin{split}
b_{k,k'}^{\alpha,\alpha'} =\lim_{n \to \infty} n\big(\E[X_k^{\alpha} X_{k'}^{\alpha'}]-\E[X_k^{\alpha}]\E[X_{k'}^{\alpha'}]\big).
\end{split}
\end{equation}
Hence an application of the identity from equation (\ref{EQ2ProofRestMultil}) on equation (\ref{EQ3ProofRestMultil}) will give that
\begin{equation}\label{EQ4ProofRestMultil}
\begin{split}
b_{k,k'}^{\alpha,\alpha'} =[z^{-1}w^{-1}]\lim_{n \to \infty} \frac{n}{(k+1)(k'+1)}\E[ \big( z^{-1}+ zF_n^\alpha(z)\big)^{k+1}  \big( w^{-1}+ wF_n^{\alpha'}(w)\big)^{k'+1}]\\-\E[ \big( z^{-1}+ zF_n^\alpha(z)\big)^{k+1}]\E[\big( w^{-1}+ wF_n^{\alpha'}(w)\big)^{k'+1}]\big).
\end{split}
\end{equation}
It is enough to study what happens for each monomial on $z$ and $w$. Denote $\phi_i$ the random variable corresponding to the $i$th free cumulant of the random measure $m_K[\lambda^\alpha]$ and similarly  $\psi_i $ the random variable corresponding to the $i$th free cumulant of the random measure $m_K[\lambda^{\alpha'}]$. Then equation (\ref{EQ4ProofRestMultil}) can be restated as
\begin{equation}\label{EQ4.5ProofRestMultil}
\begin{split}
b_{k,k'}^{\alpha,\alpha'} =
\sum_{\substack{s_{-1},\dots,s_{k}\geq 0\\ s_{-1}+s_1+\dots+s_k=k+1\\s_1+2s_2+\dots+ks_k=s_{-1}-1}} \sum_{\substack{t_{-1},\dots,t_{k'}\geq 0\\ t_{-1}+t_1+\dots+t_{k'}=k'+1\\t_1+2t_2+\dots+k't_{k'}=t_{-1}-1}} 
\frac{1}{(k+1)(k'+1)}\binom{k+1}{s_{-1},s_1,\dots,s_k}\\\binom{k'+1}{t_{-1},t_1,\dots,t_{k'}} \lim_{n \to \infty}n \textup{Cov}(\prod_{i=1}^k \phi_i^{s_i},\prod_{i=1}^{k'} \psi_i^{t_i}).
\end{split}
\end{equation}

A direct application of Lemma \ref{Cumulantsofproducts} will give that
\begin{equation}\label{EQ4.6ProofRestMultil}
\lim_{n \to \infty}n \textup{Cov}(\prod_{i=1}^k \phi_i^{s_i},\prod_{i=1}^{k'} \psi_i^{t_i})= \lim_{n\to \infty}\prod_{i=1}^k \E[\phi_i]^{s_i}\prod_{j=1}^{k'} \E[\psi_j]^{t_j}\sum_{\substack{1\leq i \leq k\\ 1\leq j \leq k'}} s_i t_j \frac{n\textup{Cov}(\phi_i,\psi_j)}{\E[\phi_i] \E[\psi_j]}.
\end{equation}
We already know that $\lim_{n \to \infty} \E[\phi_i]=\alpha^i c_i$ while $\lim_{n \to \infty} \E[\psi_j]=\alpha^j c_j$, denoting $\textup{Aux}(z,w)=\sum_{i,j\geq 1}\lim_{n\to \infty} n\textup{Cov}(\phi_i,\psi_j) z^i w^j$, in a similar process on what we did to compute the covariance formula in the proof of Theorem \ref{TheoremMultilevelCLT}, we can rewrite equation (\ref{EQ4.5ProofRestMultil}), using equation (\ref{EQ4.6ProofRestMultil}) as
\begin{equation}\label{EQ5ProofRestMultil}
b_{k,k'}^{\alpha,\alpha'} =[z^{-1}w^{-1}]\Big(z^{-1}+ z\alpha F_\rho( \alpha z)\Big)^k\Big(w^{-1}+ w\alpha F_\rho( \alpha w)\Big)^{k'}\textup{Aux}(z,w). 
\end{equation}

On the other hand, we use function $H(z,w)$ as in equation (\ref{EQ11.5CLT}) to restate Theorem \ref{TheoremMultilevelCLT}  as 
\begin{equation}\label{EQ6ProofRestMultil}
    b_{k,k'}^{\alpha,\alpha'}=[z^{-1}w^{-1}]\Big(z^{-1}+ z\alpha F_\rho( \alpha z)\Big)^k\Big(w^{-1}+ w\alpha F_\rho( \alpha w)\Big)^{k'}H(z,w).
\end{equation}
We now use Lemma \ref{LemmaUnicityAuxFunction} with $\textup{Ext}=\textup{Aux}(z,w)-H(z,w)$. Equation (\ref{EQ5ProofRestMultil}) and equation (\ref{EQ6ProofRestMultil}) ensure that the hypothesis is satisfied, we conclude that $\textup{Aux}(z,w)=H(z,w)$. Finally, using the Lagrange--Bürmann formula \cite[Equation (2.2.9)]{Gessel16} gives
\begin{equation*}
\begin{split}
\lim_{n \to \infty} n\textup{Cov}(Y^{\alpha}_{k},Y^{\alpha'}_{k'})=[z^{-1}w^{-1}]\frac{1}{zw}\Big(z^{-1}+ z\alpha F_\rho( \alpha z)\Big)^k\Big(w^{-1}+ w\alpha F_\rho( \alpha w)\Big)^{k'}H(z,w). \qedhere
\end{split}
\end{equation*}
\end{proof}

\begin{example}\label{ExampleCovPlancherelContYD}
Following Example \ref{ExampleLLNandCLTPlancherel}, denote $\rho_n$ the Plancherel distribution on $\mathbb{Y}_n$. We then have that the sequence of random variables $\sqrt{n}(Y_{k}-\E Y_{k})$ converge, as $n \to \infty$, to a Gaussian vector with covariance given by  Lemma \ref{CovarianceNewCoordiantes}, that is,
\begin{align*}
    \lim_{n \to \infty} n\textup{Cov}(Y_{k},Y_{k'})&=[z^{-1}w^{-1}] \bigg[ \Big(z^{-1}+z\Big)^{k}\Big(w^{-1}+w\Big)^{k'}
\Big(-\partial_z \partial_w\ln(1 -zw\big)- 1\Big)\bigg]\\
&=\sum_{i=2}^{\infty} i\binom{k}{\frac{k-i}{2}}\binom{k'}{\frac{k'-i}{2}}.
\end{align*}
Where the binomials vanish when the indices $\frac{k-i}{2}$ or $\frac{k'-i}{2}$ are distinct from $0,\dots,k$ or $0,\dots,k'$ respectively. By using the inversion formulas for the Chebyshev polynomials of the second kind, it is a short computation to verify that this coincides with Kerov's central limit theorem for Young diagrams. See \cite[Theorem 7.1]{IO}. 

The conditioning of the GFF is already noticeable at this stage: the summation formula for the covariance starts at the index $i=2$ rather than the index $i=1$ for which the GFF occurs. See \cite{Bo} for a similar computation in the random matrix case where GFF fluctuations are obtained.
\end{example}

\subsection{Proof of Theorem  \ref{ThmCGFFforPlancherel}} \label{SectionProofofThmCGFFPlancherel}

We briefly remind our reader of the notation used to denote the moments of the fluctuations of the height function.  
\begin{equation*}
    \mathcal{M}_{\alpha,k}^{\textup{P}}=\sqrt{\pi}\int_{-\infty}^{+\infty} u^k \big[\textup{H}(\sqrt{n}u,\lfloor \alpha n\rfloor)-\E \textup{H}(\sqrt{n}u,\lfloor \alpha n\rfloor)\big]\,du.
\end{equation*}

\begin{proof}[Proof of Theorem \ref{ThmCGFFforPlancherel}]

Start by doing the change of variables $\sqrt{n}u\to x$ and integrating by parts to get
\begin{equation*}
    \mathcal{M}_{\alpha,k}^{\textup{P}}= \frac{\sqrt{\pi}}{n^{\frac{k+1}{2}}(k+1)}\bigg( \int_{-\infty}^{+\infty} x^{k+1} \,\sigma[\lambda^{\lfloor \alpha n\rfloor}](dx)-\E\int_{-\infty}^{+\infty} x^{k+1} \,\sigma[\lambda^{\lfloor \alpha n\rfloor}](dx)\bigg).
\end{equation*}

Theorem \ref{TheoremMultilevelCLT} ensures that $\big(\mathcal{M}_{\alpha,k}^{\textup{P}}\big)_{\alpha,k}$ converge to a centered Gaussian process with covariance given by Lemma \ref{CovarianceNewCoordiantes}. Denote $\textup{PCov}^{\alpha,\alpha'}_{k,k'}=\lim_{n\to \infty} \textup{Cov}(\mathcal{M}_{\alpha,k}^{\textup{P}},\mathcal{M}_{\alpha',k'}^{\textup{P}})$, since $Q(z,w)=0$ and $F_\rho(z)=1$, we have that
\begin{align*}
\textup{PCov}^{\alpha,\alpha'}_{k,k'}&=[z^{-1}w^{-1}] \bigg[\pi\frac{(\alpha z^{-1}+z)^{k+1}}{k+1} \frac{(\alpha'w^{-1}+w)^{k'+1}}{k'+1}  
\Big(- \partial_z \partial_w\big( \frac{zw}{\max(\alpha,\alpha')}\\
&\hspace{8.5cm}+\ln(1 -\frac{zw}{\max(\alpha,\alpha')})\big)\Big)\bigg]\\
&=\frac{\pi}{(2\pi i)^2}\oint\limits_{\substack{|z|^2=\alpha}} \oint\limits_{\substack{|w|^2=\alpha'}}\frac{(\alpha z^{-1}+z)^{k+1}}{k+1} \frac{(\alpha'w^{-1}+w)^{k'+1}}{k'+1}  
\Big(- \partial_z \partial_w\big( \frac{zw}{\max(\alpha,\alpha')}\\
&\hspace{8cm}+\ln(1 -\frac{zw}{\max(\alpha,\alpha')})\big)\Big)\, dz\, dw.
\end{align*}
We further integrate by parts the last expression to obtain
\begin{equation*}
\begin{split}
\textup{PCov}^{\alpha,\alpha'}_{k,k'}=\frac{\pi}{(2\pi i)^2}\oint\limits_{\substack{|z|^2=\alpha}} \oint\limits_{\substack{|w|^2=\alpha'}} x(z)^{k}x(w)^{k'}
\Big(\frac{-zw}{\max(\alpha,\alpha')}-\ln(1-\frac{zw}{\max(\alpha,\alpha')})\Big) \hspace{1cm}\\ \cdot \frac{dx(z)}{dz} \frac{dx(w)}{dw} \, dz\, dw.
\end{split}
\end{equation*}
By using the identity 
\begin{equation*}
2\ln\Biggr| \frac{\max(\alpha,\alpha')-zw}{\max(\alpha,\alpha')-z\bar{w}}\Biggr|=\ln\Bigg(\frac{\big(1-zw/\max(\alpha,\alpha')\big)\big(1-\bar{z}\bar{w}/\max(\alpha,\alpha')\big)}{\big(1-z\bar{w}/\max(\alpha,\alpha')\big)\big(1-\bar{z}w/\max(\alpha,\alpha')\big)}\Bigg),\end{equation*}
we notice that the quantity $\textup{PCov}^{\alpha,\alpha'}_{k,k'}$ is in fact equal to

\begin{equation*}
\begin{split}
\textup{PCov}^{\alpha,\alpha'}_{k,k'}=\frac{\pi}{(2\pi i )^2}\oint\limits_{\substack{|z|^2=\alpha,\\ \Im(z)>0}} \oint\limits_{\substack{|w|^2=\alpha',\\ \Im(w)>0}} x(z)^{k}x(w)^{k'}
\Big(-2\ln\Biggr| \frac{\max(\alpha,\alpha')-zw}{\max(\alpha,\alpha')-z\bar{w}}\Biggr|\\
-L_1(z,w)\Big) \frac{dx(z)}{dz} \frac{dx(w)}{dw} \, dz\, dw,
\end{split}
\end{equation*}
where 
\begin{align*}
    L_1(z,w)&=\frac{zw}{\max(\alpha,\alpha')}+\frac{\bar{z}\bar{w}}{\max(\alpha,\alpha')}-\frac{z\bar{w}}{\max(\alpha,\alpha')}-\frac{\bar{z}w}{\max(\alpha,\alpha')}\\
    &=\frac{1}{\max(\alpha,\alpha')}\big(z-\bar{z}\big)\big(w-\bar{w}\big)= -\frac{4\Im(z)\Im(w)}{\max(|z|^2,|w|^2)}.
\end{align*}
Finally, since $|z|^2=\alpha$ and $|w|^2=\alpha'$, we have that
\begin{equation*}
\ln\Bigr| \frac{\max(\alpha,\alpha')-zw}{\max(\alpha,\alpha')-z\bar{w}}\Bigr|=-\ln\Bigr| \frac{z-w}{z-\bar{w}}\Bigr| \hspace{2mm} \textup{ and } \hspace{2mm} L_1(z,w)=-4\Im(\frac{1}{z})\Im(\frac{1}{w})\min(\alpha,\alpha'),    
\end{equation*}
which further simplify the covariance formula into
\begin{equation*}
\begin{split}
\textup{PCov}^{\alpha,\alpha'}_{k,k'}=\oint\limits_{\substack{|z|^2=\alpha,\\ \Im(z)>0}} \oint\limits_{\substack{|w|^2=\alpha',\\ \Im(w)>0}} x(z)^{k}x(w)^{k'}
\Bigg[\frac{-1}{2\pi}\ln\Biggr| \frac{z-w}{z-\bar{w}}\Biggr|
-\frac{\min(\alpha,\alpha')}{\pi}\Im(\frac{1}{z})\Im(\frac{1}{w})\Bigg]\hspace{1cm}\\ \frac{dx(z)}{dz} \frac{dx(w)}{dw} \, dz\, dw.\qedhere
\end{split}
\end{equation*}
\end{proof}

\subsection{Proof of Theorem \ref{ThmCGFFforSinfty}}\label{SectionProofThmCGFFforSinfty}
Let $0<\alpha<1$ and $y \in \R$, we are interested in finding a solution $z\in \mathbb{H}$ to the equation $\frac{\alpha}{z}+zF(z)=y$. We start by informally showing how this is achieved. If we assume that $zF(z)$ is the $R$-transform of some probability measure $m$, by denoting the respective Stieltjes transform by $C_m(z)$, we can rewrite this equation as
\begin{equation}\label{EquationDiffeo}
    y=z+\frac{\alpha-1}{C_m(z)}
\end{equation}
Since $S_m(z)$ is a biholomorphism from $\mathbb{H}$ to $-\mathbb{H}$, then Lemma \ref{LemmaDiffeo} automatically provides us with a solution for the equation $\frac{\alpha}{z}+zF(z)=y$. Notice that the following values satisfy equation (\ref{EquationDiffeo}),
\begin{equation*}\label{SolutionsEquationDiffeo}
    \alpha(z)=1+\frac{(z-\bar{z})C_m(\bar{z})C_m(z)}{C_m(z)-C_m(\bar{z})}\hspace{2mm} \textup{ and } \hspace{2mm}x(z)=z+\frac{(z-\bar{z})C_m(\bar{z})}{C_m(z)-C_m(\bar{z})}.
\end{equation*}
We can then define the map $(y,\alpha):\mathbb{H}\to\R\times\R$. Denote $D_m \subseteq \R\times \R$ the image of this map.  We will need the following two technical lemmas before starting with the proof of Proposition \ref{PropInverseSinfty}.
\begin{lemma}\label{LemmaDiffeo}
Let $\alpha$ and $y$ be defined as in equation (\ref{SolutionsEquationDiffeo}). We have the following. 
\begin{itemize}
    \item Assume that $m$ is a probability measure with compact support and density $\leq 1$ with respect to the Lebesgue measure. Then the map $z\to \big(y(z),\alpha(z)\big)$ is a diffeomorphism between $\mathbb{H}$ and $D_m$.
    \item Fix $(y,\alpha)\in \R\times[0,1]$, consider equation (\ref{EquationDiffeo}), then this equation has either $0$ or $1$ root in $\mathbb{H}$. Moreover, there is a root in $\mathbb{H}$ if and only if $(y,\alpha) \in D_m$, and if we put into correspondence to the pair  $(y,\alpha) \in D_m$ the root from $\mathbb{H}$ we obtain the inverse of the map $z\to (y(z),\alpha(z))$.
\end{itemize}
\end{lemma}
\begin{proof}
This is a deformation of \cite[Proposition 3.13]{BuG2}. In that situation they consider the maps
\begin{equation*}
    \hat{\alpha}(z)=1+\frac{(z-\bar{z})\Big(\exp\big(C_m(\bar{z})\big)-1\Big)\Big(\exp\big(C_m(z)\big)-1\Big)}{\exp\big(C_m(z)\big)-\exp\big(C_m(\bar{z})\big)}
\end{equation*}
and
\begin{equation*}
    \hat{y}(z)=z+\frac{(z-\bar{z})\Big(\exp\big(C_m(\bar{z})\big)-1\Big)\exp\big(C_m(z)\big)}{\exp\big(C_m(z)\big)-\exp\big(C_m(\bar{z})\big)},
\end{equation*}
which are solution to the equation 
\begin{equation}\label{QuasiEquationSolutions}
\hat{y}=z+\frac{\hat{\alpha}-1}{\exp\big(C_m(z)\big)-1},
\end{equation}
We use that result with the probability measure $m_K=\tfrac{1}{K}m+(\tfrac{1}{K}-\tfrac{1}{K^2})\lambda_{[0,K]}$. It follows that the Stieltjes transform of $m_K$ is
\begin{equation*}
    C_{m_K}(z)=(\tfrac{1}{K}-\tfrac{1}{K^2})\big(\ln(z-K)-\ln(z)\big)+\tfrac{1}{K}C_m(z),
\end{equation*}
hence that $\exp\big(C_{m_K}(z)\big)=1+\tfrac{1}{K} C_m(z)+o(1/K)$ and finally that $\lim_{K\to \infty} \hat{\alpha}(z)=\alpha(z)$ and $\lim_{K\to \infty} \hat{y}(z)=y(z)$. Similarly, in the limit $K\to \infty$, equation (\ref{QuasiEquationSolutions}) becomes equation (\ref{EquationDiffeo}). Finally, to verify that the statement of our lemma is preserved under the limit, we simply note that the convergence is uniform when restricted to compact domains and similarly for the derivative. This concludes the proof of the lemma. 
\end{proof}

Let's compute the Young generating function of this model. It follows from the definition of the extreme characters (see section \ref{SubsectionExtremeCharacters}) that 
\begin{equation*}
    \textup{A}_{\rho_n}(\vec{x})\approx \exp(x_1)\prod_{k=2}^\infty\bigg(\sum_{j=0}^\infty \Big[\sum_{i=1}^\infty \alpha_i(n)^k +(-1)^{k-1}\sum_{i=1}^\infty \beta_i(n)^k\Big]^j n^{\frac{j(k-1)}{2}} \frac{x_k^j}{j!}  \bigg),
\end{equation*}
A direct application of the definition of the logarithm over $\R[\vec{x}]$ is that
\begin{equation*}
    \ln \textup{A}_{\rho_n}(\vec{x}) \approx x_1+ \sum_{k=2}^\infty \bigg[\sum_{i=1}^\infty \alpha_i(n)^k +(-1)^{k-1}\sum_{i=1}^\infty \beta_i(n)^k\bigg] n^{\frac{(k-1)}{2}} x_k.
\end{equation*}

Hence, we automatically have that for any $r\geq 2$ and $i_1,\dots,i_r$ integers, 

\begin{equation*}
\lim_{n\to\infty} \partial_{i_1}\cdots \partial_{i_r} \ln \textup{A}_{\rho_n}(\vec{x})=0, 
\end{equation*}
and similarly $\lim_{n\to\infty} \partial_1 \ln \textup{A}_{\rho_n}(\vec{x})=c_1=1$. While for $i\geq 2$, it follows by definition of the sequences $\alpha(n)$ and $\beta(n)$ that
\begin{align*}
    \lim_{n\to\infty} \partial_i \ln \textup{A}_{\rho_n}(\vec{x})&=\lim_{n\to \infty} n^{\frac{(k-1)}{2}}\bigg[\sum_{i=1}^\infty \alpha_i(n)^k +(-1)^{k-1}\sum_{i=1}^\infty \beta_i(n)^k\bigg]\\
    &=\int_\R x^{k+1} \,\mathcal{A}(dx)+(-1)^{k+1}\int_\R x^{k+1} \,\mathcal{B}(dx).
\end{align*}

This proves that the sequence of probability measures $(\rho_n)_{n\in\N}$ is CLT-appropriate.  Moreover, the coefficients $c_k$ in the expansion of $F(z)$ are given by $c_k=\int_\R x^{k+1} \,\mathcal{A}(dx)+(-1)^{k+1}\int_\R x^{k+1} \,\mathcal{B}(dx)$ for $k\geq 2$ and $c_1=1$. We conclude that
\begin{equation}\label{EquationlimitFSinfty}
\begin{split}
zF(z)=\frac{1}{z^2}C_\mathcal{A}\Big(\frac{1}{z}\Big)+\frac{1}{z^2}C_\mathcal{B}\Big(\frac{-1}{z}\Big)-\Big(\mathcal{A}(\R)+\mathcal{B}(\R)\Big)z+\Big(\int_\R x \mathcal{B}(dx)-\int_\R x \mathcal{A}(dx)\Big)\\+\Big(1-\int_\R x^2 \mathcal{A}(dx)-\int_\R x^2 \mathcal{B}(dx)\Big)z.
\end{split}
\end{equation}

\begin{remark}
As already discussed in section \ref{SubsectionExtremeCharacters}, when $\alpha_i(n)=\beta_i(n)=0$ for all $i\geq 1$ we recover the Plancherel distribution, in this case we have $\mathcal{A}=\mathcal{B}=0$ and equation (\ref{EquationlimitFSinfty}) gives $zF(z)=z$, which coincides with the computation done in Example \ref{ExampleArhoPlancherel}.
\end{remark}

Note that Lemma \ref{LemmaDiffeo} does not immediately solve equation (\ref{EquationDiffeo}) in this situation as it is not clear which measure $m$ to use. Fortunately, we can bypass the issue as follows. 

\begin{lemma}\label{LemmamuKforF}
There exists a sequence of probability measures $\mu_K$ with bounded by $1$ densities with respect to the Lebesgue measure on $\R$ such that theirs Stieltjes transforms satisfy:
\begin{equation*}
C_{\mu_K}(z)=\frac{1}{z}+\frac{1}{Kz^3}F(\frac{1}{z})+o(1/K).
\end{equation*}
\end{lemma}
\begin{proof}
We need to analyze each term on equation (\ref{EquationlimitFSinfty}). Notice that we can rewrite this equation as
\begin{equation}\label{NewEquationlimitFSinfty}
\begin{split}
\frac{1}{z^3}F(z)=C_\mathcal{A}(z)-\frac{\mathcal{A}(\R)}{z}+C_\mathcal{B}(-z)-\frac{\mathcal{B}(\R)}{z}+\Big(\int_\R x \mathcal{B}(dx)-\int_\R x \mathcal{A}(dx)\Big)\\
+\Big(1-\int_\R x^2 \mathcal{A}(dx)-\int_\R x^2 \mathcal{B}(dx)\Big)\frac{1}{z^3}.
\end{split}
\end{equation}

We now construct a measure $\mu_K$ by adding different measures with Stieltjes transforms corresponding to the terms of equation (\ref{NewEquationlimitFSinfty}), denote $\delta_x$ the Dirac measure on $x\in \R$. First notice that the term $\frac{1}{z}$ corresponds to the Stieltjes transform of the Dirac measure on $0$.  The term $C_\mathcal{A}(z)-\frac{\mathcal{A}(\R)}{z}$ corresponds to the Stieltjes transform of the union between the rescaled by $1/K$ measure $\mathcal{A}$ and, the rescaled rescaled Dirac $-\big(\mathcal{A}(\R)/K\big)\delta_0$ Dirac measure on $0$. 

We proceed similarly for the term $C_\mathcal{B}(-z)-\frac{\mathcal{B}(\R)}{z}$ with the difference that to obtain $C_\mathcal{B}(-z)$ we compute the Stieltjes transform  of the reflected at $0$ measure of $\mathcal{B}$, that is the measure $\textup{Refl}(\mathcal{B})(M)=\mathcal{B}(\textup{Refl(M)})$ where $M \subseteq\R$ is a measurable set and $\textup{Refl}(M)=\{-x:x\in M\}$. 

Note that if $x\sim 0$ the Stieltjes transform of $y \delta_x$ is equal to
\begin{equation*}
    \frac{y}{z-x}=\frac{y}{z}\cdot\frac{1}{1-x/z}=\frac{y}{z}+\frac{x}{z^2}+\frac{x^2}{z^3}+O(x^3).
\end{equation*}
Using this we can describe the last few terms, Denote the constant $\delta = \big(1-\int_\R x^2 \mathcal{A}(dx)-\int_\R x^2 \mathcal{B}(dx)\big)$, to obtain $\frac{\delta}{Kz^3}+o(1/K)$ we take the unions of the measures $\frac{\delta}{2\sqrt{K}}\delta_{K^{-1/4}}$, $\frac{\delta}{2\sqrt{K}}\delta_{-K^{-1/4}}$ and $\frac{-\delta}{\sqrt{K}}\delta_0$. Finally, if $x\sim 0$, the Stieltjes transform of a uniform measure on the interval $[0,x]$ with weight $y$, is equal to
\begin{equation*}
    y\ln(z-x)-y\ln(z)=y-yx+O(x^2).
\end{equation*}

Hence, for the term containing $\delta' = \big(\int_\R x \mathcal{B}(dx)-\int_\R x \mathcal{A}(dx)\big)$ we take the union of the uniform measure on the interval $[0,1/K]$ rescaled by $\delta'/K$ with the Dirac measure $\frac{-\delta'}{K^2}\delta_0$. From which we obtain $\frac{\delta'}{K}+o(1/K)$. 

Note that in the previous construction all negative measures are placed on the Dirac measure over $0$, hence for large $K$ the union of all these measures is a positive measure. Furthermore, beside the initial $\delta_0$ at each step we added a measures have total weight $0$ on the real line, hence their union is actually a probability measure. 
\end{proof}

\begin{remark}
Note that the proofs of Lemma \ref{LemmamuKforF} closely follows \cite[Lemma 9.5]{BuG2}, where the role of the uniform distribution on the interval $[0,1]$ is replaced by a Dirac measure concentrated at $0$.
\end{remark}

\begin{proof}[Proof of Proposition \ref{PropInverseSinfty}]
On one hand equation \ref{EquationDiffeo} gives $C_m(z)=\frac{\alpha-1}{y-z}$. By writing $\alpha=s_F/K$ and $y=y_F/K$, then as $K\to \infty$ we have that
\begin{equation}\label{EQ1ProofPropInverseSinf}
    C_m(z)=\frac{1}{z}-\frac{s_F}{zK}+\frac{y_F}{z^2K}.
\end{equation}
On another hand, Lemma \ref{LemmamuKforF} gives
\begin{equation}\label{EQ2ProofPropInverseSinf}
    C_{\mu_K}(z)=\frac{1}{z}+\frac{1}{Kz^3}F(\frac{1}{z})+o(1/K).
\end{equation}

Putting together equations (\ref{EQ1ProofPropInverseSinf}) with equation \ref{EQ2ProofPropInverseSinf} and solving for $Y$ gives
\begin{equation*}
    y_F=zs_F+\frac{1}{z}F\big(\frac{1}{z}\big)+o(1)
\end{equation*}
By making a further change of variables $z\to \frac{1}{z}$ we obtain $y_F=\frac{s_F}{z}+zF(z)+o(1)$, which we notice is almost the equation we want to solve. By remarking that the number of zeros of this function inside $\mathbb{H}$ cannot increase in the limit as $K \to \infty$, we have that for any pair $(y_F,s_F)$ the equation $y_F=\frac{s_F}{z}+zF(z)$ has at most one solution in $\mathbb{H}$. Furthermore, at least one solution is provided explicitly in the proof of Lemma \ref{LemmaDiffeo}. Taking the limit as $K\to \infty$, using that these solutions are analytic and converge uniformly on compact sets inside $\mathbb{H}$, allow to conclude that a solution exists. Furthermore, an application of the implicit function theorem ensures the differentiability of the map $D_F$.
\end{proof}

\begin{proof}[Proof of Theorem \ref{ThmCGFFforSinfty}]
The first few steps are as in the proof of Theorem \ref{ThmCGFFforPlancherel}, that is, Theorem \ref{TheoremMultilevelCLT} ensures that $\big(\mathcal{M}_{\alpha,k}^{\textup{C}}\big)_{\alpha,k}$ converge to a centered Gaussian process with covariance given by Lemma \ref{CovarianceNewCoordiantes}. Denote $\textup{SCov}^{\alpha,\alpha'}_{k,k'}=\lim_{n\to \infty} \textup{Cov}(\mathcal{M}_{\alpha,k}^{\textup{S}},\mathcal{M}_{\alpha',k'}^{\textup{S}})$, since $Q(z,w)=0$, we have that
\begin{align}
\textup{HCov}^{\alpha,\alpha'}_{k,k'}&=[z^{-1}w^{-1}] \bigg[\pi\frac{(\alpha z^{-1}+zF(z))^{k+1}}{k+1} \frac{(\alpha'w^{-1}+wF(w))^{k'+1}}{k'+1}  
\Big(- \partial_z \partial_w\big( \frac{zF(z)wF(w)}{\max(\alpha,\alpha')}\nonumber \\
&\hspace{7.5cm}+\ln(1 -zw\frac{zF(z)-wF(w)}{\max(\alpha,\alpha')(z-w)})\big)\Big)\bigg] \nonumber\\
&=\frac{\pi}{(2\pi i)^2}\oint\limits_{\substack{s_F(z)=\alpha}} \oint\limits_{\substack{s_F(w)=\alpha'}}\frac{(\alpha z^{-1}+zF(z))^{k+1}}{k+1} \frac{(\alpha'w^{-1}+wF(w))^{k'+1}}{k'+1} \label{EQ1CGFFforSinfty} \\
&\hspace{2cm}\cdot\Big(- \partial_z \partial_w\big( \frac{zF(z)wF(w)}{\max(\alpha,\alpha')}
+\ln(1 -zw\frac{zF(z)-wF(w)}{\max(\alpha,\alpha')(z-w)})\big)\Big)\, dz\, dw.\nonumber
\end{align}
Crucially, Proposition \ref{PropInverseSinfty} ensures that integrating over the contours $s_F(z)=\alpha$ and $s_F(w)=\alpha'$ will uniquely extract the coefficient of $z^{-1}w^{-1}$ in the term inside the integral. We can now proceed to work the expression obtained on equation (\ref{EQ1CGFFforSinfty}) to obtain the conditional GFF covariance. 

Using the identity $\partial_z\partial_w\ln(zw/\max(\alpha,\alpha'))=0$ we have that
\begin{equation}\label{EQ2CGFFforSinfty}
\begin{split}
-\partial_z\partial_w\ln(1 -zw\frac{zF(z)-wF(w)}{\max(\alpha,\alpha')(z-w)})=\partial_z\partial_w\Big[\ln(z-w)-\ln\Big(\big(\frac{\max(\alpha,\alpha')}{w}\hspace{1cm}\\+wF(w)\big)-\big(\frac{\max(\alpha,\alpha')}{z}+zF(z)\big)\Big)\Big].
\end{split}
\end{equation}
Similarly, we can write
\begin{equation}\label{EQ3CGFFforSinfty}
\begin{split}
-\partial_z\partial_w\frac{zF(z)wF(w)}{\max(\alpha,\alpha')}=\partial_z\partial_w\bigg[\Big(\frac{\alpha'}{w}+wF(w)\Big)\frac{\alpha}{z\max(\alpha,\alpha')}-\Big(\frac{\alpha}{z}+zF(z)\Big)\hspace{1cm}\\ \cdot\frac{wF(w)}{\max(\alpha,\alpha')}-\frac{\alpha\alpha'}{zw\max(\alpha,\alpha')}\bigg].
\end{split}
\end{equation}

Finally, Proposition \ref{PropInverseSinfty} ensures that $\frac{\alpha}{z}+zF(z)$ is real, $\frac{\alpha'}{w}+wF(w)$ is real and that either $\frac{\max(\alpha,\alpha')}{z}+zF(z)$ or $\frac{\max(\alpha,\alpha')}{w}+wF(w)$ are real. These facts jointly with identities (\ref{EQ2CGFFforSinfty}) and (\ref{EQ3CGFFforSinfty}) ensure that the right hand side term on equation (\ref{EQ1CGFFforSinfty}) can be simplified into 
\begin{equation*}
\begin{split}
\textup{SCov}^{\alpha,\alpha'}_{k,k'}=\oint\limits_{\substack{s_F(z)=\alpha,\\ \Im(z)>0}} \oint\limits_{\substack{s_F(w)=\alpha',\\ \Im(w)>0}} y_F(z)^{k}y_F(w)^{k'}
\Bigg[\frac{-1}{2\pi}\ln\Biggr| \frac{z-w}{z-\bar{w}}\Biggr|
-\frac{\min(\alpha,\alpha')}{\pi}\Im(\frac{1}{z})\Im(\frac{1}{w})\Bigg]\hspace{-5mm}\\ \frac{dy_F(z)}{dz} \frac{dy_F(w)}{dw} \, dz\, dw.\qedhere
\end{split}
\end{equation*}
\end{proof}

\subsection{Proof of Theorem \ref{ThmCGFFforFixedShape}}

The proofs in this section closely follow the ones in the previous section with a few modifications that we explain here.

\begin{proof}[Proof of Proposition \ref{PropInverseFixedShape}]
We proceed by composing different diffeomorphisms. First, the hypothesis provides the existence of a limiting shape $\omega$, and the Markov--Krein transform ensures the existence of the Stieltjes transform $C(z)$ for the transition measure of the continuous Young diagram $\omega$. Now Lemma \ref{LemmaDiffeo} ensures the existence of a pair of a pair of solutions $z\to (y(z),\hat{\alpha}(z))$ to the equation
\begin{equation*}
    z+\frac{\alpha-1}{C(1/z)}=y.
\end{equation*}
We compose this solution to conclude the statement from the proposition. Start by noticing that $1/z$ is holomorphic outside the real line and inverts the upper and lowers plane. Similarly,  $C(z)$ is holomorphic outside the real line, moreover it has an inverse $\tilde{z}=C^{-1}(z)=z^{-1}+zF(z)$ where $F(z)$ is given by Theorem \ref{TheoremLLN}, it inverts the upper and lowers plane. By composing these transformations we get $\hat{\alpha}(1/z)=\frac{1}{1-\hat{s}(z)}$ where $\hat{s}(z)$ is the solution to the equation from the proposition. 
\end{proof}

\begin{proof}[Proof of Theorem \ref{ThmCGFFforFixedShape}]
Since this sequence of partitions converges deterministically to a limit shape, we automatically have that $\rho$ satisfies a CLT with $b_{k,k'}=0$ for all $k,k'$, hence Theorem \ref{TheoremCLT} ensures that $\rho$ is CLT-appropriate, that is, there exist $F_\rho(z)$ and $Q_\rho(z,w)$ such that
\begin{align*}
b_{k,k'} &= [z^{-1}w^{-1}] \bigg[ \Big(z^{-1}+zF_\rho(z)\Big)^{k}\Big(w^{-1}+wF_\rho(w)\Big)^{k'}
\Big(Q_\rho(z,w)\\
& \hspace{3.5cm}- zw \partial_z \partial_w\big(zwF_\rho(z)F_\rho(w)+\ln(1 -zw\frac{zF_\rho(z)-wF_\rho(w)}{z-w})\big)\Big)\bigg]=0,
\end{align*}
For all $k,k'\geq 0$. Hence Lemma \ref{LemmaUnicityAuxFunction} guarantees
\begin{equation*}
Q_\rho(z,w)=zw \partial_z \partial_w\big(zwF_\rho(z)F_\rho(w)+\ln(1 -zw\frac{zF_\rho(z)-wF_\rho(w)}{z-w})\big).
\end{equation*}

We now use Theorem \ref{TheoremMultilevelCLT} ensures that $\big(\mathcal{M}_{\alpha,k}^{\textup{F}}\big)_{\alpha,k}$ converge to a centered Gaussian process with covariance given by Lemma \ref{CovarianceNewCoordiantes}. Denote $\textup{FCov}^{\alpha,\alpha'}_{k,k'}=\lim_{n\to \infty} \textup{Cov}(\mathcal{M}_{\alpha,k}^{\textup{F}},\mathcal{M}_{\alpha',k'}^{\textup{F}})$, we have that 

\begin{equation}\label{Equation1FCOv}
\begin{split}
\textup{FCov}^{\alpha,\alpha'}_{k,k'}=[(zw)^{-1}]\Bigg[\Big(\frac{\alpha}{z}+zF_\rho(z)\Big)^{k'}\Big(\frac{\alpha'}{w}+wF_\rho(w)\Big)^{k'}\bigg( \partial_z \partial_w\big(zwF_\rho(z)F_\rho(w)\hspace{1cm}\\
+\ln(1 -zw\frac{zF_\rho(z)-wF_\rho(w)}{z-w})\big)- \partial_z \partial_w\big(\frac{zwF_\rho(z)F_\rho(w)}{\max(\alpha,\alpha')}\\
+\ln(1 -zw\frac{zF_\rho(z)-wF_\rho(w)}{\max(\alpha,\alpha')(z-w)})\big)\bigg)\Bigg].  
\end{split}
\end{equation}

Similarly to the proof of Theorem \ref{ThmCGFFforSinfty}, we can extract the terms $\ln(z-w)$ from each logarithm and exploit the fact that $\partial_z\partial_w\ln(zw/\max(\alpha,\alpha'))=0$ to rewrite equation (\ref{Equation1FCOv}) as 

\begin{equation}\label{Equation2FCOv}
\begin{split}
\textup{FCov}^{\alpha,\alpha'}_{k,k'}=[(zw)^{-1}]\Bigg[\Big(\frac{\alpha}{z}+zF_\rho(z)\Big)^{k'}\Big(\frac{\alpha'}{w}+wF_\rho(w)\Big)^{k'}\bigg( \partial_z \partial_w\big(zwF_\rho(z)F_\rho(w)\hspace{1cm}\\
\cdot\frac{\max(\alpha,\alpha')-1)}{\max(\alpha,\alpha')}+\ln\big((\frac{1}{w}+wF(w))-(\frac{1}{z}+zF(z))\big)\big)\\-\ln\big((\frac{\max(\alpha,\alpha')}{w}+wF(w))-(\frac{\max(\alpha,\alpha')}{z}+zF(z))\big)\big)\bigg)\Bigg].  
\end{split}
\end{equation}
At this point we can repeat the analysis done in the proof of Theorem \ref{ThmCGFFforSinfty}. Since either $\frac{\max(\alpha,\alpha')}{z}+zF(z)$ or $\frac{\max(\alpha,\alpha')}{w}+wF(w)$ are real, the logarithmic term containing their difference can be ignored. We rewrite equation (\ref{Equation2FCOv}) in integral form as
\begin{equation}\label{Equation3FCOv}
\begin{split}
\textup{FCov}^{\alpha,\alpha'}_{k,k'}=\frac{\pi}{(2\pi i)^2}\oint\limits_{\substack{s_F(z)=\alpha}} \oint\limits_{\substack{s_F(w)=\alpha'}}\frac{(\alpha z^{-1}+zF(z))^{k+1}}{k+1} \frac{(\alpha'w^{-1}+wF(w))^{k'+1}}{k'+1} \\
\cdot\partial_z \partial_w\bigg[zwF_\rho(z)F_\rho(w)\frac{\max(\alpha,\alpha')-1}{\max(\alpha,\alpha')}+\ln\Big(\big(\frac{1}{w}+wF(w)\big)-\big(\frac{1}{z}+zF(z)\big)\Big)\bigg]\, dz\, dw.
\end{split}
\end{equation}
Using the following change of variables
\begin{equation*}
\tilde{z}=z^{-1}+zF(z)\hspace{2mm} \textup{ and }\hspace{2mm}\tilde{w}=w^{-1}+wF(w),
\end{equation*}
we directly get that $z=C(\tilde{z})$ and  $w=C(\tilde{w})$ where $C$ is the Stieltjes transformation of the limit shape transition measure. Hence we can rewrite equation (\ref{Equation3FCOv}) as

\begin{equation}\label{Equation4FCOv}
\begin{split}
\textup{FCov}^{\alpha,\alpha'}_{k,k'}=\frac{\pi}{(2\pi i)^2}\oint\limits_{\substack{\hat{\alpha}(z)=\alpha}} \oint\limits_{\substack{\hat{\alpha}(w)=\alpha'}}\frac{(\tilde{z}+\frac{\alpha-1}{C(\tilde{z})})^{k+1}}{k+1} \frac{(\tilde{w}+\frac{\alpha'-1}{C(\tilde{w})})^{k'+1}}{k'+1}\hspace{1cm}\\ \cdot \partial_{\tilde{z}} \partial_{\tilde{w}}\bigg[(\tilde{z}-\frac{1}{C(\tilde{z})})(\tilde{w}-\frac{1}{C(\tilde{w})})\Big(\frac{\max(\alpha,\alpha')-1}{\max(\alpha,\alpha')}\Big)+\ln\Big(\tilde{w}-\tilde{z}\Big)\bigg]\, d\tilde{z}\, d\tilde{w}.
\end{split}
\end{equation}
Finally, we further change variables $z=\tfrac{1}{\tilde{z}}$, $w=\tfrac{1}{\tilde{w}}$ and we use $\partial_z\partial_w\ln(zw)=0$ to rewrite equation (\ref{Equation4FCOv}) as
\begin{equation}\label{Equation5FCOv}
\begin{split}
\textup{FCov}^{\alpha,\alpha'}_{k,k'}=\frac{\pi}{(2\pi i)^2}\oint\limits_{\substack{\tfrac{1}{1-\hat{s}_F(z)}=\alpha}} \oint\limits_{\substack{\tfrac{1}{1-\hat{s}_F(w)}=\alpha'}}\frac{(\frac{1}{z}+\frac{\alpha-1}{C(\frac{1}{z})})^{k+1}}{k+1} \frac{(\frac{1}{w}+\frac{\alpha'-1}{C(\frac{1}{w})})^{k'+1}}{k'+1}\hspace{1cm}\\ \cdot \partial_{\tilde{z}} \partial_{\tilde{w}}\bigg[\Big(\frac{1}{z}-\frac{1}{C(1/z)}\Big)\Big(\frac{1}{w}-\frac{1}{C(1/w)}\Big)\Big(\frac{\max(\alpha,\alpha')-1}{\max(\alpha,\alpha')}\Big)+\ln(z-w)\bigg]\, dz\, dw.
\end{split}
\end{equation}

Now repeat a similar decomposition from equation (\ref{EQ3CGFFforSinfty}), that is
\begin{align*}
\Big(\frac{1}{z}-\frac{1}{C(1/z)}\Big)\Big(\frac{1}{w}-\frac{1}{C(1/w)}\Big)&=\frac{1}{(\alpha-1)(\alpha'-1)}\Big(\frac{1}{z}+\frac{\alpha-1}{C(1/z)}-\frac{\alpha}{z}\Big)\Big(\frac{1}{w}+\frac{\alpha'-1}{C(1/w)}-\frac{\alpha}{w}\Big)\\
&=\frac{\Big(\frac{1}{z}+\frac{\alpha-1}{C(1/z)}\Big)\Big(\frac{1}{w}+\frac{\alpha'-1}{C(1/w)}\Big)}{(\alpha-1)(\alpha'-1)}-\frac{\alpha'\Big(\frac{1}{z}-\frac{\alpha-1}{C(1/z)}\Big)}{w(\alpha-1)(\alpha'-1)}\\&\hspace{3cm}-\frac{\alpha\Big(\frac{1}{w}+\frac{\alpha'-1}{C(1/w)}\Big)}{z(\alpha-1)(\alpha'-1)}+\frac{\alpha\alpha'}{(1-\alpha)(1-\alpha')zw}.
\end{align*}

Furthermore, notice that
\begin{equation*}
\frac{\alpha\alpha'}{(1-\alpha)(1-\alpha')zw}\big(\frac{\max(\alpha,\alpha')-1}{\max(\alpha,\alpha')}\big)=\frac{\min\big(1-\frac{1}{\alpha},1-\frac{1}{\alpha'}\big)}{zw}.
\end{equation*}
Since both $\frac{1}{z}+\frac{\alpha-1}{C(1/z)}$ and $\frac{1}{w}+\frac{\alpha'-1}{C(1/w)}$ are real, the terms containing them can be ignored. Similarly to the proof of Theorem \ref{ThmCGFFforSinfty}, Proposition \ref{PropInverseFixedShape} ensures that no additional residue appears in the domain inside the contour curve of integration. we finally integrate by parts to get
\begin{equation*}
\begin{split}
\textup{FCov}^{\alpha,\alpha'}_{k,k'}=\oint\limits_{\substack{\hat{s}_F(z)=1-\frac{1}{\alpha},\\ \Im(z)>0}} \oint\limits_{\substack{\hat{s}_F(w)=1-\frac{1}{\alpha'},\\ \Im(w)>0}} y(z)^{k}y(w)^{k'}
\Bigg[\frac{-1}{2\pi}\ln\Biggr| \frac{z-w}{z-\bar{w}}\Biggr|
-\frac{\min(1-\frac{1}{\alpha},1-\frac{1}{\alpha'})}{\pi}\hspace{0cm}\\ \cdot \Im(\frac{1}{z})\Im(\frac{1}{w})\Bigg] \frac{dy(z)}{dz} \frac{dy(w)}{dw} \, dz\, dw.\qedhere
\end{split}
\end{equation*}
\end{proof}

\subsection{Identification of the limit object}

We briefly discuss the conditioning of the Gaussian Free Field and provide a short proof of Proposition \ref{PropConditionalGFF}.

Denote by $\mathcal{H}$ the Hilbert space formed by finite linear combinations of instances of a GFF $\mathfrak{G}$ integrated with respect to real valued polynomials over the curves $\mathcal{C}_\alpha$ with inner product given by its covariance and $K$ a closed sublinear space of $\mathcal{H}$. We say that a Gaussian field $\mathfrak{C}$ is a conditioned GFF with respect to $K$ if there exist an orthogonal projection $\textup{P}:\mathcal{H}\to K$ such that $\mathfrak{C}=\mathfrak{G}-\textup{P}[\mathfrak{G}]$. It is enough to check that their covariances coincide, see \cite[Corollary 1.10]{LG} for a precise statement on how we can interpret this as a conditioning of a Gaussian process. 

In our setting, we are integrating the Gaussian field over curves $\mathcal{C}_{\alpha}$. As already discussed in section \ref{SubsectionGaussianFields}, the conditioning of the generalized field is intrinsic to our height function, independently on the distribution of choice we deterministically have that $\int_\R \textup{H}(x,t)\,dx=t$. This identity automatically implies that $\mathfrak{C}(1)=0$, it is then natural to take this to be closed linear space $K$. 

\begin{proof}[Proof of Proposition \ref{PropConditionalGFF}]
Denote $\mathfrak{G}$ the Gaussian free field on the upper-half plane as described above. Consider the function $\textup{P}:\mathcal{H}\to K$ defined by $\textup{P}[\mathfrak{G}](f)=\mathfrak{G}\big(f-f(0)\big)$. It is clear from the definition that this is a projection and that $\textup{P}[\mathfrak{G}](1)=0$. We are left to check that $\textup{P}$ is an orthogonal projection and that the covariance of $\mathfrak{G}-\textup{P}[\mathfrak{G}]$ coincides with the covariance of $\mathfrak{C}$. We will need to verify both of these claims in each of the models considered in this paper, here we will treat the case of the Plancherel growth process, the other two being equivalent. The idea is to repeat the computations done in section \ref{SectionProofofThmCGFFPlancherel}. To verify that $\textup{P}$ is an orthogonal projection we want to check that $\textup{Cov}\big(\mathfrak{G}(f)-\textup{P}[\mathfrak{G}](f),\textup{P}[\mathfrak{G}](g)\big)=0$ for any $f$ and $g$ as in section \ref{SubsectionGaussianFields}, equivalently, following the notation of section \ref{SectionProofofThmCGFFPlancherel}, we need to check that $\textup{PCov}^{\alpha,\alpha'}_{k,0}=0$ for any $k\geq 1$, that is
\begin{equation*}
    \frac{1}{(2\pi i)^2} \oint\limits_{|z|=1/4} \oint\limits_{|w|=3/4} \Big(\frac{\alpha}{z}+z\Big)^k\Big(\frac{\alpha'}{w}+w\Big)\Big(\frac{\max(\alpha,\alpha')}{(\max(\alpha,\alpha')-zw)^2}-\frac{1}{\max(\alpha,\alpha')}\Big)dzdw=0.
\end{equation*}
Which can be verified by expanding $\Big(\frac{\max(\alpha,\alpha')zw}{(\max(\alpha,\alpha')-zw)^2}-\frac{zw}{\max(\alpha,\alpha')}\Big)$ as a power series in $zw$, it is clear that this integral is always $0$. We now compute the covariance, we have that
\begin{equation*}
    \textup{Cov}\big(\textup{P}[\mathfrak{G}](f),\textup{P}[\mathfrak{G}](g)\big)=\int_{\mathcal{C}_\alpha}\int_{\mathcal{C}_{\alpha'}} \big[f(x)-f(0)\big]\big[g(x')-g(0)\big] \frac{-1}{2\pi}\ln\Biggr| \frac{z-w}{z-\bar{w}}\Biggr|\,dx\,dx'.
\end{equation*}

By comparing with the covariance formula for $\mathfrak{C}$, it is enough to check that
\begin{equation}\label{EquationConvarianceConditioning}
\int_{\mathcal{C}_\alpha}\int_{\mathcal{C}_{\alpha'}}  \frac{-1}{2\pi}\ln\Biggr| \frac{z-w}{z-\bar{w}}\Biggr|\,dx\,dx'=\int_{\mathcal{C}_\alpha}\int_{\mathcal{C}_{\alpha'}}  \frac{\min\big(t(z),t(w)\big)}{\pi}\Im\Big(\frac{1}{z}\Big)\Im\Big(\frac{1}{w}\Big)\,dx\,dx'.
\end{equation}

Repeating the procedure of section \ref{SectionProofofThmCGFFPlancherel}, equation (\ref{EquationConvarianceConditioning}) becomes the trivial identity
\begin{equation*}
\begin{split}
\oint\limits_{|z|=1/4} \oint\limits_{|w|=3/4} \Big(\frac{\alpha}{z}+z\Big)\Big(\frac{\alpha'}{w}+w\Big)\Big(\frac{\max(\alpha,\alpha')}{(\max(\alpha,\alpha')-zw)^2}\Big)dzdw= \oint\limits_{|z|=1/4} \oint\limits_{|w|=3/4} \Big(\frac{\alpha}{z}+z\Big)\\ \cdot\Big(\frac{\alpha'}{w}+w\Big)\Big(\frac{1}{\max(\alpha,\alpha')}\Big)dzdw. \qedhere
\end{split}
\end{equation*}

\end{proof}


\begin{thebibliography}{XXX}


\bibitem[APR08]{APR08} Adin R., Postnikov A., Roichman Y., {\it{Combinatorial Gelfand models}}, J. Algebra {\bf{320}}, (2008), 1311--1325.

\bibitem[Ai79]{Ai} Aigner M., {\it{Combinatorial theory}}, Classics in Mathematics, Springer-Verlag, Berlin,
(1997), Reprint of the 1979 original.

\bibitem[AGZ09]{AGZ} Anderson G., Guionnet A., Zeitouni O., {\it{An Introduction to Random Matrices}}, Cambridge University Press  (2009).

\bibitem[Bi97]{Bi3} Biane P., {\it{Some properties of crossings and partitions}}, Discrete Math.  {\bf{175}}, (1997), 41--53.

\bibitem[Bi98]{Bi} Biane P., {\it{Representations of Symmetric Groups and Free Probability}}, Adv. in Math. {\bf{138}} (1998) 126--181.

\bibitem[Bi01]{Bi2} Biane P., {\it{Approximate factorization and concentration for characters of symmetric groups}}, Int. Math. Res. Not. {\bf{4}}, (2001), 179--192.

\bibitem[BDJ99]{BDJ} Baik J., Deift P., Johansson K., {\it{On the distribution of the length of the longest increasing subsequence of random permutations}}, J. Amer. Math. Soc.  {\bf{12}} (1999) 1119--1178.

\bibitem[BBFM25]{BBFM} Borga J.,  Boutillier C, Féray V., Méliot P.-L., {\it{A determinantal point process approach to scaling and local limits of random Young tableaux}}, Ann. Probab. {\bf{53}} (2025) 299--354.

\bibitem[Bo14]{Bo} Borodin A., {\it{CLT for spectra of submatrices of Wigner Random Matrices}}, Mosc. Math. J. {\bf{14}} (2014) 29--38.

\bibitem[BC14]{BC} Borodin A., Corwin I., {\it{Macdonald processes}}, Probab.
Theory Related Fields, {\bf{158}} (2014) 225--400.

\bibitem[BF14]{BF} Borodin A., Ferrari P., {\it{Anisotropic growth of random surfaces in 2+1 dimensions}}, Communications in Mathematical Physics, {\bf{325}} (2014) 603--684.

\bibitem[BG14]{BoG} Borodin A., Gorin V., {\it{General  $\beta$-Jacobi Corners Process and the Gaussian Free Field}}, Comm. Pure Appl. Math. {\bf{68}} (2014) 1774--1844.

\bibitem[BGG16]{BGG} Borodin A., Gorin V., Guionnet A., {\it{Gaussian asymptotics of discrete $\beta$-ensembles}}, Publ. Math. Inst. Hautes Études Sci. {\bf{125}} (2016) 1--78.

\bibitem[BOO00]{BOkOl} Borodin A., Okounkov A., Olshanski G., {\it{Asymptotics of
Plancherel measures for symmetric groups}}, J. Amer. Math. Soc.  {\bf{13}} (2000) 481--
515.

\bibitem[BO07]{BO07} Borodin A., Olshanski G., {\it{Asymptotics of Plancherel-type random partitions}}, J. Algebra {\bf{313}} (2007) 40--60.

\bibitem[BO16]{BoO} Borodin A., Olshanski G., {\it{Representations of the Infinite Symmetric Group}}, Cambridge University Press (2016).

\bibitem[BG15]{BuG1} Bufetov A., Gorin V., {\it{Representations of classical Lie groups and quantized free convolution}}, Geom. Funct. Anal. {\bf{25}} (2015) 763--814.

\bibitem[BG18]{BuG2} Bufetov A., Gorin V., {\it{Fluctuations of particle systems determined by Schur generating functions}}, Adv. Math.  {\bf{338}} (2018) 702--781.

\bibitem[BG19]{BuG3} Bufetov A., Gorin V., {\it{Fourier transform on high–dimensional unitary groups with applications to random tilings}}, Duke Math. J. {\bf{168}} (2019) 2559--2649.

\bibitem[CST10]{CST} Ceccherini-Silberstein T., Scarabotti F., and Tolli F., {\it{Representation Theory of the Symmetric Groups: The Okounkov-Vershik Approach, Character Formulas, and Partition Algebras}}, Cambridge University Press (2010).

\bibitem[CLW24]{CLW24} Chapuy G., Louf B., Walsh H., {\it{Random partitions under the Plancherel-Hurwitz measure, high genus Hurwitz numbers and maps}}, Ann. Probab. {\bf{52}} (2024) 1225--1252.

\bibitem[CDM25]{CDM} Cuenca C., Dołęga  M., Moll A., {\it{Universality of global asymptotics of Jack-deformed random Young
diagrams at varying temperatures}}, Ann. Probab. (2025) Available at \url{arXiv:2304.04089}.

\bibitem[DF16]{DF} Dołęga M., Féray V., {\it{Gaussian fluctuations of Young diagrams and structure constants of Jack characters}}, Duke Math. J. {\bf{165}} (2016), 1193--1282.

\bibitem[DS19]{DS19} Dołęga M., Śniadi P., {\it{Gaussian fluctuations of Jack--deformed random Young diagrams}}, Probab. Theory Related Fields {\bf{174}} (2019), 133--176.

\bibitem[FM12]{FM12} Féray V., Méliot P.-L., {\it{Asymptotics of q-Plancherel measures}}, Probab. Th. Rel. Fields, {\bf{152}} (2012) 589--624.

\bibitem[Ge16]{Gessel16} Gessel I. M., {\it{Lagrange inversion}},  J. Combin. Theory A {\bf{144}} (2016) 212--249.

\bibitem[Go20]{Gor} Gordenko A., {\it{Limit shapes of large skew young tableaux and a modification of the tasep process}}, Preprint (2020) Available at \url{arxiv.org/pdf/2009.10480}.

\bibitem[Go21]{Gorin} Gorin V., {\it{Lectures on random lozenge tiling}}, Cambridge Univ. Press (2021).

\bibitem[GH24]{GorinHuang} Gorin V., Huang J., {\it{Dynamical loop equation}}, Ann. Probab. {\bf{52}} (2024)  1758--1863 .

\bibitem[GR19]{GR} Gorin V., Rahman M., {\it{Random sorting networks: local statistics via random matrix laws}}, Probab. Theory Related Fields {\bf{175}} (2019) 45--96.

\bibitem[GNW79]{GNW} Greene C., Nijenhuis  A., Wilf H. S., {\it{A probabilistic proof of a formula for the number of Young tableaux of a given shape}}, Adv. Math.  {\bf{31}} (1979) 104--109.


\bibitem[GH19]{GH} Guionnet A., Huang J., {\it{Rigidity and edge universality of discrete $\beta$-ensembles}}, Comm. Pure Appl. Math. {\bf{72}} (2019) 1875--1982.

\bibitem[Ho06]{Ho} Hora A., {\it{Jucys--Murphy element and walks on modified Young graph}}, Quantum Probability
Banach Center Publications {\bf{73}} (2006), 223--235.

\bibitem[Ho16]{Ho2} Hora A., {\it{The limit shape problem for ensembles of Young diagrams}}, Springerbriefs in mathematical physics, Tokyo: Springer (2016).

\bibitem[HO07]{HoO} Hora A., Obata N., {\it{Quantum Probability and Spectral Analysis of Graphs}}, Springer (2007).

\bibitem[Hu21]{Hu21} Huang J., {\it{Law of large numbers and central limit theorems through Jack generating functions}}, Adv. Math. {\bf{380}} (2021) Paper Nº107545.


\bibitem[IO02]{IO} Ivanov, V.,  Olshanski G., {\it{Kerov’s central limit theorem for the Plancherel measure on Young diagrams}}, Symmetric functions 2001: surveys of developments and perspectives, Springer {\bf{74}} (2002) 93--151.

\bibitem[Jo01]{Jo} Johansson K., {\it{Discrete orthogonal polynomial ensembles and the
Plancherel measure}}, Ann. of Math. {\bf{153}} (2001) 259--296.

\bibitem[KP22]{KP22} Kenyon R., Prause I., {\it{Gradient variational problems in}} $\R^2$, Duke Math. J. {\bf{171}} (2022) 3003--3022.

\bibitem[Ke93]{Ke} Kerov S., {\it{Transition Probabilities of Continual Young Diagrams and Markov Moment
Problem}}, Funktsion. Anal. i Prilozhen. {\bf{27}} (1993), no. 2, 32--49; English translation: Funct. Anal. Appl, {\bf{27}} (1993), 104--117.

\bibitem[KV86]{KV} Kerov V., Vershik A., {\it{The Characters of the Infinite Symmetric Group and Probability Properties of the Robinson--Schensted--Knuth Algorithm}} SIAM J. Algebraic Discrete Methods, {\bf{7}} (1986) 116-124.

\bibitem[Kr72]{Kr} Kreweras G., {\it{Sur les partitions non croisees d’un cycle}}, Discr. Math. {\bf{1}} (1972), 333--350.

\bibitem[LG16]{LG} Le Gall J-F., {\it{Brownian Motion, Martingales, and Stochastic Calculus}}, Graduate Texts in Mathematics, Springer International Publishing, {\bf{274}} (2016).

\bibitem[Le04]{Le} Lehner F., {\it{Cumulants in noncommutative probability theory I. Noncommutative exchangeability systems}}, Math. Zeitschriftnts {\bf{248}} (2004) 67--100.

\bibitem[LS59]{LSi} Leonov V. P., Sirjaev A. N. {\it{On a method of semi-invariants}}, Theory Probab. Appl. {\bf{4}} (1959) 319--329.

\bibitem[LS77]{LS} Logan B. F., Shepp L. A. {\it{A variational problem for random Young tableaux}}, Advances in Math. {\bf{26}} (1977) 206--222.

\bibitem[Ma08]{Ma08} Matsumoto S., {\it{Jack Deformations of Plancherel Measures and Traceless Gaussian Random Matrices}}, Electron. J.
Combin. {\bf{15}} (2008).

\bibitem[Mé10a]{Mel10a} Méliot P.-L., {\it{Asymptotics of the Gelfand models of the symmetric groups}}, Preprint (2010) Available at \url{arxiv.org/abs/1009.4047}.

\bibitem[Mé10b]{Mel10b} Méliot P.-L., {\it{Kerov’s central limit theorem for Schur--Weyl measures of parameter}} $\alpha = 1/2$, Preprint (2010) Available at \url{arxiv.org/abs/1009.4034}.

\bibitem[Mé11]{Mel11} Méliot P.-L., {\it{Kerov’s central limit theorem for Schur--Weyl and Gelfand measures (extended abstract)}}, Discrete Math. Theor. Comput. Sci. (2011) 669--680.

\bibitem[Mé17]{Me} Méliot P.-L., {\it{Representation Theory of Symmetric Groups}}, CRC Press (2017).

\bibitem[Mo23]{Mo23} Moll A., {\it{Gaussian Asymptotics of Jack Measures on Partitions from Weighted Enumeration of Ribbon Paths}}, Int. Math. Res. Not. {\bf{3}} (2023) 1801--1881.

\bibitem[NS06]{NS} Nica A., R. Speicher R., {\it{Lectures on the Combinatorics of Free Probability}}, Cambridge
University Press (2006).

\bibitem[Ok00]{O} Okounkov A., {\it{Random matrices and random permutations}} Int. Math. Res. Notices {\bf{2000}} (2000) 1043--1095.

\bibitem[OV96]{OV} Okounkov A.,Vershik A., {\it{A new approach to representation theory of symmetric groups}}, Selecta Math. {\bf{2}} (1996) 581--605.

\bibitem[PT01]{PT} Peccati G., Taqqu M. S., {\it{Wiener Chaos: Moments, Cumulants and Diagrams: A Survey with Computer Implementation}}, Bocconi \& Springer Series. Springer-Verlag Italia Srl, 1st edition (2001).

\bibitem[Pe15]{Pe} Petrov L., {\it{Asymptotics of uniformly random lozenge tilings of polygons}}, Ann. Probab. {\bf{43}} (2015) 1--43.

\bibitem[PR07]{PiR} Pittel B., Romik. D., {\it{Limit shapes for random square Young tableaux}}, Adv. Appl. Math. {\bf{38}} (2007) 164--209.

\bibitem[PW21]{PW} Powell E., Werner. W., {\it{Lecture notes on the Gaussian free field}}, Cours spécialisés, Société Mathématique de France {\bf{28}} (2021).

\bibitem[Pr24]{Pr} Prause I., {\it{Random Young tableaux and the tangent plane method}}, Unpublished manuscript (2024).

\bibitem[Ra25+]{Ra25} Raposo G., {\it{Random matrix models for two-dimensional
Gaussian fields}}, in preparation (2025+).

\bibitem[Ro04]{Rom2} Romik D., {\it{Explicit formulas for hook walks on continual Young diagrams}}, Adv. Appl.
Math. \textbf{32} (2004) 625--654.

\bibitem[Ro15]{Rom} Romik D., {\it{The Surprising Mathematics of Longest
Increasing Subsequences}}, Cambridge University Press, (2015).

\bibitem[Ro02]{Ros} Rosas M. H., {\it{Specializations of MacMahon symmetric functions and the polynomial algebra}}, Discrete Mathematics {\bf{246}} (2002) 285--293.

\bibitem[Śn06a]{Sni3} Śniady P., {\it{Asymptotics of characters of symmetric groups, genus expansion and free probability}}, Discrete Math. {\bf{306}} (2006), 624--665.

\bibitem[Śn06b]{Sni} Śniady P., {\it{Gaussian fluctuations of characters of symmetric
groups and of Young diagrams}}, Probab. Theory Related Fields {\bf{136}} (2006) 263--297.

\bibitem[Śn14]{Sni4} Śniady P., {\it{Robinson--Schensted--Knuth algorithm, jeu de taquin, and Kerov-Vershik measures on infinite tableaux}}, SIAM J. Discrete Math. {\bf{28}} (2014), 598--630.

\bibitem[Śn25]{Sni2} Śniady P., {\it{Modulus of continuity of Kerov transition measure for continual Young diagrams}}, Electron. J. Probab. {\bf{30}} (2025) 1--26.

\bibitem[Sp83]{Spe} Speed T. P., {\it{Cumulants and Partition lattices}}, Australian Journal of Statistics {\bf{25}} (1983) 378--388.

\bibitem[Sp94]{Speicher} Speicher R., {\it{Multiplicative functions on the lattice of non-crossing
partitions and free convolution}}, Math. Ann. {\bf{289}} (1994) 611--628.

\bibitem[St12]{Stanley} Stanley R., {\it{Enumerative Combinatorics: Volume 2}}, Cambridge Studies in Advanced Mathematics $2$nd Edition (2012).

\bibitem[St08]{St08} Strahov E., {\it{A differential model for the deformation of the Plancherel growth process}}, Adv. Math {\bf{217}} (2008) 2625--2663.

\bibitem[Su18]{Su} Sun W., {\it{Dimer model, bead model and standard Young tableaux: finite cases and limit shapes}}, Preprint (2018) Available at \url{arxiv.org/pdf/1804.03414}.

\bibitem[Ve94]{Ve94} Vershik  A. M.,{\it{Asymptotic combinatorics and algebraic analysis}}, Proc. Internat. Congress of Mathematicians, Zurich (1994) 1384--1394.

\bibitem[VK77]{VK} Vershik  A. M., Kerov S. V.,{\it{Asymptotic behavior of the Plancherel measure of the symmetric group and the limit form of Young tableaux}}, Dokl. Akad. Nauk SSSR {\bf{233}} (1977) 1024--1027.

\end{thebibliography}
\end{document}